\documentclass[a4paper]{article}

\usepackage[leqno]{mathtools}
\usepackage{mathrsfs}
\usepackage{interval}
\usepackage{stmaryrd}
\usepackage{fullpage}
\numberwithin{equation}{section}

\usepackage{amsfonts}
\usepackage[widespace]{fourier}
\usepackage{fourier-orns}

\RequirePackage[inline]{enumitem}
\setlist[enumerate]{font=\upshape}
\setlist[enumerate,1]{label=(\arabic*),ref=(\arabic*)}
\setlist[enumerate,2]{label={(\arabic{enumi}.\alph*)},
  ref=(\arabic{enumi}.\alph*)}
\setlist[enumerate,3]{label={(\arabic{enumi}.\alph{enumii}.\roman*)},
  ref=(\arabic{enumi}.\alph{enumii}.\roman*)}

\providecommand*{\transp}{\intercal}

\newcommand{\hull}[1]{\ensuremath{\operatorname{hull}\left(#1\right)}}

\newcommand{\vect}{\ensuremath\operatorname{Vect}}

\mathchardef\mhyphen="2D

\DeclareMathOperator{\alg}{alg}

\DeclareMathOperator{\coeff}{\bf Coeff}

\DeclareMathOperator{\cop}{cop}

\DeclareMathOperator{\Endo}{End}

\DeclareMathOperator{\Hom}{Hom}
\DeclareMathOperator{\id}{\iota}

\DeclareMathOperator{\image}{Image}

\DeclareMathOperator{\inv}{inv}
\DeclareMathOperator{\irr}{Irr}

\DeclareMathOperator{\op}{op}
\DeclareMathOperator{\pol}{Pol}

\DeclareMathOperator{\res}{Res}

\DeclareMathOperator{\supp}{supp}

\DeclarePairedDelimiter{\norm}{\|}{\|}
\DeclarePairedDelimiter{\abs}{\lvert}{\rvert}
\DeclarePairedDelimiterX{\pairing}[2]{\langle}{\rangle}%
{\mskip1mu#1\mskip2mu,\mskip3mu#2\,}
\DeclarePairedDelimiterX\set[1]\lbrace\rbrace{\def\given{\;\delimsize\vert\;}#1}
\DeclarePairedDelimiterX\prosca[1]\lparen\rparen{#1}

\usepackage{amsthm}
\newtheoremstyle{mystyle}%                % Name
{}%                                     % Space above
{}%                                     % Space below
{\itshape}%                                     % Body font
{}%                                     % Indent amount
{\bfseries}%                            % Theorem head font
{.}%                                    % Punctuation after theorem head
{ }%                                    % Space after theorem head, ' ', or \newline
{\thmname{#1}\thmnumber{ #2}\thmnote{ (#3)}}%

\theoremstyle{definition}
\newtheorem{defi}{Definition}[section]

\theoremstyle{mystyle}
\newtheorem{lemm}[defi]{Lemma}
\newtheorem{prop}[defi]{Proposition}
\newtheorem{coro}[defi]{Corollary}
\newtheorem{theo}[defi]{Theorem}
\newtheorem{nota}[defi]{Notation}

\newtheorem{rema}[defi]{Remark}

\usepackage{tikz-cd}

\usepackage{hyperref}
\usepackage{xcolor}
\hypersetup{
  colorlinks,
  linkcolor={gray},
  citecolor={gray},
  urlcolor={blue},
}

\usepackage{amsrefs}
\AtBeginDocument{%
  \def\MR#1{}
}

\newcommand{\C}{\ensuremath{\mathbb{C}}}

\newcommand{\N}{\ensuremath{\mathbb{N}}}

\newcommand{\R}{\ensuremath{\mathbb{R}}}

\mathchardef\myh="2D

\begin{document}

\title{A theory of locally convex Hopf algebras}

\author{Hua Wang}

\date{}

\maketitle

\abstract{Using the completed inductive, projective and injective tensor
  products of Grothendieck for locally convex topological vector spaces, we
  develop a systematic theory of locally convex Hopf algebras with an emphasis
  on Pontryagin-type dualities. We describe how classical Hopf algebras, real
  and complex Lie groups, as well as compact and discrete quantum groups, can
  all give rise to natural examples of this theory in a variety of different
  ways. We also show that the space of all continuous functions on a topological
  group \( G \) whose topological structures are compactly generated has an
  \( \varepsilon \)-Hopf algebra structure, and we can recover \( G \) fully as
  a topological group from this locally convex Hopf algebra. The latter is done
  via a generalization of Gelfand duality, which is of its own interest.
  Certain projective and inductive limits are also considered in this framework,
  and it is shown that how this can lead to examples seemingly outside of the
  framework of locally compact quantum groups in the sense of
  Kustermans-Vaes. As an illustration, we propose a version of the infinite
  quantum permutation group \( S^{+}_{\infty} \), the free orthogonal group
  \( O^{+}_{\infty} \), and the free unitary group \( U^{+}_{\infty} \) as
  certain strict inductive limits, all of which still retain a nice
  duality. Combined with our duality theory, this may be seen as an alternative
  tentative approach to the Kac program of developing a Pontryagin-type duality
  to a wider class, while at the same time, we include many more interesting
  examples of classical and quantum groups.}

\section*{Introduction}
\label{sec:a8b281cedaa9569c}

The aim of this paper is to develop a theory of locally convex Hopf algebras
that formally enjoys many nice properties of classical Hopf algebras, that are
closely related to at least a large portion of classical and quantum groups that
frequently appear in (commutative or not) analysis, and that are flexible
enough to accommodate many new examples.

The main motivation for this work is two-fold. On the one hand, our current
framework of locally compact quantum groups (\cites{MR1951446,MR1832993} and
\cite{MR2020804}) heavily utilizes the machinery in operator algebras. While
this culminates in an amazingly nice theory that concludes the Kac program of
generalizing the Pontryagin dual to all locally compact groups and later certain
quantum groups, it seems that one gets further and further away from Hopf
algebras. For example, the antipode and the counit are often no longer
everywhere defined (even for non-Kac type compact quantum groups); the
multiplication in general won't factor through the tensor product in use (this
seems a general issue even for Kac algebras as in \cite{MR1215933}); and in the
\( C^{\ast} \)\nobreakdash-algebra version, one has to consider multiplier
algebras which often makes the theory even more technically demanding. Thus it
seems that it is a worthwhile goal to try to make connections between locally
compact quantum groups and some sort of topological versions of Hopf algebras
that are already alluded in Drinfeld's famous ICM address \cite{MR934283}. On
the other hand, there are certainly many interesting topological groups that are
not locally compact, so one naturally wonders what are other interesting
topological versions of quantum groups that are not locally compact. As it is
the consensus of the community that the operator algebra approach, at least the
\( C^{\ast} \)\nobreakdash-algebraic version, should describe only the locally
compact quantum groups, this naturally begs the question of how can one develop
a theory of topological Hopf algebras that include at least many interesting
examples of locally compact quantum groups while also include many interesting
non locally compact classical topological groups.

We now elaborate a bit on our motivation and describe briefly some background
and related works. On the Kac program side, there is a long history starting
from the Pontryagin duality of locally compact abelian groups, for which we
refer to the excellent introduction in \cite{MR1832993}, and only mentions the
independent works of Kac and Vainerman \cites{MR0348038,MR0338259} in the Soviet
union, as well as Enock and Schwartz \cites{MR1215933,MR1207551,MR0442710},
resulting in the notion of Kac algebras. Formally, Kac algebras are certain von
Neumann algebras that closely resemble a Hopf algebra (together with invariant
weights playing the roles of the left and right Haar measures on locally compact
groups), and later there are also \( C^{\ast} \)\nobreakdash-algebra
version. They form a category that enjoys the Pontryagin duality, include all
locally compact groups and restricts to classical Pontryagin duality for locally
compact abelian groups, thus provides a very satisfying answer to the Kac
program. However, it turns out later that (e.g.\ \cite{MR901157}) that the
framework of Kac algebras is too restricted to include some then newly
discovered quantum groups. After a lot of efforts, among which we mention
\cites{MR1658585,MR1220906} in which van Daele makes systematic use of the
multiplier algebras to treat certain locally compact non compact cases,
Kustermans \& Vaes first proposed a satisfactory theory \cite{MR1832993} of
locally compact quantum groups. The powerful machinery of operator algebras are
fully employed in this theory, enabling one to establish many beautiful results
which are too extensive to mention here. However, if one can say that Kac
algebras are still somehow related to Hopf algebras, as the theories progress,
one gets further and further away from a Hopf algebra structure, as mentioned
above. On the non locally compact topological quantum group side, one first
natural step of course is to find a good formulation of certain non locally
compact classical groups so that one might recast it successfully and
meaningfully in the non-commutative paradigm. It seems that due to the fact that
we know preciously little on locally convex algebras when compared to our
knowledge of \( C^{\ast} \)\nobreakdash-algebras and von Neumann algebras, there
are relatively few attempts in this direction. We here mention Akbarov's work
based on the notion of stereotype spaces and envelops, see
\cites{MR1346445,MR1794594,MR1965073,MR2143915}, as well as the recent book
\cite{MR4544403}. However, it seems that Akbarov's work focuses more on the
formally nice properties of the category of stereotype spaces, and less on the
various examples that we shall consider in this paper. We also mention that
still using multiplier algebras, Voigt \cite{MR2379773} develops an interesting
theory of bornological (instead of topological) Hopf algebras. Also,
\cite{MR1266072} considers topological Hopf algebras much in our spirit, but
only for nuclear ones that are either of class \( (\mathcal{F}) \) (Fréchet
spaces) or of class \( (\mathcal{DF}) \) (see \S~\ref{sec:2ffdde5596825e91}),
which they termed ``well-behaved''. While the scope of \cite{MR1266072} is
limited for our purposes, it can already be seen there that by introducing the
topological tensor products, one can deal with certain quantum deformation that
the usual algebraic tensor product can not handle.

With our motivation and some background being said, we now briefly describe our
approach. Just as the theory of operator algebras play an important technical
background for the development of locally compact quantum groups, for our
purposes, we naturally considers various algebras and coalgebras structures
based on the highly developed, albeit currently much less active, theory of
locally convex spaces, especially various topological tensor products and the
related approximation properties introduced by Grothendieck in his thesis
\cite{MR0075539}. As we want to use topological tensor products to add more
elements and still be able to use arguments by continuity and density, it is
natural for us to work with locally convex spaces that are complete. We will
easily see (\S~\ref{sec:d078c89fa6b6cafa}) that equipped with some commonly
considered topological tensor products, the category
\( \widehat{\mathsf{LCS}} \) of all complete locally convex spaces admits a
symmetric monoidal category structure, hence provides the natural background for
us to consider locally convex algebras, coalgebra, bialgebras and of course,
Hopf algebras. The latter part of the paper, on the technical side, can be
summarized by saying that it is an application of the powerful general theory of
locally convex spaces, as well as certain well-known explicit function spaces.

We now describe the organization of this paper and summarize certain main
results along the way. As the theory of locally convex spaces as we are going to
use seems not in the toolbox of our typical reader, an unusually lengthy
preliminary section (\S~\ref{sec:63dd23720298d074}) is devoted to nail down the
relevant results, while simultaneously fix some notation. The goal here is to
give a quick summary of the results that play an essential role later, with
explicit references on where these results can be found. At the same time, some
effort has been made to make the account concise and coherent, and hopefully,
more readable. The main work starts with \S~\ref{sec:060ad0d40e57b25a}, in which
we introduce the basic notions and establish some easy results. We point out
that we considered two kinds of duality, the strong one as well as a polar one
(see \S~\ref{sec:9b15e9ed66625847}), while the former is what one usually
expects, the latter sometimes is still applicable while the former fails (see
\S~\ref{sec:7848cae39b1cc5b0}). In \S~\ref{sec:7848cae39b1cc5b0}, we develop
some deeper structural results on duality of locally convex Hopf algebras, and
these results shall be illustrated by examples described in
\S~\ref{sec:2a8ab6bf60a9b3ca}. We highlight that in
\S~\ref{sec:b6aa227ba43c21a5}, it is shown that all classical Hopf algebras over
the field of real or complex numbers fit into our framework, and unlike the
purely algebraic case, also enjoys a nice topological version of duality. We
think that this new point of view might be of interest to some experts working
only with classical Hopf algebras. We also show that all compact and discrete
quantum groups fit nicely into our framework in
\S~\ref{sec:4d1fce794188b9cb}. The attempt to include certain non locally
compact topological groups starts with \S~\ref{sec:7b59ee62a46ef144}, the class
of topological groups we have in mind is the ones whose topology is compactly
generated, and \S~\ref{sec:7b59ee62a46ef144} develops the basics on certain
function spaces on the spaces of compactly generated spaces, to the point that
one can associate a locally convex Hopf algebra with any topological group with
compactly generated topology. In particular, we can now describe not only
locally compact groups, but also all metrizable groups, topological groups with
a CW complex structure etc, using locally convex Hopf algebras. In
\S~\ref{sec:7b59ee62a46ef144}, various types of results are established,
allowing one to recover the underlying topological groups from a locally convex
Hopf algebra under certain mild conditions. It is worth mentioning that we
establish a generalization of Gelfand duality in \S~\ref{sec:e2f8637d41ae033e},
which enables us to recover many topological groups as the character group of
the locally convex Hopf algebra equipped with the topology of compact
convergence. We also point out in \S~\ref{sec:3e2f92039fbda631} that all the
information on the Pontryagin dual of a locally compact abelian group can also
be retrieved merely from the corresponding locally convex Hopf algebra. Hence at
this point, our theory can already be seen as a seemingly interesting alternate
approach to the Kac program. Certain projective and inductive limit
constructions are considered next. The original goal is develop enough of the
theory to be able to describe certain topological quantum groups that seemingly
go beyond the framework of locally compact quantum groups, which is achieved in
the last section \S~\ref{sec:4c8d08ba8bcefe46}. The needed theory is developed
in \S~\ref{sec:7a5cda2c4078b89b}, and it is worth pointing out that after the
theory is developed, one can use a variant of Bruhat's regular functions to give
an alternative description of second countable locally compact groups in our
framework (\S~\ref{sec:de1592686b0dd205}), based on some beautiful structural
results related to Hilbert's fifth problem.

We end this introduction with some convention and notation.

In this paper, unless stated otherwise, we make the following convention about
our terminologies. The scalar field \( \mathbb{K} \) always denotes either
\( \R \) or \( \C \). All compact and locally compact spaces, manifolds, as well
as topological groups, are assumed to be Hausdorff. Smooth manifolds need not be
second countable, but are assumed to be finite dimensional, with each component
of the same dimension, and paracompact. Lie groups, be it over \( \R \) or
\( \C \), are \emph{not} assumed to be second countable.  Similarly, locally
compact groups are \emph{not} assumed to be second-countable or
\( \sigma \)-compact. Locally convex spaces are always assumed to be Hausdorff,
unless it is explicitly mentioned that we are talking about a not necessarily
Hausdorff locally convex space. For basic categorical notions such as projective
and inductive limits, adjoint functors et cetera, we refer to
\cite{MR0354798}. By a projective (resp.\ inductive) system in a category
\( \mathcal{C} \), we mean a system \( (x_{i}, f_{i,j}) \) given by a family of
objects \( (x_{i}) \) in \( \mathcal{C} \) (often denoted by
\( x_{i} \in \mathcal{C} \) later) indexed by a partially ordered set \( I \)
that is directed above in the sense that for all \( i, j \in I \), there exists
\( k \in I \), such that \( i \leq k \) and \( j \leq k \); as well as a family
\( (f_{i,j}: x_{j} \to x_{i}) \) (resp.\ \( f_{i,j}: x_{i} \to x_{j} \)) of
morphisms, one for each \( i, j \in I \) with \( i \leq j \), such that
\( f_{i,i} = \id_{x_{i}} \), and \( f_{i,k} = f_{i,j}f_{j,k} \) (resp.\
\( f_{i,k} = f_{j,k}f_{i,j} \)) whenever \( i \leq j \leq k \). By projective
(resp. inductive) limits, we always mean the limit (resp.\ colimit) with respect
to a projective (resp.\ inductive) system.

\begin{nota}
  \label{nota:b18fef8a2ee12975}
  Unless otherwise stated, the following notation convention are valid
  throughout the paper. For topological vector spaces \( E \), \( F \), the
  symbol \( L(E, F) \) denotes the space of all linear maps from \( E \) to
  \( F \), \( \mathcal{L}(E, F) \) the space of all \emph{continuous} linear
  maps. The symbol \( \odot \) denotes the algebraic tensor product, with
  \( \otimes \) reserved for tensor products of elements or linear
  maps. Topological tensor products on locally convex spaces is always suffixed,
  such as \( \otimes_{\tau} \), with \( \tau \) being a compatible topology. We
  denote the category of locally convex spaces with continuous linear maps as
  morphisms by \( \mathsf{LCS} \).
\end{nota}

\tableofcontents
\addtocontents{toc}{\protect{\pdfbookmark[1]{\contentsname}{toc}}}

\section{Preliminaries on locally convex spaces}
\label{sec:63dd23720298d074}

Since we shall make significant use of the uniform structures in several places,
it seems natural to begin this preliminary \S~\ref{sec:63dd23720298d074} by
briefly summarizing some part of the theory of uniform spaces in
\S~\ref{sec:bb19dab8a74cccc2} and \S~\ref{sec:3a8d43b32f489a0c}, which is
slightly more general than the implicit uniform structure of a locally convex
space when one talks about pre-compactness (equivalent to total boundedness),
completeness, et cetera. The treatment here follows closely
\cite{bourbaki_topologie_2006}*{Chapitres~II \& III} and
\cite{bourbaki_topologie_2006-1}*{Chapitre~X}. In the rest of
\S~\ref{sec:63dd23720298d074}, we give a very cursory account of some aspects of
the theory of locally convex spaces that are seemingly less well-known nowadays,
aiming to briefly lay out the analytic foundations for our subsequent work, so
that the main flow of the presentation in later sections is not disturbed. A
systematic treatment can be found in some excellent books such as
\cite{MR633754}, \cite{MR0077884}, \cite{MR0205028},
\cites{MR0248498,MR0551623}, \cite{MR0342978} and \cite{MR0225131}, and of
course the systematic treatment of topological tensor product by Grothendieck
\cite{MR0075539} as well as the summary of this monumental work
\cite{MR0061754}. Deeper or more specialized results concerning locally convex
spaces that are necessary for our treatment shall be stated later as they come
along. All results will be either given a proof, a sketch of a proof if it is
indeed easy, or an explicit reference in the literature. Due to his own
ignorance, the author does not keep proper attributions, nor does he claim
originality when a proof is given or sketched in the paper. For the purpose of
avoiding unnecessary interrupting of the flow of the text, some special results
that are needed later will not be in this preliminary part, but shall be only
recalled as the situation arises.

To avoid making this already lengthy preliminary subsection overly long, we make
free use of some basic notions and facts about topological and locally convex
spaces that can be found in e.g.\ \cite{MR0342978}*{Chapter~I \& II}, and only
treat the needed topics that are more or less beyond the basic level. In
particular, we assume basic knowledge on Fréchet spaces, or \( (F) \)-spaces, as
in e.g.\ \cite{MR4182424}*{Ch.~7}.

To fix the terminologies, in a vector space \( V \), we say that a subset
\( A \subseteq V \) is \textbf{absolutely convex} if it is convex and circled
(aka. balanced). Each subset \( S \subseteq V \) has an absolutely convex hull,
denoted by \( \Gamma(S) \) of \( \Gamma(S) \), which is the smallest absolutely
convex set containing \( S \) and consists of absolutely convex combinations of
elements in \( S \), i.e.\ elements of the form
\( \sum_{i=1}^{n} \lambda_{i}x_{i} \), with
\( \sum_{i=1}^{n} \abs*{\lambda_{i}} \leq 1 \), where
\( \lambda_{i} \in \mathbb{K} \), \( x_{i} \in S \) for each \( i \).

\begin{nota}
  \label{nota:ee765bb793c0ed9f}
  Let \( E \) be a vector space equipped with a locally convex topology. We use
  \( \mathcal{N}(x) \) to denote the filter of neighborhoods of \( x \in E \),
  and \( \mathcal{N}_{\Gamma}(0) \) the filter basis for \( \mathcal{N}(0) \)
  consisting of all absolutely convex neighborhoods of \( 0 \). We use \( E' \)
  to denote the topological dual of \( E \), and \( E^{\sharp} \) the algebraic
  linear dual.

  We denote the class of all \( (F) \)-spaces by \( (\mathcal{F}) \), and
  similarly for notation such as \( (\mathcal{DF}) \) (\( (DF) \)-spaces),
  \( (\mathcal{M}) \) (\( (M) \)-spaces, or Montel spaces) etc.

  We are going to make a lot of references to Köthe's monographs
  \cites{MR0248498,MR0551623}, and thus caution our reader that we will follow
  him by using \( E'_{c} \), or more generally \( \mathcal{L}_{c}(E, F) \), to
  denote the topology of \emph{precompact} convergence; whereas our other main
  references such as \cite{MR633754} or \cite{MR0342978}, the same symbol stands
  for the topology of \emph{compact} convergence, with the topology of
  precompact convergence denoted by \( E'_{pc} \), or more generally
  \( \mathcal{L}_{pc}(E, F) \).
\end{nota}

\subsection{Uniform spaces}
\label{sec:bb19dab8a74cccc2}

\begin{defi}
  \label{defi:99f3d2bd6f3c8540}
  A \textbf{uniform space} is a couple \( (X, \mathcal{U}) \), where \( X \) is
  a non-empty set, \( \mathcal{U} \) a collection of subsets of
  \( X \times X \), whose elements are called \textbf{entourages}, such that
  \begin{enumerate}
  \item \label{item:7e6e88fd1cb2d444} \( \mathcal{U} \) has the filter property
    in the sense that
    \begin{enumerate}%[label=\textup{(\arabic{enumi}\alph*)}]
    \item \label{item:aad432e8a9b783fe} \( U \in \mathcal{U} \),
      \( U \subseteq W \subseteq X \times X \) implies \( W \in \mathcal{U} \);
    \item \label{item:cbe9726bd06f51fb} \( U, W \in \mathcal{U} \) implies
      \( U \cap W \in \mathcal{U} \);
    \end{enumerate}
  \item \label{item:3b955de6d0375b06} every entourage contains the diagonal \( \Delta_{X} \);
  \item \label{item:e413d7c2654f8fb1} \( U \in \mathcal{U} \) implies
    \( U^{-1} \in \mathcal{U} \), where
    \( U^{-1} = \set*{(y,x) \given (x,y) \in U} \);
  \item \label{item:28711a6e3cd2b2bf} for every \( U \in \mathcal{U} \), there
    exists \( V, W \in \mathcal{U} \), such that \( V \circ W \subseteq U \),
    where
    \begin{equation}
      \label{eq:94ed800a5ddc8f66}
      V \circ W = \set*{(x, z) \given (x,y) \in V, (y,z) \in W}.
    \end{equation}
  \end{enumerate}
  We often simply say that \( X \) is a uniform space if the collection
  \( \mathcal{U} \) of entourages are clear from context, and say that
  \( \mathcal{U} \) is the uniform structure on \( X \).
\end{defi}

It is clear that on a given set \( X \), there is a \textbf{finest} uniform
structure, where the entourage is any subset of \( X \times X \) containing the
diagonal, and a \textbf{coarsest} one, with the diagonal being the only
entourage.

\begin{rema}
  \label{rema:e4a073c717300aab}
  When \( X \) is non-empty, one often specifies a uniform structure
  \( \mathcal{U} \) on it by giving a filter basis for \( \mathcal{U} \), called
  \textbf{a fundamental system of entourages} for \( \mathcal{U} \). For
  example, given a metric space \( (X, d) \), collection of sets of the form
  \( B_{r}(d) = \set*{(x,y) \given d(x, y) < r} \subseteq X \times X \),
  \( r > 0 \), forms a filter basis for a uniform structure on \( X \), called
  \textbf{the uniform structure of the metric} \( d \). We say a uniform
  structure \( \mathcal{U} \) on \( X \) is \textbf{metrizable} if it is the
  uniform structure of some metric \( d \) on \( X \).
\end{rema}

\begin{defi}
  \label{defi:2a27e093cab51057}
  Let \( U \) be an entourage of a uniform space \( X \). For each
  \( x \in X \), define
  \begin{equation}
    \label{eq:4d8cc04d7143f1e0}
    U(x) :=
    \set*{y \in X \given ((x, y) \in U)}.
  \end{equation}
  Then there is a unique topology on
  \( X \), such that \( \mathcal{U}(x) = \set*{U(x) \given U \in \mathcal{U}} \)
  is the filter of neighborhoods of \( x \) for each \( x \in X \), namely, a
  set \( O \subseteq X \) is open if and only if for each \( x \in O \), there
  exists \( U \in \mathcal{U} \), such that \( U(x) \in O \). This topology on
  \( X \) is called \textbf{the topology of the uniform structure}
  \( \mathcal{U} \). Conversely, given a topology \( \tau \) on a set \( X \),
  we say \( \tau \) \textbf{is compatible with a uniform structure}
  \( \mathcal{U} \) on \( X \) if the topology of \( \mathcal{U} \) is
  \( \tau \). When speaking of certain topological properties of a uniform space
  such as being Hausdorff, compact, et cetera, we are referring the topology of
  the given uniform structure.
\end{defi}

\begin{theo}[\cite{bourbaki_topologie_2006}*{Chapitre~II, p5, Proposition~3} and
  \cite{bourbaki_topologie_2006-1}*{\S~IX.2.4}]
  \label{theo:1396f3f298bc2cb3}
  Let \( X \) be a uniform space with the uniform structure \( \mathcal{U} \). Then
  \begin{enumerate}
  \item \label{item:db992ae78b701244} \( X \) is Hausdorff if and only if the
    intersection of all entourages is the diagonal of \( X \);
  \item \label{item:a08aad6b01df60c7} \( (X, \mathcal{U}) \) is metrizable if
    and only if \( X \) is Hausdorff, and \( \mathcal{U} \) admits a countable
    filter basis.
  \end{enumerate}
\end{theo}

\begin{defi}
  \label{defi:d78530a7993833e1}
  Given two uniform spaces \( (X, \mathcal{U}) \) and \( (Y, \mathcal{V}) \), a
  map \( f : X \to Y \) is called \textbf{uniformly continuous}, if
  \( (f \times f)^{-1}(V) \in \mathcal{U} \) for any \( V \in \mathcal{V}
  \).
  The collection of all uniform spaces together with uniformly continuous
  maps among them forms \textbf{the category of uniform spaces} \( \mathsf{Unif} \).
\end{defi}

It is clear that uniformly continuous maps are always continuous with respect to
the topologies of the corresponding uniform structures.

\begin{defi}
  \label{defi:d52915ac6c7610ba}
  A filter \( \mathcal{F} \) on a uniform space \( (X, \mathcal{U}) \) is called
  a \textbf{Cauchy-filter}, if for each entourage \( V \), there exists
  \( M \in \mathcal{F} \) that is \( V \)-\textbf{small} in the sense that
  \( M \times M \subseteq V \). A uniform space is called \textbf{complete} if
  each Cauchy filter on it converges.
\end{defi}

It is well-known that this notion of completeness coincides with the one for
metric spaces, see e.g.\ \cite{bourbaki_topologie_2006-1}*{\S~IX.2}.

\begin{defi}[\cite{bourbaki_topologie_2006}*{\S~II.2.3}]
  \label{defi:dd6ea09f8ac64745}
  \textbf{The initial uniform structure} on a set \( X \) with respect to a
  family \( (f_{i})_{i \in I} \) of maps \( f_{i}: X \to X_{i} \) into uniform
  spaces \( X_{i} \) is the unique uniform structure on \( X \), which always
  exists, that is the coarsest one making all \( f_{i} \) uniformly
  continuous. In particular, unless stated otherwise, the \textbf{subspace}
  \( A \) of a given uniform space \( Y \) is equipped with the initial uniform
  structure with respect to the inclusion \( i : A \hookrightarrow Y \), and the
  \textbf{projective limit} \( \varprojlim Y_{i} \) of a projective system
  \( (Y_{i}, f_{i,j})_{i \in I} \) in \( \mathsf{Unif} \) is the set-theoretic
  projective limit \( L:= \varprojlim Y_{i} \) equipped with the initial uniform
  structure with respect to all the canonical projections
  \( p_{i} : L \to Y_{i} \). The (Cartesian) product of a family of uniform
  spaces, as a uniform space, is defined similarly.
\end{defi}

\begin{prop}[\cite{bourbaki_topologie_2006}*{\S~II.3, Proposition~8,
    Proposition~10 \& Corollaire}]
  \label{prop:903bfccf4a781ceb}
  The following hold:
  \begin{enumerate}
  \item \label{item:d564b3a44ec3bcf0} Closed subspace of a complete uniform
    space is complete. A complete subspace of a Hausdorff uniform space is
    closed.
  \item \label{item:a20beac8bb890c25} The product of a family of non-empty
    uniform space is complete if and only if each of its factor is complete.
  \item \label{item:69c05ffb5eeff953} The projective limit of any projective
    system of complete Hausdorff uniform spaces in \( \mathsf{Unif} \) remains
    complete.
  \end{enumerate}
\end{prop}

\begin{defi}
  \label{defi:da4d0e8821475034}
  A \textbf{separated completion} of \( (X, \mathcal{U}) \) is a pair
  \( (\widehat{X}, i) \), with \( \widehat{X} \) being a \emph{Hausdorff}
  uniform space, and \( i : X \to \widehat{X} \) a map, called \textbf{canonical}, such that
  \begin{enumerate}
  \item \label{item:cdeddbbe6f8034ee} \( i(X) \) is dense in \( \widehat{X} \);
  \item \label{item:14637679aba6429b} the uniform structure on \( X \) is the
    initial one with respect to the map \( i \);
  \item \label{item:36fcb718dcb7a3c4} any uniformly continuous map from \( X \)
    to a complete Hausdorff uniform space \( Y \) factors uniquely through \( i \).
  \end{enumerate}
\end{defi}

\begin{theo}[\cite{bourbaki_topologie_2006}*{\S~II.3.7}]
  \label{theo:5c48561db339a0d3}
  The following hold:
  \begin{enumerate}
  \item \label{item:fe2f257052001631} The separated completion
    \( (\widehat{X}, i) \) of a uniform space \( X \) always exists and is
    unique up to isomorphism in \( \mathsf{Unif} \).
  \item \label{item:197ce1367e450831} If \( X \) is Hausdorff,
    then \( i: X \to i(X) \) is an isomorphism of uniform spaces.
  \item \label{item:d2e7dcd0233a36b2} If \( f : X \to Y \) is a morphism in
    \( \mathsf{Unif} \), and let \( (\widehat{X}, i) \), \( (\widehat{Y}, j) \)
    be the separated completion of \( X \) and \( Y \) respectively, then there
    exists a unique continuous \( \widehat{f}: \widehat{X} \to \widehat{Y} \),
    such that \( j f = \widehat{f} i \). Moreover, this \( \widehat{f} \) is
    uniformly continuous.
  \end{enumerate}
\end{theo}

We now treat the compactness of uniform spaces.

\begin{theo}[\cite{bourbaki_topologie_2006}*{\S~II.4.1, Théorème~1}]
  \label{theo:af81aff906fa7aa4}
  Let \( X \) be a compact space (recall that we've assumed that compact spaces
  are Hausdorff), then there exists a unique uniform structure \( \mathcal{U} \)
  on \( X \) that is compatible with the topology of \( X \). Moreover,
  \( U \in \mathcal{U} \) if and only if \( U \) is a neighborhood of the
  diagonal in \( X \times X \).
\end{theo}

Recall that a set \( A \) in a topological space \( X \) is called
\textbf{relatively compact} if its closure is compact.

\begin{defi}
  \label{defi:8b4e3fc5c7fd7dcf}
  Let \( X \) be a uniform space, \( (\widehat{X}, i) \) its separated
  completion. We say that a set \( A \subseteq X \) is \textbf{precompact}, if
  \( i(A) \) is relatively compact in \( \widehat{X} \). We say
  \( A \subseteq X \) is \textbf{totally bounded}, if for each entourage \( U \)
  of \( X \), there exists finite many points \( a_{1}, \ldots, a_{n} \) in
  \( A \), such that \( A \subseteq \cup_{i=1}^{n}U(a_{i}) \) (see
  \eqref{eq:4d8cc04d7143f1e0} for the notation).
\end{defi}

\begin{prop}[\cite{bourbaki_topologie_2006}*{II.29, Théorème~3}]
  \label{prop:cdb6e313569d3c6e}
  Using the above notation, \( A \) is precompact if and only if it is totally
  bounded.
\end{prop}

We end our general discussion of uniform spaces with a very useful notion.

\begin{defi}
  \label{defi:94460281f200f584}
  Let \( Y \) be a uniform space, \( X \) a topological space (resp.\ uniform
  space), \( H \) a family of maps from \( X \) to \( Y \). We say that \( H \)
  is \textbf{equicontinuous} (resp.\ \textbf{uniformly equicontinuous}), if for
  each open set \( V \) in \( Y \) (resp.\ an entourage \( \mathcal{V} \) of
  \( Y \)), the set \( \cap_{f \in H}f^{-1}(V) \) (resp.\
  \( \cap_{f \in H}(f \times f)^{-1}(\mathcal{V}) \)) is an open set (resp.\ an
  entourage) of \( X \).
\end{defi}

\begin{rema}
  \label{rema:ad31e8284450a2ef}
  It is clear that uniform equicontinuity implies equicontinuity, and a map
  \( f : X \to Y \) is continuous (resp.\ uniformly continuous) if and only if
  \( \set*{f} \) is equicontinuous (resp.\ uniformly equicontinuous).
\end{rema}
\subsection{The topology of
  \texorpdfstring{\( \mathfrak{S} \)}{S}-convergence}
\label{sec:3a8d43b32f489a0c}

\begin{defi}[\cite{bourbaki_topologie_2006-1}*{\S~X.1.1}]
  \label{defi:a4863de2d22b6ea1}
  Let be \( X \) a set, \( (Y, \mathcal{V)} \) a uniform space,
  \( \mathcal{F}(X, Y) \) a collection of maps from \( X \) to \( Y \). The
  \textbf{uniform structure of uniform convergence} on \( \mathcal{F}(X, Y) \)
  is the one where a filter basis for the collection of entourages is given by
  all sets of the form
  \begin{displaymath}
    W(V) : = \set[\big]{(u, v) \in \mathcal{F}(X, Y) \times \mathcal{F}(X, Y) \given \forall x \in X, \, (u(x), v(x)) \in V},
    \; V \in \mathcal{V}.
  \end{displaymath}
  We denote the resulting uniform space by \( \mathcal{F}_{u}(X, Y) \).
\end{defi}

\begin{defi}[\cite{bourbaki_topologie_2006-1}*{\S~X.1.2}]
  \label{defi:acbf0c33176b6907}
  Let be \( X \) a set, \( (Y, \mathcal{V)} \) a uniform space,
  \( \mathcal{F}(X, Y) \) a collection of maps from \( X \) to \( Y \), and
  \( \mathfrak{S} \) a collection of subsets of \( X \). The \textbf{uniform
    structure of \( \mathfrak{S} \)-convergence}, or more precisely, the
  \textbf{uniform structure of uniform convergence on sets of
    \( \mathfrak{S} \)}, on \( \mathcal{F}(X, Y) \), is the initial uniform
  structure with respect to the family all restriction map
  \( r_{A} : \mathcal{F}(X, Y) \to \mathcal{F}_{u}(X, Y) \),
  \( A \in \mathfrak{S} \). The resulting uniform space is denoted by
  \( \mathcal{F}_{\mathfrak{S}}(X, Y) \), and its topology is called \textbf{the
    topology of uniform convergence on sets of \( \mathfrak{S} \)}, or simply
  \textbf{the topology of \( \mathfrak{S} \)-convergence}.
\end{defi}

We will make extensive use of Bourbaki's version of Ascoli theorem.
\begin{theo}[\cite{bourbaki_topologie_2006-1}*{X.17, Théorème~2}]
  \label{theo:957848d2ddc8435e}
  Let \( X \) be a topological (resp.\ uniform) space, \( Y \) a uniform space,
  \( \mathfrak{S} \) a family of subsets of \( X \) that covers \( X \), and
  \( H \) a family of maps from \( X \) to \( Y \). Suppose the restriction of
  any map in \( H \) onto any \( A \in \mathfrak{S} \) is continuous (resp.\
  uniformly continuous). For \( H \) to be precompact with respect to the
  uniform structure of \( \mathfrak{S} \)-convergence, it is necessary in all
  cases, and also sufficient in the case when all sets in \( \mathfrak{S} \) are
  compact (resp.\ precompact), that the following hold:
  \begin{enumerate}
  \item \label{item:2aa19cbd96eb8ada} the restriction
    \( H\vert_{A}:= \set*{f\vert_{A}\given f \in H} \) is equicontinuous;
  \item \label{item:bd4632330be01ea0} the set
    \( H(x) := \set*{f(x) \given x \in X} \) is precompact in \( Y \).
  \end{enumerate}
\end{theo}

Given a topological vector space
\( E \) (locally convex or not, Hausdorff or not), when we speak of the uniform
structure property of \( E \), unless stated otherwise, we are always referring
to the uniform structure on \( E \) where a fundamental system of entourages is
given by sets of the form
\( \widetilde{V} := \set*{(x, y) \in E \given x - y \in V} \), \( V \) being a
neighborhood of \( 0 \in E \). This uniform structure is compatible with the
topology on \( E \).

One often give a locally convex topology on a vector space \( E \) by specifying
a fundamental system of absolutely convex neighborhoods of \( 0 \). Here's a
criterion for when this can be done.

\begin{prop}
  \label{prop:fd1b582617d3215e}
  Let \( E \) be a vector space, and \( \mathcal{F} \) a filter basis consisting
  of absolutely convex subsets of \( E \). Suppose the following hold:
  \begin{enumerate}
  \item \label{item:119aac6c1751a765} each \( M \in \mathcal{F} \) is absorbing;
  \item \label{item:07cb8057a6172eb7} \( M \in \mathcal{F} \) implies that
    \( \lambda M \in \mathcal{F} \) for all nonzero
    \( \lambda \in \mathbb{K} \).
  \end{enumerate}
  Then, there exists a unique locally convex topology on \( E \) such that
  \( \mathcal{F} \) is a fundamental system of neighborhoods of \( 0 \in E \).
  This topology is Hausdorff if and only if the intersection of all sets in
  \( \mathcal{F} \) is \( \set*{0} \).
\end{prop}
\begin{proof}
  This follows easily from \cite{MR3154940}*{p52, Theorem~3.2}.
\end{proof}

\begin{prop}
  \label{prop:eb240fbacebaf554}
  Let \( E \) be a locally convex space, \( T \) a set, \( \mathfrak{S} \) a
  collection of subsets of \( T \), \( G \) a linear subspace of the space of
  maps from \( T \) to \( E \). Denote by \( G_{\mathfrak{S}} \) the space
  \( G \) equipped with the topology of \( \mathfrak{S} \)-convergence.
  \begin{enumerate}
  \item \label{item:945ea785188d4872} If \( f(A) \) is bounded in \( E \) for all
    \( f \in G \), \( A \in \mathfrak{S} \), then the topology of \( G_{\mathfrak{S}} \) is locally convex.
  \item \label{item:86e45c617c23d6ec} If \ref{item:945ea785188d4872} holds, and
    for any \( f \in G \), \( f(A) = 0 \) for all \( A \in \mathfrak{S} \)
    implies \( f = 0 \), then \( G_{\mathfrak{S}} \) is a locally convex space.
  \end{enumerate}
\end{prop}
\begin{proof}
  Unwinding the definitions, a fundamental system of neighborhoods of
  \( 0 \in G \), which we denote by \( \mathcal{F} \), can be given by
  \emph{finite intersections of} absolutely convex sets of the form
  \begin{equation}
    \label{eq:53722af48bb64390}
    M(S, V):= \set*{f \in G \given f(S) \subseteq V}, \qquad S \in \mathfrak{S}, \, V \in \mathcal{N}_{\Gamma}(0).
  \end{equation}
  The condition in \ref{item:945ea785188d4872} implies the filter basis each set
  in \( \mathcal{F} \) is absorbing, hence the topology of
  \( \mathfrak{S} \)-convergence on \( G \) is locally convex by
  Proposition~\ref{prop:fd1b582617d3215e}. Since \( E \) is Hausdorff,
  \( \cap_{V \in \mathcal{N}_{\Gamma}(0)} V = \set*{0} \), the condition in
  \ref{item:86e45c617c23d6ec} implies that the intersection of all sets in
  \( \mathcal{F} \) is
  \begin{displaymath}
    \bigcap_{S \in \mathfrak{S}, V \in \mathcal{N}_{\Gamma}(0)}M(S, V) = \bigcap_{S \in \mathfrak{S}}\bigl(\bigcap_{V \in \mathcal{N}_{\Gamma}(0)}M(S, V)\bigr)
    = \bigcap_{S \in \mathfrak{S}} \left(\set*{f \in G \given f(S) \subseteq \bigcap_{V \in \mathcal{N}_{\Gamma}(0)}V = \set*{0}}\right)
    = \set*{0},
  \end{displaymath}
  thus \( G_{\mathfrak{S}} \) is a locally convex space (it is Hausdorff) by the
  same proposition.
\end{proof}

We will use Proposition~\ref{prop:eb240fbacebaf554} to define locally convex
topologies on certain spaces of bilinear forms when treating topological tensor
products. For now, we focus on the space of linear maps, in which case the
standard way is to introduce the notion of a bornology, following \cite{MR633754}*{\S~III.1}

\begin{defi}
  \label{defi:8503d41d4ff22cde}
  Let \( E \) be a vector space. A \textbf{bornology} \( \mathfrak{B} \) on
  \( E \) is a collection of subsets of \( E \) if it is
  \begin{enumerate*}
  \item filtered below in the sense that \( B \in \mathfrak{B} \) and
    \( A \subseteq B \) implies \( A \in \mathfrak{B} \);
  \item closed under taking finite unions.
  \end{enumerate*}
  We say the bornology \( \mathfrak{B} \) is \textbf{convex}, if in addition, we have
  \begin{enumerate}[resume]
  \item \( B \in \mathfrak{B} \) and \( \lambda \in \mathsf{K} \) implies \( \lambda B \in \mathfrak{B} \);
  \item \( B \in \mathfrak{B} \) implies \( \Gamma(B) \in \mathfrak{B} \), where
    \( \Gamma(B) \) denotes the absolutely convex hull of \( B \).
  \end{enumerate}
  If \( E \) is equipped with a linear topology \( \mathfrak{T} \) (locally
  convex or not), we say that a bornology \( \mathfrak{B} \) on \( E \) is
  \textbf{bounded} if it consists of only bounded sets; we say
  \( \mathfrak{B} \) is \textbf{adapted} to \( \mathfrak{T} \), or simply
  \( \mathfrak{B} \) is an \textbf{adapted bornology} when the topology
  \( \mathfrak{T} \) is clear from context, if \( \mathfrak{B} \) is convex,
  bounded and is stable under taking closure. The smallest adapted bornology of
  a bornology \( \mathfrak{B} \), which always exist, is called the
  \textbf{saturation} of \( \mathfrak{B} \). We say \( \mathfrak{B} \) is
  \textbf{total} if \( \cup_{B \in \mathfrak{B}}B \) is total in \( E \), i.e.\
  it spans a dense linear subspace; it is \textbf{covering} if
  \( \mathfrak{B} \) covers \( E \).
\end{defi}

\begin{prop}[\cite{MR0551623}*{\S~39.1(1) \& (2)}]
  \label{prop:2ac8943873ea3985}
  Let \( E \), \( F \) be locally convex spaces, \( \mathfrak{S} \) a total
  bornology on \( E \), then \( \mathcal{L}_{\mathfrak{B}}(E, F) \), i.e.\ the
  space of all continuous linear maps from \( E \) to \( F \) equipped with the
  \( \mathfrak{S} \)-topology, is a locally convex space. Let
  \( \mathfrak{S}' \) be another such bornology, then on
  \( \mathcal{L}(E, F) \), the \( \mathfrak{S} \)-topology coincides with the
  \( \mathfrak{S}' \)-topology if and only if \( \mathfrak{S} \) and
  \( \mathfrak{S}' \) have the same saturation.
\end{prop}

\begin{nota}
  \label{nota:6a66d03bb5d7af70}
  Let \( E \) be a not necessarily Hausdorff locally convex space. We use
  \( \mathfrak{F}(E) \), \( \mathfrak{C}(E) \) and \( \mathfrak{B}(E) \), or
  simply \( \mathfrak{F} \), \( \mathfrak{C} \) and \( \mathfrak{B}\) if \( E \)
  is clear form context, to denote respectively the collection of all finite
  sets, all absolutely convex weakly compact sets, all precompact sets and all
  bounded sets in \( E \).
\end{nota}

\begin{prop}
  \label{prop:6c0374a8634e2b65}
  Let \( E \) be a not necessarily Hausdorff locally convex space, then
  \( \mathfrak{F} \), \( \mathfrak{C} \) and \( \mathfrak{B}\) are all bounded
  covering bornologies on \( E \), among which \( \mathfrak{C} \) and
  \( \mathfrak{B} \) are adapted.
\end{prop}
\begin{proof}
  We sketch the proof of \( \mathfrak{C} \) is stable under taking the closed
  absolutely convex hull, which is the only nontrivial part. By
  Proposition~\ref{prop:cdb6e313569d3c6e}, it coincides with the collection of
  totally bounded subsets. By \cite{MR0342978}*{\S~II.4.3}, \( \mathfrak{C} \)
  is stable under taking the absolutely convex hull, and it's stable under
  taking closure follows directly from Definition~\ref{defi:8b4e3fc5c7fd7dcf}
  and the fact that any closed subspace of compact spaces remains compact.
\end{proof}

\begin{defi}
  \label{defi:3852fa5b63ee88dd}
  Let \( F \) be another locally convex space.  On \( \mathcal{L}(E, F) \), the
  \( \mathfrak{F} \), \( \mathfrak{C} \) and \( \mathfrak{B} \)-topologies are
  called respectively the \textbf{topology of simple, precompact and bounded
    convergence}, and is denoted respectively by \( \mathcal{L}_{s}(E, F) \),
  \( \mathcal{L}_{c}(E, F) \) (see the caution in
  Notation~\ref{nota:ee765bb793c0ed9f}) and \( \mathcal{L}_{b}(E, F) \). When
  \( F = \mathbb{K} \), the resulting locally convex space is denoted by
  respectively \( E'_{s} \), \( E'_{c} \), \( E'_{b} \), and is respectively
  called the \textbf{weak dual}, the \textbf{polar dual} and the \textbf{strong
    dual} of \( E \).
\end{defi}

We shall have the occasion to use the Banach-Dieudonné theorem.
\begin{theo}[\cite{MR633754}*{IV.24, Théorème~1}]
  \label{theo:4b05f97fb7cc0705}
  Let \( E \) be a metrizable locally convex space. Then on \( E' \), the
  following topologies are identical:
  \begin{enumerate}    
  \item \label{item:de72e2299e543a79} the \( \mathfrak{N} \)-topology, where
    \( \mathfrak{N} \) is the collection of the image of all null sequences in \( E \);
  \item \label{item:3cbc087d2a4bb383} the topology of precompact convergence;
  \item \label{item:06cf100592514f1e} the topology of compact convergence;
  \item \label{item:f9243b61245c753d} the finest topology that coincides with
    \( \sigma(E', E) \) (the topology of simple convergence) on all
    equicontinuous parts of \( E' \).
  \end{enumerate}
\end{theo}

\subsection{Dual pairs of vector spaces and polar topologies}
\label{sec:17562a18b9fdb5d3}

\begin{defi}
  \label{defi:12b42590a911540b}
  A \textbf{dual pair (or pairing, or duality pairing), of vector spaces} is a
  bilinear form \( \pairing*{\cdot}{\cdot}: V \times W \to \mathbb{K} \) that is
  separated on both variables, and shall of be denoted by \( \pairing*{V}{W} \).
  We say that a locally convex topology \( \mathfrak{T} \) on \( V \) is
  \textbf{compatible with the pairing}, if
  \( w \in W \mapsto \pairing*{\cdot}{w} \in V^{\sharp} \) maps \( W \)
  bijectively onto \( \bigl(V[\mathfrak{T}]\bigr)' \). In this case, we shall
  often write \( W = \bigl(V[\mathfrak{T}]\bigr)' \) to mean we've identified
  \( W \) with the topological dual of \( V[\mathfrak{T}] \) in this
  way. Compatibility of a locally convex topology on \( W \) is defined and
  notated similarly. If \( E \), \( F \) are locally convex spaces, a dual pair
  \( \pairing*{E}{F} \) is called \textbf{compatible} if it is compatible with
  both the topology on \( E \) and the topology on \( F \).
\end{defi}

\begin{defi}
  \label{defi:058d62b099daf673}
  Given a dual pair \( \pairing*{V}{W} \). The \textbf{(absolute) polar} of a
  subset \( A \subseteq V \) is the absolutely convex set
  \begin{equation}
    \label{eq:329b590c57db4dff}
    A^{\circ}:= \set[\big]{w \in W \given \forall a \in A, \, \abs*{\pairing*{a}{w}} \leq 1}.
  \end{equation}
  The polar \( B^{\circ} \) of a set \( B \subseteq W \) is defined
  similarly. We say a set \( A \subseteq V \) is \textbf{simply bounded} (with
  respect to the pairing) if
  \( \set*{\pairing*{a}{w} \given a \in A} \subseteq \mathbb{K} \) is bounded
  for each \( w \in W \). Simply bounded sets in \( W \) (with respect to the
  pairing) is defined similarly.
\end{defi}

\begin{prop}
  \label{prop:7dcbda305a2f1ba2}
  Let \( \pairing*{E}{F} \) be a dual pair of vector spaces, and
  \( \mathfrak{B} \) a collection of simply bounded sets on \( F \), and let
  \( \widetilde{\mathfrak{B}} \) be the smallest collection of simply bounded
  subsets of \( F \) that contains \( \mathfrak{B} \) and is stable under taking
  dilations and finite unions. Then
  \( \set*{B^{\circ} \given B \in \widetilde{\mathfrak{B}}} \) specifies a fundamental
  system of neighborhoods of \( 0 \) for a unique locally convex topology on
  \( E \), and this topology is Hausdorff if \( \mathfrak{B} \) is covering.
\end{prop}
\begin{proof}
  This follows from Proposition~\ref{prop:fd1b582617d3215e} and
  Definition~\ref{defi:8503d41d4ff22cde}.
\end{proof}

\begin{defi}
  \label{defi:5913407834094598}
  Recall Notation~\ref{nota:6a66d03bb5d7af70}. The topology in
  Proposition~\ref{prop:7dcbda305a2f1ba2} is called the \textbf{polar topology}
  of \( \mathfrak{B} \). The polar topology on \( E \) of \( \mathfrak{F}(F) \)
  is called \textbf{the weak topology with respect to the dual pair}
  \( \pairing*{E}{F} \), and is often denoted by \( \sigma(E, F) \); and that of
  \( \mathfrak{B}(F) \) the \textbf{strong topology with respect to}
  \( \pairing*{E}{F} \), and denoted by \( b(E, F) \).

  Equipped with \( \sigma(E, F) \) (resp.\ \( b(E, F) \)), the locally convex
  space \( E \) is called the \textbf{weak dual} (resp.\ \textbf{strong dual})
  of \( F \) (with respect to the pairing \( \pairing*{E}{F} \)).

  When \( F = E' \) and the pairing \( \pairing*{E}{E'} \) is given by
  evaluation, the topology \( \sigma(E, E') \) (resp.\ \( b(E, E') \),
  \( c(E, E') \)) is simply called the \textbf{weak} ((resp.\ \textbf{strong}),
  \textbf{topology} on \( E \).
\end{defi}

It is easy to check that \( \sigma(E, F) \) is the initial topology on \( E \)
with respect to
\( \set*{\pairing*{\cdot}{y} \given y \in F}: E \to \mathbb{K} \), and every
polar \( A^{\circ} \) of some \( A \subseteq E \) is a weakly (with respect to
\( \pairing*{E}{F} \)) closed absolutely convex set in \( F \). The converse
also holds by the bipolar theorem.

\begin{theo}[Bipolar theorem, \cite{MR0248498}*{\S~20.8(5)}]
  \label{theo:45bc53b6ce4e0dc3}
  Let \( E \) be a locally convex space, and consider the canonical dual pair
  \( \pairing*{E}{E'} \). For any \( A \subseteq E \), the (absolute) bipolar
  \( A^{\circ\circ} \) of \( A \) is exactly the weakly closed absolutely convex
  hull of \( A \) in \( E \).
\end{theo}

Recall the notion of equicontinuity (Definition~\ref{defi:94460281f200f584}). It
is clear that for a locally convex space \( E \), a set
\( H \subseteq E' = \mathcal{L}(E, \mathbb{K}) \) if and only if
\( H \subseteq U^{\circ} \) for some \( U \in \mathcal{N}(0) \). We also have
the following version of Alaoglu's theorem.

\begin{theo}[\cite{MR0342978}*{\S~III.4.3, Corollary}]
  \label{theo:bf2abf9c9ce8aa7a}
  Any equicontinuous set in \( E' \) is relatively compact for
  \( \sigma(E', E) \).
\end{theo}

Moreover, every locally convex Hausdorff topology can be seen as a polar
topology.

\begin{prop}
  \label{prop:f9256f38e904319e}
  Let \( E \) be a locally convex space. Then closed absolutely convex
  neighborhoods of \( 0 \in E \) are exactly the polars of equicontinuous sets
  in \( E' \).
\end{prop}
\begin{proof}
  If \( A \supseteq U^{\circ} \) for some \( U \in \mathcal{N}(0) \), then
  \( A^{\circ} \supseteq U^{\circ\circ} \supseteq U \), hence
  \( A^{\circ} \in \mathcal{N}(0) \) and it is absolutely continuous and closed
  since
  \begin{displaymath}
    A = \cap_{f \in A}f^{-1}\left(\set[\big]{\lambda \in \mathbb{K} \given
        \abs*{\lambda} \leq 1}\right).
  \end{displaymath}
  Conversely, every neighborhood of \( 0 \) contains some neighborhood \( U \)
  of \( 0 \) that is closed and absolutely convex. By
  Theorem~\ref{theo:45bc53b6ce4e0dc3}, \( U \) is the polar of the
  equicontinuous set \( U^{\circ} \).
\end{proof}

Proposition~\ref{prop:f9256f38e904319e} motivates us to describe compatible
topologies on the algebraic tensor product using some suitable polar topology,
see \S~\ref{sec:696b043dbe6fbd93}, Proposition~\ref{prop:5084d766fecfec2d}.

\subsection{The Mackey topology, boundedness}
\label{sec:4d080905b6153a35}

\begin{nota}
  \label{nota:abe825c6b234ab51}
  Let \( \pairing*{E}{F} \) be a dual pair of vector space. The collection of
  all \( \sigma(F, E) \)-compact absolutely convex is denoted by
  \( \mathfrak{K}(F, E) \). If \( E \) is a locally convex space, and the
  pairing \( \pairing*{E'}{E} \) is evaluation, then \( \mathfrak{K}(E, E') \)
  is abbreviated \( \mathfrak{K}(E) \), or simply \( \mathfrak{K} \) if \( E \)
  is clear from context. If \( E \), \( F \) are locally convex spaces,
  \( \mathcal{L}_{k}(E, F) \) denotes \( \mathcal{L}(E, F) \) equipped with the
  topology of \( \mathfrak{K}(E) \)-convergence, and \( E'_{k} \) denotes
  \( \mathcal{L}_{k}(E, \mathbb{K}) \).
\end{nota}

\begin{defi}
  \label{defi:78058b27cd954f7a}
  Let \( \pairing*{E}{F} \) be a duality pairing. The polar topology on \( E \)
  of \( \mathfrak{K}(F, E) \) is called the \textbf{Mackey topology}, and is
  denoted by \( \tau(E, F) \). If \( E \) is a locally convex space, then
  \( E' \), equipped with \( \tau(E', E) \) for the evaluation pairing
  \( \pairing*{E'}{E} \) is also denoted by \( E'_{k} \); whereas \( E \), now
  equipped with \( \tau(E, E') \), is denoted by \( E_{\tau} \). We say that the
  locally convex space \( E \) is a \textbf{Mackey space}, if
  \( E = E_{\tau} \).
\end{defi}

We have the following consequences of a result of Mackey \cite{MR0020214}.

\begin{theo}[Mackey, \cite{MR633754}*{IV.2, Théorème~1}]
  \label{theo:52188c82b45e07e3}
  Let \( \pairing*{E}{F} \) be a duality pairing. Then a locally convex topology
  on \( E \) is compatible with this pairing if and only if it is finer than
  \( \sigma(E, F) \) and coarser than \( \tau(E, F) \).
\end{theo}

\begin{coro}
  \label{coro:6ce070ab8838d3f8}
  Let \( E \) be a locally convex space, \( A \subseteq E \). Then the following
  are equivalent:
  \begin{enumerate}
  \item \label{item:50eeb88c1bc06852} \( A \) is bounded;
  \item \label{item:03473de24d3f4c00} \( A \) is bounded for one of the locally
    convex topology that is coarser than \( \tau(E, E') \) and finer than
    \( \sigma(E, E') \);
  \item \label{item:6e48edca069218c9} \( A \) is bounded for all of the locally
    convex topology that is coarser than \( \tau(E, E') \) and finer than
    \( \sigma(E, E') \);
  \item \label{item:fe19658c3bae1266} \( u(A) \subseteq \mathbb{K} \) is bounded
    for every \( u \in E' \).
  \end{enumerate}
\end{coro}
\begin{proof}
  Equivalence between \ref{item:50eeb88c1bc06852} and
  \ref{item:fe19658c3bae1266} follows from \cite{MR633754}*{\S~III.4,
    Théorème~2}, the rest of the equivalences now follows from Mackey's theorem.
\end{proof}

\subsection{Projective and inductive locally convex topologies}
\label{sec:e39c2dc88f4df7b1}

For the following results in \S~\ref{sec:e39c2dc88f4df7b1}, we refer our readers
to \cite{MR0342978}*{\S~II.5 \& \S~II.6}, as well as \cite{MR0248498}*{\S~19}.

\begin{defi}
  \label{defi:6205fbfabe1845eb}
  Let \( E \) be a vector space, \( f_{i}: E \to E_{i} \), \( i \in I \) a
  collection of linear maps into locally convex spaces \( E_{i} \). The
  \textbf{projective topology} on \( E \) with respect to the family
  \( (f_{i}) \) is the initial topology with respect to this family. It is a
  locally convex topology on \( E \), and enjoys the universal property of
  initial topologies.
\end{defi}

\begin{rema}
  \label{rema:0361c678857d395c}
  Being an initial topology, projective topology on \( E \) enjoys the universal
  property of initial topologies with respect to the family
  \( (f_{i})_{i \in I} \) as topological spaces.
\end{rema}

\begin{prop}
  \label{prop:438d10ce1f2149df}
  For a projective system of locally convex spaces
  \( (E_{i}, p_{ij})_{i, j \in I} \), \textbf{the projective limit}
  \( E:= \varprojlim E_{i} \) in \( \mathsf{LCS} \), and can be realized by
  equipping the algebraic projective limit \( \varprojlim E_{i} \) with the
  projective topologies with respect to the canonical projections
  \( p_{i}: E \to E_{i} \), \( i \in I \). Moreover,
  \begin{displaymath}
    \varprojlim E_{i} = \set[\big]{(x_{i}) \in \prod_{i} E_{i}
      \given \forall i \leq j,\, p_{ij}(x_{j}) = x_{i}}
  \end{displaymath}
  is a closed subspace of \( \prod_{i \in I}E_{i} \), and
  \( \varprojlim E_{i} \) is complete if every \( E_{i} \), \( i \in I \) is complete.
\end{prop}
\begin{proof}
  The part on the projective limit follows from
  Remark~\ref{rema:0361c678857d395c} (the universal property is straightforward,
  then it follows that
  \( \set*{\alpha_{i}p_{i} \in E' \given \alpha_{i} \in E'_{i}, \, i \in I} \)
  is a separating family of linear maps on \( E \), hence \( E \) is
  Hausdorff). The completeness follows from
  Proposition~\ref{prop:903bfccf4a781ceb}.
\end{proof}

\begin{defi}
  \label{defi:d7ad55a6124fe9a4}
  In Proposition~\ref{prop:438d10ce1f2149df}, the projective limit is called
  \textbf{reduced}, if \( p_{i}(E) \) is dense in each \( E_{i} \),
  \( i \in I \).
\end{defi}

Reduced projective limits behaves well with the projective and injective tensor
products, see \S~\ref{sec:33ea3663864561d2} \& \S~\ref{sec:ac51753bf78a1da5}.

The dual notion of inductive locally convex topologies is more subtle, e.g.\ the
Hausdorff condition might not be preserved, see
Remark~\ref{rema:67bc53cfb5df5400}.

\begin{prop}
  \label{prop:8a4b66d6bb64e8de}
  Let \( f_{i}: F_{i} \to F \), \( i \in I \) be a collection of linear maps
  from locally convex space \( F_{i} \) to the same linear space \( F \). There
  exists a finest locally convex topology \( \tau_{0} \) on \( F \) making each
  \( f_{i} \) continuous.
\end{prop}
\begin{proof}
  Let \( \mathfrak{T} \) denotes the collection of all locally convex topologies
  on \( F \) making each \( f_{i} \) continuous. First of all,
  \( \mathfrak{T} \ne \emptyset \) since it contains the indiscrete
  topology. For each \( \tau \in \mathfrak{T} \), let \( F_{\tau} \) denote
  \( F \) equipped with topology and \( i_{\tau}: F \to F_{\tau} \) the identity
  map. Let \( \tau_{0} \) be the projective topology on \( F \) with respect to
  the family \( (i_{\tau}: F \to F_{\tau})_{\tau \in \mathfrak{T}} \), which is
  convex. From the universal property of the initial topology, it follows that
  \( \tau_{0} \in \mathfrak{T} \), and from construction,
  \( i_{\tau}: F_{\tau_{0}} \to F_{\tau} \) is continuous for all
  \( \tau \in \mathfrak{T} \), meaning \( \tau_{0} \) is the finest locally
  convex topology in \( \mathfrak{T} \).
\end{proof}

\begin{defi}
  \label{defi:0c468afc634e0bc1}
  Using the notation in Proposition~\ref{prop:8a4b66d6bb64e8de}, the topology
  \( \tau_{0} \) on \( F \) is called \textbf{the inductive topology} with
  respect to the family \( (f_{i}) \). We call \( (F, \tau_{0}) \) the
  \textbf{locally convex hull} with respect to \( (f_{i}) \), if the union
  \( \cup_{i \in I}f_{i}(F_{i}) \) spans \( F \) linearly, and we denote it by
  \( \hull{(f_{i})} \) inductive system of locally convex spaces, and
  \( v_{i}: E_{i} \to \varinjlim^{\alg}E_{i} \) the canonical insertions into
  the inductive limit in the category of vector spaces, then
  \( \hull{(v_{i})} \) is called the \textbf{locally convex hull} of the system
  \( (E_{i}, u_{ij}) \), and we denote it by
  \( \hull{(E_{i}, u_{i,j})_{i,j \in I}} \).
\end{defi}

\begin{prop}[Universal property of the inductive topology]
  \label{prop:c90d867cb6afb94d}
  Using the notation in Proposition~\ref{prop:8a4b66d6bb64e8de}, let \( G \) be
  a topological vector space that is locally convex. Then a linear map
  \( f: F \to G \) is continuous if and only if each \( ff_{i} : E \to G \) is
  continuous.
\end{prop}
\begin{proof}
  Let \( \tau \) be the unique locally convex topology on \( F \) with
  \( f^{-1}(V) \), \( V \in \mathcal{N}^{G}_{\Gamma}(0) \) being a fundamental
  system of neighborhoods of \( 0 \). Clearly, \( ff_{i} \) is continuous if
  \( f \) is so. Conversely, each \( ff_{i} \) being continuous, we see that
  \( f_{i}^{-1}(f^{-1}(V)) = (ff_{i})^{-1}(V) \in
  \mathcal{N}^{F_{i}}_{\Gamma}(0) \), meaning \( f_{i}: F_{i} \to (F, \tau) \)
  is continuous for each \( i \). Hence \( \tau \subseteq \tau_{0} \) and
  \( f \) is continuous.
\end{proof}

\begin{prop}
  \label{prop:8c1bccabeb0c2d99}
  Let \( f_{i}: F_{i} \to F \), \( i \in I \) be a collection of linear maps
  from locally convex space \( F_{i} \) to the same linear space \( F \), and we
  equip \( F \) with the corresponding inductive topology. Then for each
  \( V \subseteq F \), we have \( V \in \mathcal{N}^{F}_{\Gamma}(0) \) if and
  only if \( f^{-1}_{i}(V) \in \mathcal{N}^{F_{i}}_{\Gamma}(0) \). Moreover, if
  \( F \) is the locally convex hull of \( \set*{f_{i}} \), then a fundamental
  system of neighborhoods of \( 0 \) for \( F \) is given by absolutely convex
  sets of the form \( \Gamma\left(\cup_{i \in I}V_{i}\right) \), with
  \( V_{i} \in \mathcal{N}^{F_{i}}_{\Gamma}(0) \) for each \( i \in I \).
\end{prop}
\begin{proof}
  The first statement follows easily from the existence of the inductive
  topology on \( F \). The second statement follows from the first, while noting
  that the condition of being an inductive hull guarantees that
  \( \Gamma\left(\cup_{i \in I}V_{i}\right) \) is absorbing, cf
  Proposition~\ref{prop:fd1b582617d3215e}.
\end{proof}

\begin{rema}
  \label{rema:67bc53cfb5df5400}
  Let \( E \) be a locally convex space, \( N \) a subspace, and
  \( q: E \to E/N \) the quotient map. It is clear that the quotient topology on
  \( E/N \) is the inductive topology with respect to \( \set*{q} \), which is
  Hausdorff if and only if \( N \) is closed.
\end{rema}

\begin{defi}
  \label{defi:69f4413fd7d99b7b}
  Let \( E \) be a vector space equipped with a locally convex topology, the
  \textbf{separated quotient} of \( E \) is the quotient locally convex space
  \( E/N \), where \( N \) is the closed subspace of \( E \) given by the
  intersection of all neighborhoods of \( 0 \).
\end{defi}

\begin{prop}
  \label{prop:c2ffa4009a3940c2}
  Let \( (E_{i}, u_{ij})_{i,j \in I} \) be an inductive system of locally convex
  space. Then the inductive limit \( \varinjlim E_{i} \) exists in
  \( \mathsf{LCS} \), and is given by the separated quotient of
  \( \hull{(E_{i}, u_{ij})} \).
\end{prop}
\begin{proof}
  This follows from Proposition~\ref{prop:c90d867cb6afb94d} and the universal
  property of the separated quotient.
\end{proof}

\begin{defi}
  \label{defi:8c937f7304a52741}
  The \textbf{(locally convex) direct sum} of a family \( E_{i} \),
  \( i \in I \) of locally convex space is the locally convex hull with respect
  to the family of canonical injections
  \( u_{i}: E_{i} \hookrightarrow \bigoplus_{i} E_{i} \). We shall often denote
  this locally convex direct sum by the same notation \( \bigoplus_{i}E_{i} \).
\end{defi}

\begin{prop}[\cite{MR0248498}*{p212, (3)}]
  \label{prop:d842a95c764bfd01}
  Let \( E_{i} \), \( i \in I \) be a family of locally convex spaces, and
  \( \bigoplus_{i}E_{i} \) the locally direct sum. Then,
  \begin{enumerate}
  \item \label{item:057c871221feccf1} \( \bigoplus_{i}E_{i} \)  is a locally convex space;
  \item \label{item:8f7f5910e0175668}
    \( \widehat{\bigoplus_{i}E_{i}} = \bigoplus \widehat{E_{i}} \) (we use
    \( \widehat{\cdot} \) to denote completion, as introduced in
    \S~\ref{sec:bb19dab8a74cccc2}, Definition~\ref{defi:da4d0e8821475034}).
  \end{enumerate}
  In particular, \( \bigoplus_{i} E_{i} \) is a complete locally convex space if
  each \( E_{i} \), \( i \in I \) is complete.
\end{prop}

We also have the following characterization of bounded sets in a locally convex direct sum.

\begin{prop}
  \label{prop:746e91e6638b3d9d}
  Let \( E = \oplus_{i \in I} E_{i} \) be a locally convex direct sum of locally
  convex spaces. For each index \( i \), let \( p_{i}: E \to E_{i} \) be the
  corresponding canonical projection. Then a set \( A \subseteq E \) is bounded
  if and only if there exists a finite \( I_{0} \subseteq I \), such that
  \( p_{i}(A) = 0 \) for \( i \notin I_{0} \) and \( p_{i}(B) \) is bounded in
  \( E_{i} \) for each \( i \in I_{0} \). In particular, a set in \( E \) is
  bounded if and only if it is contained the sum of finitely many bounded sets
  in each of the \( E_{i} \)'s.
\end{prop}

We shall need the following duality results.

\begin{prop}[\cite{MR0342978}*{p137, 4.3}]
  \label{prop:e5dd78c27c8725c9}
  Let \( (E_{i})_{i \in I} \) be a family of locally convex spaces, and
  \( E = \prod_{i \in I}E_{i} \) the (locally convex) product. Then the
  topological dual \( E' \) of \( E \), as a vector space, is canonically
  identified with the algebraic direct sum \( \oplus_{i \in I}E'_{i} \). Consider
  canonical pairing \( \pairing*{E}{E'} \), we have
  \begin{enumerate}
  \item \label{item:1c9160cfb356df54} \( \sigma(E, E') = \prod_{i} \sigma(E_{i}, E'_{i}) \);
  \item \label{item:081ba93c47a3eb24} \( \tau(E, E') = \prod_{i} \tau(E_{i}, E'_{i}) \);
  \item \label{item:df8b51f0915a1e5c}
    \( \tau(E', E) = \oplus_{i} \tau(E_{i}, E'_{i}) \) (right side is the
    locally convex direct sum).
  \end{enumerate}
\end{prop}

\begin{prop}[\cite{MR0342978}*{p139, 4.4}]
  \label{prop:bd7af82b294982b6}
  Let \( E = \varprojlim E_{\beta} \) be a reduced projective limit of some
  projective system \( (E_{\beta}, g_{\alpha \beta}) \) in \( \mathsf{LCS}
  \). Then the topological dual \( E' \), when equipped with the Mackey topology
  \( \tau(E', E) \) is canonically identified with the inductive limit
  \( \varinjlim (E_{\alpha})_{\tau}' \) in \( \mathsf{LCS} \) of the dual system
  \( (E'_{\alpha}, h_{\beta \alpha}) \) with
  \( h_{\beta \alpha}: E'_{\alpha} \to E'_{\beta} \) being the transpose of
  \( g_{\alpha \beta}: E_{\beta} \to E_{\alpha} \) whenever
  \( \alpha < \beta \).
\end{prop}

\subsection{Strict morphism, countable strict inductive limits}
\label{sec:a2c822e6edcf5f15}

Contrary to the projective limit case (Proposition~\ref{prop:438d10ce1f2149df}),
taking inductive limit in \( \mathsf{LCS} \) doesn't preserve
completeness. However, it behaves well if we restrict ourselves to strict
inductive limits.

\begin{defi}
  \label{defi:2c61a2eb8f753159}
  A continuous linear map \( f : E \to F \) between locally convex spaces is
  called \textbf{strict morphism}, or simply \textbf{strict}, if
  \( f : E \to f(E) \) is an open map, or equivalently, if the induced map
  \( \overline{f}: E/\ker(f) \to f(E) \) is an isomorphism of locally convex
  spaces. In \( \mathsf{LCS} \), we call an inductive system of the form
  \( (E_{n}, u_{n}: E_{n} \to E_{n+1})_{n \geq 1} \) \textbf{strict}, if each
  \( u_{n} \) is a strict embedding.
\end{defi}

\begin{prop}
  \label{prop:cdd9376419588ce4}
  The locally convex hull of a countable strict inductive system
  \( (E_{n}, u_{n}) \) is Hausdorff, hence is already a locally convex space and
  coincides with the inductive limit \( \varinjlim E_{n} \) in
  \( \mathsf{LCS} \). Moreover, each canonical insertion
  \( v_{n}: E_{n} \to \varinjlim E_{n} \) is a strict morphism onto a closed
  subspace of \( \varinjlim E_{n} \).
\end{prop}
\begin{proof}
  All statements except the one on the inductive limit follows from
  \cite{MR0342978}*{\S~II.6.4}. Now the assertion on the inductive limit follows
  from Proposition~\ref{prop:c2ffa4009a3940c2}.
\end{proof}

\subsection{Completion as a left adjoint}
\label{sec:e2b9bce4f6334a93}

We shall have the occasion of using the commutation of the completion with
taking colimit in the category \( \mathsf{LCS} \), which generalizes
Proposition~\ref{prop:d842a95c764bfd01}~\ref{item:8f7f5910e0175668}. Although
the consideration is elementary, the author couldn't find an explicit statement
in the literature, so we give a brief treatment here.

\begin{prop}
  \label{prop:6f06da3d7291c7cf}
  The completion functor
  \( \widehat{(\cdot)}: \mathsf{LCS} \to \widehat{\mathsf{LCS}} \) is left
  adjoint to the forgetful functor
  \( \mathscr{F}: \widehat{\mathsf{LCS}} \to \mathsf{LCS} \).  More
  specifically, for any \( X \in \mathsf{LCS} \),
  \( Y \in \widehat{\mathsf{LCS}} \), we have a natural isomorphism of vector
  spaces
  \begin{equation}
    \label{eq:d14217924871cd6c}
    \begin{split}
      \mathcal{L}(\widehat{X}, Y) & \simeq \mathcal{L}(X, \mathscr{F}Y) \\
      f & \mapsto f\vert_{X} = f i_{X} \\
      \widehat{g} & \mapsfrom g
    \end{split}
  \end{equation}
  where \( \widehat{X} = (X, i_{X}) \), and \( \widehat{g} \) is the unique
  continuous extension of \( g \).
\end{prop}
\begin{proof}
  The verification of the alleged bijection \eqref{eq:d14217924871cd6c} follows
  from the discussion of separated completion in \S~\ref{sec:bb19dab8a74cccc2}
  (Definition~\ref{defi:da4d0e8821475034} and
  Theorem~\ref{theo:5c48561db339a0d3}). The naturalness of this bijection is
  readily seen by a routine check.
\end{proof}

\begin{coro}
  \label{coro:a9570db3191db003}
  Completion commutes with taking small colimit in \( \mathsf{LCS} \).
\end{coro}
\begin{proof}
  This is a formal property of all left adjoints, see, e.g.\ the dual version of
  \cite{MR0354798}*{\S~V.5, Theorem~1}.
\end{proof}

\begin{coro}
  \label{coro:7ae986967f0d1033}
  If \( E \) is the strict inductive limit of a countable strict inductive
  system \( (E_{n}, u_{n})_{n \geq 1} \) in \( \mathsf{LCS} \), then the following hold:
  \begin{enumerate}
  \item \label{item:609bacd01d0f98af} the induced inductive system
    \( (\widehat{E_{n}}, \widehat{u_{n}})_{n \geq 1} \) remain strict;
  \item \label{item:320c6e598192438d} the completion
    \( \widehat{v_{n}} : \widehat{E_{n}} \hookrightarrow \widehat{E} \) of each
    canonical insertion \( v_{n} : E_{n} \hookrightarrow E \) remains a strict
    embedding;
  \item \label{item:abb6971c4b95c91e} the completion \( \widehat{E} \), together
    with the canonical injections
    \( \widehat{v_{n}}: E_{n} \hookrightarrow \widehat{E} \), is the strict
    inductive limit of \( (\widehat{E_{n}}, \widehat{u_{n}})_{n \geq 1} \).
  \end{enumerate}
\end{coro}
\begin{proof}
  First two statements follow from Theorem~\ref{theo:5c48561db339a0d3}, and the
  last follows from the first two and Proposition~\ref{prop:6f06da3d7291c7cf}.
\end{proof}

\subsection{Barrelled and bornological spaces}
\label{sec:8d70a40bdd5aa3e4}

\begin{defi}
  \label{defi:be959980ca5aafcd}
  Let \( E \) be a locally convex space. A \textbf{barrel} in \( E \) is a
  closed absolutely convex set that is also absorbing, i.e.\ absorbs every
  point, in the sense that \( A \) absorbs \( B \) if
  \( \lambda A \supseteq B \) for all scalars \( \lambda \) that are large
  enough. If every barrel in \( E \) is a neighborhood of \( 0 \), then we say
  that \( E \) is \textbf{barrelled}.
\end{defi}

\begin{defi}
  \label{defi:97d946f8a8d3ae40}
  A set \( B \subseteq E \) is called \textbf{bornivorous}, if it absorbs all
  bounded sets. We say \( E \) is \textbf{bornological} if every absolutely
  convex bornivorous (not necessarily closed) set of \( E \) is a neighborhood
  of \( 0 \).
\end{defi}

\begin{prop}[\cite{MR0342978}*{\S\S~II.7, 8, \& p138, Corollary~4}]
  \label{prop:7963cb2832fabbd4}
  The following hold:
  \begin{enumerate}
  \item \label{item:39f63d7a7460a912} A locally convex space \( E \) is
    barrelled if any of the following holds:
    \begin{enumerate}
    \item \label{item:27b0bff54489f9ec} \( E \) is a Baire space as a
      topological space;
    \item \label{item:9f20dc20c94d8474} \( E \) is the inductive topology with
      respect to a family of linear maps from barrelled spaces;
    \item \label{item:97a8e5338108cf5f} \( E \) is the product of a family of
      barrelled spaces.
    \end{enumerate}
  \item \label{item:ab6f69487743234f} A locally convex space \( E \) is
    bornological if any of the following holds:
  \item \label{item:217ead75b12b4308} \( E \) is metrizable;
  \item \label{item:f7fc60e143d597d3} \( E \) is the inductive topology with
    respect to a family of linear maps from bornological spaces.
  \end{enumerate}
\end{prop}

\begin{rema}
  \label{rema:c87031bba5990077}
  Countable product of bornological spaces remain bornological, however, if the
  family is uncountable, the situation seems unknown. See
  \cite{MR0342978}*{p61}.
\end{rema}

\begin{prop}[\cite{MR0342978}*{IV.3.4}]
  \label{prop:a944f1bb11e881f1}
  Let \( E \) be a locally convex space.  If \( E \) is either barreled or
  bornological, then it is Mackey.
\end{prop}

\subsection{Quasi-completeness and the completeness of the strong and polar
  duals}
\label{sec:db161d8f8f30bbe7}

\begin{defi}
  \label{defi:56805be540a1b16d}
  We say a locally convex space \( E \) is \textbf{quasi-complete}, if every
  closed \emph{bounded} set in \( E \) is complete.
\end{defi}

\begin{prop}[\cite{MR633754}*{III.23, Proposition~12}]
  \label{prop:2ff9c3b3fe361ef4}
  Let \( E \) be a bornological locally convex space, \( F \) a complete locally
  convex space, \( \mathfrak{S} \) a collection of bounded sets in \( E \). If
  \( \mathfrak{S} \) contains the image of every null sequence in \( E \), then
  \( \mathcal{L}_{\mathfrak{S}}(E, F) \) is complete.
\end{prop}

\begin{coro}
  \label{coro:f7e5d81ca583056d}
  Let \( E \) be a bornological locally convex space, then the strong dual
  \( E'_{b} \) and the polar dual \( E'_{c} \) (cf.\
  Definition~\ref{defi:3852fa5b63ee88dd}) are both complete.
\end{coro}

\subsection{\texorpdfstring{\( (DF) \)}{(DF)}-spaces}
\label{sec:2ffdde5596825e91}

Motivated by the studies of strong duals of metrizable locally convex spaces,
and based on the observation in \cite{MR0038554}, Grothendieck introduced the
notion of a \( (DF) \)-space in \cite{MR0075542}.

\begin{defi}[\cite{MR4182424}*{p53}]
  \label{defi:5cdb26a0af71ffa7}
  We say a locally convex space \( E \) is a \( (DF) \)-space, if it satisfies
  the following conditions:
  \begin{enumerate}
  \item \label{item:da386a70af7a0673} \( E \) possesses a fundamental sequence
    of bounded sets \( (B_{n})_{n \geq 1} \) in the sense that every
    \( B \in \mathfrak{B}(E) \) is contained in some \( B_{n} \);
  \item \label{item:ccad66f3327f6fec} if a countable union of equicontinuous
    sets in \( E' \) is strongly bounded, i.e.\ with respect to the topology
    \( b(E', E) \), then this union remains equicontinuous.
  \end{enumerate}
\end{defi}

\begin{nota}
  \label{nota:de1e39c9c4c63d75}
  The class of all \( (DF) \)-spaces is denoted by \( (\mathcal{DF}) \).
\end{nota}

\begin{prop}
  \label{prop:150a19f0054c21da}
  The strong dual of a metrizable locally convex is a \( (DF) \) space, and the
  strong dual of a \( (DF) \) space is an \( (F) \)-space.
\end{prop}
\begin{proof}
  See \cite{MR4182424}*{p54, Theorem~7.1 \& p56, Theorem~7.5}, note that Voigt
  used a different definition of \( (DF) \)-spaces \cite{MR4182424}{p53}, which
  is equivalent to Definition~\ref{defi:5cdb26a0af71ffa7} by
  \cite{MR633754}*{Chapitre IV, \S~3, Proposition~1}.
\end{proof}

\subsection{Reflexivity and polar reflexivity (aka.\ stereotype duality)}
\label{sec:673b6d13dedd649b}

\begin{defi}
  \label{defi:46a9774dd14859c7}
  Recall Notation~\ref{nota:6a66d03bb5d7af70}. Let \( \pairing*{E}{F} \) be a
  duality pair of vector spaces \( E \) and \( F \). The polar topology on
  \( E \) of \( \mathfrak{C}(F) \) is called the \textbf{precompact topology
    with respect to} \( \pairing*{E}{F} \) (with respect to the pairing
  \( \pairing*{E}{F} \)), and denoted by \( c(E, F) \).  Equipped with
  \( c(E, F) \), the locally convex space \( F \) is called the \textbf{polar
    dual} of \( E \) with respect to the pairing \( \pairing*{E}{F} \).
  Similarly, one obtains the notion of the precompact topology on \( E \), and
  the polar dual of \( F \), both with respect to the same pairing.
\end{defi}

When \( F = E' \) and the pairing is evaluation, we recover the polar dual
\( E'_{c} \) of \( E \) as defined in Definition~\ref{defi:3852fa5b63ee88dd}.

To facilitate our discussion, we also introduce the notion of reflexive and
polar reflexive duality pairs.

\begin{defi}
  \label{defi:d9087abf5fbcb620}
  Let \( E \), \( F \) be locally convex spaces, and \( \pairing*{E}{F} \) a
  duality pair. We say that the duality pair \( \pairing*{E}{F} \) is
  \begin{itemize}
  \item \label{defi:57594be8a177355f} \textbf{reflexive}, if it is compatible
    (Definition~\ref{defi:12b42590a911540b}), the strong topology \( b(E, F) \)
    on \( E \) with respect to the pair coincides with the topology on \( E \),
    and the strong topology \( b(F,E) \) on \( F \) coincides with the topology
    on \( F \);
  \item \label{defi:95148302c9ab2984} \textbf{polar reflexive}, if it is
    compatible (Definition~\ref{defi:12b42590a911540b}), the strong topology
    \( b(E, F) \) on \( E \) with respect to the pair coincides with the
    topology on \( E \), and the strong topology \( b(F,E) \) on \( F \)
    coincides with the topology on \( F \);
  \end{itemize}

  We say a locally convex space \( E \) is \textbf{reflexive} (resp.\
  \textbf{polar reflexive}), if the evaluation pairing \( \pairing*{E}{E'} \) is
  reflexive (resp.\ polar reflexive).
\end{defi}

The definition above of reflexivity of a locally convex space is equivalent to
the usual one (see, e.g.\ \cite{MR0342978}*{p144}). Actually, a direct unwinding
of the definition establishes the following result.
\begin{prop}
  \label{prop:a6c89f0b75eda9eb}
  Let \( E \) be a locally convex space. Consider the canonical embedding
  \( \kappa: E \to E'^{\sharp} \), \( x \mapsto \pairing*{\cdot}{x} \) of vector
  spaces.  Then the following are equivalent:
  \begin{enumerate}
  \item \label{item:2f5a9dae3c826a16} \( E \) is reflexive (resp.\ polar
    reflexive) in the sense of Definition~\ref{defi:d9087abf5fbcb620};
  \item \label{item:bf41f079d3e119d6} \( \kappa \) restricts to an isomorphism
    of locally convex spaces from \( E \) onto \( (E'_{b})'_{b} \) (resp.\
    \( (E'_{c})'_{c} \)).
  \end{enumerate}
\end{prop}

\begin{prop}
  \label{prop:6ac1bb9a146288e4}
  Let \( E \) be a locally convex space. Then,
  \begin{enumerate}
  \item \label{item:3ed7b6a910677afe} the following are equivalent
    \begin{enumerate}
    \item \label{item:8fea49133d196448} \( E \) is reflexive;
    \item \label{item:7bda6f793590ebdc} \( E'_{b} \) is reflexive;
    \item \label{item:e173ceadde4d907f} \( E \) is barrelled and every bounded
      set in \( E \) is relatively weakly compact.
    \end{enumerate}
  \item \label{item:0f58155dc720d2d8} consider the following statements:
    \begin{enumerate}
    \item \label{item:15310318d1e858fe} \( E \) is an \( (F) \)-space;
    \item \label{item:b80ad8ba8642c00f} \( E \) is quasi-complete and Mackey,
      and \( E'_{c} \) is quasi-complete;
    \item \label{item:3f346ecb9ab732a0} \( E \) is polar reflexive.
    \end{enumerate}
    Then \ref{item:15310318d1e858fe} implies \ref{item:b80ad8ba8642c00f}, which
    in turn implies \ref{item:3f346ecb9ab732a0}. Moreover,
    \ref{item:8fea49133d196448} implies \ref{item:3f346ecb9ab732a0}.
  \end{enumerate}
\end{prop}
\begin{proof}
  \ref{item:3ed7b6a910677afe} follows from \cite{MR0342978}*{p145,
    Theorem~IV.5.6 \& Corollary~1}. Now let \( E \) be an \( (F) \)-space, then
  it is complete, hence quasi-complete. Being metrizable, \( E \) is
  bornological (Proposition~\ref{prop:7963cb2832fabbd4}), hence Mackey
  (Proposition~\ref{prop:a944f1bb11e881f1}). Thus \ref{item:15310318d1e858fe}
  implies \ref{item:b80ad8ba8642c00f}. That \ref{item:b80ad8ba8642c00f} implies
  \ref{item:3f346ecb9ab732a0} is shown in \cite{MR0248498}*{p309, (4)}, and that
  \ref{item:8fea49133d196448} implies \ref{item:3f346ecb9ab732a0} in
  \cite{MR0248498}*{p309, (3)}.
\end{proof}

\begin{nota}
  \label{nota:dff66701d54396e5}
  We use \( (\mathcal{F}'_{c}) \) to denote the class of all locally convex
  spaces that are isomorphic to the polar duals of some \( (F) \)-space. Thus
  the classes \( (\mathcal{F}'_{c}) \) and \( (\mathcal{F}) \) are the polar
  duals of each other.
\end{nota}

\begin{rema}
  \label{rema:3e0976f839979573}
  What we call polar dual here of a locally convex space \( E \) is termed as
  the \textbf{stereotype dual} of \( E \), and polar reflexive spaces are called
  \textbf{stereotyped spaces}, by Akbarov \cite{MR1346445}, and is then used by
  him in his series of works. The term ``polar dual'' is adopted here mainly
  because it is systematically used in \cite{MR0248498}*{\S~23.9}, and the work
  of Köthe is much earlier. We also note that M.\ Smith characterized the polar
  duals of Banach spaces in \cite{MR0049479}, and later Brauner characterized
  the polar duals of Fréchet spaces in \cite{MR0330988}.
\end{rema}

\subsection{Hypocontinuity of bilinear maps}
\label{sec:32baf6e5b8321b30}

The notion of hypocontinuity is not directly used in the following, but it
facilitates some discussions and plays an crucial role in the development of
topological tensor products. Therefore, we include here a brief scratch of the
surface.
\begin{nota}
  \label{nota:6a6df2f58b919f3e}
  Let \( E \), \( F \), \( G \) be locally convex space. We denote by
  \( B(E, F; G) \) the space of all bilinear maps (no continuity required) from
  \( E \times F \) to \( G \). By \( \mathfrak{B}(E,F; G) \), we mean the
  subspace of \( B(E, F; G) \) of separately continuous bilinear maps; and by
  \( \mathcal{B}(E, F; G) \), the subspace of \( \mathfrak{B}(E, F; G) \) of all
  (jointly) continuous bilinear maps.
\end{nota}

Lying between the separate continuity and continuity, there's also another
fruitful notion of continuity, called hypocontinuity and is due to Bourbaki
\cite{MR633754}*{\S~III.5}, for bilinear maps.

\begin{defi}
  \label{defi:00c926c5039e93d1}
  Let \( E \), \( F \), \( G \) be locally convex spaces,
  \( \mathfrak{M} \subseteq \mathfrak{B}(E) \),
  \( \mathfrak{N} \subseteq \mathfrak{B}(F) \). A map
  \( f \in \mathfrak{B}(E, F; G) \) is called \( \mathfrak{M} \)-(resp.\
  \( \mathfrak{N} \)-)\textbf{hypocontinuous} if for any
  \( W \in \mathcal{N}^{G}_{\Gamma}(0) \) and \( A \in \mathfrak{M} \) (resp.\
  \( B \in \mathfrak{N} \)), there exists
  \( V \in \mathcal{N}^{F}_{\Gamma}(0) \) (resp.\
  \( U \in \mathcal{N}^{E}_{\Gamma}(0) \)), such that
  \( f(A \times V) \subseteq W \) (resp.\ \( f(U \times B) \subseteq W \)). We
  say \( f \) is \( (\mathfrak{M}, \mathfrak{N}) \)-\textbf{hypocontinuous}, if
  it is both \( \mathfrak{M} \)-hypocontinuous and
  \( \mathfrak{N} \)-hypocontinuous. We simply say \( f \) is
  \textbf{hypocontinuous}, if it is
  \( \bigl(\mathfrak{B}(E), \mathfrak{B}(F)\bigr) \)-hypocontinuous.
\end{defi}

\begin{prop}[\cite{MR0342978}*{p89, 5.2}]
  \label{prop:921772454be0e6bd}
  Let \( E \), \( F \), \( G \) be locally convex spaces and
  \( f \in \mathfrak{B}(E, F; G) \). If \( E \) is barrelled, then \( f \) is
  \( \mathfrak{B}(F) \)-hypocontinuous. If \( E \) and \( F \) are both
  barrelled, then \( f \) is hypocontinuous.
\end{prop}

\subsection{Compatible topologies on the tensor product}
\label{sec:696b043dbe6fbd93}

All of \S~\ref{sec:696b043dbe6fbd93} is due to Grothendieck
\cite{MR0075539}. Our presentation follows closely the treatment in
\cite{MR0551623}.

\begin{defi}[\cite{MR0551623}*{p.264}]
  \label{defi:9bd5e00768234cac}
  Let \( E \) and \( F \) be locally convex spaces, a locally convex topology
  \( \mathfrak{T}_{\tau} \) on \( E \odot F \) is called \textbf{compatible}, if
  \begin{enumerate}
  \item \label{item:bf7ac5d66f888d35} the canonical bilinear map from
    \( E \times F \) into \( E \otimes_{\tau} F = (E \odot F, \mathfrak{T}_{\tau}) \) is separately
    continuous;
  \item \label{item:0730820be0a913a1}
    \( u \otimes v \in (E \otimes_{\tau} F)' \) for all \( u \in E' \), \( v \in F' \);
  \item \label{item:a398b6ca1f6e6725} if \( A \subseteq E' \),
    \( B \subseteq F' \) are equicontinuous on \( E \) and \( F \) respectively,
    then \( A \otimes B \) is equicontinuous on \( E \otimes_{\tau} F \).
  \end{enumerate}
\end{defi}

\begin{rema}
  \label{rema:cdcb9afde04a7d07}
  By \ref{item:0730820be0a913a1}, there are enough continuous linear forms to
  separate points in \( E \otimes_{\tau} F \), hence a compatible topology on
  the tensor product is always Hausdorff.
\end{rema}

The compatible topology on the tensor product enjoys the following important property.

\begin{prop}[\cite{MR0551623}*{p265,(4)}]
  \label{prop:671875b580127e64}
  Using the above notation, let \( x_{0} \in E \), \( y_{0} \in F \) be
  \emph{nonzero}, the maps
  \( y \in F \mapsto x_{0} \otimes y \in E \otimes_{\tau} F \) and
  \( x \in E \mapsto x \otimes y_{0} \in E \otimes_{\tau} F \) are both strict
  monomorphism.
\end{prop}

It is clear that we have a duality pairing of vector spaces
\( \pairing*{E \odot F}{\mathfrak{B}(E, F)} \), given by the bilinear form
\begin{equation}
  \label{eq:2ccbd263c9730d60}
  \begin{split}
    \pairing*{\cdot}{\cdot}: (E \odot F) \times \mathfrak{B}(E, F) & \to \mathbb{K} \\
    \pairing*{\sum_{i}x_{i} \otimes y_{i}}{f} & \mapsto \sum_{i}f(x_{i}, y_{i})
  \end{split}
\end{equation}

Motivated by Proposition~\ref{prop:f9256f38e904319e}, we give the following
characterization of compatible topologies using polars.

\begin{prop}[\cite{MR0551623}*{p265,(3)}]
  \label{prop:5084d766fecfec2d}
  Using the above notation and the duality pairing \eqref{eq:2ccbd263c9730d60},
  a locally convex topology on \( E \odot F \) is compatible if and only if it
  is the polar topologies of a collection \( \mathfrak{M} \) of subsets of
  \( \mathfrak{B}(E, F) \), such that
  \begin{enumerate}
  \item \label{item:6cf381639686af9f} every \( M \in \mathfrak{M} \) is separately equicontinuous;
  \item \label{item:20f6c0d6ee2a362e} for any equicontinuous
    \( G_{1} \subseteq E' \) and \( G_{2} \subseteq F' \), we have
    \( G_{1} \otimes G_{2} \in \mathfrak{M} \), viewed canonically in
    \( \mathfrak{B}(E, F) \).
  \end{enumerate}
\end{prop}

\begin{nota}
  \label{nota:9ede2d4175a2259e}
  For any locally convex space \( E \), we use \( \mathfrak{E}(E) \) to denote
  the collection of all equicontinuous sets in \( E' \).
\end{nota}

\begin{coro}
  \label{coro:33816b8f283c3930}
  Using the notation in Proposition~\ref{prop:5084d766fecfec2d}, the following
  hold:
  \begin{enumerate}
  \item \label{item:7658c3dd7c07e174} There exists a finest compatible tensor
    product on \( E \odot F \), which is given by taking \( \mathfrak{M} \) to be the
    collection of all separately equicontinuous sets in \( \mathfrak{B}(E, F) \).
  \item \label{item:b9bef197b4c6b09e} There exists a coarsest compatible tensor
    product on \( E \odot F \), which is given by taking \( \mathfrak{M} \) to
    be the collection
    \( \set[\big]{G_{1} \otimes G_{2} \given G_{1} \in \mathfrak{E}(E), \, G_{2}
      \in \mathfrak{E}(F)} \).
  \end{enumerate}
\end{coro}

\subsection{The inductive tensor product}
\label{sec:c43804468b41e821}

We can give an alternative description of the finest compatible topology as in
Corollary~\ref{coro:33816b8f283c3930}.

\begin{prop}
  \label{prop:5a80b25896f61451}
  Let \( E \), \( F \) be locally convex spaces. Let \( \mathfrak{T}_{\iota} \)
  be the inductive topology on \( E \odot F \) with respect to the family
  \begin{equation}
    \label{eq:308aaf3235ce3c87}
    \set*{x \otimes (\cdot): F \to E \odot F \given x \in E}
    \cup \set*{(\cdot) \otimes y: F \to E \odot F \given y \in F},
  \end{equation}
  then \( \mathfrak{T}_{\iota} \) is the finest compatible topology on
  \( E \odot F \).
\end{prop}
\begin{proof}
  By condition~\ref{item:bf7ac5d66f888d35} of
  Definition~\ref{defi:9bd5e00768234cac}, \( \mathfrak{T}_{\iota} \) is finer
  than every compatible topology on \( E \odot F \). It remains to check that it
  is compatible. For all \( u \in E' \), \( v \in F' \), let \( \tau_{u,v} \) be
  the projective topology on \( E \odot F \) with respect to
  \( u \otimes v: E \odot F \to \mathbb{K} \). Then we have form
  \( (u \otimes v)\bigl(x \otimes (\cdot)\bigr) = v \in F' \), and
  \( (u \otimes v)\bigl((\cdot) \otimes y\bigr) = u \in E' \) for all
  \( x \in E \), \( y \in F \). Hence \( \tau_{u, v} \) is coarser than
  \( \mathfrak{T}_{\iota} \) by the definition of the inductive topology, and
  \( u \otimes v \in \bigl((E \otimes F)[\mathfrak{T}_{\iota}]\bigr)' \). It
  remains to check that if \( G \in \mathfrak{E}(E) \),
  \( H \in \mathfrak{E}(F) \), then
  \( G \otimes H \in \mathfrak{E}\bigl((E \otimes F)[\mathfrak{T}_{\iota}]\bigr)
  \). It suffices to show that the polar \( (G \otimes H)^{\circ} \) with
  respect to the canonical pairing \( \pairing*{E \odot F}{E' \odot F'} \) is a
  neighborhood of \( 0 \) for \( \mathfrak{T}_{\iota} \), which holds if and
  only if the inverse image of \( (G \otimes H)^{\circ} \) along each map in the
  family \eqref{eq:308aaf3235ce3c87} is a neighborhood of \( 0 \) of the domain
  of the map. Take any \( x \in E \), since \( G \) is equicontinuous, the polar
  \( G^{\circ} \) (with respect to the evaluation pairing \( \pairing*{E}{E'} \)
  of course) is a neighborhood of \( 0 \in E \), hence there exists
  \( \varepsilon > 0 \), such that for any \( \lambda \in \mathbb{K} \) with
  \( \abs*{\lambda} \leq \varepsilon \), we have \( \lambda x \in G^{\circ}
  \). If \( x \ne 0 \), then
  \( x \otimes (\cdot) = \varepsilon x \otimes \varepsilon^{-1}(\cdot) \). Thus
  \begin{equation}
    \label{eq:d86ab27882088e45}
    \forall y \in \varepsilon H^{\circ} \in \mathcal{N}^{F}_{\Gamma}(0), \qquad
    x \otimes y = \varepsilon x \otimes \varepsilon^{-1}y \in \varepsilon x \otimes H^{\circ} \subseteq G^{\circ} \otimes H^{\circ} \subseteq (G \otimes H)^{\circ}.
  \end{equation}
  Let \( f_{x} : F \to E \odot F \) be the map \( y \mapsto x \). Then
  \eqref{eq:d86ab27882088e45} means
  \( \varepsilon H^{\circ} \subseteq f_{x}^{-1}\bigl((G \otimes H)^{\circ}\bigr)
  \), thus the latter is in \( \mathcal{N}^{F}_{\Gamma}(0) \). Similarly, let
  \( f^{y}: E \to E \odot F \) be the map \( x \mapsto x \otimes y \), we have
  \( (f^{y})^{-1}\bigl((G \otimes H)^{\circ}\bigr) \in
  \mathcal{N}^{E}_{\Gamma}(0) \). As \( f^{0} \) and \( f_{0} \) are the zero
  map, which are the trivial cases, the inverse image of
  \( (G \otimes H)^{\circ} \) along any maps in the family
  \eqref{eq:308aaf3235ce3c87} is indeed a neighborhood of \( 0 \), and the proof
  is complete.
\end{proof}

\begin{defi}
  \label{defi:f447ded1ddad6051}
  In Proposition~\ref{prop:5a80b25896f61451}, the topology
  \( \mathfrak{T}_{\iota} \) on \( E \odot F \) is called the \textbf{inductive
    tensor product topology} on \( E \odot F \), and we denote the resulting
  locally convex space by \( E \otimes_{\iota} F \), and call it the
  \textbf{inductive tensor product}. The completion of \( E \otimes_{\iota} F \)
  as a locally convex space, denoted by \( E \overline{\otimes}_{\iota} F \), is
  called the \textbf{completed inductive tensor product}.
\end{defi}

\begin{prop}[Universal property of inductive tensor product]
  \label{prop:6989e716dae21a7b}
  Let \( E \), \( F \) and \( G \) be locally convex spaces, then for every
  bilinear map \( f : E \times F \to G \), we have
  \( f \in \mathfrak{B}(E, F; G) \) if and only if there exists a unique
  \( \widetilde{f}: E \otimes_{\iota} F \to G \), such that
  \( f = \widetilde{f}\chi \), with
  \( \chi: E \times F \to E \otimes_{\iota}F \) being the canonical map.
\end{prop}
\begin{proof}
  This follows Proposition~\ref{prop:5a80b25896f61451} and the universal
  property of inductive topologies.
\end{proof}

\begin{coro}
  \label{coro:9dd4c5414e7497a4}
  For \( E \in \mathsf{LCS} \), consider the functors
  \( E \otimes_{\iota}(\cdot) \), \( (\cdot) \otimes_{\iota} E \) and
  \( \mathcal{L}_{s}(E, \cdot) \) from \( \mathsf{LCS} \) to itself. Then
  currying yields the natural bijections
  \begin{equation}
    \label{eq:1237ef8f1ce74fce}
    \mathcal{L}(E \otimes_{\iota} F, G) \simeq  \mathcal{L}(F, \mathcal{L}_{s}(E, G)) \simeq \mathcal{L}(F \otimes_{\iota}E, G),
    \qquad F, G \in \mathsf{LCS}.
  \end{equation}
  In particular, \( E \otimes_{\iota}(\cdot) \) and
  \( (\cdot) \otimes_{\iota} E \) are both left adjoint to
  \( \mathcal{L}_{s}(E, \cdot) \), hence commute with taking colimits in
  \( \mathsf{LCS} \).
\end{coro}
\begin{proof}
  This is a routine check using the universal property of the inductive tensor
  product as well as the algebraic tensor product of vector spaces.
\end{proof}

\begin{rema}
  \label{rema:4c49ce7f5d1573ae}
  One might be tempted to say that \( E \overline{\otimes}_{\iota}(\cdot) \)
  commutes with taking colimits in the category \( \widehat{\mathsf{LCS}} \) of
  complete locally convex spaces, by claiming it is the left adjoint of
  \( \widehat{\mathcal{L}_{s}(E, \cdot)} \), as functors on
  \( \widehat{\mathsf{LCS}} \). This is however \emph{false}, since
  \( \mathcal{L}_{s}(E, G) \) is in general not complete, hence a similar
  adjunction as \eqref{eq:1237ef8f1ce74fce} need not hold in
  \( \widehat{\mathsf{LCS}} \), and the left side of \eqref{eq:1237ef8f1ce74fce}
  will miss to detect continuous linear maps from \( F \) to
  \( \widehat{\mathcal{L}_{s}(E, G)} \) whose range is not contained in
  \( \mathcal{L}_{s}(E, G) \).
\end{rema}

In spite of Remark~\ref{rema:4c49ce7f5d1573ae}, we still have the following
commutation result with locally convex direct sums.
\begin{coro}
  \label{coro:06a795bc15551930}
  Let \( (E_{i})_{i \in I} \) and \( (F_{j \in J}) \) be two families of locally
  convex spaces, we have
  \begin{equation}
    \label{eq:9f32254c2561fd04}
    \left(\bigoplus_{i}E_{i}\right) \overline{\otimes}_{\iota}
    \left(\bigoplus_{j} F_{j}\right)
    = \widehat{\bigoplus_{i,j}E_{i} \otimes_{\iota} F_{j}}
    = \bigoplus_{i,j} E_{i} \overline{\otimes}_{j} F_{j}.
  \end{equation}
\end{coro}
\begin{proof}
  By Proposition~\ref{prop:6989e716dae21a7b}, we have
  \begin{equation}
    \label{eq:7e6e56eb324b6c83}
    \left(\bigoplus_{i}E_{i}\right) \otimes_{\iota}
    \left(\bigoplus_{j} F_{j}\right)
    = \bigoplus_{i,j}E_{i} \otimes_{\iota} F_{j}
    = \bigoplus_{i,j} E_{i} \otimes_{j} F_{j}.
  \end{equation}
  Now \eqref{eq:9f32254c2561fd04} follows from \eqref{eq:7e6e56eb324b6c83} by
  Proposition~\ref{prop:d842a95c764bfd01}.
\end{proof}

\begin{prop}
  \label{prop:963932029cf324b4}
  The following holds for the (completed) inductive tensor product.
  \begin{enumerate}
  \item \label{item:e8987e1abfde6201} If \( f: A \to E \), \( g : B \to F \) are
    continuous linear maps of locally convex spaces, then the induced linear map
    \( f \otimes g: A \otimes_{\iota} B \to E \otimes_{\iota} F \) is
    continuous, hence extends by continuity to a continuous linear map
    \( f \overline{\otimes}_{\iota} g: A \overline{\otimes}_{\iota} B \to E
    \overline{\otimes}_{\iota} F \).
  \item \label{item:a9e9f676c5d31d25} For all locally convex spaces \( E \) and
    \( F \), the swap
    \begin{displaymath}
      s: E \otimes_{\iota} F \to F \otimes_{\iota} E ,\quad
      \sum_{i}x_{i} \otimes y_{i} \mapsto \sum_{i}y_{i} \otimes x_{i}
    \end{displaymath}
    is an isomorphism of locally convex spaces, hence extends by continuity to
    an isomorphism of locally convex spaces
    \( \overline{s}: E \overline{\otimes}_{\iota} F \to F
    \overline{\otimes}_{\iota} E \).
  \item \label{item:913b343d957a2af9} For all locally convex spaces \( E \),
    \( F \) and \( G \), the association map
    \begin{displaymath}
      a: (E \otimes_{\iota} F) \otimes_{\iota} G \to E \otimes_{\iota} (F
      \otimes_{\iota} G), \quad \sum_{i}(x_{i} \otimes y_{i}) \otimes z_{i} \mapsto \sum_{i} x_{i}
      \otimes (y_{i} \otimes z_{i})
    \end{displaymath}
    is an isomorphism of locally convex spaces, hence extends by continuity to
    an isomorphism of locally convex spaces
    \( \overline{a}: (E \overline{\otimes}_{\iota} F) \overline{\otimes}_{\iota}
    G \to E \overline{\otimes}_{\iota} (F \overline{\otimes}_{\iota} G) \).
  \end{enumerate}
\end{prop}
\begin{proof}
  \ref{item:e8987e1abfde6201} follows from the universal property of
  \( A \otimes_{\iota} B \) by noting that the composition
  \( A \times B \xrightarrow{f \times g} E \times F \xrightarrow{\chi_{E,F}} E
  \otimes_{\iota} F \) is a separately continuous bilinear map.
  \ref{item:a9e9f676c5d31d25} is similar to \ref{item:e8987e1abfde6201}. To
  prove \ref{item:913b343d957a2af9}, note that a similar argument as in
  Proposition~\ref{prop:5a80b25896f61451} establishes that there exists a finest
  locally convex topology on \( E \odot F \odot G \) making the trilinear map
  \( (x, y, z) \in E \times F \times G \mapsto x \otimes y \otimes z \in E \odot
  F \odot G \) separately continuous. And one checks that this locally convex
  topology on \( E \odot F \odot G \) transports to exactly both sides of
  \( a \) via the canonical identifications.
\end{proof}

\subsection{The injective tensor product}
\label{sec:ac51753bf78a1da5}

From Corollary~\ref{coro:33816b8f283c3930}, we can also deduce that there exists
a coarsest compatible topology on \( E \odot F \).

\begin{defi}
  \label{defi:4eeb99ac2dee068f}
  Let \( E \), \( F \) be locally convex space. The space \( E \odot F \),
  equipped coarsest compatible topology, is called the \textbf{injective tensor
    product} of \( E \) and \( F \), and is denoted by
  \( E \otimes_{\varepsilon} F \), its completion, denoted by
  \( E \overline{\otimes}_{\varepsilon} F \), is called the \textbf{completed
    tensor product}.
\end{defi}

We may embed \( E \otimes F \) into \( \mathfrak{B}(E'_{\sigma}, F'_{\sigma}) \)
by identifying the tensor \( \sum_{i}x_{i} \otimes y_{i} \) with the bilinear
form \( (\alpha, \beta) \mapsto \sum_{i}\alpha(x_{i})\beta(y_{i}) \).

\begin{nota}
  \label{nota:6f4b4313d9c387ba}
  Define
  \begin{displaymath}
    \mathfrak{E} = \set*{A \times B \given A \in \mathfrak{E}(E), \, B \in \mathfrak{E}(B)}.
  \end{displaymath}
  Let \( G \) be another locally convex space. We use
  \( \mathfrak{B}_{e}(E'_{s}, F'_{s}; G) \) to denote the
  \( \mathfrak{S} \)-topology on \( \mathfrak{B}(E'_{s}, F'_{s}; G) \) viewed as
  a space of maps from \( E'_{s} \times F'_{s} \) to \( G \).  When
  \( G = \mathbb{K} \), we use \( \mathfrak{B}_{e}(E'_{s}, F'_{s}) \) to denote
  \( \mathfrak{B}_{e}(E'_{s}, F'_{s}; \mathbb{K}) \).
\end{nota}

\begin{prop}
  \label{prop:6a4e157f8736cd87}
  Using the above notation the following hold:
  \begin{enumerate}
  \item \label{item:5bd123c30080cc97} the
    \( \mathfrak{B}_{e}(E'_{s}, F'_{s}; G) \) is a locally convex space;
  \item \label{item:b3b706cbbaa67a3d} \( \mathfrak{B}_{e}(E'_{e}, F'_{e}) \) is
    complete if both \( E \) and \( F \) are;
  \item \label{item:6aa6d16fe7fa8c79} the embedding
    \( E \otimes_{\varepsilon} F \hookrightarrow \mathfrak{B}_{e}(E'_{s},
    F'_{s}) \) is a strict monomorphism.
  \end{enumerate}
  In particular, if both \( E \) and \( F \) are complete, then
  \( E \overline{\otimes}_{\varepsilon} F \) is identified with the closure of
  \( E \otimes_{\mathfrak{\varepsilon}} F \) in
  \( \mathfrak{B}_{e}(E'_{s}, F'_{s}) \).
\end{prop}
\begin{proof}
  \ref{item:5bd123c30080cc97} holds by \cite{MR0551623}*{p167,(4)}, and
  \ref{item:b3b706cbbaa67a3d} by \cite{MR0551623}*{p167,(5)}.  We now prove
  \ref{item:6aa6d16fe7fa8c79}. Let \( B_{\mathbb{K}} \) denote the unit disk of
  the scalar field \( \mathbb{K} \). Consider sets of the form
  \begin{equation}
    \label{eq:d792b216757c8477}
    M(U, V) = \set*{f \in \mathfrak{B}(E'_{s}, F'_{s}) \given f(U^{\circ}, V^{\circ}) \subseteq B_{\mathbb{K}}},
    \qquad (U, V) \in \mathcal{N}^{E}_{\Gamma}(0) \times \mathcal{N}^{F}_{\Gamma}(0).
  \end{equation}
  Since every equicontinuous set in \( E'_{s} \) (resp.\ \( F'_{s} \)) is
  contained in \( U^{\circ} \) (resp.\ \( V^{\circ} \)) for some
  \( U \in \mathcal{N}_{\Gamma}^{E}(0) \) (resp.\
  \( V \in \mathcal{N}_{\Gamma}^{E}(0) \)), it follows that all sets in
  \eqref{eq:d792b216757c8477} form a fundamental system of neighborhoods of
  \( 0 \) in \( \mathfrak{B}_{e}(E'_{s}, F'_{s}) \). Clearly,
  \( M(U, V) \cap (E \odot F) = (U^{\circ} \otimes V^{\circ})^{\circ} \), where
  the outer polar is taken with respect to the duality pairing
  \eqref{eq:2ccbd263c9730d60}. It follows that the subspace topology on
  \( E \odot F \) induced from the topology of
  \( \mathfrak{B}_{e}(E'_{s}, F'_{s}) \) is precisely the injective tensor
  product topology on \( E \odot F \), which finishes the proof.
\end{proof}

\begin{rema}
  \label{rema:3df7d46a1f1a52fb}
  Due to the above proposition, the injective tensor product topology is also
  called the topology of bi-equicontinuous convergence, which is the original
  terminology in \cite{MR0075539}.
\end{rema}

The following description of \( E \otimes_{\varepsilon} F \) (and of
\( E \overline{\otimes}_{\varepsilon} F \)) using seminorms is a consequence by
a routine translation between a continuous semi-norm and its associate unit
ball, via the Minkowski functionals (aka.\ the gauge functionals).

\begin{prop}[\cite{MR0551623}*{p267, (2) \& (3)}]
  \label{prop:2204037e873ab1a3}
  Let \( (p_{i})_{i \in I} \) (resp.\ \( (q_{j})_{j \in J} \)) be a generating
  family of seminorms for the locally convex space \( E \) (resp.\ \( F \)). For
  each index, let \( U_{i} \) (resp.\ \( V_{j} \)) be the unit ball for the
  seminorm \( p_{i} \) (resp.\ \( q_{j} \)). Define
  \begin{displaymath}
    \begin{split}
      p_{i} \otimes_{\varepsilon} q_{j}: E \odot F &\to \R_{+} = \interval[open right]{0}{+\infty} \\
      t &\mapsto \sup_{u \in U_{i}^{\circ}, v \in V_{j}^{\circ}}\abs*{(u \otimes v)(t)},
    \end{split}
  \end{displaymath}
  then \( p_{i} \otimes_{\varepsilon} q_{j} \), \( (i,j) \in I \times J \), is a
  generating family of seminorms for \( E \otimes_{\varepsilon} F \) and the
  respective extensions \( p_{i} \overline{\otimes}_{\varepsilon} q_{j} \),
  \( (i, j) \in I \times J \), is a generating family of seminorms on
  \( E \overline{\otimes}_{\varepsilon} F \).
\end{prop}

\begin{rema}
  \label{rema:339c7da9f040d789}
  When \( E \), \( F \) are \( B \)-spaces, or more generally normed spaces,
  taking \( p \) to be the norm on \( E \), and \( q \) the norm on \( F \),
  then \( p \otimes_{\varepsilon} q \) is a norm generating the topology on
  \( E \otimes_{\varepsilon} F \). Using this norm, we recover the usual
  definition of injective tensor product of Banach (or normed) spaces, see e.g.\
  \cite{takesaki_theory_2002}*{\S~IV.2}.
\end{rema}

\begin{coro}
  \label{coro:97adf837ae63af02}
  If \( E \), \( F \) are \( (F) \)-spaces, then so is \( E \overline{\otimes}_{\varepsilon} F \).
\end{coro}
\begin{proof}
  It suffices to take a countable generating family of seminorms for \( E \) and
  for \( F \) in Proposition~\ref{prop:2204037e873ab1a3}.
\end{proof}

\begin{prop}
  \label{prop:51d98c7dd4f5950d}
  The following holds for the (completed) injective tensor product.
  \begin{enumerate}
  \item \label{item:4410bd42693473ea} If \( f: A \to E \), \( g : B \to F \) are
    continuous linear maps of locally convex spaces, then the induced linear map
    \( f \otimes g: A \otimes_{\varepsilon} B \to E \otimes_{\varepsilon} F \) is
    continuous, hence extends by continuity to a continuous linear map
    \( f \overline{\otimes}_{\varepsilon} g: A \overline{\otimes}_{\varepsilon} B \to E
    \overline{\otimes}_{\varepsilon} F \).
  \item \label{item:604a5aa25ff0ca12} For all locally convex spaces \( E \) and
    \( F \), the swap
    \begin{displaymath}
      s: E \otimes_{\varepsilon} F \to F \otimes_{\varepsilon} E ,\quad
      \sum_{i}x_{i} \otimes y_{i} \mapsto \sum_{i}y_{i} \otimes x_{i}
    \end{displaymath}
    is an isomorphism of locally convex spaces, hence extends by continuity to
    an isomorphism of locally convex spaces
    \( \overline{s}: E \overline{\otimes}_{\varepsilon} F \to F
    \overline{\otimes}_{\varepsilon} E \).
  \item \label{item:37a1c4b17ef9a65c} For all locally convex spaces \( E \),
    \( F \) and \( G \), the association map
    \begin{displaymath}
      a: (E \otimes_{\varepsilon} F) \otimes_{\varepsilon} G \to E \otimes_{\varepsilon} (F
      \otimes_{\varepsilon} G), \quad \sum_{i}(x_{i} \otimes y_{i}) \otimes z_{i} \mapsto \sum_{i} x_{i}
      \otimes (y_{i} \otimes z_{i})
    \end{displaymath}
    is an isomorphism of locally convex spaces, hence extends by continuity to
    an isomorphism of locally convex spaces
    \( \overline{a}: (E \overline{\otimes}_{\varepsilon} F) \overline{\otimes}_{\varepsilon}
    G \to E \overline{\otimes}_{\varepsilon} (F \overline{\otimes}_{\varepsilon} G) \).
  \end{enumerate}
\end{prop}
\begin{proof}
  \ref{item:4410bd42693473ea} is
  \cite{MR0551623}*{p275,(1)}. \ref{item:604a5aa25ff0ca12} follows directly from
  Definition~\ref{defi:4eeb99ac2dee068f} based on
  Corollary~\ref{coro:33816b8f283c3930}. \ref{item:37a1c4b17ef9a65c} can be seen
  from, e.g.\ Proposition~\ref{prop:2204037e873ab1a3} by writing out the
  corresponding generating families of seminorms.
\end{proof}

The following result justifies our terminology of using ``injective tensor
product'' instead of the more precise term ``tensor product of bi-equicontinuous
convergence'' as Grothendieck did (\cite{MR0075539}*{p89, Définition~5})

\begin{prop}[\cite{MR0551623}*{p278, (6)}]
  \label{prop:710a683d733970db}
  Let \( E_{i} \), \( F_{i} \) be locally convex spaces and
  \( f_{i}: E_{i} \to F_{i} \) a strict monomorphism for \( i = 1,2 \). Then
  both
  \( f_{1} \otimes_{\varepsilon} f_{2} : E_{1} \otimes_{\varepsilon} F_{1} \to
  E_{2} \otimes_{\varepsilon} F_{2} \) and
  \( f_{1} \overline{\otimes} f_{2}: E_{1} \overline{\otimes}_{\varepsilon}
  F_{1} \to E_{2} \overline{\otimes} F_{2} \) are strict monomorphisms.
\end{prop}

The completed injective tensor product behaves well with products and reduced
projective limits.

\begin{prop}[\cite{MR0551623}*{p282, (5)}]
  \label{prop:8056525c3c4c9aa5}
  The following hold.
  \begin{enumerate}
  \item \label{item:2479d97cbebb5d9d} Let \( E \) be a locally convex space,
    \( (F_{j})_{j \in J} \) a family of locally convex spaces. We have a
    canonical isomorphism
    \( E \overline{\otimes}_{\varepsilon} \prod_{j \in J} F_{j} \simeq \prod_{j
      \in J} E \overline{\otimes}_{\varepsilon} F_{j} \).
  \item \label{item:ab3601b8e4ba1df9} Let \( E = \varprojlim E_{i} \),
    \( F = \varprojlim F_{j} \) be reduced projective limits of locally convex
    spaces, then \( E \overline{\otimes}_{\varepsilon} F \) can be canonically
    identified as the projective limit
    \( \varprojlim E_{i} \overline{\otimes}_{\varepsilon} F_{j} \), which is
    already complete by Proposition~\ref{prop:438d10ce1f2149df}.
  \end{enumerate}
\end{prop}

\begin{rema}
  \label{rema:16eeb08f48eda694}
  The injective tensor product does not behave well with taking locally convex
  direct sums, see e.g.\ \cite{MR0551623}*{p283} for a counterexample, as well a
  partial positive result.
\end{rema}

\subsection{The projective tensor product}
\label{sec:33ea3663864561d2}

Informally, just as inductive tensor product is the universal one witnessing
separate continuity of bilinear maps, the projective tensor product is the
universal one witnessing (joint) continuity. This is formalized in the following
result.

\begin{prop}[\cite{MR0551623}*{pp176,177}]
  \label{prop:dda5a3f4762f22f6}
  Let \( E \), \( F \) be locally convex spaces. The absolutely convex hulls
  \( \Gamma(U \otimes V) \), \( U \in \mathcal{N}^{E}_{\Gamma}(0) \),
  \( V \in \mathcal{N}^{F}_{\Gamma}(0) \) form a fundamental system of
  neighborhoods of \( 0 \) for a unique locally convex Hausdorff topology
  \( \mathfrak{T}_{\pi} \) on \( E \odot F \). Moreover,
  \begin{enumerate}
  \item \label{item:9fadcb3c3659cfa1} Let \( p_{U} \) (resp.\ \( q_{V} \)) be
    the gauge functional for \( U \), then the gauge functional
    \( p_{U} \otimes_{\pi} q_{V} \) for \( \Gamma(U \otimes V) \) is given by the formula
    \begin{equation}
      (p_{U} \otimes_{\pi} q_{V})(t) = \inf\set[\big]{\sum_{k}p_{U}(x_{k})q_{V}(y_{k})
        \given t = \sum_{k} x_{k} \otimes y_{k}, \, (x_{k}, y_{k}) \in E \times F}.
    \end{equation}
    In particular, \( p_{U} \otimes_{\pi} q_{V} \) is (uniformly) continuous and
    extends by continuity to a unique continuous seminorm on the completion of
    \( (E \odot F, \mathfrak{T}_{\pi}) \).
  \item \label{item:81e89e6631572c66} If \( (p_{i}) \), \( (q_{j}) \) are
    generating families of seminorms for \( E \) and \( F \) respectively, then
    \( p_{i} \otimes_{\pi} q_{j} \), \( (i,j) \in I \times J \) is a generating
    family of seminorms for
    \( E \otimes_{\pi}F:= (E \odot F, \mathfrak{T}_{\pi}) \); and the unique
    continuous extension \( p_{i} \overline{\otimes}_{\pi} q_{j} \),
    \( (i,j) \in I \times J \) is a generating family of seminorms for the
    completion \( E \overline{\otimes}_{\pi} F \) of \( E \otimes_{\pi} F \).
  \item \label{item:4a6bca6839439643} Let \( \chi: E \times F \to E \odot F \)
    be the canonical map. Then \( \mathfrak{T}_{\pi} \) is the finest locally
    convex topology making \( \chi \) continuous.
  \item \label{item:b0505895928e3373} The locally convex tensor product
    \( E \otimes_{\pi} F \) enjoys the following universal property: for each
    locally convex space \( G \), the map
    \( \mathcal{L}(E \otimes_{\pi} F, G) \to \mathcal{B}(E, F; G) \),
    \( f \mapsto f\chi \) is a bijection.
  \end{enumerate}
\end{prop}

\begin{coro}
  \label{coro:43e2166622970bc3}
  Using the notation in Proposition~\ref{prop:dda5a3f4762f22f6}, the topology
  \( \mathfrak{T}_{\pi} \) on \( E \odot F \) is compatible in the sense of
  Definition~\ref{defi:9bd5e00768234cac}.
\end{coro}
\begin{proof}
  The only nontrivial thing to check is that \( U^{\circ} \otimes V^{\circ} \)
  is equicontinuous in \( (E \otimes_{\pi} F)' \), for any
  \( U \in \mathcal{N}^{E}_{\Gamma}(0) \) and
  \( V \in \mathcal{N}^{F}_{\Gamma}(0) \). For each
  \( t = \sum_{k} \lambda_{k}x_{k} \otimes y_{k} \in \Gamma(U \otimes V) \),
  \( x_{k} \in U \), \( y_{k} \in V \), and \( \lambda_{k} \in \mathbb{K} \)
  with \( \sum_{k} \abs*{\lambda_{k}} \leq 1 \), as well as each
  \( f \in U^{\circ} \), \( g \in V^{\circ} \), we have
  \begin{displaymath}
    \abs*{\pairing*{t}{f \otimes g}} = \abs*{\sum_{k} \lambda_{k} f(x_{k}) g(y_{k})}
    \leq \sum_{k} \abs*{\lambda_{k}} \cdot \abs*{f(x_{k})} \cdot \abs*{g(y_{k})}
    \leq \sum_{k} \abs*{\lambda_{k}} \leq 1.
  \end{displaymath}
  Hence \( U^{\circ} \otimes V^{\circ} \) lies in the polar of the neighborhood
  \( \Gamma(U \otimes V) \) of \( 0 \) in \( E \otimes_{\pi} F \), hence
  equicontinuous.
\end{proof}

\begin{coro}
  \label{coro:0be09f15f7e01dce}
  If \( E \) and \( F \) are \( (F) \)-spaces, then so is \( E \overline{\otimes}_{\pi} F \).
\end{coro}
\begin{proof}
  It suffices to take a countable generating family of seminorms for \( E \) and
  for \( F \) in Proposition~\ref{prop:dda5a3f4762f22f6}
  \ref{item:81e89e6631572c66}.
\end{proof}

\begin{defi}
  \label{defi:8c492948622145a0}
  Using the notation in Proposition~\ref{prop:dda5a3f4762f22f6}, we call the
  locally convex space \( E \otimes_{\pi} F \) the \textbf{projective tensor
    product} of \( E \) and \( F \), and its completion
  \( E \overline{\otimes}_{\pi} F \) the \textbf{completed projective tensor
    product}.
\end{defi}

\begin{rema}
  \label{rema:9a1ecf4df00f5761}
  When \( E \) and \( F \) are normed spaces, and \( U \), \( V \) their
  respective unit balls, then \( p_{U} \) and \( p_{V} \) are the norms on
  \( E \) and \( F \), and \eqref{eq:d792b216757c8477} describes the norm on the
  projective tensor product \( E \otimes_{\pi} F \) of normed spaces, cf.\
  \cite{takesaki_theory_2002}*{\S~IV.2}.
\end{rema}

\begin{prop}
  \label{prop:6f861986d67606e0}
  The following holds for the (completed) projective tensor product.
  \begin{enumerate}
  \item \label{item:23f9770ba3660a41} If \( f: A \to E \), \( g : B \to F \) are
    continuous linear maps of locally convex spaces, then the induced linear map
    \( f \otimes g: A \otimes_{\pi} B \to E \otimes_{\pi} F \) is
    continuous, hence extends by continuity to a continuous linear map
    \( f \overline{\otimes}_{\pi} g: A \overline{\otimes}_{\pi} B \to E
    \overline{\otimes}_{\pi} F \).
  \item \label{item:24920149d656ce66} For all locally convex spaces \( E \) and
    \( F \), the swap
    \begin{displaymath}
      s: E \otimes_{\pi} F \to F \otimes_{\pi} E ,\quad
      \sum_{i}x_{i} \otimes y_{i} \mapsto \sum_{i}y_{i} \otimes x_{i}
    \end{displaymath}
    is an isomorphism of locally convex spaces, hence extends by continuity to
    an isomorphism of locally convex spaces
    \( \overline{s}: E \overline{\otimes}_{\pi} F \to F
    \overline{\otimes}_{\pi} E \).
  \item \label{item:a640f255a320010a} For all locally convex spaces \( E \),
    \( F \) and \( G \), the association map
    \begin{displaymath}
      a: (E \otimes_{\pi} F) \otimes_{\pi} G \to E \otimes_{\pi} (F
      \otimes_{\pi} G), \quad \sum_{i}(x_{i} \otimes y_{i}) \otimes z_{i} \mapsto \sum_{i} x_{i}
      \otimes (y_{i} \otimes z_{i})
    \end{displaymath}
    is an isomorphism of locally convex spaces, hence extends by continuity to
    an isomorphism of locally convex spaces
    \( \overline{a}: (E \overline{\otimes}_{\pi} F) \overline{\otimes}_{\pi}
    G \to E \overline{\otimes}_{\pi} (F \overline{\otimes}_{\pi} G) \).
  \end{enumerate}
\end{prop}
\begin{proof}
  The proof parallels to the case for inductive tensor product
  (Proposition~\ref{prop:963932029cf324b4}) by using continuous bilinear or
  trilinear maps instead of the separately continuous ones.
\end{proof}

The completed projective tensor product also behaves well with products and
reduced projective limits.

\begin{prop}[\cite{MR0551623}*{p193 (3), \& p194, (5)}]
  \label{prop:046bbcf972e4452d}
  The following hold.
  \begin{enumerate}
  \item \label{item:c7eb6fdad9e43d69} Let \( E \) be a locally convex space,
    \( (F_{j})_{j \in J} \) a family of locally convex spaces. We have a
    canonical isomorphism
    \( E \overline{\otimes}_{\pi} \prod_{j \in J} F_{j} \simeq \prod_{j
      \in J} E \overline{\otimes}_{\pi} F_{j} \).
  \item \label{item:ccfeff3bc431544d} Let \( E = \varprojlim E_{i} \),
    \( F = \varprojlim F_{j} \) be reduced projective limits of locally convex
    spaces, then \( E \overline{\otimes}_{\pi} F \) can be canonically
    identified as the projective limit
    \( \varprojlim E_{i} \overline{\otimes}_{\pi} F_{j} \), which is
    already complete by Proposition~\ref{prop:438d10ce1f2149df}.
  \end{enumerate}
\end{prop}

\begin{rema}
  \label{rema:304c55c05ed83842}
  As injective tensor product, in general, the project tensor product does not
  behave well with locally convex direct sums, see \cite{MR0551623}*{p195, (6),
    (7)} for some partial positive result, and the counterexample in
  \cite{MR0551623}*{p196}.
\end{rema}

\subsection{Comparison of the topological tensor products and nuclear spaces}
\label{sec:27526e4ba41269f2}

\begin{prop}
  \label{prop:6ecdcba4c320179a}
  Let \( E \), \( F \) be locally convex spaces. Then
  \begin{enumerate}
  \item \label{item:14b17bbd79368ce2} The identity map
    \( E \otimes_{\iota} F \to E \otimes_{\pi} F \) is continuous, and extends
    by continuity to a unique continuous linear map
    \( E \overline{\otimes}_{\iota} F \to E \overline{\otimes}_{\pi} F \).
  \item \label{item:e88486857a3d79b6} The identity map
    \( E \otimes_{\pi} F \to E \otimes_{\varepsilon} F \) is continuous, and extends
    by continuity to a unique continuous linear map
    \( E \overline{\otimes}_{\pi} F \to E \overline{\otimes}_{\varepsilon} F \).
  \end{enumerate}
\end{prop}
\begin{proof}
  All of this follows from projective tensor product topology is a compatible
  topology (Corollary~\ref{coro:43e2166622970bc3}), and among compatible
  topologies, and the inductive tensor product topology is the finest, while the
  injective one coarsest.
\end{proof}

\begin{defi}
  \label{defi:aec05a12ba4169ac}
  All the linear maps in Proposition~\ref{prop:6ecdcba4c320179a} are called
  \textbf{canonical}.
\end{defi}

It shall be important for us to have some criterion for the canonical maps to be
isomorphisms in \( \mathsf{LCS} \). We start with the comparison of the
inductive and projective tensor product.

\begin{prop}
  \label{prop:89e27b078637bfc2}
  Let \( E \), \( F \) be locally convex spaces. Then the following are
  equivalent:
  \begin{enumerate}
  \item \label{item:45939c3a32b69dd9} the canonical map
    \( E \otimes_{\iota}F \to E \otimes_{\pi} F \) is an isomorphism;
  \item \label{item:1d6b11df9106f9bd} the canonical map
    \( E \overline{\otimes}_{\iota}F \to E \overline{\otimes}_{\pi} F \) is an isomorphism;
  \item \label{item:8b3ae8da37750c0d} the canonical map
    \( \chi : E \times F \to E \otimes_{\iota} F \) is continuous;
  \item \label{item:94a8bacfaa6dd579} for any locally convex space \( G \), we
    have \( \mathfrak{B}(E, F; G) = \mathcal{B}(E, F; G) \), i.e.\ every
    separately continuous bilinear map from \( E \times F \) to \( G \) is
    continuous.
  \end{enumerate}
\end{prop}
\begin{proof}
  It follows from the uniqueness of the completion of uniform spaces
  (Theorem~\ref{theo:5c48561db339a0d3}) that \ref{item:45939c3a32b69dd9} and
  \ref{item:1d6b11df9106f9bd} are equivalent. By the universal properties of the
  inductive and projective tensor product respectively, we see that
  \ref{item:45939c3a32b69dd9} and \ref{item:94a8bacfaa6dd579} are
  equivalent. Clearly, \ref{item:94a8bacfaa6dd579} implies
  \ref{item:8b3ae8da37750c0d}. Finally, suppose \ref{item:8b3ae8da37750c0d}
  holds, by the universal property of the projective tensor product, we see that
  the identity map \( E \otimes_{\pi} F \to E \otimes_{\iota} F \) is
  continuous, and is clearly the inverse of the continuous canonical map
  \( E \otimes_{\iota} F \to E \otimes_{\pi}F \), hence we have
  \ref{item:45939c3a32b69dd9}.
\end{proof}

It is well-known that \( (F) \)-spaces and barrelled \( (DF) \)-spaces satisfy
the equivalent conditions in the above proposition.

\begin{prop}
  \label{prop:89675519c2bb97f8}
  Let \( E \), \( F \), \( G \) be locally convex spaces. Then
  \( \mathfrak{B}(E, F; G) = \mathcal{B}(E, F; G) \) if either one of the
  following conditions holds:
  \begin{enumerate}
  \item \label{item:e198854a8ab48468} \( E \) and \( F \) are both metrizable
    and barrelled, in particular if they are \( (F) \)-spaces;
  \item \label{item:897272a3e73841f9} \( E \) and \( F \) are both
    \( (DF) \)-spaces and barrelled.
  \end{enumerate}
\end{prop}
\begin{proof}
  Case~\ref{item:e198854a8ab48468} can be found in \cite{MR0551623}*{p158, (2)},
  while case~\ref{item:897272a3e73841f9} in \cite{MR0551623}*{p161, (11)}.
\end{proof}

For the comparison of projective and injective tensor products, we introduce the
well-known notion of nuclear spaces due to Grothendieck~\cite{MR0075539}.

\begin{defi}[\cite{MR0354652}*{Ch.II, p34, Définition~4}]
  \label{defi:a13a9b24c4226f4e}
  A locally convex space \( E \) is called \textbf{nuclear}, if for any locally
  convex space \( F \), the canonical map
  \( E \otimes_{\pi} F \to E \otimes_{\varepsilon} F \) is an isomorphism, or
  equivalently, the canonical map
  \( E \overline{\otimes}_{\pi} F \to E \overline{\otimes}_{\varepsilon} F \) is
  an isomorphism.
\end{defi}

\begin{rema}
  \label{rema:080318a5a743ba15}
  There are a variety of equivalent conditions for a locally convex space to be
  nuclear, which we will not recall here nor shall we use directly later, but
  are nevertheless very important in the development of the theory, cf.\ e.g.\
  \cite{MR0225131}*{p511, Theorem~50.1}.
\end{rema}

\begin{nota}
  \label{nota:0b0569840477cd74}
  We use \( (\mathcal{N}) \) to denote the class of nuclear locally convex spaces.
\end{nota}

\begin{prop}
  \label{prop:a8139d3432a18d9e}
  The following holds for the class \( (\mathcal{N}) \).
  \begin{enumerate}
  \item \label{item:99ceefcf15b602f7} The class \( (\mathcal{N}) \) is stable
    under taking countable locally direct sums, separated quotients, inductive
    limits of a countable family, arbitrary products, subspaces, and projective
    limits in \( \mathsf{LCS} \).
  \item \label{item:5fe0be7a72283790} The class \( (\mathcal{N}) \) is stable
    under taking projective (or the same, injective) tensor product.
  \item \label{item:61531a6a26abaf1f} A locally convex space \( E \) is nuclear
    if and only if its completion \( \widehat{E} \) is.
  \item \label{item:8fdf68fedf3451b4} An \( (F) \)-space is nuclear if and only
    if its strong dual is.
  \item \label{item:6a23cd80cf1e5fb3} A normed space is nuclear if and only if
    it is finite dimensional.
  \end{enumerate}
\end{prop}
\begin{proof}
  \ref{item:99ceefcf15b602f7} is in \cite{MR0342978}*{p103, 7.4} and
  \ref{item:5fe0be7a72283790} in \cite{MR0342978}*{p105,
    7.5}. \ref{item:61531a6a26abaf1f} is clear from some equivalent definitions
  of nuclear spaces but can also be found explicitly in \cite{MR0075539}*{Ch~II,
    p39, Corollaire~3}. \ref{item:8fdf68fedf3451b4} is in
  \cite{MR0225131}*{p523, Proposition~50.6}, and \ref{item:6a23cd80cf1e5fb3} in
  \cite{MR0252485}*{p520, Corollary~2}.
\end{proof}

On the negative side.
\begin{prop}[\cite{MR0225131}*{p527, Theorem~51.2}]
  \label{prop:eb0c853264f8729c}
  Uncountable locally convex direct sum of \( \mathbb{K} \) is \emph{not}
  nuclear.
\end{prop}

\subsection{The approximation property}
\label{sec:50d4bdef0f766b83}

Let \( E \) be a locally convex space. Then we obtain an algebraic linear
embedding \( E \odot E' \to \mathcal{L}(E) = \mathcal{L}(E, E) \),
\( \sum_{i}x_{i} \otimes \omega_{i} \mapsto \{x \in E \mapsto
\sum_{i}\omega_{i}(x)x_{i}\} \), whose image we shall denote by
\( \mathfrak{F}(E) \). By the theorem of Hahn-Banach, one easily sees that
\( \mathfrak{F}(E) \) consists exactly of all continuous linear operators on
\( E \) that has finite rank.
\begin{defi}
  \label{defi:dfd30565fa05096b}
  A locally convex space \( E \) is said to have the \textbf{approximation
    property}, if \( \mathfrak{F}(E) \) is dense in
  \( \mathcal{L}_{c}(E) = \mathcal{L}_{c}(E, E) \).
\end{defi}
\begin{rema}
  \label{rema:551801ef6f2d0268}
  There are many well-known characterization of the approximation property,
  which is of minor importance in this paper, cf.\ \cite{MR0075539}*{Ch.I,
    \S~5}. Using these characterization, the problem of whether every locally
  convex space enjoys the approximation property quickly reduces to the case of
  Banach spaces, see e.g.\ \cite{MR0342978}*{p110, Corollary~3}. It is
  well-known that the latter is settled in the negative by a remarkable work of
  Enflo \cite{MR0402468}.
\end{rema}

\begin{nota}
  \label{nota:862ae2531ca3605a}
  We denote the class of all locally convex spaces that has the approximation
  property by \( (\mathcal{AP}) \).
\end{nota}

\begin{prop}[\cite{MR0551623}*{pp245--247}]
  \label{prop:90e4c2c8f8b18525}
  The class \( (\mathcal{AP}) \) is stable under taking arbitrary locally convex
  direct sums, strict inductive limits, arbitrary products and reduced
  projective limits.
\end{prop}

\begin{prop}[\cite{MR0551623}*{p256, (3)}]
  \label{prop:91bfdedd58d67453}
  Let \( K \) be a compact space. The Banach space \( C(K) \) has the
  approximation property.
\end{prop}

We will also have the occasion to use the fact that the Banach space
\( \ell^{1}(X) \) on a discrete set \( X \) equipped with the counting measure,
has the approximate property.  For this, we track a more general result.

We adopt the convention as in \cite{hewitt2012abstract}*{Chapter~III}. Let
\( R \) be a locally convex space, and \( \mu \) a Radon measure on \( R \). The
Banach spaces \( L^{p}(X, \mu) \), \( 1 \leq p \leq \infty \) is defined as
usual, but we modulo out functions that are \emph{locally almost} null. This
yields the same \( L^{p}(X, \mu) \) spaces (ones that modulo out functions that
are almost null), when \( 1 \leq p < +\infty \), but there can be some
difference for \( L^{\infty}(X, \mu) \) due to the possible existence of sets
that are locally \( \mu \)-null, but not \( \mu \)-null. For details, see
\cite{hewitt2012abstract}*{\S~135}.

\begin{prop}[\cite{MR0551623}*{p259, (10) \& (11)}]
  \label{prop:0f7d595efc04bfe4}
  Let \( \mu \) be a Radon measure on a locally compact space \( R \). Then
  \( L^{p}(X, \mu) \) has the approximation property for every
  \( p \in \interval{1}{+\infty} \).
\end{prop}

\begin{prop}[\cite{MR0342978}*{p110, Corollary~2}]
  \label{prop:0bc5d06c0f42ff63}
  Every nuclear space has the approximation property.
\end{prop}

\begin{prop}[\cite{MR0551623}*{p248, (11)}]
  \label{prop:273d1ea40b816ac9}
  Let \( E \) be a Mackey space such that both \( E \) and \( E'_{c} \) are
  quasi-complete (so that it is polar reflexive by
  Proposition~\ref{prop:6ac1bb9a146288e4}), then \( E \in (\mathcal{AP}) \) if
  and only if \( E_{c} \in (\mathcal{AP}) \). This applies in particular when
  \( E \) is an \( (F) \)-space.
\end{prop}

The following result, due to Schwartz based on his theory of
\( \varepsilon \)-products, shall be important for us.

\begin{prop}[\cite{MR0107812}*{pp46-48, Proposition~11, Corollaries 1 \& 2}]
  \label{prop:20a96aff7596345b}
  If \( E \), \( F \) are complete locally convex spaces with the approximation
  property, then \( E \overline{\otimes}_{\varepsilon} F \) has the
  approximation property.
\end{prop}

\subsection{Montel spaces}
\label{sec:053293c5d037bcf2}

\begin{defi}
  \label{defi:817190e978bef08d}
  A locally convex space \( E \) is called a \textbf{Montel space}, or simply
  Montel, if it is barrelled and all bounded sets are relatively compact.
\end{defi}

\begin{rema}
  \label{rema:d35dad705aace610}
  It follows immediately from the definition that all Montel spaces are
  \textbf{quasi-complete}. However, they \emph{need not be complete}, cf.\
  \cite{MR0665588}*{p48, Example~32}.
\end{rema}

\begin{nota}
  \label{nota:b75f85ba53697a31}
  Montel spaces are also called \( (M) \)-spaces, and the class of all Montel
  spaces is denoted by \( (\mathcal{M}) \). An \( (F) \)-space that is also an
  \( (M) \)-space is called an \( (FM) \) space, and \( (\mathcal{FM}) \)
  denotes the class of all \( (FM) \)-spaces.
\end{nota}

\begin{prop}
  \label{prop:99b4fc84a9b1d8a1}
  In a Montel space \( E \), for an arbitrary \( A \subseteq E \), the following
  are equivalent:
  \begin{enumerate}
  \item \label{item:d13964f29d418d8b} \( A \) is bounded;
  \item \label{item:62ef31ffaf7581d4} \( A \) is relatively compact;
  \item \label{item:7e85aadfe8bb3f2e} \( A \) is weakly-bounded;
  \item \label{item:977c46a543af65bf} \( A \) is weakly relatively compact;
  \item \label{item:9a1231cdc2cc918c} \( A \) is equicontinuous with respect to
    the evaluation pairing \( \pairing*{E}{E'} \), where \( E' \) is equipped
    with the strong topology.
  \end{enumerate}
  In particular, \( E'_{c} = E'_{b} = E'_{\tau} \), i.e.\ all Montel spaces are
  Mackey, for which the strong and polar duals are the same.
\end{prop}
\begin{proof}
  The equivalence of \ref{item:d13964f29d418d8b} and \ref{item:7e85aadfe8bb3f2e}
  holds for any locally convex space \( E \) by the general theory of locally
  convex spaces (see, e.g.\ \cite{MR3154940}). Since bounded sets are relatively
  compact in a Montel space and the converse always holds in a locally convex
  space, we also have equivalence between \ref{item:d13964f29d418d8b} and
  \ref{item:62ef31ffaf7581d4}. Clearly \ref{item:62ef31ffaf7581d4} implies
  \ref{item:977c46a543af65bf} as weak topology has less open sets, and it is
  trivial that \ref{item:977c46a543af65bf} implies
  \ref{item:7e85aadfe8bb3f2e}. Thus all first four conditions are equivalent and
  \( E'_{c} = E'_{b} = E'_{\tau} \).

  That \ref{item:9a1231cdc2cc918c} implies \ref{item:977c46a543af65bf} follows
  from Theorem~\ref{theo:bf2abf9c9ce8aa7a}.

  Finally, since Montel spaces are reflexive by
  Proposition~\ref{prop:6ac1bb9a146288e4}, if \( A \) is equicontinuous, the
  polar \( A^{\circ} \) is a neighborhood of \( 0 \in E' \), hence the bipolar
  \( A^{\circ\circ} \) in \( E \) is weakly relatively compact (in fact, weakly
  compact) by Theorem~\ref{theo:bf2abf9c9ce8aa7a} again. Thus
  \( A \subseteq A^{\circ\circ} \) is weakly relatively compact, and
  \ref{item:9a1231cdc2cc918c} implies \ref{item:977c46a543af65bf}.
\end{proof}

We shall need the following.
\begin{prop}[\cite{MR0248498}*{pp369-370}]
  \label{prop:a017fe989c7853e0}
  Consider the class \( (\mathcal{M}) \) of all \( (M) \)-spaces.
  \begin{enumerate}
  \item \label{item:a716828b3f6c61e2} Every \( (M) \)-space is reflexive, and
    the strong dual of an \( (M) \)-space remains in \( (\mathcal{M}) \).
  \item \label{item:5e956eb0e84b6729} The class \( (\mathcal{M}) \) is stable
    under taking arbitrary products and locally convex direct sums.
  \item \label{item:e7b3dc1809350acc} Strict inductive
    limits of \emph{complete} \( M \)-spaces remains in \( (\mathcal{M}) \).
  \end{enumerate}
\end{prop}

The following result is due to Dieudonné.
\begin{prop}[\cite{MR0058854}, see also \cite{MR0248498}*{p370, (5)}]
  \label{prop:c9ea2347581dc852}
  Every \( (FM) \)-space is separable.
\end{prop}

The following is due to Grothendieck.
\begin{prop}[\cite{MR0075539}*{Ch.II, p38, Corollaire~1}, see also
  \cite{MR0225131}*{p520, Corollary~3}]
  \label{prop:f4cbc15f463be1bc}
  All bounded sets in a nuclear space is precompact. In particular, a barrelled
  quasi-complete nuclear space is a Montel space.
\end{prop}

\section{Locally convex Hopf algebras, dualizability and reflexivity}
\label{sec:060ad0d40e57b25a}

Put simply and abstractly, we are going to work with Hopf monoid in a symmetric
monoidal category, in the concrete setting of the category
\( \widehat{\mathsf{LCS}} \) of complete locally convex spaces with various
completed topological tensor product being the monoidal structure. We mention
that there is a systematic treatment of Hopf monoids. See e.g.\
\cite{MR1243637}*{\S~10.5}, or even the monograph \cite{MR2724388}. But we are
not going into that abstract direction. Instead, as will be demonstrated in
later sections of this paper, our concrete setting of working over
\( \widehat{\mathsf{LCS}} \) allows us to benefit greatly from the powerful
tools as presented in \S~\ref{sec:63dd23720298d074}. As a consequence, most
results in \S~\ref{sec:060ad0d40e57b25a} can be proved by an obvious parallel
argument (with added continuity and density) in the proof of their classical
counterparts, and we consider these results elementary. Less elementary results
shall be established gradually in later sections.

\subsection{Monoidal categories and symmetric monoidal categories}
\label{sec:d078c89fa6b6cafa}

We begin by recalling just enough abstract language in order to be able to put
our main objects of study into perspective.

\begin{nota}
  \label{nota:3d951f22614583f1}
  Let \( k \) be a field, which we fix throughout
  \S~\ref{sec:d078c89fa6b6cafa}. An additive category \( \mathcal{C} \) is
  \( k \)-linear if each \( \Hom \)-set is equipped with a vector space
  structure over \( k \), and the composition map is \( k \)-bilinear. We shall
  often use \( A \in \mathcal{C} \) to mean \( A \) is an \emph{object} of the
  category \( \mathcal{C} \). We also use \( \mathcal{C}(A, B) \) to denote the
  collection of all morphisms from \( A \) to \( B \) in a category
  \( \mathcal{C} \), where \( A, B \in \mathcal{C} \).
\end{nota}

We now briefly recall some basic notions of monoidal categories following
\cite{MR3242743}.

\begin{defi}
  \label{defi:1fefc3f86867d064}
  A \( k \)-\textbf{linear monoidal category} is a \( k \)-additive category
  equipped with the extra structures \( \otimes \), \( a \), \( \iota \) and
  \( \mathbb{1} \), where
  \begin{itemize}
  \item \( \otimes : \mathcal{C} \times \mathcal{C} \to \mathcal{C} \) is a
    bifunctor that is bilinear over \( k \), called the \textbf{tensor product};
  \item
    \( a: (- \otimes -) \otimes - \xrightarrow[]{\sim} - \otimes (- \otimes -) \)
    a natural isomorphism between functors from
    \( \mathcal{C} \times \mathcal{C} \times \mathcal{C} \) to \( \mathcal{C} \),
    called the \textbf{associativity constraint} or the \textbf{associator};
  \item \( \mathbb{1} \) is a distinguished object in \( \mathcal{C} \);
  \item \( \iota: \mathbb{1} \otimes \mathbb{1} \to \mathbb{1} \) an isomorphism, such that
  \end{itemize}
  \begin{itemize}
  \item the associator \( a \) satisfies \textbf{the pentagonal identity} in the
    sense that
    \begin{displaymath}
      (\id_{W} \otimes a_{X. Y, X}) a_{W, X \otimes Y, Z} (a_{W, X, Y} \otimes \id_{Z})
      = a_{W, X, Y \otimes Z} a_{W \otimes X, Y ,Z}
    \end{displaymath}
    for all objects \( W, X, Y, Z \in \mathcal{C} \), i.e.\ diagram~\eqref{eq:e66f40562381918f} is commutative.
    \begin{equation}
      \label{eq:e66f40562381918f}
      \begin{tikzpicture}[commutative diagrams/every diagram]
        \node (P0) at (90:2.3cm) {\( ((W \otimes X) \otimes Y) \otimes Z \)};
        \node (P1) at (90+72:2cm) {\( (W \otimes (X \otimes Y)) \otimes Z \)} ;
        \node (P2) at (90+2*72:2cm) {\makebox[5ex][r]{\( W \otimes ((X \otimes Y) \otimes Z) \)}};
        \node (P3) at (90+3*72:2cm) {\makebox[5ex][l]{\( W \otimes (X \otimes (Y \otimes Z)) \)}};
        \node (P4) at (90+4*72:2cm) {\( (W \otimes X) \otimes (Y \otimes Z) \)};

        \path[commutative diagrams/.cd, every arrow, every label]  
        (P0) edge node[swap] {\( a_{W, X, Y} \otimes \id_{Z} \)} (P1)
        (P1) edge node[swap] {\( a_{W, X \otimes Y, Z} \)} (P2)
        (P2) edge node {\( \id_{W} \otimes a_{X. Y, X} \)} (P3)
        (P4) edge node {\( a_{W, X, Y \otimes Z} \)} (P3)
        (P0) edge node {\( a_{W \otimes X, Y ,Z} \)} (P4);
      \end{tikzpicture}
    \end{equation}
  \item the functors \( L_{\mathbb{1}} : X \mapsto \mathbb{1} \otimes X \) and
    \( R_{\mathbb{1}}: X \mapsto X \otimes \mathbb{1} \) (with the obvious way
    of defining how these functors map the morphisms) are auto-equivalence of
    \( \mathcal{C} \), called \textbf{the left and right unit constraints}
    respectively.
  \end{itemize}
  The pair \( (\mathbb{1}, \iota) \), or only \( \mathbb{1} \) if \( \iota \) is
  clear from context, shall be called the \textbf{unit object}, and \( \iota \) the
  \textbf{unit isomorphism}.
\end{defi}

\begin{rema}
  \label{rema:65a5457d702593e9}
  This is not the usual way of defining a monoidal category which also features
  the left and right constraints (\( L_{\mathbb{1}} \) and \( R_{\mathbb{1}} \)
  above), see e.g.\ \cite{mac1998categories}*{\S~VII.1}. It is shown in
  \cite{MR3242743}*{Ch.~2} that the above slightly more neat formulation is
  nevertheless equivalent to the usual one.
\end{rema}

\begin{rema}
  \label{rema:eeda598152abaa2e}
  If the associator \( a \) as well as \( \iota \) in the above are both
  identities, then we say that the monoidal category \( \mathcal{C} \) is
  \textbf{strict}. There is a whole theory on monoidal categories. Here we only
  mentions that for our purposes in this paper, MacLane's coherence theorem and
  strictification theorem (\cite{MR0414648}, see also \cite{MR3242743}*{Ch.~2}
  or \cite{mac1998categories}*{Ch.~VII}) allows us to work as though the
  underlying monoidal categories are strict. We refer to the above references
  for the precise formulation of how these strictification and coherence results
  work.
\end{rema}

We shall also need a good formulation of the comultiplication map preserves
multiplication in order to talk about bialgebras and Hopf algebras. For this, we
need to recall the structure of a symmetric braiding.

\begin{defi}
  \label{defi:1508b4999ae67617}
  Let \( (\mathcal{C}, \otimes, \mathbb{1}, a, \iota) \) (later abbreviated as
  \( \mathcal{C} \)) be a \( k \)-linear monoidal category. and denote the
  swapping functor \( (X, Y) \mapsto (Y, X) \), \( (f, g) \mapsto (g, f) \) on
  \( \mathcal{C} \times \mathcal{C} \) by \( \mathsf{S} \), and
  \( \otimes^{\op}: \mathcal{C} \times \mathcal{C} \to \mathcal{C} \) the
  composition \( \otimes \circ \mathsf{S} \). A \textbf{symmetric braiding}
  \( s \) on \( \mathcal{C} \) is an involutive natural isomorphism from
  \( \otimes \) to \( \otimes^{\op} \) that is also a braiding. More precisely,
  this means that \( s \) consists a family of isomorphisms
  \( (s_{X, Y})_{(X, Y) \in \mathcal{C} \times \mathcal{C}}: X \otimes Y \to Y
  \otimes X \), natural in \( X \) and \( Y \), such that
  \begin{enumerate}
  \item \label{item:22ccfc6ca31b5302} \( s \) is its own inverse, i.e.\
    \( s_{X, Y} \) and \( s_{Y, X} \) are inverses to each other for all \( X, Y \in \mathcal{C} \);
  \item \label{item:edc6f82ee34e08da} \( s \) satisfies the \textbf{hexagon
      axiom}, i.e.\ for all \( X, Y, Z \in \mathcal{C} \),
    diagram~\eqref{eq:aeb1d2a29f83b8e3} is commutative.
    \begin{equation}
      \label{eq:aeb1d2a29f83b8e3}
      \begin{tikzcd}
        (X \otimes Y) \otimes Z \arrow{r}{s_{X \otimes Y, Z}} \arrow{d}{a_{X, Y, Z}}
        & Z \otimes (X \otimes Y) \arrow{d}{a^{-1}_{Z, X, Y}} \\
        X \otimes (Y \otimes Z) \arrow{d}{\id_{X} \otimes s_{Y, Z}} &
        (Z \otimes X) \otimes Y \arrow{d}{s_{Z, X} \otimes \id_{Y}} \\
        X \otimes (Z \otimes Y) \arrow{r}{a^{-1}_{X, Z, Y}}
        & (X \otimes Z) \otimes Y
      \end{tikzcd}
    \end{equation}
  \end{enumerate}
\end{defi}

\begin{rema}
  \label{rema:274535ea20f3154a}
  The braiding condition involves another axiom, which is automatic for
  symmetric braidings as formulated above, hence is omitted. There's also
  version of coherence theorem for symmetric monoidal category. Basically, this
  means the symmetric braiding does behave as ``a swapping should behave'', in
  the sense that every diagram (involving the morphisms in the defining
  structures of a symmetric monoidal category) that should commute does
  commute. For details of these remarks, we refer to
  \cite{MR0354798}*{S~IX.1}. In our concrete settings of locally convex spaces
  starting from \S~\ref{sec:060ad0d40e57b25a}, one may also easily check the
  required commutativity of the diagrams.
\end{rema}

\begin{rema}
  \label{rema:4f014d44fede02c3}
  Take \( X = Y = Z = \mathbb{1} \) in \eqref{eq:aeb1d2a29f83b8e3} and assume
  \( \mathcal{C} \) is strict, we obtain
  \begin{equation}
    \label{eq:e2d5a79cb88f6c87}
    s_{\mathbb{1}, \mathbb{1}} =
    \id_{\mathbb{1}} \otimes s_{\mathbb{1}, \mathbb{1}} = (s_{\mathbb{1},
      \mathbb{1}} \otimes \id_{\mathbb{1}}) s_{\mathbb{1} \otimes \mathbb{1}, \mathbb{1}}
    = s_{\mathbb{1}, \mathbb{1}} s_{\mathbb{1}, \mathbb{1}} = \id_{\mathbb{1}}.
  \end{equation}
\end{rema}

Our proto-example of a (symmetric) monoidal category is the category
\( \mathsf{Vect} \) of all vector spaces over \( k \), where the tensor functor
\( \otimes \) is given by the usual tensor product, associators by the canonical
identifications describing the associativity of such tensor products, \( k \)
being the unit, and the unit law being the canonical identification again (and
the symmetric braiding is the usual swap). Later, starting from
\S~\ref{sec:060ad0d40e57b25a}, we shall consider topological versions of
\( \mathsf{Vect} \) by working over the category \( \widehat{\mathsf{LCS}} \) of
complete locally convex spaces over \( \mathbb{K} = \R \) or \( \C \), and the
tensor product being various topological tensor products
(\S\S~\ref{sec:c43804468b41e821}\ref{sec:ac51753bf78a1da5}\ref{sec:33ea3663864561d2}).

\subsection{Compatible symmetric monoidal functors on the category of complete
  locally convex spaces}
\label{sec:cbdc5edf0dbbb84c}

Intuitively, locally convex Hopf algebras should be the Hopf algebra in some
(symmetric) monoidal category of locally convex spaces. Since we wish to work
with complete topological tensor products in order to benefit from the larger
range for the comultiplication, we restrict our attention to the category
\( \widehat{\mathsf{LCS}} \) of complete locally convex spaces. The notation
here is consistent with the common convention of denoting the (separated if
\( E \) is non-Hausdorff) completion of a locally convex space \( E \) by
\( \widehat{E} \).

We now formalize the categorical structures with which we are going to work.

\begin{defi}
  \label{defi:ec19c5fb89871bda}
  A covariant bifunctor
  \( \overline{\otimes}_{\tau}: \widehat{\mathsf{LCS}} \times
  \widehat{\mathsf{LCS}} \to \widehat{\mathsf{LCS}} \) is called a
  \textbf{compatible symmetric monoidal functor}, if it satisfies the following
  conditions:
  \begin{enumerate}
  \item \label{item:c83911d814721127} for each pair \( (E, F) \) in
    \( \widehat{\mathsf{LCS}} \), denote \( \overline{\otimes}_{\tau}(E, F) \)
    by \( E \overline{\otimes}_{\tau} F \), there exists a separately continuous
    bilinear map \( \chi: E \times F \to E \overline{\otimes}_{\tau} F \),
    called \textbf{the canonical map}, such that the image of \( \chi \) is
    total in \( E \overline{\otimes}_{\tau} F \), and
    \( \chi: E \times F \to \vect\set*{\chi(E \times F)} \) is a realization of
    the algebraic tensor product \( E \odot F \) of the vector spaces \( E \)
    of \( F \), and we shall often identify \( E \odot F \)  with
    \( \vect\set*{\chi(E \times F)} \) (together with \( \chi \) to be precise)
    in this way;
  \item \label{item:3049a942b7707f9e} the subspace topology on \( E \odot F \)
    induced by \( E \overline{\otimes}_{\tau} F \) is a compatible topology in
    the sense of Definition~\ref{defi:9bd5e00768234cac}, and we use
    \( E \otimes_{\tau} F \) to denote \( E \odot F \) equipped with this
    topology;
  \item \label{item:3615d2adbc6b540a} for all
    \( E \in \widehat{\mathsf{LCS}} \), we have
    \( E \overline{\otimes}_{\tau} \mathbb{K} = E = \mathbb{K}
    \overline{\otimes}_{\tau} E \);
  \item \label{item:03802e7cbdb6ecfe} for all
    \( E, F \in \widehat{\mathsf{LCS}} \), the swap
    \( s_{E,F}: E \otimes_{\tau} F \mapsto F \otimes_{\tau}E \) determined by
    \( x \otimes y \mapsto y \otimes x \) and linearity, is continuous;
  \item \label{item:b8917369763b9f15} for all
    \( E, F, G \in \widehat{\mathsf{LCS}} \), the associator
    \( a_{E,F,G}: (E \otimes_{\tau} F) \otimes_{\tau} G \to E \otimes_{\tau} (F
    \otimes_{\tau} G) \) determined by
    \( (x \otimes y) \otimes z \mapsto x \otimes (y \otimes z) \) and linearity,
    is continuous.
  \end{enumerate}
\end{defi}

\begin{prop}
  \label{prop:6d54eef950984ead}
  Let \( \overline{\otimes}_{\tau} \) be a compatible symmetric monoidal functor
  on \( \widehat{\mathsf{LCS}} \). Then
  \begin{enumerate}
  \item \label{item:ed741a9bcd8fef7c} for all
    \( E, F \in \widehat{\mathsf{LCS}} \), the swap \( s_{E,F} \) is an
    isomorphism of locally convex spaces, and extends uniquely by continuity to
    a morphism
    \( \overline{s}_{E, F}: E \overline{\otimes}_{\tau} F \to F
    \overline{\otimes}_{\tau} E \) in \( \widehat{\mathsf{LCS}} \), which is
    also an isomorphism;
  \item \label{item:22874fae581750a6} for all
    \( E, F, G \in \widehat{\mathsf{LCS}} \), the associator \( a_{E,F,G} \) is
    an isomorphism of locally convex spaces, and extends uniquely by continuity
    to a morphism
    \( \overline{a}_{E,F,G}: (E \overline{\otimes}_{\tau}F)
    \overline{\otimes}_{\tau}G \to E \overline{\otimes}_{\tau}(F
    \overline{\otimes}_{\tau}G) \) in \( \widehat{\mathsf{LCS}} \), which is
    also an isomorphism.
  \end{enumerate}
\end{prop}
\begin{proof}
  \ref{item:ed741a9bcd8fef7c} follows by noting that \( s_{E,F} \) and
  \( s_{F,E} \) are inverses to each other and
  Theorem~\ref{theo:5c48561db339a0d3}. By the same theorem, \( a_{E,F,G} \)
  extends uniquely by continuity to \( \overline{a}_{E, F, G} \) for all
  \( E,F,G \in \widehat{\mathsf{LCS}} \). Now consider the composition in \( \widehat{\mathsf{LCS}} \),
  \begin{displaymath}
    \begin{tikzcd}
      b_{E,F,G}: & E \overline{\otimes}_{\tau} (F \overline{\otimes}_{\tau} G)
      \arrow{r}{\overline{s}_{E, F \overline{\otimes}_{\tau} G}} & (F
      \overline{\otimes}_{\tau} G) \overline{\otimes}_{\tau} E
      \arrow{r}{\overline{s}_{F,G} \overline{\otimes} \id_{E}} \arrow[d,
      phantom, ""{coordinate, name=Z}]& (G \overline{\otimes}_{\tau} F)
      \overline{\otimes}_{\tau} E \arrow[dll, rounded corners, to path={ --
        ([xshift=2ex]\tikztostart.east) |- (Z) [near end]\tikztonodes -|
        ([xshift=-2ex]\tikztotarget.west) -- (\tikztotarget)
      }] \\
      & G \overline{\otimes}_{\tau} (F \overline{\otimes}_{\tau} E)
      \arrow{r}{\id_{G} \overline{\otimes}_{\tau} \overline{s}_{F, E}} & G
      \overline{\otimes}_{\tau} (E \overline{\otimes}_{\tau} F)
      \arrow{r}{\overline{s}_{G, E \overline{\otimes}_{\tau}F}} & (E
      \overline{\otimes}_{\tau} F) \overline{\otimes}_{\tau} G,
    \end{tikzcd}
  \end{displaymath}
  where the middle connecting ``snake arrow'' is \( \overline{a}_{G,F,E} \). One
  checks immediately that the morphism \( b_{E, F, G} \) and
  \( \overline{a}_{E,F,G} \) in \( \widehat{\mathsf{LCS}} \) are inverses to
  each other, and \( b_{E,F,G} \) restricts to the inverse of \( a_{E,F,G} \),
  which establishes \ref{item:22874fae581750a6}.
\end{proof}

\begin{prop}
  \label{prop:cf8ab0f606af3a4a}
  Equipped with a compatible monoidal tensor product
  \( \overline{\otimes}_{\tau} \), \( \overline{a} = (\overline{a}_{E,F,G}) \)
  as the associator, \( \overline{s} = (\overline{s}_{E,F}) \) as the symmetry,
  \( \mathbb{K} \) as the unit object, and the identity maps as the unit
  isomorphism as well as left and right constraints,
  \( \widehat{\mathsf{LCS}} \) becomes a symmetric monoidal category.
\end{prop}
\begin{proof}
  This is a routine check. Every axiom follows from the corresponding one for
  the symmetric monoidal category \( \mathsf{Vect} \), by extending by
  continuity according to Proposition~\ref{prop:6d54eef950984ead}.
\end{proof}

\begin{prop}
  \label{prop:f1859d673771c63a}
  The completed inductive tensor product \( \overline{\otimes}_{\iota} \), the
  completed injective tensor product \( \overline{\otimes}_{\varepsilon} \), and
  the completed projective tensor product \( \overline{\otimes}_{\pi} \), are
  all compatible symmetric monoidal functors on \( \widehat{\mathsf{LCS}} \).
\end{prop}

\begin{proof}
  This follows from Proposition~\ref{prop:963932029cf324b4},
  Proposition~\ref{prop:51d98c7dd4f5950d} and
  Proposition~\ref{prop:6f861986d67606e0}.
\end{proof}

\begin{nota}
  \label{nota:fb4caa132d89863a}
  To facilitate our discussion, we adopt the following convention on our
  notation. Given a compatible monoidal functor \( \overline{\otimes}_{\tau} \)
  on \( \widehat{\mathsf{LCS}} \), we use \( \widehat{\mathsf{LCS}}_{\tau} \) to
  denote the symmetric monoidal category in
  Proposition~\ref{prop:cf8ab0f606af3a4a}. Whenever it is clear from the
  context, for any \( E, F \in \widehat{\mathsf{LCS}}_{\tau} \), the topology on
  \( E \overline{\otimes}_{\tau} F \), as well as subspace topology on the
  canonical copy of the subspace \( E \odot F \), shall be denoted by
  \( \mathfrak{T}_{\tau} \). We use \( E \otimes_{\tau} F \) to denote the space
  \( E \odot F \) equipped with the topology \( \mathfrak{T}_{\tau} \).  This
  applies in particular to \( \tau = \iota, \pi, \varepsilon \).
\end{nota}

\begin{rema}
  \label{rema:113a388d37745bd7}
  In the following, we will treat \( \widehat{\mathsf{LCS}}_{\tau} \) as though
  it is a strict monoidal category, which is valid by MacLane's coherence and
  strictification theorem, see Remark~\ref{rema:eeda598152abaa2e}.
\end{rema}

\subsection{Locally convex coalgebras, algebras and bialgebras}
\label{sec:d1604ccb16145a3f}

Unless stated otherwise, the scalar field shall be either \( \R \) or \( \C \),
and we shall often use \( \mathbb{K} \) to denote this scalar field. From now
on, we fix a compatible symmetric monoidal functor
\( \overline{\otimes}_{\tau} \) (on \( \widehat{\mathsf{LCS}} \) of course), and
we work as though our symmetric monoidal categories
\( \widehat{\mathsf{LCS}}_{\tau} \) is strict, cf.\
Remark~\ref{rema:eeda598152abaa2e}.

\begin{defi}
  \label{defi:6ccaa28e838d7d44}
  An \textbf{algebra in} \( \widehat{\mathsf{LCS}}_{\tau} \), or a
  \( \tau \)-\textbf{algebra}, is a triplet \( (A, m, \eta) \), with
  \( A \in \widehat{\mathsf{LCS}}_{\tau} \), and
  \( m : A \overline{\otimes}_{\tau} A \to A \) (called the
  \textbf{multiplication}), \( \eta: \mathbb{K} \to A \) (called the
  \textbf{unit}), such that the diagram
  \begin{equation}
    \label{eq:4278e8ce2b471651}
    \begin{tikzcd}
      A \overline{\otimes}_{\tau} A \overline{\otimes}_{\tau} A \arrow{r}{m \overline{\otimes}_{\tau} \id_{A}} \arrow{d}{\id_{A} \overline{\otimes}_{\tau} m}
      & A \overline{\otimes}_{\tau} A \arrow{d}{m} \\
      A \overline{\otimes}_{\tau} A \arrow{r}{m} & A
    \end{tikzcd}
  \end{equation}
  commutes (called \textbf{associativity of} \( m \)), i.e.\
  \( m(m \overline{\otimes}_{\tau} \id_{A}) = m(\id_{A} \overline{\otimes}_{\tau} m) \); and \( \eta \) satisfies the \textbf{unit law}, meaning the diagram
  \begin{equation}
    \label{eq:d39b49d679de395a}
    \begin{tikzcd}
      A \overline{\otimes}_{\tau} \mathbb{K} = A = \mathbb{K} \overline{\otimes}_{\tau} A \arrow{r}{\eta \overline{\otimes}_{\tau} \id_{A}}
      \arrow{d}{\id_{A} \overline{\otimes}_{\tau} \eta}
      & A \overline{\otimes}_{\tau} A \arrow{d}{m} \\
      A \overline{\otimes}_{\tau} A \arrow{r}{m} & A
    \end{tikzcd}
  \end{equation}
  commutes, or equivalently, \( m (\eta \overline{\otimes}_{\tau} \id_{A}) = m (\id_{A} \overline{\otimes}_{\tau} \eta) \).

  Let \( (A, m_{A}, \eta_{A}) \), \( (B, m_{B}, \eta_{B}) \) be algebras in
  \( \widehat{\mathsf{LCS}}_{\tau} \). We say a morphism \( f : A \to B \) is a
  \textbf{morphism of (\( \tau \)-)algebras} if it preserves multiplication in
  the sense that \( m_{B}(f \overline{\otimes}_{\tau} f)= f m_{A} \), i.e.\ the
  diagram
  \begin{equation}
    \label{eq:ba29f204e8d43593}
    \begin{tikzcd}
      A \overline{\otimes}_{\tau} A \arrow{r}{m_{A}} \arrow{d}{f \overline{\otimes}_{\tau} f} & A \arrow{d}{f} \\
      B \overline{\otimes}_{\tau} B \arrow{r}{m_{B}} & B
    \end{tikzcd}
  \end{equation}
  commutes; and \( f \) preserves the unit, i.e.\ \( f \eta_{A} = \eta_{B} \),
  or equivalently the diagram
  \begin{equation}
    \label{eq:6b157c36dc2b0875}
    \begin{tikzcd}
      \mathbb{K} \arrow{r}{\eta_{A}} \arrow{rd}{\eta_{B}} & A \arrow{d}{f} \\
      & B
    \end{tikzcd}
  \end{equation}
  commutes.
\end{defi}

\begin{nota}
  \label{nota:39f3e83af9d2b14d}
  Given a \( \tau \)-algebra \( (A, \eta, m) \), following old traditions, we
  shall often write \( xy \) instead of \( m(x \otimes y) \) for
  \( x, y \in A \), and confuse the unit \( \eta \) of \( A \) with the image
  \( \eta(1) \in A \), and we also denote \( \eta(1) \) by \( 1 \), or
  \( 1_{A} \) if we want to emphasize it is the multiplicative neutral element
  of \( A \).
\end{nota}

Coalgebras in \( \widehat{\mathsf{LCS}}_{\tau} \), as well as their morphisms, are defined by
formally by reversing all the arrows in the diagrams \eqref{eq:4278e8ce2b471651},
\eqref{eq:d39b49d679de395a}, \eqref{eq:ba29f204e8d43593} and
\eqref{eq:6b157c36dc2b0875}.

\begin{defi}
  \label{defi:0fb1526e3c72db6a}
  A \textbf{coalgebra} in \( \widehat{\mathsf{LCS}}_{\tau} \), or a
  \( \tau \)-\textbf{coalgebra}, is a triplet \( (C, \Delta, \varepsilon) \),
  where \( C \in \widehat{\mathsf{LCS}}_{\tau} \),
  \( \Delta : C \to C \overline{\otimes}_{\tau} C \), called the
  \textbf{comultiplication}, and \( \varepsilon: C \to \mathbb{K} \), called the
  \textbf{counit}, such that the comultiplication is \textbf{coassociative},
  i.e.\
  \( (\Delta \overline{\otimes}_{\tau} \id_{C})\Delta = (\id_{C}
  \overline{\otimes}_{\tau} \Delta)\Delta \), and \( \varepsilon \) satisfies
  the \textbf{counit law}, in the sense that
  \( (\varepsilon \odot \id_{C})\Delta = \id_{C} = (\id
  \overline{\otimes}_{\tau} \varepsilon) \Delta \).

  A morphism of \( \tau \)-coalgebras from a \( \tau \)-coalgebra
  \( (C, \Delta_{C}, \varepsilon_{C}) \) to a \( \tau \)-coalgebra
  \( (D, \Delta_{D}, \varepsilon_{D}) \) is a morphism \( f: C \to D \) that
  preserves both the comultiplication and the counit:
  \( (f \overline{\otimes}_{\tau} f)\Delta_{C} = \Delta_{D}f \) and
  \( \varepsilon_{D} f = \varepsilon_{C} \).
\end{defi}

\begin{rema}
  \label{rema:9d43b29d99f15c3b}
  The unit (resp.\ counit) of a \( \tau \)-algebra (resp.\ \( \tau \)-coalgebra)
  is \emph{unique}, once the multiplication (comultiplication) is fixed. Indeed,
  let \( \eta_{1} \), \( \eta_{2} \) are both unit for a \( \tau \)-algebra
  \( A \) equipped with \( m \) as the multiplication. We have
  \begin{displaymath}
    \eta_{1} = m (\id_{A} \overline{\otimes}_{\tau} \eta_{2}) \eta_{1} = m (\eta_{1} \overline{\otimes}_{\tau} \eta_{2})
    = m (\eta_{1} \overline{\otimes}_{\tau} \id_{A}) \eta_{2} = \eta_{2}.
  \end{displaymath}
  The uniqueness of a counit is proved dually.
\end{rema}

In a symmetric monoidal category, by using the swap, one may consider the
opposite (resp.\ coopposite) of algebras (resp.\ coalgebras), as well as forming
their tensor products.

The proof of the following proposition is a routine check and is left to the
reader.
\begin{prop}
  \label{prop:63fea142ee144ad4}
  Consider the symmetric monoidal category \( \widehat{\mathsf{LCS}}_{\tau} \).
  \begin{enumerate}
  \item \label{item:04f778ba6a90ade7} Suppose \( (A, m, \eta) \) (resp.\
    \( (C, \Delta, \varepsilon) \)) is an algebra (resp.\ a coalgebra) in
    \( \widehat{\mathsf{LCS}}_{\tau} \). Define \( m^{\op} = m s_{A,A} \)
    (resp.\ \( \Delta^{\cop} = s_{C,C}\Delta \)), then \( (A, m^{\op}, \eta) \)
    (resp.\ \( (C, \Delta^{\cop}, \varepsilon) \)) is an algebra (resp.\ a
    coalgebra) in \( \overline{\mathsf{LCS}}_{\tau} \).
  \item \label{item:7e8247473a6477ca} For \( i = 1, 2 \), let
    \( (A_{i}, m_{i}, \eta_{i}) \) (resp.\
    \( (C_{i}, \Delta_{i}, \varepsilon_{i}) \)) be an algebra (resp.\ coalgebra)
    in \( \widehat{\mathsf{LCS}}_{\tau} \). Define the multiplication
    \( m : (A_{1} \overline{\otimes}_{\tau} A_{2}) \overline{\otimes}_{\tau}
    (A_{1} \overline{\otimes}_{\tau} A_{2}) \to A_{1} \overline{\otimes}_{\tau}
    A_{2} \) as
    \( (m_{1} \overline{\otimes}_{\tau} m_{2})(\id_{A_{1}}
    \overline{\otimes}_{\tau} s_{A_{2},A_{1}} \overline{\otimes}_{\tau}
    \id_{A_{2}}) \), and the unit
    \( \eta: \mathbb{K} \to A_{1} \overline{\otimes}_{\tau} A_{2} \) as
    \( \eta_{A_{1}} \overline{\otimes}_{\tau} \eta_{A_{2}} \) (we identified
    \( \mathbb{K} \) with \( \mathbb{K} \overline{\otimes}_{\tau} \mathbb{K} \))
    (resp.\ the comultiplication
    \( \Delta: C_{1} \overline{\otimes}_{\tau} C_{2} \to (C_{1}
    \overline{\otimes}_{\tau} C_{2}) \overline{\otimes}_{\tau} (C_{1}
    \overline{\otimes}_{\tau} C_{2}) \) as
    \( (\id_{C_{1}} \overline{\otimes}_{\tau} s_{C_{1}, C_{2}}
    \overline{\otimes}_{\tau} \id_{C_{2}})(\Delta_{1} \overline{\otimes}_{\tau}
    \Delta_{2}) \), and the counit
    \( \varepsilon: C_{1} \overline{\otimes}_{\tau} C_{2} \to \mathbb{K} \) as
    \( \varepsilon_{1} \overline{\otimes}_{\tau} \varepsilon_{2} \)), then
    \( (A_{1} \overline{\otimes}_{\tau} A_{2}, m, \eta) \) (resp.\
    \( (C_{1} \overline{\otimes}_{\tau} C_{2}, \Delta, \varepsilon) \)) is an
    algebra (resp.\ a coalgebra) in \( \overline{\mathsf{LCS}}_{\tau} \).
  \end{enumerate}
\end{prop}

\begin{defi}
  \label{defi:b3554ee336047666}
  Using the settings of Proposition~\ref{prop:63fea142ee144ad4}, we call the
  algebra \( (A, m^{\op}, \eta) \) (resp.\ \( (C, \Delta, \varepsilon) \)) the
  \textbf{opposite} (resp.\ \textbf{coopposite}) of the \( \tau \)-algebra
  \( (A, m, \eta) \) (resp.\ the \( \tau \)-coalgebra
  \( (C, \Delta, \varepsilon) \)), and we call the resulting algebra (resp.\
  coalgebra) in
  Proposition~\ref{prop:63fea142ee144ad4}~\ref{item:7e8247473a6477ca} the
  \textbf{(\( \tau \)-)tensor product} of the \( \tau \)-algebras (resp.\
  \( \tau \)-coalgebras).
\end{defi}

\begin{nota}
  \label{nota:bf99759a3f8300c9}
  If there is no risk of confusion, following the old tradition of notation
  abuse, we will often omit the structure maps, and speak of a
  \( \tau \)-algebra \( A \) (instead of \( (A, m, \eta) \)),
  \( \tau \)-coalgebra \( C \), the opposite algebra \( A^{\op} \), the
  coopposite coalgebra \( C^{\cop} \), the tensor product (algebra)
  \( A_{1} \overline{\otimes}_{\tau} A_{2} \), the tensor product coalgebra
  \( C_{1} \overline{\otimes}_{\tau} C_{2} \), etc.
\end{nota}

Note that in the following result, the unit object \( \mathbb{K} \) has a
trivial algebra and coalgebra structure where all structure maps are identities
(after the canonical identifications).
\begin{prop}
  \label{prop:717c7c133e9e974c}
  Let \( B \) be an object in \( \widehat{\mathsf{LCS}}_{\tau} \) that is
  equipped both with an algebra structure \( (m, \eta) \) and a coalgebra
  structure \( (\Delta, \varepsilon) \). Then the following are equivalent:
  \begin{enumerate}
  \item \label{item:f57f223592789fac} \( \Delta: B \to B \overline{\otimes}_{\tau} B \) is a
    morphism of algebras and \( \varepsilon \) preserves multiplication;
  \item \label{item:74acd821b1b29155} \( m : B \overline{\otimes}_{\tau} B \to B \) is a morphism
    of coalgebras and \( \eta \) preserves comultiplication;
  \item \label{item:e32e1fa4753a028d}
    \( \Delta m = (m \overline{\otimes}_{\tau} m)(\id \overline{\otimes}_{\tau} s \overline{\otimes}_{\tau} \id)(\Delta \overline{\otimes}_{\tau}
    \Delta) \), \( \Delta \eta = \eta \overline{\otimes}_{\tau} \eta \) and
    \( \varepsilon m = \varepsilon \overline{\otimes}_{\tau} \varepsilon \).
  \end{enumerate}
\end{prop}
\begin{proof}
  Unwinding the definitions, one checks immediately both
  \ref{item:f57f223592789fac} and \ref{item:74acd821b1b29155} are equivalent to
  \ref{item:e32e1fa4753a028d}.
\end{proof}

\begin{defi}
  \label{defi:e67e64a18cd4e77b}
  We say \( B \) as Proposition~\ref{prop:717c7c133e9e974c} is a
  \textbf{bialgebra} in \( \widehat{\mathsf{LCS}}_{\tau} \), or simply a
  \( \tau \)-\textbf{bialgebra}, if it satisfies one of the equivalent
  conditions above. Morphisms between \( \tau \)-bialgebras are the ones that
  preserves both the \( \tau \)-algebra and the \( \tau \)-coalgebra structures.
\end{defi}

\begin{prop}
  \label{prop:6a271f71782832ec}
  If \( B \) is a bialgebra in \( \widehat{\mathsf{LCS}}_{\tau} \), then so is
  \( B^{\op} \) and \( B^{\cop} \).  If \( B_{1} \), \( B_{2} \) are bialgebras
  in \( \widehat{\mathsf{LCS}}_{\tau} \), so is
  \( B_{1} \overline{\otimes}_{\tau} B_{2} \) equipped with the structures of
  tensor products of algebras and coalgebras.
\end{prop}
\begin{proof}
  Again, this is a routine check just as in the theory for classical bialgebras
  and is left to the reader.
\end{proof}

Naturally, we call \( B^{\op} \) (resp.\ \( B^{\cop} \)) the \textbf{opposite}
(\textbf{coopposite}) bialgebra of \( B \), and the bialgebras
\( B_{1} \overline{\otimes}_{\tau} B_{2} \) the \textbf{(\( \tau \)-)tensor
  product} of \( B_{1} \) and \( B_{2} \).

\subsection{The convolution algebra of continuous linear maps, locally convex Hopf algebras}
\label{sec:62f5f3547854deb9}

Suppose \( C \) (resp.\ \( A \)) a \( \tau \)-coalgebra (resp.\
\( \tau \)-algebra), and consider the linear space \( \mathcal{L}(C, B) \) of
all continuous linear maps from the locally convex space \( C \) to \( B \). For
any \( f, g \in \mathcal{L}(C, B) \), we may define their \textbf{convolution}
\( f \ast g \) as
\begin{equation}
  \label{eq:1490c12531f219a0}
  f \ast g := m_{B} (f \overline{\otimes}_{\tau} g) \Delta_{C} \in \mathcal{L}(C, B).
\end{equation}
Clearly, \eqref{eq:1490c12531f219a0} is bilinear in \( f \) and \( g \). The
convolution is also \textbf{associative}. Indeed, let
\( f, g, h \in \mathcal{L}(C, B) \), we have
\begin{displaymath}
  \begin{split}
    (f \ast g) \ast h &= m_{B}\Bigl(\bigl(m_{B}(f \overline{\otimes}_{\tau} g) \Delta_{C}\bigr) \overline{\otimes}_{\tau} h\Bigr) \Delta_{C}
                        = m_{B}(m_{B} \overline{\otimes}_{\tau} \id_{B}) (f \overline{\otimes}_{\tau} g \overline{\otimes}_{\tau} h) (\Delta_{C} \overline{\otimes}_{\tau} \id_{C}) \Delta_{C} \\
                &= m_{B}(\id_{B} \overline{\otimes}_{\tau} m_{B}) (f \overline{\otimes}_{\tau}  g \overline{\otimes}_{\tau} h) (\id_{C} \overline{\otimes}_{\tau} \Delta_{C}) \Delta_{C}
                  = f \ast (g \ast h),
  \end{split}
\end{displaymath}
where the first equality on the second line follows from the associativity of
\( m_{B} \) and the coassociativity of \( \Delta_{C} \).

We also have a distinguished element, namely \( \eta_{B}\varepsilon_{C} \), in
\( \mathcal{L}(C, B) \). It is the neutral element for the convolution product,
since for each \( f \in \mathcal{L}(C, B) \), we have
\begin{displaymath}
  f \ast (\eta_{B} \varepsilon_{C}) = m_{B} \bigl(f \overline{\otimes}_{\tau} (\eta_{B} \varepsilon_{C})\bigr) \Delta_{C}
  = m_{B}(\id_{B} \overline{\otimes}_{\tau} \eta_{B})f(\id_{C} \overline{\otimes}_{\tau} \varepsilon_{C})\Delta_{C} = f
\end{displaymath}
where the last equality follows from the defining property of the unit and
counit. Similarly, \( (\eta_{B} \varepsilon_{C}) \ast f = f \).

\begin{defi}
  \label{defi:52d0aeb9c0ac1362}
  We call the \( \mathbb{K} \)-algebra \( \mathcal{L}(C, B) \) with convolution
  as the product, and \( \eta_{B}\varepsilon_{C} \) as the multiplicative
  neutral element (this is a classical unital associative algebra, or an algebra
  in \( \mathsf{Vect} \)), the \textbf{convolution algebra} of all continuous
  linear maps from the \( \tau \)-coalgebra \( C \) to the \( \tau \)-algebra
  \( B \), and shall often denote it by the same symbol \( \mathcal{L}(C, B) \).
\end{defi}

\begin{prop}
  \label{prop:be353ca74ee009d0}
  Let \( \mathsf{Alg}_{\tau} \) (resp.\ \( \mathsf{Coalg}_{\tau} \)) denote the
  category of \( \tau \)-algebras (resp.\ \( \tau \)-coalgebras), and
  \( \mathsf{Alg}_{\mathbb{K}} \) the category of classical
  \( \mathbb{K} \)-algebras. The association
  \( \mathsf{Coalg}_{\tau} \times \mathsf{Alg}_{\tau} \to \mathsf{Alg}_{\mathbb{K}} \),
  \( (C, B) \mapsto \mathcal{L}(C, B) \),
  \( (f, g) \mapsto f^{\ast}g_{\ast} = g_{\ast}f^{\ast} \) (\( (\cdot)^{\ast} \)
  denotes the pull-back, and \( (\cdot)_{\ast} \) the push-forward) is a functor
  that is contravariant in the first variable and covariant in the second.
\end{prop}
\begin{proof}
  This follows again from a routine check and is left to the reader.
\end{proof}

\begin{defi}
  \label{defi:a3ee2dc588f4bf0f}
  We say that a \( \tau \)-bialgebra \( H \) in \( \mathcal{C} \) is a
  \textbf{Hopf algebra in } \( \widehat{\mathsf{LCS}}_{\tau} \), or simply a
  \( \tau \)-\textbf{Hopf algebra}, if \( \id_{H} \in \mathcal{L}(H, H) \)
  admits a multiplicative inverse in the convolution algebra
  \( \mathcal{L}(H, H) \). We call this multiplicative inverse the
  \textbf{antipode} of the \( \tau \)-Hopf algebra, and often denote it by
  \( S_{H} \), or simply \( S \). \textbf{Morphisms} between \( \tau \)-Hopf
  algebras are the ones that are morphisms of the underlying
  \( \tau \)-bialgebras.
\end{defi}

As in the classical case, for morphisms of \( \tau \)-Hopf algebras, the
antipode takes care of itself, which is no surprise since it is uniquely
determined by the \( \tau \)-bialgebra structure.
\begin{prop}
  \label{prop:dfd77ff270dde74a}
  Let \( H_{1} \), \( H_{2} \) be \( \tau \)-Hopf algebras with antipodes
  \( S_{1} \) and \( S_{2} \) respectively. If \( f : H_{1} \to H_{2} \) is a
  morphism of \( \tau \)-Hopf algebras, then \( f S_{1} = S_{2}f \).
\end{prop}
\begin{proof}
  Proof again parallels the classical case (see, e.g.\ \cite{MR2397671}*{p22,
    Proposition~1.3.17}).
\end{proof}

\begin{rema}
  \label{rema:ddd970d827d67782}
  Unwinding the definition, we see that \( S \in \mathcal{L}(H, H) \) is the
  antipode if and only if
  \begin{equation}
    \label{eq:c8e35b71db95f592}
    m(\id \overline{\otimes}_{\tau} S)\Delta = m(S \overline{\otimes}_{\tau} \id)\Delta = \id.
  \end{equation}
  Formally, this is the usual definition of an antipode. We use our equivalent
  Definition~\ref{defi:a3ee2dc588f4bf0f} to emphasize the fact the antipode is
  unique once it exists---it is already uniquely determined by the bialgebra
  structure.
\end{rema}

In \cite{MR1220906}, van Daele gives a characterization for when a bialgebra
admits an antipode. Here's a version for locally convex bialgebras.
\begin{prop}
  \label{prop:74b9c50fedf2379f}
  Let \( \overline{\otimes}_{\tau} \) be a compatible symmetric monoidal functor
  and \( H \) a \( \tau \)-bialgebra. Then \( H \) is a \( \tau \)-Hopf algebra
  if and only if the maps
  \( T_{1} := (\id_{H} \overline{\otimes}_{\tau} m)(\Delta
  \overline{\otimes}_{\tau} \id_{H}) \) and
  \( T_{2}: = (m \overline{\otimes}_{\tau} \id_{H})(\id_{H}
  \overline{\otimes}_{\tau} \Delta) \) are isomorphisms of locally convex spaces
  from \( H \overline{\otimes}_{\tau} H \) onto itself.
\end{prop}
\begin{proof}
  This basically boils down to an obvious adaptation of van Daele's work on the
  algebraic case \cite{MR1220906}. More precisely, for necessity, let \( S \) be
  the antipode of \( H \), following \cite{MR1220906}*{Proposition~2.1}, we
  define
  \begin{displaymath}
    R_{1}:= (\id_{H} \overline{\otimes}_{\tau} m)(\id_{H} \overline{\otimes}_{\tau} S \overline{\otimes}_{\tau} \id)
    (\Delta \overline{\otimes}_{\tau} \id_{H})
  \end{displaymath}
  and
  \begin{displaymath}
    R_{2}:= (m \overline{\otimes}_{\tau} \id)(\id_{H} \overline{\otimes}_{\tau} S \overline{\otimes}_{\tau} \id_{H})
    (\id_{H} \overline{\otimes}_{\tau} \Delta).
  \end{displaymath}
  For each \( i = 1, 2 \), it follows from a routine check that \( R_{i} \) is
  clearly a continuous linear map from the locally convex space
  \( H \overline{\otimes}_{\tau} H \) to itself, and is the inverse of
  \( T_{i} \), thus \( T_{i} \) is indeed an isomorphism.

  For sufficiency, define \( S: H \to H \) by
  \( S(a):= (\varepsilon \overline{\otimes}_{\tau} \id_{H}) T_{1}^{-1}(a \otimes
  1) \). It follows from Proposition~\ref{prop:671875b580127e64} that \( S \) is
  continuous. Now the verification of \( S \) being an antipode for the
  \( \tau \)-bialgebra of \( H \) follows from (an even simpler due to the
  existence of the unit and counit) adaptation of \cite{MR1220906}*{\S~4} using
  routine continuity and density arguments.
\end{proof}

\begin{prop}
  \label{prop:bc869f35c39bbe44}
  Let \( H \) be a \( \tau \)-Hopf algebra with antipode \( S \). Then \( S \),
  viewed as a map from the \( \tau \)-bialgebra to the \( \tau \)-bialgebra
  \( H^{\op,\cop} \), is a morphism of \( \tau \)-bialgebras.
\end{prop}
\begin{proof}
  This follows again from a routine adaptation of the classical argument (see,
  e.g.\ \cite{MR2397671}*{p18, Proposition~1.3.12}) using continuity and
  density.
\end{proof}

Observe that in general, the antipode for a Hopf algebra need \emph{not} be an
isomorphism.
\begin{rema}
  \label{rema:2331bbbfef8d22c5}
  As a counterexample, consider classical coalgebras and Hopf algebras etc.\
  over \( \mathbb{K} \). There is a construction called the \textbf{free Hopf
    algebra} \( H(C) \) of a coalgebra \( C \), given by Takeuchi
  \cite{MR0292876}*{pp262--265}. Recall that a coalgebra is pointed if each of
  its simple subalgebra is one-dimensional, cf.\ e.g.\ \cite{MR2894855}*{p100,
    Definition~3.4.4}. In our case where \( k \) is assumed to be algebraically
  closed, a result of Takeuchi \cite{MR0292876}*{p569, Theorem~18} showed that
  the antipode of \( H(C) \) is invertible if and only if the coalgebra \( C \)
  pointed. But the comatrix coalgebra \( C_{S}(\mathbb{K}) \)
  (\cite{MR2894855}*{p26, Definition~2.1.16}) on a finite set \( S \) of more
  than two elements is clearly simple, hence not pointed. So the antipode of
  \( H\bigl(C_{S}(\mathbb{K})\bigr) \) is not invertible. Moreover, note that as
  a vector space, \( H\bigl(C_{S}(\mathbb{K})\bigr) \) is of countable dimension
  from its construction \cite{MR0292876}*{p262}. In
  \S~\ref{sec:b6aa227ba43c21a5}, we will see how to fit all classical Hopf
  algebras into our framework of locally convex Hopf algebras, so these
  classical counterexamples also work in our setting.
\end{rema}

\begin{defi}
  \label{defi:625bdb1be01b0b3c}
  We say a \( \tau \)-Hopf algebra is \textbf{regular}, if its antipode is an
  isomorphism of locally convex spaces.
\end{defi}

\begin{prop}
  \label{prop:5e51bd7b92ec1787}
  Let \( H \) be a \( \tau \)-Hopf algebra with antipode \( S \), the following
  are equivalent:
  \begin{enumerate}
  \item \label{item:1c0fec38e8aecac8} \( H \) is regular with \( T \) being the
    inverse of \( S \);
  \item \label{item:10d6fb3d3b0b9531} The bialgebra \( H^{\op} \) admits
    antipode \( T \) (so is a \( \tau \)-Hopf algebra);
  \item \label{item:0784fd83f6cedc62} The bialgebra \( H^{\cop} \) admits
    antipode \( T \) (so is a \( \tau \)-Hopf algebra).
  \end{enumerate}
\end{prop}
\begin{proof}
  The proof is a routine check and parallels that of \cite{MR2397671}*{p21,
    Lemma~1.3.15}.
\end{proof}

\subsection{Involutive locally convex Hopf algebras}
\label{sec:7d03cf6e99381be5}

Throughout \S~\ref{sec:7d03cf6e99381be5}, we suppose \( \mathbb{K} = \C \), and
introduce the notion the involutive locally convex Hopf algebras, or simply
locally convex Hopf-\( \ast \) algebras.

\begin{defi}
  \label{defi:548ebf8ed72d45ca}
  An \textbf{involution} \( (\cdot)^{\ast}: A \to A \) on a \( \tau \)-algebra
  \( A \) is an antilinear homeomorphism, such that \( x^{\ast\ast} = x \),
  \( (yx)^{\ast} = x^{\ast}y^{\ast} \) for all \( x, y \in A \). An
  \textbf{involutive \( \tau \)-algebra}, or simply a \( \tau
  \)-\( \ast \)-\textbf{algebra}, is a \( \tau \)-algebra equipped with an
  involution. An \textbf{involutive \( \tau \)-bialgebra}, or simply
  \( \tau \)-\( \ast \)-\textbf{bialgebra}, \( B \), is one such that the
  comultiplication preserves the involution, where the involution on
  \( B \overline{\otimes}_{\tau} B \) is the unique continuous extension of the
  map
  \( \sum_{k}x_{k} \otimes y_{k} \mapsto \sum_{k}x_{k}^{\ast} \otimes
  y_{k}^{\ast} \), which is an antilinear homeomorphism from
  \( B \otimes_{\tau} B \) onto itself. A \( \tau \)-\( \ast \)-bialgebra that
  admits antipode is called an \textbf{involutive \( \tau \)-Hopf algebra}, or
  simply a \( \tau \)-\textbf{Hopf-\( \ast \) algebra}.
\end{defi}

\begin{prop}
  \label{prop:b358de0ebf54512e}
  Let \( H \) be a \( \tau \)-Hopf-\( \ast \) algebra. The following hold:
  \begin{enumerate}
  \item \label{item:f59f2b277b538b62} the counit \( \varepsilon \) is
    involutive, i.e.\ \( \varepsilon(x^{\ast}) = \overline{\varepsilon(x)} \)
    for all \( x \in H \);
  \item \label{item:5a806a0923e2e291} The antipode \( S \) satisfies
    \( S \circ (\cdot)^{\ast} \circ S \circ (\cdot)^{\ast} = \id_{H} \). In
    particular, \( S \) is an isomorphism and \( H \) is regular.
  \end{enumerate}
\end{prop}
\begin{proof}
  The proofs parallels the classical theory (see, e.g.\ \cite{MR2397671}*{p27}).
\end{proof}

Among many of its usage, we mention in advance that locally convex
Hopf-\( \ast \) algebras will play a vital role when we describe how Pontryagin
duality for second countable locally compact abelian groups can manifest in our
theory.

\subsection{Dualizability and reflexivity}
\label{sec:9b15e9ed66625847}

Recall Definition~\ref{defi:12b42590a911540b} and
Definition~\ref{defi:d9087abf5fbcb620}.

\begin{defi}
  \label{defi:a6ba1501ad62d100}
  Let \( \overline{\otimes}_{\tau} \), \( \overline{\otimes}_{\sigma} \) be
  compatible symmetric monoidal functors and \( E \in \widehat{\mathsf{LCS}} \). We say \( E \) is
  \begin{itemize}
  \item \( (\tau, \sigma) \)-\textbf{dualizable}, if the strong dual
    \( E'_{b} \) is complete, and if the canonical algebraic pairings
    \begin{displaymath}
      \pairing*{E}{E'} , \;  \pairing*{E \odot E}{E' \odot E'},\;
      \pairing*{E \odot E \odot E}{E' \odot E' \odot E'},
      \;\text{ and }\;
      \pairing*{E \odot E \odot E \odot E}{E' \odot E' \odot E' \odot E'}
    \end{displaymath}
    extends respectively to duality pairings
    \begin{equation}
      \label{eq:9dc1aa4a05ea2221}
      \begin{gathered}
        \pairing*{E}{E'_{b}} ,\;
        \pairing*{E \overline{\otimes}_{\tau} E}{E'_{b}
          \overline{\otimes}_{\sigma} E'_{b}}, \;
        \pairing*{E \overline{\otimes}_{\tau} E \overline{\otimes}_{\tau}
          E}{E'_{b} \overline{\otimes}_{\sigma} E'_{b} \overline{\otimes}_{\sigma}
          E'_{b}} \\
        \text{ and }\;
        \pairing*{E \overline{\otimes}_{\tau} E \overline{\otimes}_{\tau}
          E \overline{\otimes}_{\tau} E}{E'_{b} \overline{\otimes}_{\sigma} E'_{b} \overline{\otimes}_{\sigma}
          E'_{b} \overline{\otimes}_{\sigma} E'_{b}}
      \end{gathered}
    \end{equation}
    that are compatible respectively with \( E \),
    \( E \overline{\otimes}_{\tau} E \),
    \( E \overline{\otimes}_{\tau} E \overline{\otimes}_{\tau} E \) and
    \( E \overline{\otimes}_{\tau} E \overline{\otimes}_{\tau} E
    \overline{\otimes}_{\tau} E \);
  \item \( (\tau, \sigma) \)-\textbf{reflexive}, if it is dualizable and the
    pairings in \eqref{eq:9dc1aa4a05ea2221} are reflexive;
  \item  \( (\tau, \sigma) \)-\textbf{polar dualizable}, if the polar dual
    \( E'_{c} \) is complete, and if the canonical algebraic pairings
    \begin{displaymath}
      \pairing*{E}{E'} , \;  \pairing*{E \odot E}{E' \odot E'}, \;
      \pairing*{E \odot E \odot E}{E' \odot E' \odot E'},
      \;\text{ and }\;
      \pairing*{E \odot E \odot E \odot E}{E' \odot E' \odot E' \odot E'}
    \end{displaymath}
    extends respectively to duality pairings
    \begin{equation}
      \label{eq:9b4a9def286218d1}
      \begin{gathered}
        \pairing*{E}{E'_{c}} ,\;
        \pairing*{E \overline{\otimes}_{\tau} E}{E'_{c}
          \overline{\otimes}_{\sigma} E'_{c}}, \;
        \pairing*{E \overline{\otimes}_{\tau} E \overline{\otimes}_{\tau}
          E}{E'_{c} \overline{\otimes}_{\sigma} E'_{c} \overline{\otimes}_{\sigma}
          E'_{c}} \\
        \text{ and }\;
        \pairing*{E \overline{\otimes}_{\tau} E \overline{\otimes}_{\tau}
          E \overline{\otimes}_{\tau} E}{E'_{c} \overline{\otimes}_{\sigma} E'_{c} \overline{\otimes}_{\sigma}
          E'_{c} \overline{\otimes}_{\sigma} E'_{c}}
      \end{gathered}
    \end{equation}
    that are polar compatible respectively with \( E \),
    \( E \overline{\otimes}_{\tau} E \),
    \( E \overline{\otimes}_{\tau} E \overline{\otimes}_{\tau} E \) and
    \( E \overline{\otimes}_{\tau} E \overline{\otimes}_{\tau} E
    \overline{\otimes}_{\tau} E \);
  \item \( (\tau, \sigma) \)-\textbf{polar reflexive}, if it is dualizable and
    the pairings in \eqref{eq:9b4a9def286218d1} are polar reflexive;
  \end{itemize}
\end{defi}

\begin{prop}
  \label{prop:844fe110a4f00d8b}
  Let \( \overline{\otimes}_{\tau} \), \( \overline{\otimes}_{\sigma} \) be
  compatible symmetric monoidal functors and \( H \in \widehat{\mathsf{LCS}}
  \). Then the following holds:
  \begin{enumerate}
  \item \label{item:4e35c3dbe6ccf7b3} If \( H \) is
    \( (\tau, \sigma) \)-dualizable and \( H \) has a structure of
    \( \tau \)-algebra (resp.\ coalgebra, bialgebra, Hopf algebra, regular Hopf
    algebra), then the strong dual \( H'_{b} \) has a structure of
    \( \sigma \)-coalgebra (resp.\ algebra, bialgebra, Hopf algebra, regular
    Hopf algebra) by taking transposes of the structure maps of \( H \).
  \item \label{item:539e6db5ab32a67f} If \( \mathbb{K} = \C \), \( H \) is
    \( (\tau, \sigma) \)-dualizable, and admits a \( \tau \)-Hopf-\( \ast \)
    algebra structure, then \( H'_{b} \) has a \( \sigma \)-Hopf-\( \ast \)
    algebra structure where the \( \sigma \)-Hopf algebra structure is as in
    \ref{item:4e35c3dbe6ccf7b3}, and the involution \( \omega \in H'_{b} \mapsto \omega^{\ast} \in H'_{b} \) given by
    \begin{equation}
      \label{eq:650baecac6902597}
      \omega^{\ast}(x) = \overline{\pairing*{[S(x)]^{\ast}}{\omega}}, \qquad x \in H,
    \end{equation}
    where \( S \) is the antipode on \( H \).
  \item \label{item:f927c48f1f02bd29} If \( H \) is \( (\tau, \sigma) \)-polar
    dualizable and \( H \) has a structure of \( \tau \)-algebra (resp.\
    coalgebra, bialgebra, Hopf algebra, regular Hopf algebra), then the polar
    dual \( H'_{c} \) has a structure of \( \sigma \)-coalgebra (resp.\ algebra,
    bialgebra, Hopf algebra, regular Hopf algebra) by taking transposes of the
    structure maps of \( H \).
  \item \label{item:97082c3bb390a292} If \( \mathbb{K} = \C \), \( H \) is
    \( (\tau, \sigma) \)-polar dualizable, and admits a \( \tau
    \)-Hopf-\( \ast \) algebra structure, then \( H'_{c} \) has a
    \( \sigma \)-Hopf-\( \ast \) algebra structure where the \( \sigma \)-Hopf
    algebra structure is as in \ref{item:f927c48f1f02bd29}, and the involution
    \( \omega \in H'_{c} \mapsto \omega^{\ast} \in H'_{c} \) given by
    \begin{equation}
      \label{eq:e9cb18c145d094fe}
      \omega^{\ast}(x) = \overline{\pairing*{[S(x)]^{\ast}}{\omega}}, \qquad x \in H,
    \end{equation}
    where \( S \) is the antipode on \( H \).
  \end{enumerate}
\end{prop}
\begin{proof}
  Note that since bounded sets (resp.\ precompact sets) are preserved under
  continuous linear maps, the transpose of continuous linear maps remains
  continuous as maps on the strong (resp.\ polar) duals. Note also that the
  involution on \( H \) also preserves bounded sets (resp.\ precompact
  sets). The rest of the proof is parallel to the case of classical (finite
  dimensional) Hopf algebras, which are well-known, and can be easily
  checked. We therefore leave it as a routine exercise, and merely point out
  that the (co)associativity of (co)multiplication involves three-fold tensor
  products, and the compatibility of the algebra and coalgebra structures
  (Proposition~\ref{prop:717c7c133e9e974c}) involves four-fold tensor products.
\end{proof}

\begin{defi}
  \label{defi:783a70411adbf39d}
  We place ourselves in the settings of Proposition~\ref{prop:844fe110a4f00d8b}.
  We say that a \( \tau \)-algebra (resp.\ coalgebra, bialgebra, Hopf algebra,
  Hopf-\( \ast \)-algebra) \( H \) is \( (\tau, \sigma) \)-\textbf{dualizable}
  (resp.\ \textbf{polar dualizable}), if \( H \) is so as a complete locally
  convex space. We define \( (\tau, \sigma) \)-\textbf{reflexivity} (resp.\
  \textbf{polar reflexivity}) similarly. In these cases, we say that \( H \) is
  \textbf{reflexive} (resp.\ \textbf{polar reflexive}), if \( H \) is
  \( (\tau, \sigma) \)-\textbf{reflexive} (resp.\ \textbf{polar reflexive}) for
  some compatible symmetric monoidal functor \( \overline{\otimes}_{\sigma} \).

  In the dualizable (resp.\ polar dualizable) case, we call \( H'_{b} \) (resp.\
  \( H'_{c} \)), equipped the corresponding obtained structure maps, the
  \textbf{strong dual} (resp.\ \textbf{polar dual}) of \( H \), which is a
  \( \sigma \)-coalgebra (resp.\ algebra, bialgebra, Hopf algebra,
  Hopf-\( \ast \) algebra) if \( H \) is a \( \tau \)-algebra (resp.\ coalgebra,
  bialgebra, Hopf algebra, Hopf-\( \ast \) algebra).

  In all of the above, if \( H \) is nuclear and
  \( \tau \in \set*{\varepsilon, \pi} \), then for simplicity, we merely say
  \( \sigma \)-dualizable instead of \( (\tau, \sigma) \)-dualizable, and
  similarly for \( \sigma \)-reflexive and the polar versions of dualizability
  and reflexivity.
\end{defi}

\begin{rema}
  \label{rema:28e309a0f4e6146b}
  As the duality theory of topological tensor product is quite subtle, we shall
  often specify which compatible symmetric monoidal functors are involved in
  studying notions defined above.
\end{rema}

Here's a version of Pontryagin type duality for locally convex algebras (resp.\
coalgebras, bialgebras, Hopf algebras, regular Hopf algebras,
Hopf-\( \ast \)-algebras).

\begin{prop}
  \label{prop:9ae8172142f23b99}
  Le \( \overline{\otimes}_{\tau} \), \( \overline{\otimes}_{\sigma} \) be
  compatible symmetric monoidal functors and \( H \) a \( \tau \)-algebra
  (resp.\ coalgebras, bialgebras, Hopf algebras, regular Hopf algebras,
  Hopf-\( \ast \)-algebras) that is \( (\tau, \sigma) \)-reflexive (polar
  reflexive), then the second strong dual \( (H'_{b})'_{b} \) (resp.\ second
  polar dual \( (H'_{c})'_{c} \)) remains so, and is isomorphic to \( H \) via
  the natural bijection.
\end{prop}
\begin{proof}
  This follows from an easy unwinding of Definition~\ref{defi:783a70411adbf39d}
  and Proposition~\ref{prop:844fe110a4f00d8b}.
\end{proof}

In the following, many general results have an analogue for the duality between
locally convex algebras and coalgebras, as well between locally convex
bialgebras. Since our main focus is on Hopf algebras, to avoid being too
pedantic and for the sake of readability as well as clarity, we will leave the
easy corresponding formulation and proofs for structures to our readers and only
state and prove the ones involving locally convex Hopf algebras.

\section{Duality of topological tensor products and locally convex Hopf
  algebras}
\label{sec:7848cae39b1cc5b0}

\begin{nota}
  \label{nota:fcc41733902f400b}
  We adopt the following notation to denote the following several classes of
  locally convex spaces:
  \( (\mathcal{FM}) = (\mathcal{F}) \cap (\mathcal{M}) \),
  \( (\mathcal{DFM}) = (\mathcal{DF}) \cap (\mathcal{M}) \),
  \( (\mathcal{FN}) = (\mathcal{F}) \cap (\mathcal{N}) \). The class of all
  reflexive (polar reflexive) spaces is denoted by \( (\mathcal{R}) \) (resp.\
  \( (\mathcal{R}_{p}) \)). Recall that we use \( (\mathcal{F}'_{c}) \) to
  denote the class of all locally convex spaces that are isomorphic to the polar
  dual of some \( (F) \)-space.
\end{nota}

\subsection{The Buchwalter duality}
\label{sec:0b2e5745467c3681}

As a first step out of framework outside Banach algebras such as \( C^{\ast} \)
and \( W^{\ast} \)-algebras (see, e.g.\
\cites{takesaki_theory_2002,takesaki_theory_2003}), which are crucial in the
theory of (locally) compact quantum groups (cf.\
\cites{MR1616348,woronowicz_compact_1987,MR1378538,MR1277138,MR1832993,MR1951446}),
we now consider the question of (polar) reflexivity of locally convex Hopf
algebras of class \( (\mathcal{F}) \).

Buchwalter has established the following duality result which plays a
fundamental role in our subsequent study.

\begin{prop}[\cite{MR0451189}*{p56, Proposition~2.7}]
  \label{prop:2a9c783caf76b6fb}
  Let \( E \), \( F \) be \( (F) \)-spaces. If at least one of \( E \), \( F \)
  has the approximation property, then
  \begin{enumerate}
  \item \label{item:32d8c9487a724d01}
    \( (E \overline{\otimes}_{\pi} F)'_{c} = E'_{c} \overline{\otimes}_{\varepsilon} F'_{c} \) and
    \( (E'_{c} \overline{\otimes}_{\varepsilon} F'_{c})'_{c} = E \overline{\otimes}_{\pi} F \);
  \item \label{item:93ec69e754c7fd39}
    \( (E \overline{\otimes}_{\varepsilon} F)'_{c} = E'_{c}
    \overline{\otimes}_{\pi} F'_{c} \) and
    \( (E'_{c} \overline{\otimes}_{\pi} F'_{c})'_{c} = E
    \overline{\otimes}_{\varepsilon} F \).
  \end{enumerate}
\end{prop}

Before presenting our first result on the (polar) reflexivity of locally convex
Hopf algebras, we also need the following preparation.

\begin{prop}
  \label{prop:8ffb913b5f5ddea6}
  The following hold.
  \begin{enumerate}
  \item \label{item:272241af97d82589} All locally convex spaces in
    \( (\mathcal{F}) \cap (\mathcal{AP}) \) and
    \( (\mathcal{F}'_{c}) \cap (\mathcal{AP}) \) are complete; and viewing
    \( (\mathcal{F}) \cap (\mathcal{AP}) \) and
    \( (\mathcal{F}'_{c}) \cap (\mathcal{AP}) \) as full subcategory of
    \( \widehat{\mathsf{LCS}} \), taking polars duals and transposes establishes
    an anti-equivalence of categories between \( (\mathcal{F}) \cap (\mathcal{AP}) \) and
    \( (\mathcal{F}'_{c}) \cap (\mathcal{AP}) \).
  \item \label{item:9e56bc6f15cd09bf} The class
    \( (\mathcal{F}) \cap (\mathcal{AP}) \) is stable under taking completed
    injective tensor product, and the class
    \( (\mathcal{F}'_{c}) \cap (\mathcal{AP}) \) stable under taking
    completed projective tensor product.
  \item \label{item:4f036dcd1d2124e3} All locally convex spaces in
    \( (\mathcal{FM}) \cap (\mathcal{AP}) \) and
    \( (\mathcal{DFM}) \cap (\mathcal{AP}) \) are complete, and taking polars
    duals establishes a bijection between
    \( (\mathcal{FM}) \cap (\mathcal{AP}) \) and
    \( (\mathcal{DFM}) \cap (\mathcal{AP}) \).
  \item \label{item:434fda161e81e6c0} The class
    \( (\mathcal{FM}) \cap (\mathcal{AP}) \) is stable under taking completed
    injective tensor product, and the class
    \( (\mathcal{DFM}) \cap (\mathcal{AP}) \) is stable under taking completed
    projective tensor products.
  \end{enumerate}
\end{prop}
\begin{proof}
  Being metrizable, an \( (F) \)-space is barrelled and bornological
  (Proposition~\ref{prop:7963cb2832fabbd4}), hence Mackey
  (Proposition~\ref{prop:a944f1bb11e881f1}). Now \ref{item:272241af97d82589}
  follows from Proposition~\ref{prop:2ff9c3b3fe361ef4},
  Proposition~\ref{prop:6ac1bb9a146288e4},
  Proposition~\ref{prop:273d1ea40b816ac9}. \ref{item:9e56bc6f15cd09bf}, the
  \( (F) \)-space case follows from Proposition~\ref{prop:20a96aff7596345b} and
  Corollary~\ref{coro:97adf837ae63af02}, and the case for
  \( (\mathcal{F}'_{c}) \) follows from Buchwalter's duality
  (Proposition~\ref{prop:2a9c783caf76b6fb}). Note that the strong dual and the
  polar dual of a Montel space coincide
  (Proposition~\ref{prop:99b4fc84a9b1d8a1}), then \ref{item:4f036dcd1d2124e3}
  follows from \ref{item:272241af97d82589} and \ref{item:434fda161e81e6c0} from
  \ref{item:9e56bc6f15cd09bf}.
\end{proof}

\begin{rema}
  \label{rema:1ffe0a62ed31a181}
  It seems, however, unclear that whether \( E \overline{\otimes}_{\pi} F \)
  remains a \( (FM) \)-space, if both \( E \) and \( F \) are, cf.\
  \cite{MR0551623}*{p304, (7)}
\end{rema}

Now recall Definition~\ref{defi:a6ba1501ad62d100}.
\begin{coro}
  \label{coro:10966fc71ea114e8}
  The following holds.
  \begin{enumerate}
  \item \label{item:01f3c816210da0cc} Every space in
    \( (\mathcal{F}) \cap (\mathcal{AP}) \) (resp.\ in
    \( (\mathcal{F}'_{c}) \cap (\mathcal{AP}) \)) is both
    \( (\varepsilon, \pi) \)-polar reflexive and \( (\pi, \varepsilon) \)-polar
    reflexive.
  \item \label{item:07b55e42c6627a79} Every space in
    \( (\mathcal{FM}) \cap (\mathcal{AP}) \) (resp.\
    \( (\mathcal{DFM}) \cap (\mathcal{AP}) \)) is
    \( (\varepsilon, \pi) \)-reflexive (resp.\
    \( (\pi, \varepsilon) \)-reflexive).
  \end{enumerate}
\end{coro}
\begin{proof}
  Even though we do not know whether \( (\mathcal{FM}) \) (nor for
  \( (\mathcal{AP}) \)) is stable under taking the completed projective tensor
  products. Nevertheless, \( (\mathcal{F}) \) is still stable under
  \( \overline{\otimes}_{\pi} \) (Corollary~\ref{coro:0be09f15f7e01dce}). Note
  that Buchwalter's duality (Proposition~\ref{prop:2a9c783caf76b6fb}) holds as
  soon as merely one factor has the approximation property, the corollary now
  follows from Proposition~\ref{prop:8ffb913b5f5ddea6} and repeated use of the
  Buchwalter duality (for three-fold completed projective or injective tensor
  products).
\end{proof}

We now easily obtain our first criterion on polar reflexivity and reflexivity of
locally convex Hopf algebras.

\begin{theo}
  \label{theo:5f3cb1a05114cdca}
  The following hold.
  \begin{enumerate}
  \item \label{item:02787e032fdd7835} If \( H \) is an \( \varepsilon \)-Hopf
    algebra (resp.\ \( \pi \)-Hopf algebra) of class
    \( (\mathcal{F}) \cap (\mathcal{AP}) \), then \( H \) is
    \( (\varepsilon, \pi) \)-polar reflexive (resp.\
    \( (\pi, \varepsilon) \)-reflexive), and the polar dual \( H'_{c} \) is of
    class \( (\mathcal{F}'_{c}) \).
  \item \label{item:63279b35f6379efd} If \( H \) is an \( \varepsilon \)-Hopf
    algebra (resp.\ \( \pi \)-Hopf algebra) of class
    \( (\mathcal{F}'_{c}) \cap (\mathcal{AP}) \), then \( H \) is
    \( (\varepsilon, \pi) \)-polar reflexive (resp.\
    \( (\pi, \varepsilon) \)-reflexive), and the polar dual \( H'_{c} \) is of
    class \( (\mathcal{F}) \).
  \item \label{item:86814a0291a8a634} If \( H \) is an \( \varepsilon \)-Hopf
    algebra of class \( (\mathcal{FM}) \cap (\mathcal{AP}) \), then \( H \) is
    \( (\varepsilon, \pi) \)-reflexive, and the strong dual \( H'_{b} \) is of
    class \( (\mathcal{DFM}) \cap (\mathcal{AP}) \).
  \item \label{item:fc7a957164b71231} If \( H \) is an \( \pi \)-Hopf algebra of
    class \( (\mathcal{DFM}) \cap (\mathcal{AP}) \), then \( H \) is
    \( (\pi, \varepsilon) \)-reflexive, and the strong dual \( H'_{b} \) is of
    class \( (\mathcal{FM}) \cap (\mathcal{AP}) \).
  \end{enumerate}
  All of the above still holds if \( \mathbb{K} = \C \), and we replace the
  corresponding locally convex Hopf algebras by locally convex
  Hopf-\( \ast \)-algebras.
\end{theo}
\begin{proof}
  This follows from Corollary~\ref{coro:10966fc71ea114e8} and
  Proposition~\ref{prop:844fe110a4f00d8b}.
\end{proof}

\begin{rema}
  \label{rema:c1d7fd8f084c0e17}
  Note that due to Remark~\ref{rema:1ffe0a62ed31a181}, we don't know yet that
  whether a \( \pi \)-Hopf algebra of class
  \( (\mathcal{FM}) \cap (\mathcal{AP}) \) is
  \( (\pi, \varepsilon) \)-reflexive, since we don't now whether the locally
  convex space \( H \overline{\otimes}_{\pi} H \in (\mathcal{M}) \) to continue
  the use of the Buchwalter's duality for the three-fold tensor products of
  \( H \) with itself (the strong dual and the polar dual for
  \( H \overline{\otimes}_{\pi} H \) might now be different).
\end{rema}

\begin{rema}
  \label{rema:18c5fa54f0e4fe49}
  In \cite{MR1266072}*{p126, Definition~1.2}, the term ``well-behaved
  topological Hopf algebra'' is used to refer to what we call
  \( \varepsilon \)-Hopf algebra (or equivalent, \( \pi \)-Hopf algebra) of
  class \( (\mathcal{F}M) \cap (\mathcal{N}) \) or of class
  \( (\mathcal{DFM}) \cap (\mathcal{N}) \) (note
  Proposition~\ref{prop:a8139d3432a18d9e} and
  Proposition~\ref{prop:f4cbc15f463be1bc}). Our
  Theorem~\ref{theo:5f3cb1a05114cdca} extends the main duality result for
  well-behaved topological Hopf algebras \cite{MR1266072}*{p129,
    Proposition~1.3}. Thus all the examples of well-behaved topological Hopf
  algebras described in \cite{MR1266072} are examples of our theory. Of course,
  many more examples of our theory (that fit in
  Theorem~\ref{theo:5f3cb1a05114cdca} or go beyond) shall be given later.
\end{rema}

\subsection{The polar duality between products and locally convex direct sums}
\label{sec:12265e858f6868a5}

As a preparation for  \S~\ref{sec:d0b3014d3c129a16}, we
establish the following duality results.

\begin{prop}
  \label{prop:9f3f8bda0402a4ff}
  Let \( (E_{i})_{i \in I} \) be a family of polar reflexive spaces. Then the
  canonical pairing
  \begin{equation}
    \label{eq:a5ea2c5bf5c2ee22}
    \begin{split}
      \pairing*{\cdot}{\cdot}:
      \left(\prod_{i} E_{i}\right) \times \left(\bigoplus_{i \in I}(E_{i})'_{c}\right)
      & \to \mathbb{K} \\
      \Bigl((x_{i})_{i \in I}, \, (l_{i})_{i \in I}\Bigr)
      & \mapsto \sum_{i} l_{i}(x_{i})
    \end{split}
  \end{equation}
  is polar reflexive. More precisely, \( \prod_{i} E_{i} \) and
  \( \oplus_{i}(E_{i})'_{c} \) are polar reflexive and they are the polar dual
  of each other.
\end{prop}

\begin{proof}
  Apply Proposition~\ref{prop:e5dd78c27c8725c9} to the weak topologies and note
  that the topology of precompact convergence lies between the weak and the
  Mackey topologies, we see that the pairing \eqref{eq:a5ea2c5bf5c2ee22} is
  compatible.

  It is clear from the definition of product (uniform) spaces that a set
  \( A \subseteq \prod_{i} E_{i} \) is precompact if and only if
  \( p_{i}(A) \in \mathfrak{C}(E_{i}) \) for each \( i \in I \) (recall
  Notation~\ref{nota:6a66d03bb5d7af70}). By
  Proposition~\ref{prop:6c0374a8634e2b65}, we see that \( A \) is precompact if
  and only if there exists an absolutely convex
  \( C_{i} \in \mathfrak{C}(E_{i}) \) for each \( i \in I \), such that
  \( A \subseteq \prod_{i} C_{i} \). Hence absolutely convex sets of the form
  \( \left(\prod_{i} C_{i}\right)^{\circ} \),
  \( \Gamma(C_{i}) = C_{i} \in \mathfrak{C}(E_{i}) \) for each \( i \), form a
  fundamental system of neighborhoods of \( 0 \) for the polar dual of
  \( \prod_{i} E_{i} \), which we identify with \( \oplus_{i} E'_{i} \), and the
  polar is calculated with respect to the canonical pairing
  \( \pairing*{\prod_{i}E_{i}}{\oplus_{i} E'_{i}} \).

  We claim that
  \begin{equation}
    \label{eq:920af95588c06022}
    \left(\prod_{i}C_{i}\right)^{\circ} = \overline{\Gamma\left(\cup_{i}
        C_{i}^{\circ}\right)}^{\sigma\left(\oplus_{i} E'_{i},\, \prod_{i}E_{i}\right)},
  \end{equation}
  where the polar \( C_{i}^{\circ} \) is taken with respect to the canonical
  pairing \( \pairing*{E_{i}}{E'_{i}} \), then seen as in
  \( \oplus_{i}E'_{i} \). Indeed, it is clear that
  \( (\prod_{i} C_{i})^{\circ} \) contains
  \( \Gamma\left(\cup_{i}C_{i}^{\circ}\right) \). For the reverse inclusion, let
  \( (l_{i}) \in \oplus_{i} E'_{i} \) lie in
  \( \left(\prod_{i}C_{i}\right)^{\circ} \). Note that all but finitely many
  \( l_{i} \in E'_{i} \) is zero. Define
  \( \lambda_{i} = \sup\set[\big]{\abs*{l_{i}(x_{i})} \given x_{i} \in C_{i}} \)
  and \( I_{0} = \set*{i \in I \given \lambda_{i} > 0} \), which is finite since
  \( l_{i} = 0 \) clearly implies \( \lambda_{i} = 0 \). Let
  \( \omega_{i} = \lambda_{i}^{-1}l_{i} \) for \( i \in I_{0} \). By definition,
  \( \omega_{i} \in C_{i}^{\circ} \) for each \( i \in I_{0} \). Since
  \( (l_{i}) \in \left(\prod_{i \in I}C_{i}\right)^{\circ} \), we have
  \begin{equation}
    \label{eq:ffd885827ba2b39a}
    \forall(x_{i}) \in \prod_{i} C_{i}, \qquad
    \abs*{\sum_{i} l_{i}(x_{i})}
    = \abs*{\sum_{i \in I_{0}}\lambda_{i}\omega_{i}(x_{i})} \leq 1.
  \end{equation}
  As each \( C_{i} \) is absolutely convex, for each \( i \in I_{0} \), we may
  choose \( x_{i} \in C_{i} \) so that \( l_{i}(x_{i}) \) is arbitrarily close
  to \( \lambda_{i} \), i.e.\ \( \omega_{i}(x_{i}) \) arbitrarily close to
  \( 1 \), in \eqref{eq:ffd885827ba2b39a}, and obtain
  \begin{equation}
    \label{eq:e1ae535b3e9db197}
    \sum_{i \in I_{0}} \lambda_{i} \leq 1.
  \end{equation}
  Moreover, by definition we have
  \begin{equation}
    \label{eq:3e67ce71d99379ac}
    \sum_{i \in I_{0}} l_{i} = \sum_{i \in I_{0}} \lambda_{i} \omega_{i}.
  \end{equation}
  Now let
  \( I_{1} = \set*{i \in I \given i \notin I_{0}, \text{ and } l_{i} \ne 0}
  \). By the definition of \( I_{0} \), we know that \( l_{i}(C_{i}) = 0 \) for
  each \( i \in I_{1} \), and \( (l_{i}) = \sum_{i}l_{i} \) is supported in the
  disjoint union \( I_{0} \cup I_{1} \), which is finite. For each
  \( i \in I_{1} \), the condition \( l_{i}(C_{i}) = 0 \) implies that
  \( t l_{i} \in C_{i}^{\circ} \) for any \( t > 0 \). Thus for each
  \( t > 0 \), define
  \( l_{t} = \sum_{i \in I_{1}}t^{-1}l_{i} \), it is trivial
  that \( l_{t} \in \Gamma\bigl(\cup_{i \in I}C_{i}^{\circ}\bigr) \) and combined with \eqref{eq:3e67ce71d99379ac}, we have
  \begin{equation}
    \label{eq:d11812836ea6016f}
    (l_{i}) = l_{1} + \sum_{i \in I_{0}} \lambda_{i} \omega_{i}.
  \end{equation}
  Now for each \( \varepsilon > 0 \), by \eqref{eq:d11812836ea6016f}, we have
  \begin{equation}
    \label{eq:52fdbac8d8527ef2}
    (1 + \varepsilon)^{-1} (l_{i}) = (1 + \varepsilon)^{-1} l_{1} + \sum_{i \in I_{0}}(1 + \varepsilon)^{-1} \lambda_{i} \omega_{i}
    = \varepsilon(1 + \varepsilon)^{-1} l_{\varepsilon} + \sum_{i \in I_{0}} (1 + \varepsilon)^{-1} \lambda_{i} \omega_{i}.
  \end{equation}
  Now that \( l_{\varepsilon} \in \Gamma\left(\cup_{i}C_{i}^{\circ}\right) \)
  and for each \( i \in I_{0} \),
  \( \omega_{i} \in C_{i}^{\circ} \subseteq
  \Gamma\left(\cup_{i}C_{i}^{\circ}\right) \), and note
  \eqref{eq:e1ae535b3e9db197}, equation \eqref{eq:52fdbac8d8527ef2} writes
  \( (1 + \varepsilon)^{-1}(l_{i}) \) as an absolutely convex combination of
  elements in the absolutely convex hull
  \( \Gamma\left(\cup_{i}C_{i}^{\circ}\right) \), which is of course absolutely
  convex itself. Hence
  \( (1 + \varepsilon)^{-1}(l_{i}) \in \Gamma\left(\cup_{i}C_{i}^{\circ}\right)
  \) for each \( \varepsilon > 0 \), and letting \( \varepsilon \to 0 \), we
  obtain \eqref{eq:920af95588c06022}.

  Since each \( E_{i} \) is polar reflexive, as \( C_{i} \) runs through
  absolutely convex sets in \( \mathfrak{C}(E_{i}) \), \( C_{i}^{\circ} \) runs
  through a fundamental system of neighborhoods of \( 0 \) in \( E_{i} \). By
  Proposition~\ref{prop:8c1bccabeb0c2d99}, polars of the form in
  \eqref{eq:920af95588c06022} lies in
  \( \mathcal{N}^{\oplus_{i} (E_{i})'_{c}}_{\Gamma}(0) \), hence
  \( c\left(\oplus_{i} E'_{i}, \, \prod_{i \in I}E_{i}\right) \) is coarser than
  the locally convex direct sum topology \( \oplus_{i} (E_{i})'_{c} \).
  Conversely, let \( V \) be a \emph{closed} absolutely convex neighborhood of
  \( 0 \) in \( \oplus_{i}(E_{i})'_{c} \). Since the polar dual is finer than
  the weak dual, \( V \) is also weakly closed.  On the other hand, by the above
  discussion and Proposition~\ref{prop:8c1bccabeb0c2d99} again, \( V \) being a
  neighborhood of \( 0 \) means that it contains a set of the form
  \( \Gamma\left(\cup_{i \in I}C_{i}^{\circ}\right) \), where, without loss of
  generality, we may assume that \( C_{i} \in \mathfrak{C}(E_{i}) \) is
  absolutely convex (\( E_{i} \) is polar reflexive). Since \( V \) is weakly
  closed, we have
  \begin{equation}
    \label{eq:a50dc5aec14aee3d}
    V \supseteq \overline{\Gamma\left(\cup_{i}
        C_{i}^{\circ}\right)}^{\sigma\left(\oplus_{i} E'_{i},\, \prod_{i}E_{i}\right)}
    = \left(\prod_{i}C_{i}\right)^{\circ}.
  \end{equation}
  As closed absolutely convex neighborhoods of \( 0 \) form a fundamental system
  of neighborhoods \( 0 \) in any locally convex space,
  \eqref{eq:a50dc5aec14aee3d} entails that the polar topology
  \( c\left(\oplus_{i}E'_{i}, \, \prod_{i} E_{i}\right) \) is also finer than
  the locally convex direct sum topology \( \oplus_{i} (E_{i})'_{c} \). Hence
  \begin{equation}
    \label{eq:dd04293b16de5336}
    \left(\prod_{i} E_{i}\right)'_{c} = \bigoplus_{i} (E_{i})'_{c},
  \end{equation}
  as desired.

  We now determine the polar dual of \( \oplus_{i} (E_{i})'_{c} \), i.e.\ the
  polar topology \( c\left(\prod_{i}E_{i}, \, \oplus_{i}(E_{i})'_{c}\right)
  \). Note that precompact sets are preserved under continuous linear maps and
  all precompact sets are bounded, it follows from
  Proposition~\ref{prop:746e91e6638b3d9d} that a set in
  \( \oplus_{i \in I}(E_{i})'_{c} \) is precompact if and only if it is
  contained in a sum of the form \( \sum_{i \in I_{0}}B_{i} \), where
  \( I_{0} \subseteq I \) is finite, and
  \( B_{i} \in \mathfrak{C}\bigl((E_{i})'_{c}\bigr) \). Hence, a fundamental
  system of neighborhoods of \( 0 \) for
  \( c\left(\prod_{i}E_{i}, \, \oplus_{i}(E_{i})'_{c}\right) \) is given by
  polars of the form \( (\sum_{i \in I_{0}}B_{i})^{\circ} \) as above. Fix
  arbitrarily such a polar, define \( U_{i} = B_{i}^{\circ} \) for each
  \( i \in I_{0} \), and let \( U_{i} = E_{i} \) if \( i \notin I_{0} \), then a
  simple calculation yields
  \begin{equation}
    \label{eq:e702e1968ef2665e}
    \frac{1}{\abs*{I_{0}}} \prod_{i}U_{i} \subseteq \left(\sum_{i \in I_{0}}B_{i}\right)^{\circ}
    \subseteq \prod_{i}U_{i}.
  \end{equation}
  Note that as \( B_{i} \) runs through
  \( \mathfrak{C}\bigl((E_{i})'_{c}\bigr) \), since \( E_{i} \) is polar
  reflexive, \( B_{i}^{\circ} \) runs through a fundamental neighborhoods of
  \( 0 \) in \( E_{i} \). Now \eqref{eq:e702e1968ef2665e} implies that
  \( c\left(\prod_{i}E_{i}, \, \oplus_{i}(E_{i})'_{c}\right) \) is both coarser
  and finer than the product topology on \( \prod_{i} E_{i} \), hence is
  identical with it. This finishes our proof.
\end{proof}

\subsection{Direct products and polar reflexivity of locally convex Hopf algebras}
\label{sec:d0b3014d3c129a16}

Recall Definition~\ref{defi:a6ba1501ad62d100}. We now establish the following
(partial) permanence result.

\begin{prop}
  \label{prop:323fe7df2f5a5c21}
  Let \( (E_{i})_{i \in I} \) be a family of \( (\varepsilon, \pi) \)-polar
  reflexive spaces. Suppose that
  \begin{enumerate}
  \item \label{item:52c295c5ed4be513} for all \( i, j \in I \), the canonical
    map
    \begin{equation}
      \label{eq:2cd82491cd4c19c7}
      \chi_{i,j} : (E_{i})'_{c} \times (E_{j})'_{c} \to (E_{i})'_{c}
      \otimes_{\iota} (E_{j})'_{c}
    \end{equation}
    is continuous;
  \item \label{item:cae80c9b74869833} for all \( i, j, k \in I \), the canonical
    map
    \begin{equation}
      \label{eq:31a28463fb9afa5a}
      \chi_{(i,j), k} : \bigl((E_{i})'_{c} \overline{\otimes}_{\pi} (E_{j})'_{c}\bigr)
      \times (E_{k})'_{c}
      \to \bigl((E_{i})'_{c} \overline{\otimes}_{\pi} (E_{j})'_{c}\bigr)
      \otimes_{\iota} (E_{k})'_{c}
    \end{equation}
   is continuous;
 \item \label{item:3480c1e4478c8a02} for all \( i, j, k, l \in I \), the canonical
    map
    \begin{equation}
      \label{eq:03ffea1efcfe805b}
      \chi_{(i,j, k), l} : \Bigl(\bigl((E_{i})'_{c} \overline{\otimes}_{\pi} (E_{j})'_{c}\bigr)
      \otimes_{\iota} (E_{k})'_{c}\Bigr)
      \times (E_{l})'_{c}
      \to \Bigl(\bigl((E_{i})'_{c} \overline{\otimes}_{\pi} (E_{j})'_{c}\bigr)
      \otimes_{\iota} (E_{k})'_{c}\Bigr) \overline{\otimes}_{\pi} (E_{l})'_{c}
    \end{equation}
   is continuous;
 \end{enumerate}
 Then the product \( E = \prod_{i} E_{i} \) is
 \( (\varepsilon, \iota) \)-polar reflexive, with
 \( E'_{c} = \bigoplus_{i} (E_{i})'_{c} \).
\end{prop}

\begin{proof}
  By Proposition~\ref{prop:89e27b078637bfc2}, the interchanges of
  \( \overline{\otimes}_{\pi} \) and \( \overline{\otimes}_{\iota} \) in the
  following calculation is valid.

  First of all, note that both \( \prod_{i}E_{i} \) and
  \( \oplus_{i} (E_{i})'_{c} \) are complete
  (Proposition~\ref{prop:903bfccf4a781ceb} and
  Proposition~\ref{prop:d842a95c764bfd01}).  Note also that
  \( \overline{\otimes}_{\varepsilon} \) commutes with direct products
  (Proposition~\ref{prop:8056525c3c4c9aa5}), while
  \( \overline{\otimes}_{\varepsilon} \) commutes with locally convex direct
  sums (Corollary~\ref{coro:06a795bc15551930}).

  By Proposition~\ref{prop:9f3f8bda0402a4ff}, we know that the duality pairing
  \begin{displaymath}
    \pairing*{\prod_{i}E_{i}}{\bigoplus_{i} (E_{i})'_{c}}
  \end{displaymath}
  is polar reflexive

  For two-fold tensor products, by hypothesis, applying
  Proposition~\ref{prop:9f3f8bda0402a4ff}, we have
  \begin{equation}
    \label{eq:48221104b43e64d9}
    \begin{split}
      (E \overline{\otimes}_{\varepsilon} E)'_{c}
      &= \left(\prod_{i,j} E_{i} \overline{\otimes}_{\varepsilon} E_{j}\right)'_{c}
        = \bigoplus_{i,j} \left(E_{i} \overline{\otimes}_{\varepsilon} E_{j}\right)'_{c}
        = \bigoplus_{i,j} (E_{i})'_{c} \overline{\otimes}_{\pi} (E_{j})'_{c} \\
      &= \bigoplus_{i,j} (E_{i})'_{c} \overline{\otimes}_{\iota} (E_{j})'_{c}
        = \left(\bigoplus_{i}(E_{i})'_{c}\right)
        \overline{\otimes}_{\iota} \left(\bigoplus_{i}(E_{i})'_{c}\right).
    \end{split}
  \end{equation}
  Conversely, applying Proposition~\ref{prop:9f3f8bda0402a4ff} again, we have
  \begin{equation}
    \label{eq:600f57d98553e86c}
    \begin{split}
      \left(\left(\bigoplus_{i}(E_{i})'_{c}\right)
      \overline{\otimes}_{\iota} \left(\bigoplus_{i}(E_{i})'_{c}\right)\right)'_{c}
      &= \left( \bigoplus_{i,j} (E_{i})'_{c} \overline{\otimes}_{\iota} (E_{j})'_{c}\right)'_{c}
        = \prod_{i,j} \left(  (E_{i})'_{c} \overline{\otimes}_{\iota} (E_{j})'_{c} \right)'_{c} \\
      &=  \prod_{i,j} \left(  (E_{i})'_{c} \overline{\otimes}_{\pi} (E_{j})'_{c} \right)'_{c}
        =  \prod_{i,j}  E_{i} \overline{\otimes}_{\varepsilon} E_{j}
        = \left(\prod_{i}E_{i}\right) \overline{\otimes} \left(\prod_{i}E_{i}\right).
    \end{split}
  \end{equation}
  Hence, the pairing
  \begin{displaymath}
    \pairing*{\left(\prod_{i}E_{i}\right) \overline{\otimes}_{\varepsilon} \left(\prod_{i}E_{i}\right)}
    {\left(\bigoplus_{i} (E_{i})'_{c}\right) \overline{\otimes}_{\iota} \left(\bigoplus_{i} (E_{i})'_{c}\right)}
  \end{displaymath}
  is also polar reflexive.

  We also have the corresponding polar reflexive pairings for three-fold and
  four-fold tensor products, following a similar calculation as in
  \eqref{eq:48221104b43e64d9} and \eqref{eq:600f57d98553e86c}, which finishes
  the proof.
\end{proof}

In the case of \( (F) \)-spaces, conditions in
Proposition~\ref{prop:323fe7df2f5a5c21} can be greatly weakened, as a
consequence of Hollstein \cite{MR0470696}*{p241, Satz~2.1}, see also
\cite{MR0551623}*{p302, (3)}. We won't dive into detailing the notion of
\( \sigma \)-locally topological, but merely point out that all spaces in
\( (\mathcal{F}'_{c}) \) are \( \sigma \)-locally topological
\cite{MR0551623}*{p302, paragraph above (4)}. Then it is trivial from
\cite{MR0551623}*{p302, (3)} that we have the following, which we track for
future reference.

\begin{prop}
  \label{prop:f0b1250b423a54b3}
  Let \( E, F \in (\mathcal{F}'_{c}) \) and \( G \) a locally convex space. If
  \( f : E \times F \to G \) is a bilinear map that is hypocontinuous, then
  \( f \) is already continuous.
\end{prop}

\begin{rema}
  \label{rema:a932c6ef4b516901}
  Now it is clear that by Proposition~\ref{prop:f0b1250b423a54b3} and
  Proposition~\ref{prop:921772454be0e6bd} that conditions in
  Proposition~\ref{prop:323fe7df2f5a5c21} are satisfied if the relevant polar
  duals are all barrelled. But it is easy to see that for an \( (F) \)-space
  \( E \), \( E'_{c} \) is barrelled if and only if \( E \) is Montel. This
  motivates us to consider the situation in
  \S~\ref{sec:c5a60681556c38a0}.
\end{rema}

\subsection{Products of \texorpdfstring{\( (FM) \)}{(FM)}-spaces with the
  approximation property}
\label{sec:c5a60681556c38a0}

\begin{prop}
  \label{prop:a56a0c072645abf2}
  Let \( (E_{i})_{i \in I} \) be a family of spaces in
  \( (\mathcal{FM}) \cap (\mathcal{AP}) \), and \( E = \prod_{i} E_{i} \). Then
  \( E \) is \( (\varepsilon, \iota) \)-reflexive, and
  \( E'_{b} = \bigoplus_{i} (E_{i})'_{b} \). Dually, if \( (F_{j})_{j \in J} \)
  is a family of spaces in \( (\mathcal{DFM}) \cap (\mathcal{AP}) \), and
  \( F = \bigoplus_{j} F_{j} \) is the locally convex direct sum. Then \( F \)
  is \( (\iota, \varepsilon) \)-reflexive, and
  \( F'_{b} = \prod_{j}(F_{j})'_{b} \).
\end{prop}
\begin{proof}
  We prove only the first statement, as the second follows by duality.

  Note again that the polar dual coincides with the strong for Montel spaces
  (Proposition~\ref{prop:99b4fc84a9b1d8a1}). Each
  \( (E_{i})'_{b} = (E_{i})'_{c} \) is therefore in
  \( (\mathcal{DFM}) \cap (\mathcal{AP}) \)
  (Proposition~\ref{prop:8ffb913b5f5ddea6}). By
  Proposition~\ref{prop:89675519c2bb97f8} (and
  Proposition~\ref{prop:89e27b078637bfc2}), we know that the family
  \( (E_{i}) \) satisfies condition~\ref{item:52c295c5ed4be513} in
  Proposition~\ref{prop:323fe7df2f5a5c21}. Since
  \( (\mathcal{DFM}) \cap (\mathcal{AP}) \) is stable under taking completed
  projective tensor products (Proposition~\ref{prop:8ffb913b5f5ddea6}),
  condition~\ref{item:cae80c9b74869833} in
  Proposition~\ref{prop:323fe7df2f5a5c21} is also satisfied. We may now conclude
  with Proposition~\ref{prop:323fe7df2f5a5c21} by noting that the class
  \( (\mathcal{M}) \) is stable under taking arbitrary product and locally
  convex direct sums (Proposition~\ref{prop:a017fe989c7853e0}).
\end{proof}

\begin{coro}
  \label{coro:5a8cb8e4d424f5b2}
  Let \( (E_{i})_{i \in I} \) be a family of spaces in \( (\mathcal{FN}) \), and
  \( E = \prod_{i} E_{i} \). Then \( E \) is
  \( (\varepsilon, \iota) \)-reflexive, and
  \( E'_{b} = \bigoplus_{i} (E_{i})'_{b} \). Dually, if \( (F_{j})_{j \in J} \)
  is a family of spaces in \( (\mathcal{DFN}) \), and
  \( F = \bigoplus_{j} F_{j} \), then \( F \) is
  \( (\iota, \varepsilon) \)-reflexive and \( F'_{b} = \prod_{j} (F_{j})'_{b} \).
\end{coro}
\begin{proof}
  This is clear from Proposition~\ref{prop:a56a0c072645abf2} by noting that
  \( (\mathcal{FN}) \subseteq (\mathcal{FM}) \cap (\mathcal{AP}) \)
  (Proposition~\ref{prop:0bc5d06c0f42ff63} and
  Proposition~\ref{prop:f4cbc15f463be1bc}), and also
  \( (\mathcal{DFN}) \subseteq (\mathcal{DFM}) \cap (\mathcal{AP}) \) by duality
  (Proposition~\ref{prop:a8139d3432a18d9e}).
\end{proof}

\begin{theo}
  \label{theo:a2230cfdc7b9adaf}
  Let \( H \) be an \( \varepsilon \)-Hopf algebra (resp.\ \( \iota \)-Hopf
  algebra). If as a locally convex space, \( H \) is isomorphic to a product of
  spaces in \( (\mathcal{FM}) \cap (\mathcal{AP}) \) (resp.\ a locally convex
  direct sum of spaces in \( (\mathcal{DFM} \cap (\mathcal{AP})) \)), then as an
  \( \varepsilon \)-Hopf algebra (resp.\ a \( \pi \)-Hopf algebra), it is
  \( (\varepsilon, \iota) \)-reflexive (resp.\
  \( (\iota, \varepsilon) \)-reflexive). In particular, this applies when the
  locally convex space \( H \) is isomorphic to a product of spaces in
  \( (\mathcal{FN}) \) (resp.\ a locally convex direct sum of spaces in
  \( (\mathcal{DFN}) \)).
\end{theo}
\begin{proof}
  This follows from Proposition~\ref{prop:844fe110a4f00d8b}, combined with
  Proposition~\ref{prop:a56a0c072645abf2} and
  Corollary~\ref{coro:5a8cb8e4d424f5b2}.
\end{proof}

\section{Some basic examples of locally convex Hopf algebras}
\label{sec:2a8ab6bf60a9b3ca}

We now describe some basic examples of locally convex Hopf algebras. To start,
as noted in Remark~\ref{rema:18c5fa54f0e4fe49}, all well-behaved topological
Hopf algebras as studied in \cite{MR1266072} are \( (\sigma, \tau) \)-reflexive
Hopf algebras, where \( \sigma, \tau \in \set*{\varepsilon, \pi} \). We are now
interested in examples that could go beyond their settings. More sophisticated
examples that require further work shall be described later.

\subsection{Discrete groups and
  \texorpdfstring{\( \ell^{1}(\Gamma) \)}{ell1(Gamma)} as a
  \texorpdfstring{\( \pi \)}{pi}-Hopf algebra}
\label{sec:de569e5c9166cc4a}

We start by providing a class of examples that are
\( (\pi, \varepsilon) \)-polar reflexive, but not reflexive.

Let \( \Gamma \) be a discrete group. Consider \( H = \ell^{1}(\Gamma) \) as a
Banach space. Consider the normed space
\( \ell^{1}(\Gamma) \otimes_{\pi} \ell^{1}(\Gamma) \), where the norm
\( \norm*{\cdot}_{\pi} \) is determined as in
Remark~\ref{rema:9a1ecf4df00f5761}. We may embed
\( \ell^{1}(\Gamma) \otimes \ell^{1}(\Gamma) \) as a subspace of \( \ell^{1}(\Gamma \times \Gamma)  \), as follows.
Let \( t \) be a tensor in
\( \ell^{1}(\Gamma) \otimes \ell^{1}(\Gamma) \), which can be written as
\begin{equation}
  \label{eq:78357ea6201909d7}
  t = \sum_{k=1}^{n}\left(\sum_{s \in \Gamma} a_{k, s} \delta_{s} \right)
  \otimes \left(\sum_{t \in \Gamma} b_{k, t} \delta_{t}\right)
  = \sum_{s, t \in \Gamma} \left(\sum_{k=1}^{n}a_{k,s} b_{k,t}\right) \delta_{s} \otimes \delta_{t}.
\end{equation}
Clearly, identifying \( \delta_{s} \otimes \delta_{t} \) with
\( \delta_{(s,t)} \), we may identify \( t \) with
\( \sum_{s, t \in \Gamma}\left(\sum_{k=1}^{n}a_{k,s}b_{k,t}\right) \delta_{(s,
  t)} \). This does not depend on how one writes \( t \) in
\eqref{eq:78357ea6201909d7}, since the coefficient of \( \delta_{(s, t)} \) is given by
\begin{displaymath}
  \sum_{k=1}^{n}a_{k,s}b_{k,t}
  = (e_{s} \otimes e_{t})(t),
\end{displaymath}
where \( \epsilon_{s} \in \ell^{1}(\Gamma)' \) and
\( \epsilon_{t} \in \ell^{1}(\Gamma)' \) denote the evaluation at
\( s \in \Gamma \) and at \( t \in \Gamma \) respectively. From this, it follows
easily that \( \ell^{1}(\Gamma) \otimes_{\pi} \ell^{1}(\Gamma) \) identifies
isometrically with a dense subspace of \( \ell^{1}(\Gamma \times \Gamma) \), and
we may thus write
\( \ell^{1}(\Gamma) \overline{\otimes}_{\pi} \ell^{1}(\Gamma) = \ell^{1}(\Gamma
\times \Gamma) \). More generally, a similar argument shows that
\( \ell^{1}(S) \overline{\otimes}_{\pi} \ell^{1}(T) = \ell^{1}(S \times T) \)
for any sets \( S \) and \( T \).

We now define the unit \( \eta : \mathbb{K} \to \ell^{1}(\Gamma) \) by
\( 1 \mapsto \delta_{e} \), where \( e \in \Gamma \) is the neutral element of
\( \Gamma \); and the counit \( \varepsilon: \ell^{1}(\Gamma) \to \mathbb{K} \)
as \( \epsilon_{e} \). Furthermore, define the multiplication
\begin{displaymath}
  \begin{split}
    m: \ell^{1}(\Gamma) \overline{\otimes}_{\pi} \ell^{1}(\Gamma) = \ell^{1}(\Gamma \times \Gamma)
    & \to \ell^{1}(\Gamma) \\
    \sum_{(s, t) \in \Gamma \times \Gamma} \lambda_{s,t} \delta_{(s,t)}
    \mapsto \sum_{\gamma \in \Gamma} \left(\sum_{s \in \Gamma} \lambda_{s, s^{-1}\gamma}\right) \delta_{\gamma},
  \end{split}
\end{displaymath}
the comultiplication
\begin{displaymath}
  \begin{split}
    \Delta: \ell^{1}(\Gamma) & \to \ell^{1}(\Gamma \times \Gamma) = \ell^{1}(\Gamma) \overline{\otimes}_{\pi} \ell^{1}(\Gamma) \\
    \sum_{s \in \Gamma}\lambda_{s} \delta_{s} & \mapsto \sum_{s \in \Gamma} \lambda_{s} \delta_{(s, s)},
  \end{split}
\end{displaymath}
and the antipode
\begin{displaymath}
  \begin{split}
    S: \ell^{1}(\Gamma) & \to \ell^{1}(\Gamma) \\
    \sum_{s \in \Gamma}\lambda_{s} \delta_{s} & \mapsto \sum_{s \in \Gamma} \lambda_{s^{-1}} \delta_{s}.
  \end{split}
\end{displaymath}
It is readily checked that each of \( \eta, \varepsilon, m, \Delta \) and
\( S \) is a contractive linear map, in particular, continuous, and restricts to
the corresponding structure maps on the Hopf algebra \( \mathbb{K}[G] \), which
is a dense subspace of \( \ell^{1}(\Gamma) \). Hence by continuity and density,
\( H = \bigl(\ell^{1}(\Gamma), \eta, \varepsilon, m, \Delta, S\bigr) \) is a
\( \pi \)-Hopf algebra.

Since \( \ell^{1}(\Gamma) \) is a Banach space (in particular, an
\( (F) \)-space), and has the approximation property
(Proposition~\ref{prop:0f7d595efc04bfe4}), Theorem~\ref{theo:5f3cb1a05114cdca}
applies, and we see that \( H \) is \( (\pi, \varepsilon) \)-polar
reflexive. However, when \( \Gamma \) is an infinite group, it is well-known
that \( \ell^{1}(\Gamma) \) is not reflexive as a locally convex space, hence
\( H \) is not \( (\varepsilon, \tau) \)-reflexive for any choice of compatible
symmetric monoidal functor \( \overline{\otimes}_{\tau} \).

Thus the polar dual \( \ell^{1}(\Gamma)'_{c} \) has an \( \varepsilon \)-Hopf
algebra structure by duality. Note that as a vector space, we do have
\( \ell^{1}(\Gamma)'_{c} = \ell^{\infty}(\Gamma) \), however, the topology on
\( \ell^{1}(\Gamma)'_{c} \) is no longer normable once \( \Gamma \) is
infinite. It is interesting to note that the vector \( \ell^{\infty}(\Gamma) \),
when equipped with the \( c(\ell^{\infty}(\Gamma), \ell^{1}(\Gamma)) \)
topology, carries an \( \varepsilon \)-Hopf algebra structure, which we denote
of course by \( H'_{c} \) as in Definition~\ref{defi:783a70411adbf39d}.

One checks also easily that
\begin{displaymath}
  \set*{\delta_{s} \given s \in \Gamma}
  = \set*{x \in H \given \Delta(x) = x \otimes x, \, x \ne 0},
\end{displaymath}
and \( m(\delta_{s} \otimes \delta_{t}) = \delta_{st} \), thus one may recover
the discrete group \( \Gamma \) from \( H \), as expected.

We end this class of examples by noting that when \( \mathbb{K} = \C \), we may
even give an involution on \( H \), by letting
\begin{displaymath}
  \left(\sum_{s \in \Gamma} \lambda_{s}\delta_{s}\right)^{\ast}
  = \sum_{s \in \Gamma} \overline{\lambda_{s^{-1}}} \delta_{s},
\end{displaymath}
making \( H \) a \( \pi \)-Hopf-\( \ast \)-algebra.

\subsection{Locally convex Hopf algebras related to Lie groups}
\label{sec:374fdcc9344806b0}

In \S~\ref{sec:374fdcc9344806b0}, all smooth manifolds (hence Lie groups) are
assumed to be Hausdorff and paracompact, so that they always admits a smooth
partitions of unity \cite{MR0448362}*{p43, Theorem~2.1} (note that this reduces
to the \( \sigma \)-compact case by \cite{bourbaki_topologie_2006}*{I.70,
  Théorème~5}, or \cite{MR1700700}*{p38, Theorem~12.11}).

Consider a (paracompact) smooth manifold \( M \). We shall work with the
so-called weak topology (terminology in \cite{MR0448362}*{p35}) on the algebra
\( C^{\infty}(M) \) of smooth functions, which we will always use when referring
to the locally convex topology on \( C^{\infty}(M) \) unless stated
otherwise. We briefly describe here how this topology is defined. First, fix
\( r \in \N_{+} \). Let \( (\varphi, U) \) be a (smooth) chart on \( M \) and
\( K \subseteq U \) compact, in particular \( \varphi(U) \) is open in
\( \R^{d} \) with \( d \) being the dimension of \( M \), so that the right side
of \eqref{eq:af14908c0470e2f1} makes sense. For each \( f \in C^{r}(M) \),
define
\begin{equation}
  \label{eq:af14908c0470e2f1}
  p_{K, \varphi, r}(f) = \sup_{k = 0, 1, \ldots, r}\max_{x \in K}\norm*{D^{k}(f \circ \varphi^{-1}\vert_{\varphi(U)})(x)},
\end{equation}
where \( D^{k} \) is the total derivation of order \( k \) (we can of course
also use partial derivatives instead, then locally convex topologies thus
obtained can be easily seen to be the same). Clearly, \( p_{K, \varphi, r} \) is
a semi-norm on \( C^{r}(M) \). We denote by \( C^{r}_{W}(M) \) the space
\( C^{r}(M) \) equipped with the locally convex topology generated by semi-norms
of the form \( p_{K, \varphi, r} \). It can be checked that if \( M \) is second
countable, \( C^{r}_{W}(M) \) is a space of type \( (F) \) (cf.\
\cite{MR0448362}*{pp.33-34}). Now by the weak (smooth) topology on
\( C^{\infty}(M) \), we mean the coarsest locally convex topology on
\( C^{\infty}(M) \) making all restriction maps
\( C^{\infty}(M) \to C^{r}_{W}(M) \) continuous. Or equivalently, the weak
topology on \( C^{\infty}(M) \) is the one generated by all semi-norms of the
form \( p_{K, \varphi, r}\vert_{C^{\infty}(M)} \), with \( (\varphi, U) \) a
chart, \( K \subseteq U \) compact, and \( r \in \N_{+} \). From this
observation, it is clear that if \( M \) is second countable,
\( C^{\infty}(M) \) is metrizable (a suitable countable sub-family of
\( p_{K, \varphi, r}\vert_{C^{\infty}(M)} \) suffices to define the
topology). It is also complete by using the completeness of each
\( C^{r}_{W}(M) \) and (after a standard partitions of unity argument) the
standard fact in elementary analysis that a sequence \( (f_{n}) \) of functions
in \( C^{1}(V) \) (\( V \subseteq \R^{d} \) open and connected) converges
uniformly on compact sets if their first partial derivatives converges uniformly
on compacts, and \( (f_{n}(x_{0})) \) converges for some \( x_{0} \in V \).
Hence \( C^{\infty}(M) \) is a space of type \( (F) \) if \( M \) is second
countable.

The following lemmas, at least in the case where the underlying manifolds are
second countable, are well-known. We nevertheless includes a brief sketches of
the proofs for completeness and convenience of the reader.
\begin{lemm}
  \label{lemm:975e581967ecec25}
  Using the above notation, if \( M \) is second countable, then
  \( C^{\infty}(M) \) is of class \( (\mathcal{FN}) \). If \( M \) is merely
  paracompact, then let \( (M_{i})_{i \in I} \) be the connected component of
  \( M \), then each \( M_{i} \) is second countable, and
  \( C^{\infty}(M) = \prod_{i \in I} C^{\infty}(M_{i}) \in (\mathcal{N}) \).
\end{lemm}
\begin{proof}
  We first treat the second countable case. It is already noted in the above
  that \( C^{\infty}(M) \in (\mathcal{F}) \). To see that
  \( C^{\infty}(M) \in (\mathcal{N}) \), one needs the theory of nuclear maps,
  then using \cite{MR0225131}*{Theorem~50.1} and a standard partitions of unity
  argument, we reduce to the case where \( M \) is an open set of an Euclidean
  space. Now the nuclearity of \( C^{\infty}(M) \) follows from
  \cite{MR0225131}*{p530, Corollary}.

  In the general case, by \cite{MR1700700}*{p38, Theorem~12.11}, each component
  \( M_{i} \) is \( \sigma \)-compact, hence second countable since it is
  locally Euclidean. Note that being smooth is a local property, our definition
  clearly shows that \( C^{\infty}(M) = \prod_{i \in I}C^{\infty}(M) \) as
  locally convex spaces, and the proof of the lemma is complete by the above
  second countable case.
\end{proof}

\begin{lemm}
  \label{lemm:c00a7742f85d9ace}
  Let \( M \), \( N \) be paracompact smooth manifold, then the completed
  topological tensor product
  \( C^{\infty}(M) \overline{\otimes}_{\varepsilon} C^{\infty}(N) \) is canonically isomorphic to \( C^{\infty}(M \times N) \).
\end{lemm}
\begin{proof}
  Note first that we may as well use
  \( C^{\infty}(M) \overline{\otimes}_{\pi} C^{\infty}(N) \) by nuclearity
  (Lemma~\ref{lemm:975e581967ecec25}). By Lemma~\ref{lemm:975e581967ecec25} and
  the commutativity of products with completed injective tensor products
  (Proposition~\ref{prop:8056525c3c4c9aa5}), we reduce the problem to the case
  where \( M \), \( N \) are both second countable. Then a standard partitions
  of unity argument reduces further to the case where \( M \), \( N \) are open
  subsets of Euclidean spaces, in which case, the lemma follows from
  \cite{MR0225131}*{p530, Theorem~51.6}.
\end{proof}

By a Lie group, we mean a real Lie group that is not necessarily
second-countable. So in particular, all discrete groups (countable or not)
counts as Lie groups. It is well-known that the connected components of a Lie
group is second countable (\cite{MR0722297}*{p83, Remark~(c)}), hence Lie groups
are paracompact \cite{MR1700700}*{p38, Theorem~12.11}, and the above lemmas
apply in this situation.

\begin{theo}
  \label{theo:6b0cb8e408044129}
  Let \( G \) be a Lie group. We identify \( C^{\infty}(G \times G) \) with
  \( C^{\infty}(G) \overline{\otimes}_{\varepsilon} C^{\infty}(G) \) canonically
  via Lemma~\ref{lemm:c00a7742f85d9ace}. Then there exists a canonical
  \( \varepsilon \)-Hopf algebra structure on \( C^{\infty}(G) \), given as
  follows
  \begin{enumerate}
  \item \label{item:a8a1fd3e21051eb5} the unit \( \eta \) is given by the constant function \( 1 \in C^{\infty}(G) \);
  \item \label{item:3d8ee7f9d44cc783} the counit \( \varepsilon \) is given by evaluation at the neutral element of \( G \);
  \item \label{item:06260d46e5e837b3} the multiplication
    \( m : C^{\infty}(G \times G) \to C^{\infty}(G) \) is given by
    \( m(F)(x) = F(x, x) \), for any \( F \in C^{\infty}(G \times G) \) and
    \( x \in G \);
  \item \label{item:0f2f12321b88e2bf} the comultiplication \( \Delta \) is by
    pulling-back the multiplication on \( G \).
  \end{enumerate}
  Furthermore, this \( \varepsilon \)-Hopf algebra is \( (\varepsilon, \iota) \)-reflexive.
\end{theo}
\begin{proof}
  Clearly, all these structure maps are linear and continuous (one need to
  observe that the product of compact sets in \( G \) remain compact for the
  continuity of \( \Delta \)).  Note that when restricted to the dense subspace
  \( C^{\infty}(G) \odot C^{\infty}(G) \subseteq C^{\infty}(G \times G) \), the
  multiplication \( m \) is exactly the pointwise multiplication. It is a
  routine verification that \( C^{\infty}(G) \) equipped with these structure
  maps is indeed an \( \varepsilon \)-Hopf algebra. This \( \varepsilon \)-Hopf
  algebra is \( (\varepsilon, \iota) \)-reflexive follows from
  Lemma~\ref{lemm:975e581967ecec25} and Theorem~\ref{theo:a2230cfdc7b9adaf}.
\end{proof}

\begin{nota}
  \label{nota:f9138e02df3c4659}
  By abuse of notation, when it's obvious from context, we will simply say the
  \( \varepsilon \)-Hopf algebra \( C^{\infty}(G) \), with the structure maps
  being understood as in Theorem~\ref{theo:6b0cb8e408044129}.
\end{nota}

\begin{rema}
  \label{rema:1c398bf82e991caa}
  If \( G \) is a complex Lie group, we may use the algebra of holomorphic
  functions \( \mathcal{H}(G) \) on \( G \). The analogues of
  Lemma~\ref{lemm:975e581967ecec25} and Lemma~\ref{lemm:c00a7742f85d9ace} still
  hold, i.e.\ \( \mathcal{H}(G) \) is a product of \( (FN) \)-spaces, and we
  have a canonical identification
  \( \mathcal{H}(G) \overline{\otimes}_{\varepsilon} \mathcal{H}(G) =
  \mathcal{H}(G \times G) \). Indeed, in the second countable case, we may still
  use \cite{MR0225131}*{p530}) to establish these lemmas, and the general
  paracompact case follows again by applying \cite{MR1700700}*{p38,
    Theorem~12.11}. Now defined similarly as in
  Theorem~\ref{theo:6b0cb8e408044129}, we see that \( \mathcal{H}(G) \) becomes
  an \( (\varepsilon, \iota) \)-reflexive \( \varepsilon \)-Hopf algebra.
\end{rema}

\begin{rema}
  \label{rema:3faab4832d082f96}
  It is easy to check that in the complex case \( \mathbb{K} = \C \), the
  \( \varepsilon \)-Hopf algebra \( C^{\infty}(G) \) becomes an
  \( \varepsilon \)-Hopf-\( \ast \) algebra with the pointwise conjugation as
  the involution. However, the same can not be applied to \( \mathcal{H}(G) \)
  as conjugation is \emph{not} holomorphic.
\end{rema}

\subsection{Classical Hopf algebras as locally convex Hopf algebras and their duals}
\label{sec:b6aa227ba43c21a5}

We now include all classical Hopf algebras over \( \mathbb{K} \) into our
framework. First, we need some preparations on locally convex spaces.

\begin{prop}
  \label{prop:12e473f3f887b8ab}
  For any vector space \( V \), there exists a unique finest locally convex
  topology \( \tau_{V} \) on \( V \), and if \( V \) is the algebraic direct sum
  of finite dimensional subspaces \( V_{i} \), \( i \in I \), then
  \( \tau_{V} \) is the locally convex direct sum of the unique Hausdorff vector
  topology on each \( V_{i} \). In particular, the space \( (V, \tau_{V}) \) is
  locally convex and complete, and every linear map out of \( (V, \tau_{V}) \)
  into a locally convex space is continuous.
\end{prop}
\begin{proof}
  Let \( \mathfrak{T}_{V} \) be the collection of all locally convex topologies
  on \( V \). Take \( \tau_{V} \) to be the inductive topology with respect to
  the family \( \id_{V}: (V, \tau) \to V \), \( \tau \in \mathfrak{T}_{V} \).
  Clearly, \( \tau_{V} \) is the finest locally convex topology on \( V \).

  If \( V = \oplus_{i}V_{i} \) algebraically, where each \( V_{i} \) if finite
  dimensional, since each \( V_{i} \hookrightarrow (V, \tau_{V}) \) is
  continuous by the axiom of a topological linear space (or the uniqueness of
  the Hausdorff linear topology on \( V_{i} \)), by the universal property of
  the locally convex direct sum, we see that the identity map from the locally
  convex direct sum \( \oplus_{i \in I}V_{i} \) onto \( (V, \tau_{V}) \) is
  continuous, forcing it to be an isomorphism since \( \tau_{V} \) is already
  the finest locally convex topology on \( V \).

  By Zorn's lemma, \( V \) admits a Hamel basis
  \( (v_{\beta})_{\beta \in B} \), hence \( \tau_{V} \) maybe taken to be the
  locally direct sum of its one dimensional subspaces
  \( \mathbb{K} v_{\beta} \), and is therefore complete and Hausdorff
  (Proposition~\ref{prop:d842a95c764bfd01}).

  Finally, let \( f : V \to E \) be a linear map into a locally convex space
  \( E \). Since the projective topology on \( V \) with respect to the
  singleton \( \set*{f} \) is a locally convex topology on \( V \), it is
  coarser than \( \tau_{V} \), meaning \( f: (V, \tau_{V}) \to E \) is
  continuous.
\end{proof}

We already can fit all classical Hopf algebras into our framework.
\begin{theo}
  \label{theo:493ddf6e91b637cd}
  Let \( H \) be a classical Hopf algebra. We equip the vector space \( H \)
  with its finest locally convex topology, then \( H \) is a \( \iota \)-Hopf
  algebra, and is \( (\iota, \varepsilon) \)-reflexive.
\end{theo}
\begin{proof}
  Let \( (x_{\beta})_{\beta \in B} \) be a vector basis (Hamel basis) for
  \( H \). We know that \( H = \oplus_{\beta \in B}\mathbb{K} x_{\beta} \) as
  locally convex spaces, hence \( H \) is complete. By
  Corollary~\ref{coro:9dd4c5414e7497a4}, we have
  \begin{displaymath}
    H \otimes_{\iota} H = \bigoplus_{(\beta_{1}, \beta_{2}) \in B \times B} \mathbb{K}
    \left(x_{\beta_{1}} \otimes x_{\beta_{2}}\right),
  \end{displaymath}
  which coincides with the algebraic tensor product \( H \odot H \) as vector
  spaces, and is already complete as a locally convex space
  (Proposition~\ref{prop:d842a95c764bfd01}), so that
  \( H \otimes_{\iota} H = H \overline{\otimes}_{\iota} H \) is exactly the
  vector space \( H \odot H \) equipped with its finest locally convex
  topology. Now using Proposition~\ref{prop:12e473f3f887b8ab}, we also know that
  every structure map that is used to define the classical Hopf algebra \( H \)
  is continuous, hence \( H \) becomes an \( \iota \)-Hopf algebra in this
  way. Being the locally convex direct sum
  \( \oplus_{\beta \in B} \mathbb{K}x_{\beta} \), clearly the locally convex
  space \( H \) is a locally convex direct sum of spaces in
  \( (\mathcal{DFN}) \), hence Theorem~\ref{theo:a2230cfdc7b9adaf} applies, and
  we see that \( H \) is \( (\iota, \varepsilon) \)-reflexive.
\end{proof}

\begin{rema}
  \label{rema:7cc044cc3cec42ec}
  It seems that in the purely algebraic framework, in general, only a partial
  dual can be reasonably defined for an arbitrary Hopf algebra. Of particular
  interests is the so-called ``restricted dual'' (we use the terminology in
  \cite{MR2397671}*{p35, \S~1.4.3}), considered already in
  \cite{MR0252485}*{Chapter VI, pp109-136}. Recall that for a classical algebra
  \( A \) over \( \mathbb{K} \), the restricted dual \( A^{\circ} \) is defined
  as the space of linear functionals on \( A \) that vanishes on an ideal of
  \( A \) that is of finite codimension \cite{MR2397671}*{pp35,36, Lemma~4.1.8},
  then one may pulling back the multiplication \( m : A \otimes A \to A \) to
  obtain a comultiplication
  \( \Delta : A^{\circ} \to A^{\circ} \odot A^{\circ} \), to obtain a coalgebra
  structure on \( A^{\circ} \). When \( H \) is a classical Hopf algebra over
  \( \mathbb{K} \), on \( H^{\circ} \), we may define a dual Hopf algebra
  structure, which is called the ``dual Hopf algebra'' by Sweedler
  \cite{MR0252485}*{p122, \S~6.2}. We do have some results on when the
  restricted dual \( H^{\circ} \subseteq H^{\sharp} \) is big enough to allows
  to recover \( H \) by dualization again, see e.g.\ \cite{MR0252485}*{p119,
    Lemma~6.1.0, \& p121, Theorem~6.1.3}, but the range of applicability of
  these results are limited. We may see this from an extreme example.  Let
  \( \mathbb{K}[\Gamma] \) be the group algebra of a discrete group
  \( \Gamma \), which admits a canonical Hopf algebra structure
  \cite{MR2397671}*{p9, Example~1.2.8}. One can check that, as vector spaces,
  the restricted dual \( (\mathbb{K}[\Gamma])^{\circ} \) consists exactly of the
  linear span of matrix coefficients of finite dimensional representations of
  \( \Gamma \) (\cite{MR2397671}*{p38, Example~4.1.12}). However, if we consider
  the naively looking countable group, namely, the famous Higman group
  \( \Gamma = \langle a_{k} \mid a_{k}^{-1}a_{k+1}a_{k} = a_{k+1}^{2} \rangle \)
  where \( k \in \mathbb{Z}/4\mathbb{Z} \), it follows from the main result of
  Higman \cite{MR0038348} that \( \Gamma \) admits no nontrivial finite
  dimensional representations. Hence the restricted dual
  \( (\mathbb{K}[\Gamma])^{\circ} = 0 \) in this case, so that once we take the
  restricted dual, all information on the Hopf algebra \( \mathbb{K}[\Gamma] \)
  and on the group \( \Gamma \) is lost. Nevertheless, working in our framework
  of locally convex Hopf algebras, Theorem~\ref{theo:493ddf6e91b637cd} yields a
  good duality theory for all classical Hopf algebras in our framework of
  locally convex Hopf algebras.  We can even have an alternative description of
  a discrete \( \Gamma \) using any of polar dual pair of locally Hopf algebras
  described in \S~\ref{sec:de569e5c9166cc4a}, and still able to recover the
  group \( \Gamma \) itself.
\end{rema}

We've seen how classical Hopf algebras can be seen as \( \iota \)-Hopf algebras
in Theorem~\ref{theo:493ddf6e91b637cd}. Due to the special nature of the finest
locally convex topology on a vector spaces as well as its (strong) duals, it is
interesting to note that we can actually replace
\( \overline{\otimes}_{\iota} \) with \( \overline{\otimes}_{\varepsilon} \)
(hence also with \( \overline{\otimes}_{\pi} \)) on the dual. For this finer
description, we track the following elementary results and include their proofs
for convenience of the reader.

\begin{lemm}
  \label{lemm:7cc323f1115939db}
  If \( E \) is a locally convex space where the topology is the finest locally
  convex one, then
  \( E'_{s} = E'_{k} = E'_{c} = E'_{b} \in (\mathcal{M}) \cap (\mathcal{N}) \), and is complete.
\end{lemm}
\begin{proof}
  First of all, note that for a locally convex space \( E \) where the topology
  is the finest locally convex one, the topological dual \( E' \) coincides with
  the algebraic dual \( E^{\ast} \). Being a locally convex direct sum of one
  dimensional subspaces, \( E \in (\mathcal{M}) \)
  (Proposition~\ref{prop:a017fe989c7853e0}), hence Mackey (Montel spaces are
  barrelled, then we use Proposition~\ref{prop:a944f1bb11e881f1}), and by
  Proposition~\ref{prop:e5dd78c27c8725c9}, we see that \( E'_{\tau} = E'_{s} \)
  is a product of one-dimensional spaces, hence complete. Moreover,
  \( E' \in (\mathcal{M}) \) (Proposition~\ref{prop:a017fe989c7853e0}), hence
  \( E'_{b} = E'_{c} \) (Proposition~\ref{prop:99b4fc84a9b1d8a1}). Since \( E \)
  is complete, we know \( \tau(E', E) = c(E', E) \) on \( E' \), thus
  \( E'_{\tau} = E'_{s} = E'_{b} = E'_{c} \). Being a product of one dimensional
  space, of course \( E'_{s} \in (\mathcal{N}) \)
  (Proposition~\ref{prop:a8139d3432a18d9e}).
\end{proof}

\begin{lemm}
  \label{lemm:e496d8234fc19353}
  If \( E \), \( F \) are locally convex spaces, both having the finest locally
  convex topology, then
  \begin{equation}
    \label{eq:22e7a20a6881a766}
    E \otimes_{\varepsilon} F = E \overline{\otimes}_{\varepsilon} F
    = E \otimes_{\pi} F = E \overline{\otimes}_{\pi} F
    = E \otimes_{\iota} F = E \overline{\otimes}_{\iota} F.
  \end{equation}
\end{lemm}
\begin{proof}
    In the following, we equip each one dimensional space with a norm (all these
  norms differ by a nonzero positive constant), which has no effect on the
  topology since on finite dimensional spaces, Hausdorff linear topology is
  unique. For each such one dimensional space \( S \), we use
  \( \mathsf{D}(S) \) to denote the collection of all closed disks with a finite
  nonzero radius centered at \( 0 \). It is clear that \( \mathsf{D}(S) \) are
  in bijective correspondence with \( \mathsf{D}(S') \) by taking polars with
  respect to the pairing \( \pairing*{S}{S'} \).

  Let \( E = \bigoplus_{i} E_{i} \), \( F = \bigoplus_{j} F_{j} \), with each
  \( E_{i} \) and \( F_{j} \) being one-dimensional. Then by
  Corollary~\ref{coro:9dd4c5414e7497a4}, we have
  \begin{equation}
    \label{eq:4eb1faba74816e3d}
    E \otimes_{\iota} F = \bigoplus_{i,j} E_{i} \otimes_{\iota} F_{j},
  \end{equation}
  and with \( E_{i} \otimes_{\iota} F_{j} \) being one-dimensional (hence
  complete), \( E \otimes_{\iota} F \) is already complete, and \( E \otimes_{\iota} F = E \overline{\otimes}_{\iota} F \).

  By Proposition~\ref{prop:6ecdcba4c320179a}, to finish the proof, it suffices
  now to show that \( E \otimes_{\varepsilon} F = E \otimes_{\iota} F \). By
  Proposition~\ref{prop:e5dd78c27c8725c9}, we have
  \( E'_{s} = \prod_{i} E'_{i} \), \( F'_{s} = \prod_{j} F'_{j} \), where again
  all \( E'_{i} \) and \( F'_{j} \) are one dimensional.

  Note that a fundamental system of neighborhoods of \( 0 \) in \( E \) is
  given by sets of the form \( \Gamma\left(\cup_{i}V_{i}\right) \),
  \( V_{i} \in \mathsf{D}(E_{i}) \), it follows that their polars (with respect
  to \( \pairing*{E}{E'} \)) form a fundamental system of equicontinuous sets in
  \( E' \). A simple calculation yields
  \( \Gamma\left(\cup_{i}V_{i}\right)^{\circ} = \prod_{i}V_{i}^{\circ} \).  We
  also have a similar description for equicontinuous sets in \( F' \).

  Now consider the duality pairing
  \( \pairing*{E \odot F}{\mathfrak{B}(E, F)} \) as in
  \eqref{eq:2ccbd263c9730d60}. From the above discussion and the definition of
  the \( \varepsilon \)-tensor product, it follows that a fundamental system of
  neighborhoods of \( 0 \) in \( E \otimes_{\varepsilon} F \) is given by polars
  of sets of the form
  \( \left(\prod_{i}A_{i}\right) \otimes \left(\prod_{j}B_{j}\right) \subseteq
  \mathfrak{B}(E, F) \), with \( A_{i} \in \mathsf{D}(E_{i}) \) (resp.\
  \( B_{j} \in \mathsf{D}(F_{j}) \)) Algebraically, we have the canonical
  identification
  \begin{equation}
    \label{eq:d514c45445aed46a}
    E \odot F = \bigoplus_{i,j} E_{i} \otimes F_{j}.
  \end{equation}
  It is clear by a simple computation that with respect to
  \( \pairing*{E \odot F}{\mathfrak{B}(E, F)} \), the polar
  \begin{equation}
    \label{eq:5da5d11637f91aa4}
    \left[\left(\prod_{i}A_{i}\right) \otimes \left(\prod_{j}B_{j}\right)\right]^{\circ}
    = \Gamma\left(\bigcup_{i,j} (A_{i} \otimes B_{i})^{\circ} \right),
  \end{equation}
  where on the right side of \eqref{eq:5da5d11637f91aa4},
  \( (A_{i} \otimes B_{i})^{\circ} \) is taken with respect to the duality
  pairing of one-dimensional spaces
  \( \pairing*{E_{i} \otimes F_{j}}{E'_{i} \otimes F'_{j}} \). It is clear that
  \( (A_{i} \otimes B_{i})^{\circ} \) runs through
  \( \mathsf{D}(E_{i} \otimes F_{j}) \) as \( A_{i} \) (resp.\ \( B_{j} \)) runs
  through \( \mathsf{D}(E_{i}) \) (resp.\ \( \mathsf{D}(F_{j}) \)). Hence the
  right side of \eqref{eq:5da5d11637f91aa4} describes a fundamental system of
  neighborhoods of \( 0 \) in \( E \otimes_{\varepsilon} F \), which happens to
  be a fundamental system of neighborhoods of \( 0 \) for the locally convex
  direct sum on the right side of \eqref{eq:4eb1faba74816e3d}. Hence by the
  identification \eqref{eq:d514c45445aed46a}, we see that indeed,
  \( E \otimes_{\varepsilon} F = E \otimes_{\iota} F \), and the proof is
  complete.
\end{proof}

We now have the finer description of classical Hopf algebras as locally convex ones.

\begin{theo}
  \label{theo:7ff4d334bc64ca63}
  For all \( \tau \in \set*{\iota, \pi, \varepsilon} \) and
  \( \sigma \in \set*{\pi, \varepsilon} \), any classical Hopf algebras is
  canonically a \( (\tau, \sigma) \)-reflexive \( \tau \)-Hopf algebra when
  equipped with the finest locally convex topology and exactly the same
  structure maps for the corresponding locally convex Hopf algebra structures,
  with its strong dual being of class \( (\mathcal{N}) \).
\end{theo}
\begin{proof}
  This follows from Lemma~\ref{lemm:7cc323f1115939db},
  \ref{lemm:e496d8234fc19353} and Theorem~\ref{theo:493ddf6e91b637cd}.
\end{proof}

\begin{rema}
  \label{rema:b7761beadb98816d}
  It is clear how to extend our treatment here to Hopf-\( \ast \)-algebras in
  the complex case via Proposition~\ref{prop:844fe110a4f00d8b}.
\end{rema}

\begin{rema}
  \label{rema:f8fd54f637ba985d}
  If a classical Hopf algebra \( H \) is of uncountable dimension as a vector
  space, then by Proposition~\ref{prop:eb0c853264f8729c}, \( H \) is \emph{not}
  of class \( (\mathcal{N}) \). We see that our result on all classical Hopf
  algebras go beyond the framework of well-behaved topological algebras as
  developed in \cite{MR1266072}.
\end{rema}

\begin{rema}
  \label{rema:941bc663373d0bfd}
  It is interesting to note that in the above description, classical Hopf
  algebras are equipped with the finest locally convex topology, while on its
  dual, the topology is the weak topology. We know that equipped with the weak
  topology, a locally convex space is complete if and only if it is isomorphic
  to a product of one-dimensional spaces (\cite{MR633754}*{pp.\ II.54, II.55,
    Proposition~9}), which is exactly as in our case.
\end{rema}

\subsection{Locally convex Hopf algebras of compact and discrete quantum groups}
\label{sec:4d1fce794188b9cb}

We can also put compact and discrete quantum groups in the framework of locally
convex Hopf algebras. Here, compact quantum groups are in the sense of
Woronowicz~\cites{MR901157,MR1616348}, but we do \emph{not} assume separability
of the \( C^{\ast} \)\nobreakdash-algebra of ``continuous functions'' on the
quantum group, so we are using the essentially the same theory as treated as in
\cite{MR3204665}*{Ch.1}. Discrete quantum groups are in the sense of van Daele
\cite{MR1378538}, based on his framework of multiplier Hopf algebras
\cite{MR1220906}. It is well-known that they are dual to each other, see e.g.\
\cite{MR3204665}*{pp28,29}. We will first describe how compact and discrete
groups fit into our framework of locally convex Hopf algebras, then use our
framework to give an alternative approach to these objects, and compare our
approach with that of Woronowicz's and van Daele's original approaches.

We will freely use the Peter-Weyl theory for compact quantum groups as developed
in \cite{MR1616348} or \cite{MR3204665}*{Ch.1}. Let \( \mathbb{G} \) be a
compact quantum groups, and \( \pol(\mathbb{G}) \) the unique dense
Hopf-\( \ast \)-algebra inside \( C(\mathbb{G}) \) that inherits the
comultiplication \( \Delta \) on \( C(\mathbb{G}) \). From the Peter-Weyl theory
for compact quantum groups, we may identify
\( \mathcal{H}_{\mathbb{G}}:= \pol(\mathbb{G}) \) canonically with
\( \bigoplus_{x \in \irr(\mathbb{G})}\coeff(x) \), as vector spaces (in fact
this decomposition is a direct sum of simple co-subalgebras). We set
\( \mathcal{H}_{\widehat{\mathbb{G}}} \) to be
\( \prod_{x \in \irr(\mathbb{G})}\mathcal{B}(H_{x}) \), where \( H_{x} \) is the
carrier finite dimensional Hilbert space for some irreducible unitary
representation \( U^{x} \in x \) of \( \mathbb{G} \). Note that
\( \mathcal{H}_{\widehat{\mathbb{G}}} \) is exactly the multiplier algebra of
\( c_{c}(\widehat{\mathbb{G}}):= \bigoplus_{x \in
  \irr(\mathbb{G})}^{\alg}\mathcal{B}(H_{x}) \). Of course, we equip
\( \mathcal{H}_{\widehat{\mathbb{G}}} \) with the product topology. It is
well-known that \( \mathcal{H}_{\mathbb{G}} = \pol(\mathbb{G}) \) is a
Hopf-\( \ast \) algebra (\cite{MR1616348}), and
\( c_{c}(\widehat{\mathbb{G}}) \) is a multiplier Hopf-\( \ast \) algebra
(\cite{MR1378538} and \cite{MR1658585}), with
\( \mathscr{M}\bigl(c_{c}(\widehat{\mathbb{G}})\bigr) =
\mathcal{H}_{\widehat{\mathbb{G}}} \) (\( \mathcal{M}(\cdot) \) here denotes the
multiplier algebra), with the structure maps related by the formulas in
\cite{MR3204665}*{pp28,29}.

In the following, for each \( x \in \irr(\mathbb{G}) \), we shall use a duality
pairing \( \pairing*{\coeff(x)}{\mathcal{B}(H_{x})} \), furnished by the Fourier
transform as in the following lemma.

\begin{lemm}
  \label{lemm:88522152423e594f}
  Let \( \mathbb{G} \) be a compact quantum group, \( (U, \mathscr{H}) \) a
  finite dimensional unitary representation of \( \mathbb{G} \), then the
  Fourier transform
  \begin{equation}
    \label{eq:6654361ffad5be4f}
    \begin{split}
      \mathcal{F}_{U} : \coeff(U)' & \to \mathcal{B}(\mathscr{H}) \\
      \omega & \mapsto (\id \otimes \omega)(U)
    \end{split}
  \end{equation}
  is a unital homomorphism of algebra whose range is the commutant
  \( \Endo(U)' \) of \( \Endo(U) \) in \( \mathcal{B}(\mathscr{H}) \), and
  \( \mathcal{F}_{U}(\omega^{\ast}) = \mathcal{F}_{U}(\omega)^{\ast} \), where
  \( \omega^{\ast} = \overline{\omega} \circ S \in \coeff(U)' \), with \( S \)
  being the restriction to \( \coeff(U) \) of the antipode on
  \( \pol(\mathbb{G}) \).
\end{lemm}
\begin{proof}
  For \( \omega_{1}, \omega_{2} \in \coeff(U)' \), the product \( \omega_{1}\omega_{2} \) is given by
  convolution, hence we have
  \begin{align*}
    \mathcal{F}_{U}(\omega_{1}\omega_{2}) %
    &= (\id \otimes \omega_{1} \otimes \omega_{2})(\id \otimes \Delta)(U)
      = (\id \otimes \omega_{1} \otimes \omega_{2})(U_{12}U_{13}) \\
    &= (\id \otimes \omega_{1})(\id \otimes \id \otimes \omega_{2})(U_{12}U_{13})
      = (\id \otimes \omega_{1})\Bigl(U\bigl(\mathcal{F}_{U}(\omega_{2}) \otimes 1\bigr) \Bigr) \\
    &= \mathcal{F}_{U}(\omega_{1}) \mathcal{F}_{U}(\omega_{2}),
  \end{align*}
  and
  \begin{displaymath}
    \mathcal{F}_{U}(\varepsilon_{U}) = (\id \otimes \varepsilon)(U) = \id_{\mathscr{H}},
  \end{displaymath}
  where \( \varepsilon_{U} \in \coeff(U)' \) is the restriction of the counit
  \( \varepsilon \) on \( \pol(\mathbb{G}) \). Thus \( \mathcal{F}_{U} \) is
  indeed a unital homomorphism of algebras.

  By the defining property of the antipode, we know
  \( (\id \otimes S)(U) = U^{-1} = U^{\ast} \). It follows that
  \begin{displaymath}
    \mathcal{F}_{U}(\omega)^{\ast} = (\id \otimes \overline{\omega})(U^{\ast}) = (\id \otimes \overline{\omega})
    (\id \otimes S)(U) = (\id \otimes \omega^{\ast})(U) = \mathcal{F}_{U}(\omega^{\ast}).
  \end{displaymath}
  In particular, we've shown that \( \image(\mathcal{F}_{U}) \) is a unital
  self-adjoint sub-algebra of \( \mathcal{B}(\mathscr{H}) \), hence is a von
  Neumann algebra on \( \mathscr{H} \) as \( \mathscr{H} \) is finite
  dimensional.

  Now for any \( T \in \mathcal{B}(\mathscr{H}) \), we have
  \( (T \otimes 1)U, U(T \otimes 1) \in \mathcal{B}(\mathscr{H}) \otimes
  \coeff(U) \), while applying \( \id \otimes \omega \),
  \( \omega \in \coeff(U)' \) yields
  \begin{displaymath}
    T \mathcal{F}_{U}(\omega) = (\id \otimes \omega)\bigl((T \otimes 1)U\bigr), \qquad \text{ and } \qquad
    \mathcal{F}_{U}(\omega)T = (\id \otimes \omega)\bigl(U(T \otimes 1)\bigr).
  \end{displaymath}
  Since \( \coeff(U)' \) separate points in \( \coeff(U) \), it follows that
  \begin{displaymath}
    T \in \Endo(U) \iff T \in \image(\mathcal{F}_{U})',
  \end{displaymath}
  and the proof is now complete by the bicommutant theorem of von Neumann.
\end{proof}

Using the decomposition
\( \mathcal{H}_{\mathbb{G}} = \bigoplus_{x \in \irr(\mathbb{G})} \coeff(x) \),
and
\( \mathcal{H}_{\widehat{G}} = \prod_{x \in \irr(\mathbb{G})}\mathcal{B}(H_{x})
\) and the above pairings, we thus obtain canonically a duality pairing
\( \pairing*{\mathcal{H}_{\mathbb{G}}}{\mathcal{H}_{\widehat{\mathbb{G}}}} \),
through which we may identify \( \mathcal{H}_{\widehat{\mathbb{G}}} \) with the
algebraic linear dual \( \mathcal{H}_{\mathbb{G}}^{\ast} \) of
\( \mathcal{H}_{\mathbb{G}} \). As \( \mathcal{H}_{\mathbb{G}} \) is a classical
Hopf-\( \ast \)-algebra, Theorem~\ref{theo:7ff4d334bc64ca63} and
Remark~\ref{rema:b7761beadb98816d} apply, and we obtain the following result
that include compact (and by duality, discrete) quantum groups into our
framework of locally convex Hopf algebra.

\begin{prop}
  \label{prop:11ae51a13c2f7f1d}
  For \( \sigma \in \set*{\pi, \varepsilon} \) and
  \( \tau \in \set*{\iota, \pi, \varepsilon} \), using the above symbol, every
  compact quantum group \( \mathbb{G} = \bigl(C(\mathbb{G}), \Delta\bigr) \) in
  the sense of Woronowicz give rise to an \( \tau \)-Hopf algebra
  \( \mathcal{H}_{\mathbb{G}} \) that is \( (\tau, \sigma) \)-reflexive, with
  its strong dual \( \mathcal{H}_{\widehat{\mathbb{G}}} \) being the multiplier
  algebra of \( c_{c}(\widehat{\mathbb{G}}) \) as a \( \ast \)-algebra, and the
  structure maps of the \( \sigma \)-Hopf algebra
  \( \mathcal{H}_{\widehat{\mathbb{G}}} \) restricts to the structure maps of
  \( c_{c}(\widehat{G}) \), making the latter a Hopf-\( \ast \)-algebra, which
  is a discrete quantum group in the sense of van Daele.
\end{prop}

We now go in the opposite direction of \S~\ref{sec:4d1fce794188b9cb} and
identify those locally convex Hopf-\( \ast \) algebras that are discrete/compact
quantum groups. Note that van Daele's definition of discrete quantum groups
(\cite{MR1378538}*{Definition~2.3}) is already quite close to our point of view,
so we start with characterizing discrete quantum groups. We will freely use the
construction and notation in \S~\ref{sec:4d1fce794188b9cb}.

\begin{prop}
  \label{prop:93b9dd8725d4daca}
  Let \( H \) be an \( \varepsilon \)-Hopf-\( \ast \) algebra, the following are
  equivalent:
  \begin{enumerate}
  \item \label{item:e9c3df98ead47d76} \( H \) is isomorphic to
    \( \mathcal{H}_{\widehat{\mathbb{G}}} \) for some compact quantum group \( \mathbb{G} \);
  \item \label{item:3f0292f1d8d6dcca} As an \( \varepsilon \)-algebra, \( H \)
    can be identified with a direct product of full matrix algebras
    \( H = \prod_{i \in I} \mathcal{B}(\mathfrak{H}_{i}) \) (equipped with the
    product topology), with each \( \mathfrak{H}_{i} \) being a finite
    dimensional Hilbert space; moreover, the structure map
    \begin{displaymath}
      T_{1} = (\id_{H} \overline{\otimes}_{\varepsilon} m)(\Delta
      \overline{\otimes}_{\varepsilon} \id_{H}):
      H \overline{\otimes}_{\varepsilon} H \to H \overline{\otimes}_{\varepsilon} H
    \end{displaymath}
    and the structure map
    \begin{displaymath}
      T_{2} = (m \overline{\otimes}_{\varepsilon} \id_{H})(\id_{H} \overline{\otimes}_{\varepsilon} \Delta):
      H \overline{\otimes}_{\varepsilon} H \to H \overline{\otimes}_{\varepsilon} H
    \end{displaymath}
    both restricts to a bijection of \( H_{0} \odot H_{0} \) onto itself, where
    \( H_{0} \) is the dense subspace
    \( \bigoplus_{i \in I}\mathcal{B}(\mathfrak{H}_{i}) \) (algebraic direct
    sum) of \( H \).
  \end{enumerate}
\end{prop}
\begin{proof}
  Unwinding the definition, one checks immediately that
  condition~\ref{item:3f0292f1d8d6dcca} is a restatement of
  \( (H_{0}, \Delta\vert_{H_{0}}) \) is a multiplier Hopf algebra, where
  \( \Delta\vert_{H_{0}}: H_{0} \to H \overline{\otimes}_{\varepsilon} H =
  \mathscr{M}(H_{0} \odot H_{0}) \). Now the proposition follows from van
  Daele's definition of discrete quantum groups
  \cite{MR1378538}*{Definition~2.3}, and the well-known duality between compact
  and discrete quantum groups (see the discussion in
  \S~\ref{sec:4d1fce794188b9cb} and e.g.\ \cite{MR1658585}*{5.2--5.5}).
\end{proof}

\begin{rema}
  \label{rema:7eaa33bb54416fe8}
  Due to nuclearity, in Proposition~\ref{prop:93b9dd8725d4daca}, we may as well
  use \( \overline{\otimes}_{\pi} \) instead of
  \( \overline{\otimes}_{\varepsilon} \) as the underlying compatible symmetric
  monoidal functor.
\end{rema}

Thanks to the work of van Daele \cite{MR1378538}, we may say slightly more about
the product decomposition as in \ref{item:3f0292f1d8d6dcca} of
Proposition~\ref{prop:93b9dd8725d4daca}: there is a distinguished block being
one-dimensional.
\begin{coro}
  \label{coro:f87eb9fa466c4016}
  Let \( H = \prod_{i \in I} \mathcal{B}(\mathfrak{H}_{i}) \) be as in statement
  \ref{item:3f0292f1d8d6dcca} in Proposition~\ref{prop:93b9dd8725d4daca}, then
  there exists a unique nonzero projection \( h \in H \) such that
  \( ah = ha = \varepsilon(a)h \) for any \( a \in H \). In particular, \( h \)
  linearly spans a unique matrix block.
\end{coro}
\begin{proof}
  Existence of such an \( h \) is established in
  \cite{MR1378538}*{Proposition~3.1} and its uniqueness from
  \cite{MR1658585}*{p356, remark after Definition~5.2}, or the fact that \( h \)
  corresponds to the Haar state on the strong dual of \( H \), which is of the
  form \( \mathcal{H}_{\mathbb{G}} \) for some compact quantum group
  \( \mathbb{G} \).
\end{proof}

So roughly speaking, we may characterize discrete quantum groups as
\( \varepsilon \)-Hopf algebras whose (locally convex) algebra structure is a
product of full matrix algebras, plus some additional conditions. The following
result dualizes this.

\begin{prop}
  \label{prop:0e90e7a1b601d6ed}
  Let \( H \) be an \( \varepsilon \)-Hopf-\( \ast \)-algebra. The following are equivalent:
  \begin{enumerate}
  \item \label{item:1c8a197d286fbd11} \( H \) is isomorphic to
    \( \mathcal{H}_{\mathbb{G}} \) for some compact quantum group
    \( \mathbb{G} \);
  \item \label{item:0401882f1cc0b791}
    \( H = \C 1 \oplus \bigl(\oplus_{i} C_{i}\bigr) \) both as a locally convex
    direct sums of locally convex spaces, and as a direct sum decomposition of
    coalgebras, with each \( C_{i} \) being a full matrix coalgebra; moreover,
    let \( p \) denote the linear form of taking the coefficient of \( 1 \) with
    respect to this decomposition, we have \( p({x}^{\ast}{x}) > 0 \) for any
    nonzero \( x \in H \).
  \end{enumerate}
\end{prop}
\begin{proof}
  We first show \ref{item:1c8a197d286fbd11} implies \ref{item:0401882f1cc0b791}.
  By the Peter-Weyl theory for compact quantum groups (\cite{MR1616348} or
  \cite{MR3204665}*{\S\S~1.3--1.5}), we have the decomposition
  \( \pol(\mathbb{G}) = \oplus_{x \in \irr(\mathbb{G})} \coeff(x) \), which is
  already used in \S~\ref{sec:4d1fce794188b9cb}. Note that as sub-coalgebras of
  \( \pol(\mathbb{G}) \), each \( \coeff(x) \) is a full matrix coalgebra. Note
  also that \( \C 1 \) is exactly the linear span of the matrix coefficients of
  the trivial representation. Hence we have the coalgebra decomposition in
  \ref{item:0401882f1cc0b791}. The fact that this is also a locally direct sum
  follows directly by noting that \( H \) is equipped with the finest locally
  convex topology. Finally, \( p({x}^{\ast}{x}) > 0 \) for nonzero \( x \in H \)
  follows directly from this decomposition and the orthogonality relations (see,
  e.g.\ \cite{MR3204665}*{p15, Theorem~1.4.3}).

  We now show the converse, and assume \ref{item:0401882f1cc0b791} holds. In
  this case, it is clear that \( H \) is again equipped with the finest locally
  convex topology. It is also clear that the linear functional \( p \) is a
  faithful positive linear form on the \( \ast \)-algebra \( H \). It follows
  directly from an easy computation (or from \cite{MR0252485}*{p293,
    Theorem~14.0.3}) that \( p \) is a positive invariant functional on \( H \),
  hence \( H \) is an algebraic compact quantum group in the sense of van Daele
  (see, e.g.\ \cite{MR2397671}*{CH.3}),  and we have
  \ref{item:1c8a197d286fbd11}.
\end{proof}

% \subsection{A revisit of Woronowicz's version of Krein-Tannaka reconstruction}
% \label{sec:b3de8af5394e3c98}

\section{Locally convex algebras of continuous functions on compactly generated
  Hausdorff spaces}
\label{sec:7b59ee62a46ef144}

\subsection{Compactly generated Hausdorff spaces and the
  \texorpdfstring{\( k \)}{k}-ification}
\label{sec:16fcc58a576c8ec2}

After Steenrod \cite{MR0210075} introduced them into algebraic topology, the
role played by compactly generated spaces can hardly be overemphasized. We shall
briefly describe what we need, following roughly the treatment in
\cite{MR2456045}*{\S~7.9}.

As we restrict ourselves to Hausdorff spaces, we give the definition in this case.
\begin{defi}
  \label{defi:14659ea316e3290e}
  Let \( X \) be a Hausdorff topological space. We say \( A \subseteq X \) is
  \( k \)-\textbf{closed} (resp.\ \textbf{open}), if for each compact subspace
  \( K \) of \( X \), the intersection \( A \cap K \) is closed (resp.\
  open). We say \( X \) is \textbf{compactly generated}, or a
  \( k \)-\textbf{space}, if all \( k \)-closed set in \( X \) is closed, which
  in turn is equivalent to all \( k \)-open sets are open.
\end{defi}

We include the following elementary result for convenience of the reader.
\begin{prop}
  \label{prop:42528c02cab8c326}
  Let \( X \) be a Hausdorff \( k \)-space, \( Y \) a topological space, a map
  \( f : X \to Y \) is continuous if and only if for each compact
  \( K \subseteq X \), the restriction \( f\vert_{K}: K \to Y \) is continuous.
\end{prop}
\begin{proof}
  The necessity of the condition is clear. For its sufficiency, note that if
  each \( f\vert_{K} \), \( K \subseteq X \) compact is continuous, then
  \( f^{-1}(C) \) is \( k \)-closed for any closed \( C \subseteq Y \), since
  \( f^{-1}(C) \cap K = f|_{K}^{-1}(C) \).  Hence \( f^{-1}(C) \) is closed
  since \( X \) is a \( k \)-space, and \( f \) is continuous.
\end{proof}

There are an ample supply of \( k \)-spaces, such as all \( CW \)-complexes in
algebraic topology. For us, we are mainly interested in the following two cases.

\begin{prop}[\cite{MR2456045}*{p189, Theorem~7.9.8}]
  \label{prop:68037c0ee6a99a1b}
  A Hausdorff \( X \) is a \( k \)-space if it is locally compact or first
  countable (in particular, metrizable).
\end{prop}

We also need the following permanence property of \( k \)-spaces under
quotients.

\begin{prop}[\cite{MR2456045}*{p189, Theorem~7.9.9}]
  \label{prop:a46cca900deea492}
  Let \( X \in \mathsf{CG} \) and \( \sim \) an equivalence relation on \( X
  \). If the quotient space \( X/\sim \) is Hausdorff, then \( X/\sim \) is
  still a \( k \)-space.
\end{prop}

\begin{nota}
  \label{nota:ef1a701fb9471826}
  To facilitate our discussion, let \( X \) be a Hausdorff space, we shall use
  \( \mathfrak{K}(X) \) to denote the collection of all compact subspaces of
  \( X \). We use \( \mathsf{Haus} \) to denote the category of all Hausdorff
  spaces, and \( \mathsf{CG} \) the category of all compactly generated
  Hausdorff spaces, with morphisms in both categories being the continuous maps
  between their objects.
\end{nota}

Let \( X \in \mathsf{Haus} \), let \( k(X) \) be the topological space where the
underlying set is still \( X \), but the open sets are given by the \( k \)-open
sets in \( X \). As all open sets in \( X \) is \( k \)-open, the identity map
\( i_{X} : k(X) \to X \) is continuous. If \( f : X \to Y \) is a morphism in
\( \mathsf{Haus} \), we use \( k(f): k(X) \to k(Y) \) to denote the same set
theoretic map as \( f \) but viewed as a map between possibly different
topological spaces.

\begin{prop}
  \label{prop:8c993945688f1422}
  Let \( X \in \mathsf{Haus} \).
  \begin{enumerate}
  \item \label{item:88f29f549cb936a7} The space \( k(X) \) is a compactly
    generated Hausdorff space, i.e.\ a \( k \)-space.
  \item \label{item:053e212454466d46} If \( X \) is compactly generated already,
    then \( k(X) = X \).
  \item \label{item:4da4a7403fc4211e} The spaces \( X \) and \( k(X) \) have the
    same compact subsets, i.e.\ \( \mathfrak{K}(X) = \mathfrak{K}(k(X)) \).
  \item \label{item:039bc077076e9edf} If \( f : X \to Y \) is a continuous map
    from \( X \) to another Hausdorff space \( Y \), then
    \( k(f): k(X) \to k(Y) \) is continuous.
  \end{enumerate}
\end{prop}
\begin{proof}
  \ref{item:88f29f549cb936a7}, \ref{item:053e212454466d46} and
  \ref{item:4da4a7403fc4211e} can be found in \cite{MR0210075}*{p136,
    Theorem~3.2}. To see \ref{item:039bc077076e9edf}, let
  \( C \in \mathfrak{K}(X) \) and take any \( k \)-open set \( V \) in \( Y \).
  We then have \( f(C) \in \mathfrak{K}(Y) \), hence \( f(C) \cap V \) is open
  in \( f(C) \), thus its inverse image
  \( f^{-1}\bigl(f(C) \cap V\bigr) \cap C = f^{-1}(V) \cap C \) under the
  restriction map \( C \to f(C) \), \( x \in C \mapsto f(x) \in f(C) \), is open
  in \( C \). This means \( f\vert_{C} : C \to k(Y) \) is continuous for any
  \( C \in \mathfrak{K}(X) \). But \( \mathfrak{K}(X) = \mathfrak{K}(k(X)) \) by
  \ref{item:053e212454466d46}, and \( k(X) \) is a \( k \)-space by
  \ref{item:88f29f549cb936a7}, it follows from
  Proposition~\ref{prop:42528c02cab8c326} that \( k(f) \) is continuous.
\end{proof}

\begin{coro}
  \label{coro:e6010f47670353d7}
  Let \( k: \mathsf{Haus} \to \mathsf{CG} \) be the functor
  \( X \mapsto k(X) \), and \( f \mapsto k(f) \) as above, and
  \( i : \mathsf{CG} \to \mathsf{Haus} \) the inclusion functor. Then \( k \) is
  right adjoint to \( i \), i.e.\ we have a natural bijection
  \begin{equation}
    \label{eq:cac84c9403dc0176}
    \begin{split}
      \mathsf{CG}\bigl(X, k(Y)\bigr)
      &\simeq \mathsf{Haus}\bigl(i(X), Y\bigr) \\
      \{k(f): k(X) = X \to Y\} & \mapsfrom f \\
      f & \mapsto f
    \end{split}
  \end{equation}
  for each \( X \in \mathsf{CG} \) and \( Y \in \mathsf{Haus} \), where \( f \)
  on the right of the last line of \eqref{eq:cac84c9403dc0176} means the same
  set-theoretic map from \( X \) to \( k(Y) \) but seen as from \( X \) to
  \( Y \). In particular, \( k \) commutes with all limits in
  \( \mathsf{Haus} \).
\end{coro}
\begin{proof}
  This follows directly from Proposition~\ref{prop:8c993945688f1422} and the
  observation that the topology on \( k(Y) \) is finer than that on \( Y \).
\end{proof}

\begin{defi}
  \label{defi:d4a5e987dffa7166}
  The functor \( k : \mathsf{Haus} \to \mathsf{CG} \) is called the
  \( k \)-\textbf{ification}. Note that \( \mathsf{Haus} \) has arbitrary
  products in the categorical sense, so has \( \mathsf{CG} \) since \( k \) is a
  right adjoint. For a family of spaces \( (X_{i})_{i \in I} \) in
  \( \mathsf{Haus} \), we see that \( k(\prod_{i}X_{i}) \) (together with same
  canonical projections) is the product of the family \( \bigl(k(X_{i})\bigr) \)
  in \( \mathsf{CG} \), and we call it the \( k \)-product of the family
  \( (X_{i}) \). When \( X, Y, Z \in \mathsf{CG} \), we denote
  \( k(X \times Y) \) by \( X \times_{k} Y \) and
  \( k(X \times Y \times Z) = X \times_{k} Y \times_{k} Z \), etc.
\end{defi}

We end this subsection with the following result.

\begin{prop}[\cite{MR2456045}*{p190, Theorem~7.9.12}]
  \label{prop:27398d650b82b971}
  If \( X \) is a locally compact space (hence \( X \in \mathsf{CG} \)), and
  \( Y \in \mathsf{CG} \), then \( X \times_{k} Y = X \times Y \).
\end{prop}

\begin{rema}
  \label{rema:c409b867863490d3}
  The \( k \)-product \( \times_{k} \) is the correct notion widely used in
  algebraic topology, for example, for the product of CW complexes, see
  \cite{MR1867354}*{p523-525}, especially the counter-example due to Dowker
  \cite{MR0048020} if one uses the usual product instead. We will soon see
  (Corollary~\ref{coro:b25b52005792153d}) that it is also the ``correct''
  product when considering certain completed injective tensor products.
\end{rema}

\subsection{The space \texorpdfstring{\( C(X, E) \)}{C(X, E)}}
\label{sec:173a272a70eeef93}

Recall Definition~\ref{defi:acbf0c33176b6907}. We first establish a result on
completeness.
\begin{prop}
  \label{prop:776017bb8c18481f}
  Let \( X \) be a \( k \)-space, \( E \) a locally convex space, \( C(X, E) \)
  the space of all continuous functions from \( X \) to \( E \).  We equip
  \( C(X, E) \) with the uniform structure of compact convergence. Then
  \( C(X, E) \) is a locally convex space equipped with the topology of its
  uniform structure, and \( C(X, E) \) is complete if \( E \) is complete.
\end{prop}
\begin{proof}
  That \( C(X, E) \) is a locally convex space follows from
  Proposition~\ref{prop:eb240fbacebaf554} and the fact that continuous image
  into Hausdorff spaces of compact sets remain compact, and compact sets are
  bounded in a locally convex space.

  Now assume \( E \) is complete. Let \( \mathcal{F} \) be a Cauchy filter on
  \( C(X, E) \). Since each singleton is compact, we see that for each
  \( x \in X \), the evaluation \( e_{x}: C(X, E) \to E \) is uniformly
  continuous, hence
  \( \mathcal{F}(x):= (e_{x})_{\ast}\mathcal{F} = \set*{M(x) \given M \in
    \mathcal{F}} \), where \( M(x) = \set*{f(x) \given f \in M} \), is a Cauchy
  filter in \( E \), hence convergence to a unique point
  \( f_{\mathcal{F}}(x) \) since \( E \) is complete. This defines a map
  \( f_{\mathcal{F}}: X \to E \). Now take any \( K \in \mathfrak{K}(X) \). The
  restriction \( \res_{K}: C(X, E) \to C_{u}(K, E) \) is uniformly continuous,
  where \( C_{u}(K, E) \) is the uniform subspace of the uniform space of all
  maps (continuous or not) \( F_{u}(K, E) \) as defined in
  Definition~\ref{defi:a4863de2d22b6ea1}. We denote the direct image of
  \( \mathcal{F} \) under \( \res_{K} \) by \( \mathcal{F}\vert_{K} \). It is
  clear that \( \mathcal{F}\vert_{K} \) converges uniformly to
  \( f_{\mathcal{F}}\vert_{K} \) in \( F_{u}(K, E) \). However, being the
  uniform limit of continuous functions, we have
  \( f_{\mathcal{F}}\vert_{K} \in C(K, E) \), hence
  \( f_{\mathcal{F}}: X \to E \) is continuous by
  Proposition~\ref{prop:42528c02cab8c326}. It is clear now that
  \( \mathcal{F} \) convergences \( f_{\mathcal{F}} \in C(X, E) \), and the
  latter is complete.
\end{proof}

\begin{nota}
  \label{nota:bb506b1d10eebc00}
  When \( X \) is a Hausdorff \( k \)-space, \( E \) a locally convex space,
  unless stated otherwise, the symbol \( C(X, E) \) denotes the locally convex
  space as in Proposition~\ref{prop:776017bb8c18481f}. This applies in
  particular to the case where \( E \) is the scalar field \( \mathbb{K} \), and
  we denote \( C(X, \mathbb{K}) \) simply by \( C(X) \). Note that \( C(X) \) is
  a unital algebra under pointwise multiplication, and even a unital
  \( \ast \)-algebra under conjugation when \( \mathbb{K} = \C \).
\end{nota}

Note that Ascoli's theorem (Theorem~\ref{theo:957848d2ddc8435e}) gives us a
characterization of precompact sets (relatively compact sets when \( E \) is
complete) in \( C(X, E) \) in terms of equicontinuity.

\subsection{\texorpdfstring{\( C(X) \)}{C(X)} as a projective limit and as an
  \texorpdfstring{\( \varepsilon \)}{epsilon}-algebra}
\label{sec:146893c7763861c4}

Let \( X \) be a \( k \)-space. We order \( \mathfrak{K}(X) \) by inclusion,
then \( \mathfrak{K}(X) \) become a directed set. For
\( K \in \mathfrak{K}(X) \), the locally convex space \( C(K) \) as defined
above is normable with the uniform supremum norm \( \norm*{\cdot} \), hence is a
Banach space. For \( K_{1}, K_{2} \in \mathfrak{K}(X) \) with
\( K_{1} \leq K_{2} \) (i.e.\ \( K_{1} \subseteq K_{2} \)), let
\( p_{K_{2}, K_{1}}: C(K_{2}) \to C(K_{1}) \) denote the restriction map. In
this way, we obtain a projective system \( \bigl(C(K), p_{K_{1},K_{2}}\bigr) \)
of Banach spaces, and we may form its projective limit \( \varprojlim C(K) \).
Given an element \( (f_{K}) \in \varprojlim C(K) \). We may define a function
\( f \) on \( X \) by letting \( f(x) = f_{K}(x) \) whenever
\( x \in K \in \mathfrak{K}(X) \). This is well-defined, for when
\( K_{1}, K_{2} \in \mathfrak{K}(X) \) with \( x \in K_{1} \cap K_{2} \), we have \( K_{1} \cup K_{2} \in \mathfrak{K}(X) \) and
\begin{displaymath}
  f_{K_{1}}(x) =  f_{K_{1} \cup K_{2}}(x) = f_{K_{2}}(x).
\end{displaymath}
Clearly, \( f|_{K} = f_{K} \) for any \( K \in \mathfrak{K}(X) \), and
\( f \in C(X) \) by Proposition~\ref{prop:42528c02cab8c326} (\( X \) is a
\( k \)-space). Conversely, given \( f \in C(X) \) and put
\( f_{K} = f\vert_{K} \), we obtain an element
\( (f_{K}) \in \varprojlim C(K) \). We identify \( C(X) \) with
\( \varprojlim C(K) \) in this way as vector spaces. Since the locally convex
topology on \( C(X) \) is the topology of compact convergence, we see that the
identification \( C(X) = \varprojlim C(K) \) is also an identification as
locally convex spaces. Via this identification, the canonical projection
\( p_{K} : C(X) \to C(K) \) is the restriction for any
\( K \in \mathfrak{K}(X) \), which is also an algebra morphism
(\( \ast \)-algebra morphism in the complex case).

The following result settles the problem when the projective limit
\( \varprojlim C(K) \) as considered above is reduced.

\begin{prop}
  \label{prop:d98940c76a296a22}
  Using the above notation, the projective limit \( C(X) = \varprojlim C(K) \) is
  reduced if and only if \( C(X) \) separate points in \( X \).
\end{prop}
\begin{proof}
  Necessity of the condition is clear, for if \( x, y \in X \) are distinct
  points that can not be distinguished by functions in \( C(X) \), let
  \( K = \set*{x, y} \), then the image of \( p_{K} : C(X) \to C(K) \) is just
  the set of scalar functions, which is not dense in \( C(K) \). For
  sufficiency, note that if \( C(X) \) separate points in \( X \). In
  particular, \( C(X) \) separates points in \( K \) for any
  \( K \in \mathfrak{K}(X) \), meaning the image of \( p_{K}:C(X) \to C(K) \)
  also separate points in \( K \), and is dense by the Stone-Weierstrass
  approximation theorem.
\end{proof}

\begin{coro}
  \label{coro:ff9ad8f860fc9a59}
  If a \( k \)-space \( X \) has enough continuous functions, then \( C(X) \)
  has the approximation property.
\end{coro}
\begin{proof}
  This follows from Proposition~\ref{prop:d98940c76a296a22},
  Proposition~\ref{prop:91bfdedd58d67453} and
  Proposition~\ref{prop:90e4c2c8f8b18525}.
\end{proof}

To fix terminology, we make the following definition.

\begin{defi}
  \label{defi:ff4e9416eccf6607}
  We say a topological space \( X \) \textbf{has enough continuous functions},
  if for all \( x, y \in X \), \( x \ne y \), there exists a continuous function
  (real-valued) \( f \) such that \( f(x) \ne f(y) \).
\end{defi}

\begin{coro}  
  \label{coro:b25b52005792153d}
  If \( X, Y \) are \( k \)-spaces with enough continuous functions, then the
  \( k \)-space \( X \times_{k} Y \) also has enough continuous functions, and we
  have a canonical isomorphism
  \( C(X) \overline{\otimes}_{\varepsilon} C(Y) = C(X \times_{k}Y) \).
\end{coro}
\begin{proof}
  By Proposition~\ref{prop:8c993945688f1422}, \( X \times_{k} Y \) and
  \( X \times Y \) have the same continuous functions, and it is clear that the
  latter has enough of them (it suffices already to note that the topology on
  \( X \times_{k} Y \) is finer).  It is well-known that
  \( C(K) \overline{\otimes}_{\varepsilon} C(L) = C(K
  \overline{\otimes}_{\varepsilon} L)\) when \( K \) and \( L \) are compact (or
  more generally, locally compact, see e.g.\ \cite{MR0551623}*{p288, (3)}). Now
  by Proposition~\ref{prop:d98940c76a296a22}, and the commutation of reduced
  projective limits with \( \overline{\otimes}_{\varepsilon} \)
  (Proposition~\ref{prop:8056525c3c4c9aa5}), we have
  \begin{equation}
    \label{eq:faf652e25cb9ede9}
    C(X) \overline{\otimes}_{\varepsilon} C(Y)
    = \varprojlim_{K \in \mathfrak{K}(X)}C(K) \overline{\otimes}_{\varepsilon} \varprojlim_{L \in \mathfrak{K}(Y)} C(L) 
    = \varprojlim_{K \times L} C(K \times L),
  \end{equation}
  where the projective limit at the right most of \eqref{eq:cac84c9403dc0176} is
  taken with respect to the cofinal subsystem formed by \( K \times L \),
  \( K \in \mathfrak{K}(X) \), \( L \in \mathfrak{K}(Y) \) of the projective
  systems \( \mathfrak{K}(X \times Y) = \mathfrak{K}(X \times_{k} Y) \). Hence
  \begin{equation}
    \label{eq:747b0a227e10ecb0}
    \varprojlim_{K \times L} C(K \times L) = \varprojlim_{R \in \mathfrak{K}(X \times_{k} Y)} C(R)
    = C(X \times_{k} Y).
  \end{equation}
  Now the proof is complete by combining \eqref{eq:faf652e25cb9ede9} and
  \eqref{eq:747b0a227e10ecb0}.
\end{proof}

\begin{coro}
  \label{coro:64421d9753e96b7c}
  Let \( X \) be a \( k \)-space with enough continuous functions. Then the
  pointwise multiplication \( C(X) \odot C(X) \to C(X) \),
  \( f \otimes g \mapsto fg \) extends to a unique continuous linear map
  \( m : C(X) \overline{\otimes}_{\varepsilon} C(X) \to C(X) \), together with
  \( \eta : \mathbb{K} \to C(X) \), \( c \mapsto \{x \mapsto c\} \), the triplet
  \( \bigl(C(X), m, \eta\bigr) \) becomes an \( \varepsilon \)-algebra.
\end{coro}
\begin{proof}
  Note that \( C(X) \odot C(X) \) is dense in
  \( C(X) \overline{\otimes}_{\varepsilon} C(X) = C(X \times X) \), such an
  \( m \) is unique once it exists. To see its existence, consider the diagonal
  map \( \delta_{X}: X \to X \times X \), \( x \mapsto (x, x) \), which is
  clearly continuous. Take its \( k \)-ification, we obtain
  \( \delta_{X} : X \to X \times_{k} X \) is still continuous. Now define
  \( m = \delta_{X}^{\ast} : C(X \times_{k} X) \to C(X) \), then \( m \) is
  continuous since \( \delta_{x} \) maps compact sets in \( X \) to compact sets
  in \( X \times_{k} X \). Note that \( C(X) \odot C(X) \odot C(X) \) is dense
  in
  \( C(X \times_{k} X \times_{k} X) = C(X) \overline{\otimes}_{\varepsilon} C(X)
  \overline{\otimes}_{\varepsilon} C(X) \), and both
  \( m(m \overline{\otimes}_{\varepsilon} \id_{C(X)}) \) and
  \( m(\id_{C(X)} \overline{\otimes}_{\varepsilon} m) \) restricts to the
  pointwise multiplication on \( C(X) \odot C(X) \odot C(X) \), we see that
  \( m \) is associative by continuity. Similarly, one checks that \( \eta \) is
  indeed a unit for \( m \), and the proof is complete.
\end{proof}

Later, when we treat \( C(X) \) as an \( \varepsilon \)-algebra, unless stated
otherwise, we always mean the \( \varepsilon \)-algebra structure given in
Corollary~\ref{coro:64421d9753e96b7c}.

\subsection{Topological groups with compactly generated topology}
\label{sec:ce6223de3fc1fab6}

We are now ready to study a large class of topological groups, namely, the ones
whose topology is compactly generated. This class includes all locally compact
groups, all metrizable topological groups and all topological groups whose
underlying topology has a CW complex structure. We will show that the category
of these topological groups embeds anti-equivalently to the category of
\( \varepsilon \)-Hopf algebras, and establish that the corresponding
\( \varepsilon \)-Hopf algebra is polar reflexive as soon as the topological
group is \( \sigma \)-compact.

Recall that a topological space \( X \) is called \textbf{completely
  regular}. or a \textbf{Tychnoff space}, if it is Hausdorff, and if for each
\( x \in X \) and closed \( C \subseteq X \) with \( x \notin C \), there exists
a continuous function \( f: X \to \interval{0}{1} \), such that \( f(x) = 1 \)
and \( f\vert_{C} = 0 \). We shall need the following result characterizing
Hausdorff spaces that admits a compatible uniform structure.

\begin{theo}[\cite{bourbaki_topologie_2006-1}*{IX.7, Théorème~2}]
  \label{theo:6398e30fa8e59f95}
  Let \( X \) be a Hausdorff space. Then \( X \) admits a compatible uniform
  structure if and only if it is completely regular.
\end{theo}

It is clear that a completely regular space has enough continuous functions to
separate points. It is also well-known that the topology of any topological
group admits a compatible uniform structure
(\cite{bourbaki_topologie_2006}*{III.19--III.20}), even though the uniform
structure might not be unique (see, e.g.\ \cite{takesaki_theory_2002}*{p84,
  Remark~4.10}). So in particular, the topology of any topological group is
completely regular, hence has enough continuous functions.

Now let \( G \) be a topological group whose topology is compactly generated,
i.e.\ \( G \) is a \( k \)-space as a topological space. Then by
Corollary~\ref{coro:64421d9753e96b7c}, \( C(G) \) is already an
\( \varepsilon \)-algebra under the continuous extension of pointwise
multiplication. We also have
\( C(G) \overline{\otimes}_{\varepsilon} C(G) = C(G \times_{k} G) \) by
Corollary~\ref{coro:b25b52005792153d}. Since \( G \times G \) has enough
continuous functions, so has \( G \times_{k} G \), as the latter topology is
finer. Apply Corollary~\ref{coro:b25b52005792153d} again yields
\begin{equation}
  \label{eq:280422affdba8762}
  C(G) \overline{\otimes}_{\varepsilon} C(G) \overline{\otimes}_{\varepsilon}
  C(G) = C(G \times_{k} G \times_{k} G).
\end{equation}
Let \( \mu: G \times G \to G \) be the multiplication map of \( G \), which is
continuous, so \( \mu : G \times_{k} G \to G \) remains continuous. Now the
pull-back
\( \mu^{\ast}: C(G) \to C(G \times G) = C(G) \overline{\otimes}_{\varepsilon}
C(G) \) is a well-defined continuous linear map, which we define as the
comultiplication \( \Delta \). The associativity of \( \mu \) implies the
coassociativity of \( \Delta \) (this is why we need
\eqref{eq:280422affdba8762}). Naturally, we define the antipode
\( S : C(G) \to C(G) \) to be the pull-back of the inversion on \( G \), and the
counit \( \varepsilon : C(G) \to \mathbb{K} \) the evaluation at the neutral
element of \( G \).

\begin{theo}
  \label{theo:c72b0ef0a86663be}
  Let \( G \) be a topological group whose topology is compactly generated. Then
  equipped with the above structure maps, \( C(G) \) is an
  \( \varepsilon \)-Hopf algebra. If \( \mathbb{K} = \C \) and we define
  \( (\cdot)^{\ast} \) to be pointwise conjugation, then \( C(G) \) is an
  \( \varepsilon \)-Hopf-\( \ast \) algebra.
\end{theo}
\begin{proof}
  The nontrivial part is the structure maps are well-defined, which is done in
  the above. All the rest follows from a routine check.
\end{proof}

Since all locally compact spaces, metrizable spaces and all CW complexes are
compactly generated, Theorem~\ref{theo:c72b0ef0a86663be} applies to all
topological groups whose underlying topology is of these types, as noted in the
beginning of \S~\ref{sec:ce6223de3fc1fab6}.

\begin{defi}
  \label{defi:226c6771985742e3}
  We say a Hausdorff space \( X \). a sequence of compact sets
  \( (K_{n})_{n \geq 1} \) in \( X \) is called a \textbf{fundamental sequence
    of compact sets}, if there exists , such that for any compact
  \( K \subseteq X \) is contained in some \( K_{n} \).
\end{defi}

\begin{theo}
  \label{theo:fc44c913131be011}
  Use the setting of Theorem~\ref{theo:c72b0ef0a86663be}, if in addition,
  \( G \) has a fundamental sequence of compact sets, then \( C(G) \) is an
  \( (F) \)-space, and the \( \varepsilon \)-Hopf algebra (or the
  \( \varepsilon \)-Hopf-\( \ast \) algebra in the complex case) \( C(G) \) is
  \( (\varepsilon, \pi) \)-polar reflexive.
\end{theo}
\begin{proof}
  Let \( (K_{n})_{n \geq 1} \) be a fundamental
  sequence of compact sets. It is clear by definition that the countable family
  of seminorms \( (p_{K_{n}})_{n \geq 1} \) already generates the topology on
  \( C(X) \), where
  \( p_{K_{n}}(f):= \sup\set[\big]{\abs*{f(x)}\given x \in K_{n}} \). As
  \( C(X) \) is also complete (Proposition~\ref{prop:776017bb8c18481f}), it is
  an \( (F) \)-space. Note also that \( C(G) \in (\mathcal{AP}) \).  Now the
  proof is complete by using Theorem~\ref{theo:5f3cb1a05114cdca}.
\end{proof}

\section{From locally convex Hopf algebras back to groups}
\label{sec:372fb98a797875d9}

\subsection{Characters and group like elements}
\label{sec:7cf217d179c7beb7}

Let \( \overline{\otimes}_{\tau} \) be a compatible symmetric monoidal functor
and \( H \) a \( \tau \)-Hopf algebra. Following the analogues of the classical
theory, a \textbf{character} \( \omega \) of \( H \) is a multiplicative, unital
linear \emph{continuous} functional on \( H \), i.e.\ a character of \( H \) is
simply a morphism of \( \tau \)-algebras \( \omega : H \to \mathbb{K} \). We
shall consider the set of all characters on \( H \), which we denote by
\( \chi(H) \). It follows by definition that the counit \( \varepsilon \) is
already a character, hence \( \chi(H) \ne \emptyset \).

Given \( \omega, \varphi \in \chi(H) \), consider the convolution
\( \omega * \varphi \) defined as
\( (\omega \overline{\otimes}_{\tau} \varphi) \Delta \). Using the
coassociativity of \( \Delta \), counit property of \( \varepsilon \), and the
property of the antipode \( S \), it is easily checked that \( \chi(H) \)
becomes a group under the convolution product, where \( \varepsilon \) is the
counit, and the inverse of \( \omega \) is given by \( \omega S \). The group
\( \chi(H) \) is called the \textbf{character group} of \( H \).

In the case where \( H \) is a \( \tau \)-Hopf-\( \ast \) algebra, we say a
character \( \omega \in \chi(H) \) is involutive, if \( \omega \) is an
involutive functional. The counit is again an involutive character
(Proposition~\ref{prop:b358de0ebf54512e}). It follows from the same argument
that the set \( \chi^{\inv}(H) \) of all involutive characters is a subgroup of
the character group \( \chi(H) \), and we call it the \textbf{involutive
  character group}.

\begin{rema}
  \label{rema:9e10ec2883118976}
  By contrast, note that for compact quantum groups
  \( (C(\mathbb{G}), \Delta) \), it could happen there is \emph{no} continuous
  involutive multiplicative functionals on \( C(\mathbb{G}) \). As an example,
  let \( \Gamma \) be a non-amenable discrete group, and
  \( \mathbb{G} = \bigl(C^{\ast}_{r}(\Gamma), \Delta\bigr) \) where the
  comultiplication \( \Delta \) is determined by
  \( \Delta(\delta_{\gamma}) = \delta_{\gamma} \otimes \delta_{\gamma} \), then
  \( C(\mathbb{G}) = C^{\ast}_{r}(\Gamma) \) does not admit an involutive
  character (see, e.g.\ \cite{MR2391387}*{p50, Theorem~2.6.8}).
\end{rema}

It is natural to wonder whether can one put a reasonable topology on
\( \chi(H) \) to make it a topological group, and if \( H \) comes from a
topological group \( G \), whether under this topology, \( \chi(H) \simeq G \)
as topological groups. We will soon establish various results that answers the
latter question in the affirmative. However, due to the intricacy of general
locally convex spaces, the first question seems quite nontrivial. One positive
result is tracked in Theorem~\ref{theo:5a4f6899a7115559}, for which we need a
fine result on the characterization of relatively compact sets in the completed
injective tensor product of two \( (F) \)-spaces.

\begin{lemm}[\cite{MR0551623}*{p275, (9)}]
  \label{lemm:3bdf81b675e9ccb8}
  Let \( E_{1} \), \( E_{2} \) be \( (F) \)-spaces. The relative compact sets in
  \( E_{1} \overline{\otimes}_{\varepsilon} E_{2} \) are precisely subsets of
  the form \( \bigl(C_{1}^{\circ} \otimes C_{2}^{\circ}\bigr)^{\circ} \), where
  \( C_{i} \) is absolutely convex and compact in \( E_{i} \), \( i = 1,2 \).
  Here, the inner polar is taken with respect to the canonical pairing
  \( \pairing*{E_{i}}{E'_{i}} \), and the outer polar with respect to the
  canonical pairing
  \( \pairing*{E_{1} \overline{\otimes}_{\varepsilon} E_{2}}{E'_{1} \odot
    E'_{2}} \) obtained by extending each \( t \in E'_{1} \odot E'_{2} \) first
  as a continuous linear form on \( E_{1} \otimes_{\varepsilon} E_{2} \), then
  as a continuous linear form on
  \( E_{1} \overline{\otimes}_{\varepsilon} E_{2} \) by continuous extension.
\end{lemm}

\begin{theo}
  \label{theo:5a4f6899a7115559}
  If \( H \) is an \( \varepsilon \)-Hopf algebra of type \( (\mathcal{F}) \),
  then \( \chi_{c}(H) \), i.e.\ the character group equipped with the topology
  of precompact convergence, is a topological group. Moreover, \( \chi_{c}(H) \)
  is complete as a subspace of the uniform space
  \( H'_{c} = \mathcal{L}_{c}(H, \mathbb{K}) \). If \( H \) is an
  \( \varepsilon \)-Hopf-\( \ast \) algebra of type \( (\mathcal{F}) \), then
  \( \chi_{c}^{\inv}(H) \) is a closed subgroup of \( \chi_{c}(H) \).
\end{theo}
\begin{proof}
  We first prove that \( \chi_{c}(H) \) is a topological group. As the antipode
  preserves precompact sets by continuity, the inversion on \( \chi_{c}(H) \) is
  clearly continuous. We now show the continuity of the multiplication map
  \begin{equation}
    \label{eq:c3e63bfd9f2203c4}
    \begin{split}
      \mu: \chi_{c}(H) \times \chi_{c}(H) & \to \chi_{c}(H) \\
      (\varphi, \psi) & \mapsto \varphi \ast \psi = (\varphi \overline{\otimes}_{\varepsilon} \psi)\Delta
    \end{split}
  \end{equation}
  is continuous. Recall the notation
  \begin{equation}
    \label{eq:1ccaed4d55240b00}
    M(C, B_{\mathbb{K}}(r)) = \set[\big]{f \in E' \given f(C) \subseteq B_{\mathbb{K}}(\lambda)},
    \;\text{ where }\;
    B_{\mathbb{K}}(r) = \set[\big]{\lambda \in \mathbb{K} \given \abs*{\lambda} \leq r}, \; r > 0.
  \end{equation}
  By definition of the topology of precompact convergence, as \( C \) runs
  through \( \mathfrak{C}(H) \) and \( r \) through
  \( \interval[open]{0}{+\infty} \), sets described by
  \eqref{eq:1ccaed4d55240b00} form a fundamental system of neighborhoods of
  \( 0 \) for \( E'_{c} \). By completeness of \( H \), precompact sets and
  relatively compact sets in \( H \) are the same (the same holds for the
  complete space \( H \overline{\otimes}_{\varepsilon} H \)). Thus from the
  homogeneity
  \begin{displaymath}
    M\bigl(C, B_{\mathbb{K}}(r)\bigr)
    = M\bigl(r^{-1}C, B_{\mathbb{K}}(1)\bigr)
    = r M\bigl(C, B_{\mathbb{K}}(1)\bigr),
  \end{displaymath}
  the continuity of the map \( \mu \) at
  \( (\varphi_{0}, \psi_{0}) \in \chi_{c}(H) \times \chi_{c}(H) \) amounts to
  the following: for any relative compact \( C \) in \( H \), there exists
  relative compact sets \( C_{1} \) and \( C_{2} \) in \( H \), such that for \( \varphi, \psi \in \chi_{c}(H) \),
  \begin{equation}
    \label{eq:4b60695b72b66504}
    (\varphi - \varphi_{0}, \psi - \psi_{0}) \in C_{1}^{\circ} \times C_{2}^{\circ}
    \implies \varphi \ast \psi -  \varphi_{0} \ast \psi_{0} \in C^{\circ}.
  \end{equation}
  Working inside the convolution algebra
  \( \mathcal{L}(H) = \mathcal{L}(H, H) \) (see
  Definition~\ref{defi:52d0aeb9c0ac1362}), we have
  \begin{equation}
    \label{eq:242d5e6b50554c54}
    \varphi \ast \psi -  \varphi_{0} \ast \psi_{0}
    = (\varphi - \varphi_{0}) \ast \psi_{0} + \varphi_{0} \ast (\psi - \psi_{0})
    + (\varphi - \varphi_{0}) \ast (\psi - \psi_{0}).
  \end{equation}
  Now take any relative compact \( C \) in \( H \), we are going to find the
  desired \( C_{1} \) and \( C_{2} \) as above, which will prove the continuity
  of \( \mu \) and finishes the proof that \( \chi_{c}(H) \) is a topological
  group. Consider the continuous linear map
  \( \rho(\psi_{0}):= (\id \overline{\otimes}_{\varepsilon} \psi_{0})\Delta \in
  \mathcal{L}(H) \) and
  \( \lambda(\varphi_{0}):= (\varphi_{0} \overline{\otimes} \id)\Delta \in
  \mathcal{L}(H) \). Let \( S_{1} = \rho(\psi_{0})(C) \) and
  \( S_{2} = \lambda(\varphi_{0})(C) \). Then both \( S_{1} \) and \( S_{2} \)
  are relatively compact in \( H \). Moreover,
  \begin{equation}
    \label{eq:0b3575e21462dd7e}
    \varphi - \varphi_{0} \in S_{1}^{\circ}
    \implies (\varphi - \varphi_{0}) \ast \psi_{0} = (\varphi - \varphi_{0}) \rho(\psi_{0}) \in C^{\circ},
  \end{equation}
  and similarly,
  \begin{equation}
    \label{eq:e01e97bfdb6eb711}
    \psi - \psi_{0} \in S_{2}^{\circ} \implies \varphi_{0} \ast (\psi - \psi_{0}) \in C^{\circ}.
  \end{equation}
  By Lemma~\ref{lemm:3bdf81b675e9ccb8}, we may find two absolutely convex
  compact sets \( T_{1} \) and \( T_{2} \) in \( H \), such that the relatively
  compact set \( \Delta(C) \subseteq H \overline{\otimes}_{\varepsilon} H \) is
  contained in \( \bigl(T_{1}^{\circ} \otimes T_{2}^{\circ}\bigr)^{\circ} \). By definition of the corresponding polars, we have
  \begin{equation}
    \label{eq:e4d017627befdab6}
    \varphi - \varphi_{0} \in T_{1}^{\circ} \, \text{ and } \, \psi - \psi_{0} \in T_{2}^{\circ}
    \implies (\varphi - \varphi_{0}) \ast (\psi - \psi_{0})
    = \bigl((\varphi - \varphi_{0}) \overline{\otimes}_{\varepsilon} (\psi - \psi_{0})\bigr) \Delta \in C^{\circ}.
  \end{equation}
  Now let \( C_{i} = 3(S_{i} \cup T_{i}) \), \( i = 1,2 \). Then clearly each
  \( C_{i} \) is relatively compact. Now use the homogeneity
  \eqref{eq:1ccaed4d55240b00} again, we have, by \eqref{eq:0b3575e21462dd7e}
  \begin{equation}
    \label{eq:559bdd67e9e7049e}
    \varphi - \varphi_{0} \in C_{1}^{\circ} \subseteq \frac{1}{3}S_{1}^{\circ} \implies (\varphi - \varphi_{0}) \ast \psi_{0} \in \frac{1}{3} C^{\circ},
  \end{equation}
  and similarly by \eqref{eq:e01e97bfdb6eb711},
  \begin{equation}
    \label{eq:e6f9ada5ceab380a}
    \psi - \psi_{0} \in C_{2}^{\circ} \implies \varphi_{0} \ast (\psi - \psi_{0}) \in \frac{1}{3} C^{\circ}.
  \end{equation}
  We also have, by \eqref{eq:e4d017627befdab6}, that
  \begin{equation}
    \label{eq:82c3faa5a9e751cb}
    \begin{split}
      \varphi - \varphi_{0} \in C_{1}^{\circ} \, \text{ and } \, \psi - \psi_{0} \in C_{2}^{\circ}
      & \implies \varphi - \varphi_{0} \in  \frac{1}{3} T_{1}^{\circ}
        \, \text{ and } \, \psi - \psi_{0} \in \frac{1}{3} T_{2}^{\circ} \\
      & \implies (\varphi - \varphi_{0}) \ast (\psi - \psi_{0}) \in \frac{1}{9} C^{\circ} \subseteq \frac{1}{3}C^{\circ}.
    \end{split}
  \end{equation}
  Now \eqref{eq:242d5e6b50554c54} follows from \eqref{eq:559bdd67e9e7049e},
  \eqref{eq:e6f9ada5ceab380a} and \eqref{eq:82c3faa5a9e751cb}, and \( \mu \) as
  defined in \eqref{eq:c3e63bfd9f2203c4} is indeed continuous, and
  \( \chi_{c}(H) \) is indeed a topological group.
  
  By Corollary~\ref{coro:f7e5d81ca583056d}, \( H'_{c} \) is complete. Clearly,
  \( \chi(H) \) is the subspace
  \( \cap_{x, y \in H} (\delta_{xy} - \delta_{x} \delta_{y})^{-1}(0) \) of
  \( E'_{c} \), where \( \delta_{x}: H'_{c} \to \mathbb{K} \),
  \( \omega \mapsto \omega(x) \) is continuous since \( \set*{x} \) is compact,
  and \( \delta_{x}\delta_{y}: H'_{c} \to \mathbb{K} \),
  \( \omega \mapsto \omega(x)\omega(y) \) is also continuous. Hence
  \( \chi(H) \) is a closed subspace of the complete uniform space \( H'_{c} \),
  whence \( \chi_{c}(H) \) is complete. Similarly, when in the complex case and
  \( H \) is an \( \varepsilon \)-Hopf-\( \ast \) algebra,
  \( \chi_{c}^{\inv}(H) = \chi_{c}(H) \cap \left(\cap_{x \in H}(\delta_{x} -
    \overline{\delta_{x^{-1}}})^{-1}(0)\right) \), and we see that
  \( \chi_{c}^{\inv}(H) \) is closed in \( H'_{c} \), a priori closed in
  \( \chi_{c}(H) \).
\end{proof}

Dually, for a \( \tau \)-Hopf algebra \( H \) (\( \overline{\otimes}_{\tau} \)
being a compatible monoidal symmetric functor), we may define a
\textbf{cocharcter} of \( H \) to be a morphism of \( \tau \)-coalgebras from
the trivial \( \tau \)-coalgebra \( \mathbb{K} \) to the underlying
\( \tau \)-coalgebra \( H \). Since a linear form \( \varphi \) from
\( \mathbb{K} \) is uniquely determined by \( \varphi(1) \in H \), it is clear
that \( \varphi \) is a cocharacter if and only if \( \varphi(1) \in H \) is a
group-like element in \( H \) as formalized in
Definition~\ref{defi:d68b33029bb879ee}. This being said, we track the following
formal property for future comparison when we talk about involutive cocharacters.

\begin{prop}
  \label{prop:81b50087b5e11327}
  Let \( H \) be a \( \tau \)-Hopf algebra. If If \( H \) is
  \( (\tau, \sigma) \)-reflexive (resp.\ \( (\tau,\sigma) \)-polar reflexive)
  for some compatible symmetric monoidal functor
  \( \overline{\otimes}_{\sigma} \), and \( H' \) its strong (resp.\ polar)
  dual, then a linear map \( \varphi: \mathbb{K} \to H \) is a cocharacter if
  and only if the transpose \( \varphi^{\transp}: H' \to \mathbb{K} \) is a
  character of \( H' \).
\end{prop}
\begin{proof}
  This follows directly from the construction in
  Proposition~\ref{prop:844fe110a4f00d8b}.
\end{proof}

\begin{prop}
  \label{prop:3078b43cc03b72ce}
  Let \( \overline{\otimes}_{\tau} \) be a compatible symmetric monoidal functor
  and \( H \) a \( \tau \)-Hopf algebra. Then
  \begin{enumerate}
  \item \label{item:16ad454106329fbd}
    \( \set*{x \in H \given \Delta(x) = x \overline{\otimes}_{\tau}x, \, x \ne 0}
    = \set*{x \in H \given \Delta(x) = x \overline{\otimes}_{\tau}x, \,
      \varepsilon(x) = 1} \);
  \item \label{item:02b413ded615e81a} denote the set in
    \ref{item:16ad454106329fbd} by \( \mathsf{Grp}(H) \), then
    \( \mathsf{Grp}(H) \) is a subgroup of the multiplicative group
    \( H^{\times} \) of invertible elements in \( H \);
  \item \label{item:cd0767eb6f3e78a8} \( x \in \mathsf{Grp}(H) \) if and only if
    the linear map \( \lambda \in \mathbb{K} \to \lambda x \in H \) is a
    cocharacter;
  \item \label{item:2767338a39a8bb60} If \( \tau = \pi \), then
    \( \mathsf{Grp}(H) \) is a topological group when equipped with the subspace
    topology.
  \end{enumerate}
\end{prop}
\begin{proof}
  The proof parallels the classical case, but we include a proof anyway for
  convenience of the reader.
  
  \ref{item:16ad454106329fbd}. The right side is clearly contained in the
  left. For the reverse inclusion, note that if
  \( \Delta(x) = x \overline{\otimes}_{\tau} x \) and \( x \ne 0 \), we have
  \begin{displaymath}
    0 \ne x = (\id \overline{\otimes} \varepsilon) \Delta(x) = (\id \overline{\otimes}_{\varepsilon} \varepsilon) \Delta(x)
    = \varepsilon(x)x,
  \end{displaymath}
  forcing \( \varepsilon(x) = 1 \).

  \ref{item:02b413ded615e81a}. Since \( \Delta \) is multiplicative, it is clear
  that \( \mathsf{Grp}(H) \) is multiplicative and \( 1 \in \mathsf{Grp}(H) \)
  is the neutral element for multiplication in \( \mathsf{Grp}(H) \). If
  \( x \in \mathsf{Grp}(H) \), then \( Sx \in \mathsf{Grp}(H) \) by
  Proposition~\ref{prop:bc869f35c39bbe44}. Moreover,
  \begin{displaymath}
    x (Sx) = m(\id \overline{\otimes}_{\tau}S)(x \overline{\otimes}_{\tau}x)
    = m (\id \overline{\otimes}_{\tau} S) \Delta(x) = \varepsilon(x)1 = 1,
  \end{displaymath}
  and similarly \( (Sx)x = 1 \), hence \( Sx \) is the inverse of \( x \).

  \ref{item:cd0767eb6f3e78a8} is clear by unwinding the definitions.
  
  \ref{item:2767338a39a8bb60}. Clearly \( x \mapsto Sx \) is
  continuous. Continuity of the group multiplication follows from the defining
  property of \( \otimes_{\pi} \) (Proposition~\ref{prop:dda5a3f4762f22f6}).
\end{proof}

\begin{defi}
  \label{defi:d68b33029bb879ee}
  We call the group \( \mathsf{Grp}(H) \) the group of \textbf{group-like
    elements} in \( H \), and elements of \( \mathsf{Grp}(H) \) the group-like
  elements. When \( H \) is a \( \pi \)-Hopf algebra, unless stated otherwise,
  we always equip \( \mathsf{Grp}(H) \) with the subspace topology of \( H \).
\end{defi}

\begin{rema}
  \label{rema:adde74a3e091f652}
  From a formal point of view, after considering the group of characters, it
  seems more natural to consider the group of cocharacters (see
  Proposition~\ref{prop:3078b43cc03b72ce} \ref{item:cd0767eb6f3e78a8}). But the
  notion of group-like elements are already well-established in literature on
  Hopf algebras, so we follow the tradition. However, the notion of involutive
  cocharacters as introduced below can not be identified with hermitian
  group-like elements, and seems more interesting.
\end{rema}

\begin{defi}
  \label{defi:0f95896167ed6125}
  Let \( H \) be a \( \tau \)-Hopf-\( \ast \) algebra, an \textbf{involutive
    cocharacter} \( \varphi \) of \( H \) is a morphism
  \( \varphi \in \mathsf{Coalg}_{\tau}(\mathbb{K}, H) \), such that
  \begin{equation}
    \label{eq:cee36511bfb2b406}
    S \varphi(1) = \varphi(1)^{\ast}.
  \end{equation}
  Motivated by \eqref{eq:cee36511bfb2b406}, we call a group-like element
  \( x \in \mathsf{Grp}(H) \) \textbf{involutive}, if \( Sx = x^{\ast} \). The
  collection of all involutive group-like elements of \( H \) is denoted by
  \( \mathsf{Grp}^{\inv}(H) \).
\end{defi}

We now justify our definition, especially the seemingly weird condition
\eqref{eq:cee36511bfb2b406}.
\begin{prop}
  \label{prop:f8252fbb35384d7c}
  Let \( H \) be a \( \tau \)-Hopf-\( \ast \) algebra.
  \begin{enumerate}
  \item \label{item:85e923373991bebe} Involutive group-like elements form a
    subgroup of \( \mathsf{Grp}(H) \).
  \item \label{item:e8ef63cf27ec0c90} If \( H \) is
    \( (\tau, \sigma) \)-reflexive (resp.\ \( (\tau,\sigma) \)-polar reflexive)
    for some compatible symmetric monoidal functor
    \( \overline{\otimes}_{\sigma} \), and \( H' \) its strong (resp.\ polar)
    dual, then a linear map \( \varphi: \mathbb{K} \to H \) is a cocharacter if
    and only if the transpose \( \varphi^{\transp}: H' \to \mathbb{K} \) is an
    involutive character of \( H' \).
  \end{enumerate}
\end{prop}
\begin{proof}
  \ref{item:85e923373991bebe}. By Proposition~\ref{prop:bc869f35c39bbe44}, we
  have \( S1 = 1 = 1^{\ast} \), so \( 1 \in \mathsf{Grp}^{\inv}(H) \); and
  \( x, y \in \mathsf{Grp}^{\inv}(H) \) already implies
  \( xy \in \mathsf{Grp}(H) \), and also
  \begin{displaymath}
    S(xy) = S(y)S(x) = y^{\ast}x^{\ast}= (xy)^{\ast}.
  \end{displaymath}

  \ref{item:e8ef63cf27ec0c90} again follows directly from the construction in
  Proposition~\ref{prop:844fe110a4f00d8b}.
\end{proof}

The group \( \mathsf{Grp}^{\inv}(H) \) will play an important role when we treat
classical Pontryagin duality for locally compact abelian groups in our
framework.

\subsection{A generalization of the Gelfand duality}
\label{sec:e2f8637d41ae033e}

We now present a refinement of the Gelfand duality of recovering a compact space
\( X \) as the space of characters of the algebra \( C(X) \) equipped with the
weak topology. Later we will use this result to recover many topological groups
as the character groups equipped with the topology of precompact convergence of
the associated locally convex Hopf algebras. For the sake of clarity, and since
it is of independent interests, we present the result not for certain
topological groups but more generally for certain topological spaces.

For arbitrary topological space \( X \), following Bourbaki
\cite{MR4301385}*{I.5, Définition~3}, we say a subalgebra \( A \) of \( C(X) \)
(by subalgebra, we always mean the unital ones) is \textbf{full}, if
\( f \in A \) and \( f \) is invertible in \( C(X) \), then \( f^{-1} \in A \).
Note that \( f \in C(X) \) is invertible in \( C(X) \) if and only if it doesn't
vanish everywhere. We also need a naive notion of positivity, and we denote by
\( A_{\interval{0}{1}} \) the convex set \( \set*{f \in A \given f(x) \in \interval{0}{1}} \)
in \( A \).

\begin{nota}
  \label{nota:90d693f591543e28}
  For an arbitrary algebra \( A \) equipped with a locally convex topology, we
  use \( \chi_{c}(A) \) to denote the set of all \emph{continuous} characters
  of \( A \) (characters here mean unital multiplicative functionals on
  \( A \)).
\end{nota}

We track the following elementary result on ``partitions of unity on compacts''
as a preparation for Theorem~\ref{theo:6c2d8092e7020245}. The argument is
routine but we nevertheless include a proof for convenience of the reader.

\begin{lemm}
  \label{lemm:1621bed1b9b71729}
  Let \( X \) be a topological space, and \( A \) a full subalgebra of
  \( C(X) \) such that functions in \(  A_{\interval{0}{1}} \) separates compact sets and
  closed set, i.e.\ for each compact \( K \subseteq X \) and closed
  \( C \subseteq X \) with \( K \cap C = \emptyset \), there exists
  \( f \in  A_{\interval{0}{1}} \) such that \( f(K) = \set*{1} \) and \( f(C) = \set*{0} \).
  Then for each compact set \( K \) in \( X \) and each finite covering
  \( (U_{k})_{k =1}^{n} \) of \( K \) with each \( U_{k} \) open in \( X \),
  there exists \( f_{1}, \ldots, f_{n} \in A_{\interval{0}{1}} \), with
  \( \supp f_{k} \subseteq U_{k} \) for each \( k \), and
  \( \sum_{k=1}^{n} f_{k}(x) = 1 \) for all \( x \in K \).
\end{lemm}
\begin{proof}
  Since \( (U_{k}) \) covers \( K \), for each \( x \in K \), we have
  \( x \in U_{k(x)} \) for some \( 1 \leq k(x) \leq n \). As
  \( A_{\interval{0}{1}} \) separates points (singletons are compact) and closed
  sets in \( X \), there exists \( g_{x} \in A_{\interval{0}{1}} \), such that
  \( g_{x}(x) = 1 \) and \( g_{x}(X \setminus U_{k(x)}) = \set*{0} \). Put
  \begin{displaymath}
    V_{k(x)}:= \set*{y \in X \given g(y) > \frac{1}{2}}.
  \end{displaymath}
  Then \( V_{k(x)} \) is an open neighborhood of \( 0 \), and
  \begin{displaymath}
    \overline{V_{k(x)}} \subseteq \set*{y \in X \given g(y) \geq \frac{1}{2}} \subseteq U_{k(x)}.
  \end{displaymath}
  We may thus use \( A_{\interval{0}{1}} \) separating the compact set
  \( \set*{x} \) and the closed set \( X \setminus V_{k(x)} \) again to find an
  \( h_{x} \in A_{\interval{0}{1}} \), such that \( h_{x}(x) = 1 \) and
  \( h_{x} \) vanishes outside \( V_{k(x)} \).

  Clearly
  \( h_{x}^{-1}(\interval[open]{0}{+\infty}) = h_{x}^{-1}(\interval[open
  left]{0}{1}) \), \( x \in K \) is an open cover of the compact set \( K \).
  Hence there exists finitely many \( x_{1}, \ldots, x_{m} \in K \), such that
  \begin{equation}
    \label{eq:e0039f338f614492}
    W:= \bigcup_{i=1}^{m} h_{x_{i}}^{-1}(\interval[open left]{0}{1}) \supseteq K.
  \end{equation}
  Since \( A_{\interval{0}{1}} \) separates the compact \( K \) and the closed
  set \( X \setminus W \), there exists \( h_{0} \in A_{\interval{0}{1}} \) with
  \( h_{0}\vert_{K} = 1 \) and \( h_{0}\vert_{X \setminus W} = 0 \). Take
  \( h_{1}:= 1 - h_{0} \in A_{\interval{0}{1}} \), then
  \( h_{1}\vert_{X \setminus W} = 1 \) and \( h_{1} \vert_{K} = 0 \). Define
  \( h:= h_{1} + \sum_{i=1}^{m}h_{x_{i}} \in A \). Then for each \( y \in X \),
  either \( y \in W \), in which case \( h_{x_{i}}(y) > 0 \) for some
  \( 1 \leq i \leq m \) by \eqref{eq:e0039f338f614492}, or \( y \notin W \), in
  which case \( h_{1}(y) = 1 > 0 \). Hence \( h(y) > 0 \) for all \( y \in X \),
  and \( h^{-1} \in A \) since \( A \) is a full subalgebra of \( C(X)
  \).
  Moreover, \( h_{1}\vert_{K} = 0 \) implies that
  \begin{equation}
    \label{eq:853be426cc256c73}
    y \in K \implies h(x) = \sum_{i = 1}^{m}h_{x_{i}}(y) > 0.
  \end{equation}

  Now, define \( f_{x_{i}} = h^{-1}h_{x_{i}} \in A_{\interval{0}{1}} \). Clearly, for each
  \( 1 \leq i \leq m \),
  \begin{equation}
    \label{eq:6b6c8be186e5cafa}
    \supp f_{x_{i}} \subseteq \supp h_{x_{i}} \subseteq \overline{V_{k(x_{i})}} \subseteq U_{k(x_{i})},
  \end{equation}
  and by \eqref{eq:853be426cc256c73}, we have
  \begin{equation}
    \label{eq:4463e45a227cc5ae}
    y \in K \implies \sum_{i=1}^{m} f_{x_{i}}(y) = 1.
  \end{equation}
  For each \( 1 \leq k \leq m \), define
  \begin{displaymath}
    f_{k}: = \sum_{k(x_{i}) = k}f_{x_{i}}
  \end{displaymath}
  with the empty sum being understood as \( 0 \). Then
  \eqref{eq:6b6c8be186e5cafa} implies that \( \supp f_{k} \subseteq U_{k} \) for
  each \( 1 \leq k \leq m \), and \eqref{eq:4463e45a227cc5ae} implies that
  \( \sum_{k=1}^{n}f_{k}(x) = 1 \) for all \( x \in K \). Since clearly every
  \( f_{k}(x) \geq 0 \). Hence \( f_{k} \in A_{\interval{0}{1}} \) for each
  \( 1 \leq k \leq n \), which finishes the proof.
\end{proof}

We introduce the following notion in order to formulate our generalization of
Gelfand's duality.

\begin{defi}
  \label{defi:f045397be943291d}
  Let \( X \) be a \( k \)-space, \( A \) a subspace of \( C(X) \). We say a
  locally convex topology \( \tau \) on \( A \) is \textbf{compactly localized},
  if for every continuous seminorm \( q \) on \( (A, \tau) \), there exits
  \( K \in \mathfrak{K}(X) \), such that \( q(f) = 0 \) whenever
  \( f\vert_{K} = 0 \) for all \( f \in A \).
\end{defi}

\begin{theo}
  \label{theo:6c2d8092e7020245}
  Let \( X \) be a \( k \)-space and \( A \) a full subalgebra of \( C(X) \).
  When \( \mathbb{K} = \C \), we assume in addition that \( A \) is self-adjoint
  in the sense that \( f \in A \) implies \( \overline{f} \in A \).  Let
  \( \tau \) be a locally convex topology on \( A \) making the embedding
  \( i_{A}: (A, \tau) \hookrightarrow C(X) \) continuous, and equip
  \( \chi_{c}(A) \) with the topology of compact convergence, i.e.\ the
  subspace topology of \( A'_{c} \). If \( A_{\interval{0}{1}} \) separates
  compact sets and closed sets in \( X \), then the map
  \( \delta: x \in X \mapsto \delta_{x} \in \chi_{c}(A) \) is a well-defined
  homeomorphism onto its image. If in addition, \( \tau \) is compactly
  localized, then \( \delta: X \to \chi_{c}(A) \) is surjective (hence a
  homeomorphism).
\end{theo}
\begin{proof}
  We first show that \( \delta :X \to \chi_{c}(A) \) is indeed well-defined
  and continuous. Let \( \tau_{0} \) denote the subspace topology on \( A \)
  induced by the topology of \( C(X) \), then continuity of \( i_{A} \) means
  \( \tau \) is finer than \( \tau_{0} \). As the singleton
  \( \set*{x} \in \mathfrak{K}(X) \), it is clear that
  \( \delta_{x} \in C(X)' \), and by restriction,
  \( \delta_{x} \in (A, \tau_{0})' \), hence \( \delta_{x} \in (A, \tau)' \)
  since \( \tau \) is finer than \( \tau_{0} \). It is clear that
  \( \delta_{x} \) is a unital multiplicative linear functional on \( A \), thus
  \( \set*{\delta_{x} \given x \in X} \subseteq \chi_{c}(A) \), and
  \( \delta \) is indeed well-defined. We now establish the continuity of
  \( \delta \). To this end, take an arbitrary precompact set \( H \) in
  \( (A, \tau) \), which is a priori (since \( \tau_{0} \) is coarser than
  \( \tau \)) precompact in \( (A, \tau_{0}) \) hence in \( C(X) \). By Ascoli's
  theorem (Theorem~\ref{theo:957848d2ddc8435e}), for every
  \( K \in \mathfrak{K}(X) \), we have, as a family of functions, the
  restriction \( H\vert_{K} \) is equicontinuous, which implies that
  \( \delta\vert_{K}: K \to \chi_{c}(A) \subseteq A'_{c} \) is continuous. Now
  the continuity of \( \delta \) follows from
  Proposition~\ref{prop:42528c02cab8c326}.

  Now we show that \( \delta \) is an open map onto its image. Indeed, take any
  \( x \in X \) and an open neighborhood \( U \) of \( x \). By our assumption,
  there exits \( f \in A \), such that \( f(x) = 1 \) and \( f \) vanishes
  outside \( U \). As a singleton, \( \set*{f} \) is compact in \( A \). Hence,
  by definition,
  \begin{displaymath}
    M\left(\set*{f}, B_{\mathbb{K}}(1/2)\right)
    := \set[\big]{\omega \in A' \given \omega(f) \in B_{\mathbb{K}}(1/2)}
  \end{displaymath}
  is a neighborhood of \( 0 \) in \( A'_{c} \), where
  \( B_{\mathbb{K}}(1) = \set[\big]{\lambda \in \mathbb{K} \given \abs*{\lambda}
    \leq 1/2} \). Consider any \( y \in X \) such that
  \begin{displaymath}
    \delta_{y} - \delta_{x} \in M\left(\set*{f}, B_{\mathbb{K}}(1/2)\right).
  \end{displaymath}
  If \( y \notin U \), then \( f(y) = 0 \) from our choice of \( f \). Hence we
  must have \( y \in U \), and we've established that
  \begin{displaymath}
    \delta(U) \supseteq \left(\delta_{x} + M\left(\set*{f}, B_{\mathbb{K}}(1/2)\right)\right) \cap \delta(X).
  \end{displaymath}
  This means \( \delta : X \to \delta(X) \) is open at any point \( x \in X \),
  hence is an open map.
  
  Finally, suppose \( \tau \) is compactly localized, and we need to show that
  every \( \omega \in \chi_{c}(A) \) lies in \( \delta(X) \). By continuity of
  \( \omega \), there exists a continuous seminorm \( q \) on \( (A, \tau) \),
  such that \( \abs*{\omega(f)} \leq q(f) \). Since \( \tau \) is compactly
  localized, we may choose a \( K \subseteq \mathfrak{K}(X) \), such that
  \( q(f) = 0 \) whenever \( f\vert_{K} = 0 \). Consider the restriction map
  \( r_{K}: A \to C(K) \), \( f \mapsto f\vert_{K} \). We equip \( C(K) \) with
  the inductive locally convex topology (\S~\ref{sec:e39c2dc88f4df7b1})
  \( \tau_{K} \) with respect to the singleton \( \set*{r_{K}} \). Let
  \( A_{K} = r_{K}(A) \), then equip \( A_{K} \) with the subspace topology of
  \( \bigl(C(K), \tau_{K}\bigr) \), the quotient map
  \( \pi_{K}: A/\ker(r_{K}) \to A_{K} \) is strict
  (Definition~\ref{defi:2c61a2eb8f753159}). Since \( A \) separates points in
  \( X \), clearly \( A_{K} \) separates points in \( K \), and \( A_{K} \) is
  dense in \( C(K) \) in its uniform supremum norm topology by
  Stone-Weierstrass. Our choice of \( K \) and \( \abs*{\omega(f)} \leq q(f) \)
  implies that \( \omega \) vanishes on \( \ker(r_{K}) \), hence factors through
  \( \pi_{K} \) via a unique continuous linear form
  \( \overline{\omega}: A_{K} \to \mathbb{K} \). We claim that
  \( \overline{\omega} \) is continuous with respect to the (coarser) uniform
  supremum norm topology, for which it suffices to employ the fullness of
  \( A \). Indeed, suppose \( f \in A \) with
  \begin{displaymath}
    \norm*{f}_{u, K}: = \sup\set[\big]{\abs*{f(x)} \given x \in K} < 1.
  \end{displaymath}
  Define
  \begin{displaymath}
    W:= \set[\big]{x \in X \given \abs*{f(x)} < 1}.
  \end{displaymath}
  Then \( W \) is an open set in \( X \) containing \( K \). By
  Lemma~\ref{lemm:1621bed1b9b71729}, there exists
  \( \varphi \in A_{\interval{0}{1}} \), such that \( \varphi\vert_{K} = 1 \)
  and \( \varphi\vert_{X \setminus W} = 0 \). It follows that
  \( \abs*{\varphi(x)f(x)} < 1 \) for any \( x \in X \), hence
  \( 1 - \varphi f \in A \) is invertible in \( C(X) \), which implies
  \( (1 - \varphi f)^{-1} \in A \) by the fullness of \( A \). Restricting to
  \( K \), we see that \( 1 - f\vert_{K} \) is invertible in \( A_{K} \).
  Clearly \( \overline{\omega} \) is a unital multiplicative functional on
  \( A_{K} \). The above reasoning shows that for each \( g \in A_{K} \),
  whenever \( \abs*{\lambda} > \norm*{g}_{u, K} \), then
  \( \lambda - g = \lambda\bigl(1 - \lambda^{-1}g\bigr) \) is invertible in
  \( A_{K} \), hence
  \( 0 \ne \overline{\omega}(\lambda - g) = \lambda - \overline{\omega}(g)
  \). This forces \( \overline{\omega}(g) \leq \norm*{g}_{u, K} \), and
  \( \overline{\omega} \) is continuous on \( (A_{K}, \norm*{\cdot}_{u, K}) \),
  thus extends by continuity to a unique unital multiplicative continuous
  functional on \( C(K) \), in which case, it is well-known that there exists a
  unique \( x \in K \), with \( \overline{\omega} = \delta_{x} \) as functionals
  on \( C(K) \) (see, e.g.\ \cite{MR1468229}*{p211, Corollary~3.4.2}). Now it is
  easy to check that we must have \( \omega = \delta_{x} \) from our
  construction as well, and the proof is complete.
\end{proof}

\begin{rema}
  \label{rema:bc0abb5886dfa7a0}
  There are also other generalization of the Gelfand duality, e.g.\
  \cite{MR0528590} and the more recent \cite{MR4075616}. However, the spectrum
  obtained in \cite{MR0528590} is always compact, and the functions used in
  \cite{MR4075616} are bounded, thus they do not apply in our setting. The above
  involved formulation of Theorem~\ref{theo:6c2d8092e7020245} shall enable us to
  recover various topological groups in various settings, as will be presently
  demonstrated.
\end{rema}

\begin{rema}
  \label{rema:9719618bdbf575c3}  
  Part of the proof of Theorem~\ref{theo:6c2d8092e7020245} is inspired by the
  answer provided anonymously to the author's question on the mathoverflow site
  \cite{439722}. The other question answered by van Name \cite{456094} shows
  that unlike the case for commutative \( C^{\ast} \)\nobreakdash-algebras, it
  is necessary to assume some sort of continuity on characters in
  Theorem~\ref{theo:6c2d8092e7020245}, even in the locally compact case.
\end{rema}

\begin{rema}
  \label{rema:f9b61b94b8f8e558}
  When \( C(X) \) is of class \( (\mathcal{F}) \), i.e.\ when \( X \) has a
  fundamental sequence of compacts, by the theorem of Banach-Dieudonné
  (Theorem~\ref{theo:4b05f97fb7cc0705}), we see that we may replace the topology
  of (pre)compact convergence by any of the topologies in
  Theorem~\ref{theo:4b05f97fb7cc0705}, and still end up with the same
  \( \chi_{c}\bigl(C(X)\bigr) \). In particular, if we know that, as a subset of
  \( C(X)' \), the set \( \chi_{c}\bigl(C(X)\bigr) \) is equicontinuous, then we
  may even replace the topology of (pre)compact convergence with the topology of
  simple convergence. In the case where \( X \) is compact, we have \( C(X) \)
  is a Banach algebra, and all characters (without assuming continuity) are
  automatically of norm \( 1 \) (see, \cite{takesaki_theory_2002}*{p15,
    Proposition~3.9}), hence \( \chi_{c}\bigl(C(X)\bigr) \) is
  equicontinuous. Thus Theorem~\ref{theo:6c2d8092e7020245} recovers the Gelfand
  duality for compact spaces when \( X \) is compact.
\end{rema}

When \( M \) is a paracompact smooth manifold. It is clear by smooth partitions
of unity that the algebra \( C^{\infty}(M) \) separates closed sets, in
particular, separates compact and closed sets in \( M \). Note also that the
algebra of all continuous functions on a completely regular space (see before
Theorem~\ref{theo:6398e30fa8e59f95}) also enjoys this property. Again, the proof
is elementary, but we include the proof for completeness.

\begin{lemm}
  \label{lemm:0319630a6d23a435}
  If \( X \) is a completely regular space, then \( C(X)_{\interval{0}{1}} \)
  separates compact sets and closed sets in \( X \).
\end{lemm}
\begin{proof}
  Let \( K \) (resp.\ \( C \)) be compact (resp.\ closed) in \( X \), with
  \( K \cap C = \emptyset \). For each \( x \in K \), there exists
  \( f_{x} \in C(X)_{\interval{0}{1}} \) with \( f_{x}(x) = 1 \) and
  \( f_{x} \vert_{C} = 0 \). Clearly the family of open sets
  \( f_{x}^{-1}(\interval[open]{0}{+\infty}) \), \( x \in K \) covers \( K
  \). Hence by compactness of \( K \), there exists finitely many
  \( x_{1}, \ldots, x_{n} \in K \) with
  \( \cup_{k=1}^{n} f_{x_{k}}^{-1}(\interval[open]{0}{+\infty}) \supseteq K
  \). Let \( g: = \sum_{k=1}^{n}f_{x_{k}} \), then \( g \in C(X) \) and
  \( g \geq 0 \) with \( g\vert_{C} = 0 \) and
  \( W:= g^{-1}(\interval[open]{0}{+\infty}) \) containing \( K \). Define
  \( g_{k}(x):= f_{x_{k}}(x) / g(x) \) if \( x \in W \), and \( g_{k}(x) = 0 \)
  if \( x \notin W \). It is clear that each \( g_{k} \in C(X) \), and the sum
  \( g = \sum_{k=1}^{n}g_{k} \in C(X)_{\interval{0}{1}} \) with
  \( g\vert_{K} = 1 \) and \( g\vert_{C} = 0 \).
\end{proof}

\begin{coro}
  \label{coro:8c3ea439f769498f}
  Let \( X \) be a \( k \)-space, and \( A \) a subalgebra of \( C(X) \). Then
  \( \delta: X \to \chi_{c}(A) \) is a homeomorphism in either one the following cases:
  \begin{enumerate}
  \item \label{item:66586bfcf018c603} \( X \) is completely regular and \( A = C(X) \);
  \item \label{item:a75e6b8ef145ab23} \( X \) is a paracompact smooth manifold
    and \( A = C^{\infty}(M) \).
  \end{enumerate}
\end{coro}
\begin{proof}
  \ref{item:66586bfcf018c603} follows from Lemma~\ref{lemm:0319630a6d23a435} and
  Theorem~\ref{theo:6c2d8092e7020245}, and \ref{item:a75e6b8ef145ab23} from the
  same Theorem and the existence of smooth partitions of unity.
\end{proof}

\subsection{Topological groups as the character group with (pre)compact convergence}
\label{sec:27f04a6169ab3e2c}

We may now describe how to recover the some topological groups from the
associated locally convex Hopf algebra as character groups.

\begin{theo}
  \label{theo:61825be42ff88744}
  In either of the following situations:
  \begin{enumerate}
  \item \label{item:b1ec8f5fa1fd5bf0} \( G \) is a Lie group, and
    \( \mathcal{H}_{G} \) is the \( \varepsilon \)-Hopf algebra
    \( C^{\infty}(G) \) as in Theorem~\ref{theo:6b0cb8e408044129};
  \item \label{item:d97cc18873ee9821} \( G \) is a topological group with
    compactly generated topology, and \( \mathcal{H}_{G} \) the
    \( \varepsilon \)-Hopf algebra \( C(G) \) as in
    Theorem~\ref{theo:c72b0ef0a86663be}.
  \end{enumerate}
  Then, the map \( \delta: G \to \chi_{c}\left(\mathcal{H}_{G}\right) \) is an
  isomorphism of topological groups. The same holds in the complex case, we
  consider \( \mathcal{H}_{G} \) as an \( \varepsilon \)-Hopf-\( \ast \) algebra
  and replace \( \chi_{c}(\mathcal{H}_{G}) \) by
  \( \chi_{c}^{\inv}(\mathcal{H}_{G}) \).
\end{theo}
\begin{proof}
  That \( \delta \) in these situations is a homeomorphism follows from
  Corollary~\ref{coro:8c3ea439f769498f} and
  Theorem~\ref{theo:6c2d8092e7020245}. The fact that \( \delta \) is a group
  morphism follows from a direct calculation.
\end{proof}

\subsection{An analogue of the Eymard-Stinespring-Tatsumma duality}
\label{sec:b961010092b90f57}

It is already noticed in \S~\ref{sec:de569e5c9166cc4a} that we may recover a
discrete group \( \Gamma \) from the \( \pi \)-Hopf algebra
\( \ell^{1}(\Gamma) \) as the collection of group like elements. There is a far
reaching analogous result for all locally compact groups, see
\cite{takesaki_theory_2003}*{p69, Theorem~3.9}.  In terms of the theory of
locally compact quantum groups, this result reads as follows.
\begin{theo}[Eymard-Stinespring-Tatsumma]
  \label{theo:14567563adcd04ee}
  Let \( (\mathscr{M}, \Delta) \) be the von Neumann algebra version of the dual
  of a locally compact group \( G \). Then there is a canonical bijection
  between elements in \( G \) and the set
  \( \set*{x \in \mathscr{M} \given \Delta(x) = x \otimes x, \, x \ne 0} \) of
  group like elements.
\end{theo}

\begin{rema}
  \label{rema:8872df75f79bc16c}
  The original formulation in \cite{takesaki_theory_2003}*{p69, Theorem~3.9}
  uses a certain unitary operator \( W \) instead of the comultiplication
  \( \Delta \). It is clear that the operator \( W \) there is exactly the
  corresponding multiplicative unitary (\cite{MR1235438}) associated with the
  locally compact quantum group \( (\mathscr{M}, \Delta) \). Since
  Theorem~\ref{theo:14567563adcd04ee} is still in the Kac case, the theory of
  Kac algebras (\cite{MR1215933}) already suffices to see this. In the full
  generality of locally compact quantum groups, the whole story can be rather
  involved, see also Woronowicz's contribution \cite{MR1369908} and
  \cites{MR1951446,MR1832993}.  \cite{takesaki_theory_2003}*{p69, Theorem~3.9} is
  called the \textbf{Eymard-Stinespring-Tatsumma} theorem by Takesaki
  \cite{takesaki_theory_2003}*{p90}, attributing to various contributions
  \cites{MR0228628,MR0238997,MR0102761,MR0217222}.
\end{rema}

Recall Theorem~\ref{theo:fc44c913131be011} and
Theorem~\ref{theo:6b0cb8e408044129} (when the Lie group \( G \) is second
countable, the strong dual of \( C^{\infty}(G) \) is a barreled \( (DF) \)
space, hence it is both a \( \pi \)-Hopf algebra and a \( \iota \)-Hopf
algebra, by Proposition~\ref{prop:89675519c2bb97f8}).

We now establish the following analogue of the Eymard-Stinespring-Tatsumma
duality for locally compact groups.
\begin{theo}
  \label{theo:a34d8747bbe913ff}
  If \( \mathcal{H}_{\widehat{G}} \) is a \( \pi \)-Hopf algebra given by any of
  the following:
  \begin{enumerate}
  \item \label{item:458016a3fac489b0} the strong dual of the
    \( \varepsilon \)-Hopf algebra \( \mathcal{H}_{G} = C^{\infty}(G) \) for a
    second countable Lie group \( G \);
  \item \label{item:2e3568f80838f408} the polar dual of the
    \( \varepsilon \)-Hopf algebra \( \mathcal{H}_{G} = C(G) \) for a
    topological group with compactly generated topology that admits a
    fundamental sequence of compact sets.
  \end{enumerate}
  Then the map \( G \to \mathsf{Grp}(\mathcal{H}_{\widehat{G}}) \),
  \( g \mapsto \delta_{g} \) is an isomorphism of topological groups.
\end{theo}
\begin{proof}
  In \ref{item:458016a3fac489b0}, since \( C^{\infty}(G) \) is of class
  \( (\mathcal{FN}) \) by Lemma~\ref{lemm:975e581967ecec25}, hence is of class
  \( (\mathcal{M}) \), we have
  \( \mathcal{H}_{\widehat{G}} = (\mathcal{H}_{G})'_{c} \) as locally convex
  space in both cases (Proposition~\ref{prop:99b4fc84a9b1d8a1}).

  Clearly we have a bijective correspondence
  \begin{equation}
    \label{eq:4c090eac0143241e}
    \begin{split}
      \Phi: \mathsf{Grp}(\mathcal{H}_{\widehat{G}}) & \to \mathsf{Coalg}_{\pi}(\mathbb{K}, \mathcal{H}_{\widehat{G}}) \\
      x & \mapsto \{\varphi_{x}: \lambda \mapsto \lambda x \}.
    \end{split}
  \end{equation}
  On the other hand, one checks immediately by taking transposes that we also
  have a bijective correspondence
  \begin{equation}
    \label{eq:e0f767ce001fe913}
    \begin{split}
      \Psi: \mathsf{Coalg}_{\pi}(\mathbb{K}, \mathcal{H}_{\widehat{G}})
      & \to \mathsf{Alg}_{\pi}((\mathcal{H}_{\widehat{G}})', \mathbb{K})
        = \mathsf{Alg}_{\pi}(\mathcal{H}_{G}, \mathbb{K})
        = \chi(\mathcal{H}_{G}) \\
      \varphi & \mapsto \varphi^{\transp}.
    \end{split}
  \end{equation}
  By Theorem~\ref{theo:61825be42ff88744}, \( \chi(\mathcal{H}_{G}) \simeq G \)
  in both cases, and one checks easily by following the composition
  \( \Phi^{-1}\psi^{-1} \) and the isomorphism
  \( \delta: G \to \chi(\mathcal{H}_{G}) \) that the resulting correspondence,
  still denoted by \( \delta: G \to \mathsf{Grp}(\mathcal{H}_{\widehat{G}}) \),
  \( g \mapsto \delta_{g} \) is a group isomorphism in both cases.  It remains
  to show that this is also a homeomorphism, but this already follows from
  Corollary~\ref{coro:8c3ea439f769498f} in both cases.
\end{proof}

\subsection{Recovering of the Pontryagin dual of locally compact abelian groups}
\label{sec:3e2f92039fbda631}

We now describe how Pontryagin duality for all locally compact abelian groups
can manifest itself in our theory.

Since all locally compact spaces are \( k \)-spaces, it is readily seen, as a
special case of Theorem~\ref{theo:61825be42ff88744}, that any locally compact
group \( G \), in particular the abelian ones, can already be recovered from the
associated \( \varepsilon \)-Hopf algebra \( C(G) \) as the character group
\( \chi_{c}\bigl(C(G)\bigr) \) equipped with the topology of compact
convergence.

Recall that, following Weil \cite{MR0005741}*{pp99-101}, the Pontryagin dual
\( \widehat{G} \) of a locally compact abelian group \( G \) is constructed as
follows: as a group, \( \widehat{G} \) consists of all continuous
one-dimensional \emph{unitary} representations of \( G \), where the
multiplication is pointwise; the topology on \( \widehat{G} \) is the topology
of compact convergence.

\begin{theo}
  \label{theo:bab58124687e3b85}
  Let \( G \) be a locally compact group and \( C(G) \) the associated
  \( \varepsilon \)-Hopf algebra.
  \begin{enumerate}
  \item \label{item:da8ee13534935e3b} An element \( f \in C(G) \) is group-like
    if and only if \( f : G \to \mathbb{K} \) is a continuous (one-dimensional)
    representation of \( G \).
  \item \label{item:8f6e0b3492454344} If \( \mathbb{K} = \C \) and consider
    \( C(G) \) as an \( \varepsilon \)-Hopf-\( \ast \) algebra. An element
    \( f \in C(G) \) is an involutive group-like elements if and only if
    \( f : G \to \C \) is a unitary representation. Moreover, if \( G \) is
    abelian, then \( \mathsf{Grp}^{\inv}\bigl(C(G)\bigr) \), when equipped with
    the subspace topology induced from \( C(G) \), is exactly the Pontryagin
    dual \( \widehat{G} \) of \( G \).
  \end{enumerate}
\end{theo}
\begin{proof}
  \ref{item:da8ee13534935e3b}. By definition, \( f \in \mathsf{Grp}(C(G)) \) if
  and only if \( f(1) = 1 \), and
  \( \Delta(f) = f \overline{\otimes}_{\varepsilon} f \). As continuous
  functions in \( C(G \times_{k}G) = C(G \times G) \) (see
  Proposition~\ref{prop:27398d650b82b971}), we have \( \Delta(f)(s,t) = f(st) \)
  while \( (f \overline{\otimes}_{\varepsilon} f)(s,t) = f(s)f(t) \), where
  \( s, t \in G \). Now \ref{item:da8ee13534935e3b} is clear.

  \ref{item:8f6e0b3492454344}. In this case, suppose
  \( f \in \mathsf{Grp}(C(G)) \). To say that \( f \) is involutive means
  \( \overline{f} = Sf \), i.e.\ for all \( s \in G \),
  \( \overline{f(s)} = f(s^{-1}) = f(s)^{-1} \). Hence
  \( f \in \mathsf{Grp}^{\inv}(C(G)) \) if and only if \( f : G \to \C \) is a
  unitary representation. We now have
  \( \widehat{G} = \mathsf{Grp}^{\inv}(C(G)) \) as a group. Since the topology
  on \( C(G) \) is precisely the topology of compact convergence, the identity
  \( \widehat{G} = \mathsf{Grp}^{\inv}(C(G)) \) also holds true as topological
  groups.
\end{proof}

\begin{rema}
  \label{rema:c6f6e7cc787be0fd}
  It is interesting to note that when a locally compact abelian group \( G \) is
  second countable, it has a fundamental sequence of compact sets, and
  Theorem~\ref{theo:fc44c913131be011} implies that the \( \varepsilon
  \)-Hopf-\( \ast \) algebra \( C(G) \) is \( (\varepsilon, \pi) \)-polar
  reflexive, and the polar dual \( C(G)'_{c} \) is a \( \pi \)-Hopf-\( \ast \)
  algebra. By Proposition~\ref{prop:f8252fbb35384d7c}, we see that, as groups,
  \( \widehat{G} \) can be identified canonically as
  \( \chi^{\inv}\big(C(G)'_{c}) \). One may even try to consider
  \( \chi^{\inv}_{c}\bigl(C(G)'_{c}\bigr) \), i.e.\ \( \chi^{\inv}(C(G)'_{c}) \)
  equipped with the subspace topology of
  \( \mathcal{L}_{c}\bigl(C(G)'_{c}, \C\bigr) \) (which is canonically
  isomorphic to \( C(G) \) as vector spaces since \( C(G) \), being an
  \( (F) \)-space, is polar reflexive, but the topology is different). One can
  show that in general, the topology on
  \( \chi^{\inv}_{c}\bigl(C(G)'_{c}\bigr) \) is finer than the one on
  \( \widehat{G} \). At the time of this writing, the author does not know
  whether these topologies coincide, or even whether
  \( \chi^{\inv}_{c}\bigl(C(G)'_{c}\bigr) \) is a topological group.
\end{rema}

\section{Certain projective and inductive limits}
\label{sec:7a5cda2c4078b89b}

We describe in this section some general theory concerning projective and
inductive limits of locally convex Hopf algebras. These abstract theory shall be
illustrated by concrete examples in \S~\ref{sec:e3160c56b87e4b14}.

\subsection{Reduced projective limits of projective and injective Hopf algebras}
\label{sec:fd4b65d5b8ba0635}

Using the commutativity of \( \overline{\otimes}_{\pi} \) (resp.\
\( \overline{\otimes}_{\varepsilon} \)) with reduced projective limits, we can
easily construct reduced projective limits of \( \pi \)-Hopf algebras (resp.\
\( \varepsilon \)-Hopf algebras). We describe how the construction goes in
general for \( \tau \)-Hopf algebras, where \( \overline{\otimes}_{\tau} \) is a
symmetric compatible monoidal functor that commutes with reduced projective
limits. As examples, we may take
\( \tau \in \set*{\pi, \varepsilon} \)(Proposition~\ref{prop:046bbcf972e4452d}
and Proposition~\ref{prop:8056525c3c4c9aa5}).

Let \( (H_{i}, p_{i,j})_{i \in I} \) be a projective system of \( \tau \)-Hopf
algebras. Assume that as locally convex spaces, the projective limit
\( H:= \varprojlim H_{i} \) is reduced. Consider the projective system
\( (H_{i} \overline{\otimes} H_{j}, p_{i,j} \overline{\otimes}_{\tau} p_{s,t})
\) on \( I \times I \) (equipped of course with the product order:
\( (i, j) \leq (s, t) \) means \( i \leq s \) and \( j \leq t \)). Then by the
commutation of \( \overline{\otimes}_{\tau} \) with reduced projective limits,
we see that \( H \overline{\otimes}_{\tau} H \) can be identified canonically
with the reduced projective limit
\( \varprojlim_{I \times I} H_{i} \overline{\otimes}_{\tau} H_{j} \). Moreover,
as \( I \) is directed, we see that
\( (H_{i} \overline{\otimes}_{\tau} H_{i}, p_{i,j} \overline{\otimes}_{\tau}
p_{i,j})_{i \in I} \) is cofinal in
\( (H_{i} \overline{\otimes} H_{j}, p_{i,j} \overline{\otimes}_{\tau} p_{s,t})
\). We then have a canonical identification
\( H \overline{\otimes}_{\tau} H = \varprojlim H_{i} \overline{\otimes}_{\tau}
H_{i} \), with
\( p_{i} \overline{\otimes}_{\varepsilon} p_{i}: H \overline{\otimes}_{\tau} H
\to H_{i} \overline{\otimes}_{\tau} H_{i} \) being the canonical projections.
Similarly, using the commutation of \( \overline{\otimes}_{\tau} \) with reduced
projective limits again, we have the analogous results for three-fold and
four-fold tensor products.

On the other hand, we might use the same index set \( I \) to build a projective
system \( (K_{i}, \beta_{i,j}) \), with each \( K_{i} = \mathbb{K} \) and each
\( \beta_{i,j} = \id_{\mathbb{K}} \), whenever \( i \leq j \). Since \( I \) is
directed, it is trivial that \( \varprojlim K_{i} = \mathbb{K} \) with all
canonical projections
\( \beta_{i}: \varprojlim K_{i} = \mathbb{K} \to K_{i} = K \) being
\( \id_{\mathbb{K}} \).

Denote the structure maps on each \( H_{i} \) by
\( m_{i}, \eta_{i}, \Delta_{i}, \varepsilon_{i} \) and \( S_{i} \). We now
define \( m \in \mathcal{L}(H \overline{\otimes}_{\tau} H, H) \),
\( \eta \in \mathcal{L}(\mathbb{K}, H) \),
\( \Delta \in \mathcal{L}(H, H \overline{\otimes}_{\tau}H) \),
\( \varepsilon \in \mathcal{L}(H, \mathbb{K}) = H' \) and
\( S \in \mathcal{L}(H) = \mathcal{L}(H, H) \) as follows.

Note that the family \( (\eta_{i})_{i \in I} \) is a morphism from the
projective system \( (K_{i}, \beta_{i,j}) \) of locally convex spaces to the
projective system \( (H_{i}, p_{i,j}) \) of locally convex spaces, in the sense
that whenever \( i \leq j \), \( \eta_{i}\beta_{i,j} = p_{i,j}\eta_{j} \). Hence
there exists a unique continuous linear map
\( \eta: \mathbb{K} = \varprojlim K_{i} \to H = \varprojlim H_{i} \), such that
\( p_{i}\eta = \eta_{i} \beta_{i} \) for each \( i \in I \).

Similarly, the family \( (m_{i})_{i \in I} \) is a morphism from the projective
system
\( (H_{i}\overline{\otimes}_{\tau}H_{i}, p_{i,j} \overline{\otimes}_{\tau}p_{i,j})
\) to the projective system \( (H_{i}, p_{i,j}) \), and we obtain, by passing to
the respective projective limits, a canonical linear continuous map
\( m: H \overline{\otimes}_{\tau} H \to H \). The maps
\( \Delta: H \to H \overline{\otimes}_{\tau} H \),
\( \varepsilon : \mathbb{K} \to H \) and \( S: H \to H \) are defined similarly.

\begin{theo}
  \label{theo:9db4e1dbcb4292bf}
  Let \( \overline{\otimes}_{\tau} \) be a symmetric compatible monoidal functor
  that commutes with reduced projective limits, and \( (H_{i}, p_{i,j}) \) be a
  projective system of \( \tau \)-Hopf algebras, such that as locally convex
  spaces, the projective limit \( H = \varprojlim H_{i} \) is reduced. Then
  \begin{enumerate}
  \item \label{item:2210564839769e90} Equipped with the multiplication \( m \),
    unit \( \eta \), comultiplication \( \Delta \), counit \( \varepsilon \) and
    antipode \( S \) as defined above, \( H \) is a \( \tau \)-Hopf algebra.
  \item \label{item:b8134e6f3a15c80c} The canonical projections
    \( p_{i}: H \to H_{i} \), \( i \in I \) are all morphisms of \( \tau \)-Hopf
    algebras, and \( H \), together with these \( p_{i} \), \( i \in I \), is
    the projective limit of \( (H_{i}, p_{i,j}) \) in the category
    \( \mathsf{Hopf}_{\tau} \) of \( \tau \)-Hopf algebras.
  \item \label{item:63ae727ef870a990} If in addition each \( H_{i} \) is
    regular, then so is \( H \).
  \item \label{item:b649a1d01f2bd2c3} When \( \mathbb{K} = \C \), and if each
    \( H_{i} \) is also a \( \tau \)-Hopf-\( \ast \) algebra and each
    \( p_{i,j} \) being a morphism of \( \tau \)-Hopf-\( \ast \) algebras, then
    so is \( H \) with the involution induced from the projective limit
    construction.
  \end{enumerate}
\end{theo}
\begin{proof}
  First of all, \( H \in \widehat{\mathsf{LCS}} \) by
  Proposition~\ref{prop:438d10ce1f2149df}.
  
  \ref{item:2210564839769e90}. This follows from a routine but tedious
  verification. We briefly describe this verification process for completeness.

  Note that from the defining property of \( m \), we have, for each
  \( i \in I \), that
  \begin{displaymath}    
    \begin{split}
      p_{i}m(\id_{H} \overline{\otimes}_{\tau} m)
      &= m_{i}(p_{i} \overline{\otimes}_{\tau}p_{i})(\id_{H} \overline{\otimes}_{\tau} m)
        = m_{i}\bigl(p_{i} \overline{\otimes}_{\tau} (p_{i}m)\bigr) \\
      &= m_{i} \bigl(p_{i} \overline{\otimes}_{\tau} (m_{i}(p_{i} \overline{\otimes}_{\tau} p_{i}))\bigr) 
        = m_{i}(\id_{H_{i}} \overline{\otimes}_{\tau} m_{i})(p_{i} \overline{\otimes}_{\tau} p_{i} \overline{\otimes}_{\tau} p_{i}) \\
      &= m_{i}(m_{i} \overline{\otimes}_{\tau} \id_{H})(p_{i} \overline{\otimes}_{\tau} p_{i} \overline{\otimes}_{\tau} p_{i})
        = \cdots = p_{i}m(m \overline{\otimes}_{\tau} \id_{H}).
    \end{split}
  \end{displaymath}
  By the universal property of projective limits, this forces
  \( m(\id_{H} \overline{\otimes}_{\tau} m) = m(m \overline{\otimes}_{\tau}
  \id_{H}) \), and the multiplication \( m \) is associative.

  Similarly, the comultiplication \( \Delta \) is coassociative.

  We also have, for each \( i \in i \), that
  \begin{displaymath}    
    p_{i}m(\id_{H} \overline{\otimes}_{\tau} \eta)
    = m_{i}(p_{i} \overline{\otimes}_{\tau} (p_{i}\eta))
    = m_{i}(p_{i} \overline{\otimes}_{\tau} \eta_{i})
    = p_{i},
  \end{displaymath}
  hence the universal property of projective limits again forces
  \( m(\id_{H} \overline{\otimes}_{\tau} \eta) = \id_{H} \). That
  \( m(\eta \overline{\otimes}_{\tau} \id_{H}) = \id_{H} \) follows from a
  similar argument, and \( \eta \) is the unit for \( m \).

  Similarly, one can show that \( \varepsilon \) is the counit for \( \Delta \).

  We now show that \( \Delta \) is a morphism of \( \tau \)-algebras. Clearly, it
  is unital. To see that it is also multiplicative, note that for each
  \( i \in I \), and recall that we use \( s \) to denote the swap
  (\S~\ref{sec:cbdc5edf0dbbb84c}), we have
  \begin{displaymath}
    \begin{split}
      (p_{i} \overline{\otimes}_{\tau} p_{i}) \Delta m
      &= \Delta_{i} p_{i} m = \Delta_{i} m_{i} (p_{i} \overline{\otimes}_{\tau} p_{i}) \\
      &= (m_{i} \overline{\otimes}_{\tau} m_{i})
        (\id_{H_{i}} \overline{\otimes}_{\tau} s_{H_{i}, H_{i}} \overline{\otimes}_{\tau})
        (\Delta_{i} \overline{\otimes}_{\tau} \Delta_{i})
        (p_{i} \overline{\otimes}_{\tau} p_{i}) \\ 
      &=(m_{i} \overline{\otimes}_{\tau} m_{i})
        (\id_{H_{i}} \overline{\otimes}_{\tau} s_{H_{i}, H_{i}} \overline{\otimes}_{\tau})
        (p_{i} \overline{\otimes}_{\tau} p_{i} \overline{\otimes}_{\tau} p_{i} \overline{\otimes}_{\tau} p_{i})
        (\Delta \overline{\otimes}_{\tau} \Delta) \\
      &= (m_{i} \overline{\otimes}_{\tau} m_{i})
        (p_{i} \overline{\otimes}_{\tau} p_{i} \overline{\otimes}_{\tau} p_{i} \overline{\otimes}_{\tau} p_{i})
        (\id_{H} \overline{\otimes}_{\tau} s_{H,H} \overline{\otimes}_{\tau} \id_{H})
        (\Delta \overline{\otimes}_{\tau} \Delta) \\
      &= (p_{i} \overline{\otimes}_{\tau} p_{i})(m \overline{\otimes}_{\tau} m)        
        (\id_{H} \overline{\otimes}_{\tau} s_{H,H} \overline{\otimes}_{\tau} \id_{H})
        (\Delta \overline{\otimes}_{\tau} \Delta).
    \end{split}
  \end{displaymath}
  This forces
  \begin{displaymath}
    \Delta m = (m \overline{\otimes}_{\tau} m)        
        (\id_{H} \overline{\otimes}_{\tau} s_{H,H} \overline{\otimes}_{\tau} \id_{H})
        (\Delta \overline{\otimes}_{\tau} \Delta)
  \end{displaymath}
  by the universal property of projective limits, and \( \Delta \) is indeed a
  morphism of \( \tau \)-algebras, and we've established that
  \( (H, m, \eta, \Delta, \varepsilon) \) is a \( \tau \)-bialgebra.
  
  Finally, for each \( i \in I \), we have
  \begin{displaymath}
    \begin{split}
      p_{i}m(S \overline{\otimes}_{\tau} \id_{H})\Delta
      &= m_{i}(p_{i}S \overline{\otimes}_{\tau} p_{i}) \Delta 
        = m_{i}(S_{i} p_{i} \overline{\otimes}_{\tau} p_{i}) \Delta \\
      &= m_{i}(S_{i} \overline{\otimes}_{\tau} \id_{H_{i}})\Delta_{i} p_{i}
        = \eta_{i} \varepsilon_{i} p_{i} = \eta_{i} \varepsilon = p_{i} \eta \varepsilon.
    \end{split}
  \end{displaymath}
  This forces
  \( m(S \overline{\otimes}_{\tau} \id_{H})\Delta = \eta \varepsilon \) by the
  universal property of projective limits. Similarly,
  \( m(\id_{H} \overline{\otimes}_{\tau} S)\Delta = \eta \varepsilon \), and
  \( S \) satisfies the axiom for the antipode, which finishes the proof of
  \ref{item:2210564839769e90}.

  \ref{item:b8134e6f3a15c80c}, \ref{item:63ae727ef870a990} and
  \ref{item:b649a1d01f2bd2c3} are now (easier) routine verifications.
\end{proof}

\begin{rema}
  \label{rema:84a18149794dffaf}
  In \cite{MR3059661} and \cite{MR3212682}*{\S~3}, there are already some
  construction of certain projective limits in the context of
  \( C^{\ast} \)-version or von Neumann algebraic version of compact quantum
  groups. Formally, it seems that our construction here seems more
  natural. Later in \S~\ref{sec:e3160c56b87e4b14}, we will use the construction
  presented here to describe some natural candidates for the quantum groups
  \( S^{+}_{\infty} \), \( O^{+}_{\infty} \) and \( U^{+}_{\infty} \), as well
  as their duals, and justify why these quantum groups seems not even belonging
  to the framework of locally compact quantum groups.
\end{rema}

\begin{defi}
  \label{defi:1cc1ad763eef3b30}
  We call the projective system \( (H_{i}, p_{i,j}) \) in
  Theorem~\ref{theo:9db4e1dbcb4292bf} a \textbf{reduced projective system} of
  \( \tau \)-Hopf(-\( \ast \)) algebra, and \( H \) its \textbf{projective
    limit}.
\end{defi}

\subsection{Strict inductive limits of inductive Hopf algebras}
\label{sec:32608de27100a358}

We may dualize the construction in \S~\ref{sec:fd4b65d5b8ba0635} to construct
arbitrary inductive limits in the category \( \mathsf{Hopf}_{\iota} \) of
\( \iota \)-Hopf algebras: just use the commutation of
\( \overline{\otimes}_{\iota} \) with completed inductive limits, i.e.\
inductive limits in \( \widehat{\mathsf{LCS}} \), see
Corollary~\ref{coro:a9570db3191db003} and
Corollary~\ref{coro:9dd4c5414e7497a4}. However, there are some caveats.

First of all, we've seen that (Remark~\ref{rema:67bc53cfb5df5400}) that the
inductive locally convex topology on an algebraic inductive limit might not even
be Hausdorff, so we need to take the separated quotient
(Definition~\ref{defi:69f4413fd7d99b7b}) in order to form the corresponding
inductive limit in \( \mathsf{LCS} \), see
Proposition~\ref{prop:c2ffa4009a3940c2}. In general, this separated quotient
needs not be complete, even each of locally convex space in the inductive system
is complete (for a counter-example, one may use
\cite{bourbaki2007integration}*{Ch.III, \S~1, Exercice~2}). Hence in order to
obtain the colimit in \( \widehat{\mathsf{LCS}} \) with which we are working, in
general, one also needs to take the completion. As these processes all commutes
with forming inductive limits (see Corollary~\ref{coro:a9570db3191db003} and
Corollary~\ref{coro:9dd4c5414e7497a4}), we may indeed obtain the inductive limit
of an inductive system of \( \iota \)-Hopf algebras in the category
\( \mathsf{Hopf}_{\iota} \). However, even so, we should also note that the
completion process (as well as the process of taking the separated quotient) in
general destroys the relevant duality as in \S~\ref{sec:a27a1ad55ccb128e}.

For these reasons, instead of treating general inductive limits in detail as in
\S~\ref{sec:fd4b65d5b8ba0635} for projective limits, we pay more attention to a
certain special class of strict inductive limits in \( \mathsf{Hopf}_{\iota} \).

\begin{theo}
  \label{theo:8c707fb7841e1dfb}
  Let \( (H_{n})_{n \geq 1} \) be a countable family of \( \iota \)-Hopf
  algebras. Suppose that for each \( n \), the compatible topologies on
  \begin{displaymath}
    \begin{gathered}
      H_{n} \odot H_{n}, \qquad
      (H_{n} \overline{\otimes}_{\iota} H_{n}) \odot H_{\iota}
      \qquad \text{ and }\qquad
      (H_{n} \overline{\otimes}_{\iota} H_{n} \overline{\otimes}_{\iota} H_{n}) \odot
      H_{n}
    \end{gathered}
  \end{displaymath}
  are all unique, and \( u_{n}: H_{n} \hookrightarrow H_{n+1} \) is a morphism
  of \( \iota \)-Hopf algebras that is a strict monomorphism as locally convex
  spaces. Consider the strict inductive limit \( H := \varinjlim H_{n} \) of the
  strict inductive system \( (H_{n}, u_{n}) \) in \( \widehat{\mathsf{LCS}} \)
  with \( v_{n}: H_{n} \hookrightarrow H \) being the canonical injections, and
  let \( \mathbb{K} = \varinjlim K_{n} \) be the inductive limit of the trivial
  inductive system given by \( \id_{\mathbb{K}}: K_{n} \to K_{n+1} \) for each
  \( n \), where each \( K_{n} = \mathbb{K} \). Then in
  \( \widehat{\mathsf{LCS}} \), the inductive systems
  \begin{displaymath}
    \begin{gathered}
      \left(H_{n} \overline{\otimes}_{\iota} H_{n}, u_{n}
        \overline{\otimes}_{\iota} u_{n}\right), \qquad
      \left(H_{n} \overline{\otimes}_{\iota} H_{n} \overline{\otimes}_{\iota}
        H_{n}, u_{n} \overline{\otimes}_{\iota} u_{n} \overline{\otimes}_{\iota}
        u_{n}\right)
      \\
      \text{ and } \qquad
      \left(H_{n} \overline{\otimes}_{\iota} H_{n} \overline{\otimes}_{\iota}
        H_{n} \overline{\otimes}_{\iota} H_{n}, u_{n} \overline{\otimes}_{\iota}
        u_{n} \overline{\otimes}_{\iota} u_{n} \overline{\otimes}_{\iota}
        u_{n}\right)
    \end{gathered}
  \end{displaymath}
  are all strict, with
  \begin{displaymath}
    H \overline{\otimes}_{\iota} H, \qquad
    H \overline{\otimes}_{\iota} H \overline{\otimes}_{\iota} H
    \qquad\text{ and }\qquad
     H \overline{\otimes}_{\iota} H \overline{\otimes}_{\iota} H
  \overline{\otimes}_{\iota} H
  \end{displaymath}
  being canonically identified with their respective strict inductive limit.

  Moreover, define
  \begin{enumerate}
  \item \label{item:8279f98507a59b79} \( \eta: \mathbb{K} \to H \) to be the
    unique continuous linear map such that \( \eta = u_{n} \eta_{n} \) for each
    \( n \), with each \( \eta_{n} \) being the unit of \( H_{n} \);
  \item \label{item:1a4d86685ccac009}
    \( m : H \overline{\otimes}_{\iota} H \to H \) to be the unique continuous
    linear map such that \( m(u_{n} \overline{\otimes}_{\iota}u_{n}) = m_{n} \)
    for each \( n \), with each \( m_{n} \) being the multiplication of \( H_{n} \);
  \item \label{item:d9bdf76f0be62287} \( \varepsilon: H \to \mathbb{K} \) to be
    the unique continuous linear map such that
    \( \varepsilon u_{n} = \varepsilon_{n} \) for each \( n \), with \( \varepsilon_{n} \) being the counit for \( H_{n} \);
  \item \label{item:b9b7a520d59bd261}
    \( \Delta: H \to H \overline{\otimes}_{\iota} H \) being the unique
    continuous linear map such that
    \( \Delta u_{n} = (u_{n} \overline{\otimes}_{\iota} u_{n})\Delta_{n} \),
    with \( \Delta_{n} \) being the comultiplication on \( H_{n} \);
  \item \label{item:931b70bb037e06b2} \( S: H \to H \) being the unique
    continuous linear map such that \( Su_{n} = u_{n}S_{n} \) for each \( n \),
    with \( S_{n} \) being the antipode of \( H_{n} \).
  \end{enumerate}
  Then equipped with these structure maps, \( H \in \widehat{\mathsf{LCS}} \) is
  a \( \iota \)-Hopf algebra, and is the inductive limit of the system
  \( (H_{n}, u_{n}) \) in the category \( \mathsf{Hopf}_{\iota} \) of all
  \( \iota \)-Hopf algebras. When each \( H_{n} \) is regular, then so is
  \( H \). And when \( \mathbb{K} = \C \) and each \( H_{n} \) is a
  \( \iota \)-Hopf-\( \ast \) algebra, with each \( u_{n} \) being a morphism of
  \( \iota \)-Hopf-\( \ast \) algebras, then \( H \) is an \( \iota
  \)-Hopf-\( \ast \) algebra
  under the canonical involution induced it.
\end{theo}
\begin{proof}
  We first establish the first part on the inductive limits. Note that
  \( H \in \widehat{\mathsf{LCS}} \) follows from
  Proposition~\ref{prop:cdd9376419588ce4} and
  Corollary~\ref{coro:7ae986967f0d1033}. By
  Corollary~\ref{coro:9dd4c5414e7497a4}, we have
  \begin{equation}
    \label{eq:d6dc64a379c95bc6}
    H \otimes_{\iota} H = \varinjlim_{m \geq 1, n \geq 1} H_{m} \otimes_{\iota} H_{n}
    = \varinjlim_{n \geq 1} H_{n} \otimes_{\iota} H_{n},
  \end{equation}
  where the last equality follows by noting that
  \( \set*{(n, n) \given n \geq 1} \) are cofinal in \( \N_{+} \times \N_{+} \).
  To prove the strictness of the inductive system with respect to which the
  inductive limit at the right most of \eqref{eq:d6dc64a379c95bc6} is formed,
  note that by uniqueness of the compatible topologies in the assumption and
  Proposition~\ref{prop:710a683d733970db}, we see that
  \begin{displaymath}
    u_{n} \overline{\otimes}_{\iota} u_{n} = u_{n} \overline{\otimes}_{\varepsilon} u_{n}:
    H_{n} \overline{\otimes}_{\iota} H_{n} = H_{n} \overline{\otimes}_{\varepsilon} H_{n}
    \hookrightarrow H_{n+1} \overline{\otimes}_{\varepsilon} H_{n+1} = H_{n+1} \overline{\otimes}_{\iota} H_{n+1}
  \end{displaymath}
  is indeed a strict monomorphism. Now by Corollary~\ref{coro:a9570db3191db003},
  taking completion in \eqref{eq:d6dc64a379c95bc6}, we see that
  \begin{displaymath}
    H \overline{\otimes}_{\iota} H = \varinjlim H_{n} \overline{\otimes}_{\iota} H_{n}.
  \end{displaymath}
  The case for three-fold and four-fold tensor products are established
  similarly.

  To check the second part on the locally Hopf algebras, one can dualize the
  argument in Theorem~\ref{theo:9db4e1dbcb4292bf}. The adaption is trivial and
  hence is omitted here.
\end{proof}

\begin{defi}
  \label{defi:ad896e49d004b1e6}
  We call the inductive system \( (H_{n}, u_{n})_{n \geq 1} \) in
  Theorem~\ref{theo:8c707fb7841e1dfb} a \textbf{strict inductive system} of
  \( \iota \)-Hopf(-\( \ast \)) algebras, and the \( \iota \)-Hopf(-\( \ast \))
  algebra \( H \) the \textbf{strict inductive limit} of \( (H_{n}, u_{n}) \).
\end{defi}

\subsection{A duality result}
\label{sec:a27a1ad55ccb128e}

Note that for any \( (F) \)-spaces \( E \) and \( F \), at least one of which is
nuclear, the compatible topology on \( E \odot F \) is unique
(Proposition~\ref{prop:89e27b078637bfc2},
Proposition~\ref{prop:89675519c2bb97f8} and
Definition~\ref{defi:a13a9b24c4226f4e}).

We are now ready to establish the following duality result.

\begin{theo}
  \label{theo:7f86d4ede3701552}
  Let \( (H_{n}, u_{n})_{n \geq 1} \) of a strict inductive system of
  \( \iota \)-Hopf(-\( \ast \)) algebras of class \( (\mathcal{FN}) \), and
  \( H \) its strict inductive limit. Then
  \begin{enumerate}
  \item \label{item:f7c4e2cb20f91ef6} for each \( n \), the transpose
    \( p_{n}: (H_{n+1})'_{b} \to (H_{n})'_{b} \) of \( u_{n} \) is a morphism of
    \( \varepsilon \)-Hopf(-\( \ast \)) algebras and is surjective as a morphism
    in \( \widehat{\mathsf{LCS}} \);
  \item \label{item:bebbc2c97c2c01ed} the projective system of
    \( \varepsilon \)-Hopf(-\( \ast \)) algebras \( \bigl((H_{n})'_{b}, p_{n}\bigr) \) is
    reduced;
  \item \label{item:f4d4617dcab9c51e} the \( \iota \)-Hopf(-\( \ast \)) algebra
    \( H \) is \( (\iota, \varepsilon) \)-reflexive, and its strong dual is
    canonically isomorphic to the projective limit
    \( \varprojlim (H_{n})'_{b} \).
  \end{enumerate}
\end{theo}
\begin{proof}
  \ref{item:f7c4e2cb20f91ef6}. The surjectivity of each \( p_{n} \) follows from
  the Hahn-Banach theorem. That \( p_{n} \) is a morphism of
  \( \varepsilon \)-Hopf(-\( \ast \)) algebras follows directly from our
  construction of the strong dual.

  \ref{item:bebbc2c97c2c01ed} follows from \ref{item:f7c4e2cb20f91ef6}.

  \ref{item:f4d4617dcab9c51e}. Each \( H_{n} \) is a Montel space
  (Proposition~\ref{prop:f4cbc15f463be1bc}), so is their strict inductive limit
  \( H \) (Proposition~\ref{prop:a017fe989c7853e0}). Thus for these spaces, as
  well as their duals, are all in \( (\mathcal{M}) \)
  (Proposition~\ref{prop:a017fe989c7853e0}), hence are all Mackey spaces
  (Proposition~\ref{prop:a944f1bb11e881f1}) as they are in particular all
  barreled. Moreover, \( H \) being Montel implies it is reflexive. Now apply
  Proposition~\ref{prop:bd7af82b294982b6}, we see that the strong dual
  \( H'_{b} \) can be identified canonically as the projective limit of
  \( \bigl((H_{n})'_{b}, p_{n}\bigr) \). All nuclear spaces have the
  approximation property (Proposition~\ref{prop:0bc5d06c0f42ff63}).

  Now from the discussion in the beginning of this section and
  Proposition~\ref{prop:8ffb913b5f5ddea6}, for each \( n \), we have
  \begin{displaymath}
    H_{n} \overline{\otimes}_{\iota} H_{n}
    = H_{n} \overline{\otimes}_{\pi} H_{n}
    = H_{n} \overline{\otimes}_{\varepsilon} H_{n}
  \end{displaymath}
  remains in the class \( (\mathcal{FN}) \)
  (Corollary~\ref{coro:97adf837ae63af02} and
  Proposition~\ref{prop:a8139d3432a18d9e}). It follows from the construction in
  \S~\ref{sec:32608de27100a358} that \( H \overline{\otimes}_{\iota} H \) is
  identified canonically with the strict inductive system
  \( \bigl(H_{n} \overline{\otimes}_{\iota} H_{n}, u_{n}
  \overline{\otimes}_{\iota} u_{n}\bigr) \), whose strong dual is the projective
  limit of the dual projective system
  \( \bigl((H_{n})'_{b} \overline{\otimes}_{\varepsilon}(H_{n})'_{b}, p_{n}
  \overline{\otimes}_{\varepsilon} p_{n}\bigr) \). Repeating the argument in the
  previous paragraph, we see that the strong dual of
  \( H \overline{\otimes}_{\iota} H \) is canonically identified with
  \( H'_{b} \overline{\otimes}_{\varepsilon} H'_{b} \).

  Similarly, continuing the argument, one establishes an analogous duality the
  three-fold and four-fold tensor products, which finishes the proof of
  \ref{item:f4d4617dcab9c51e}.
\end{proof}

We will soon see in \S~\ref{sec:e3160c56b87e4b14} that,
Theorem~\ref{theo:7f86d4ede3701552} already suffices to establish some
interesting locally convex Hopf algebras.

\section{Examples of projective and inductive limits}
\label{sec:e3160c56b87e4b14}

\subsection{A variant of Bruhat's regular functions on locally compact groups}
\label{sec:de1592686b0dd205}

We have seen in \S~\ref{sec:ce6223de3fc1fab6} that every locally compact group
\( G \) induces an \( \varepsilon \)-Hopf(-\( \ast \)) algebra structure on
\( C(G) \); in Theorem~\ref{theo:61825be42ff88744} that we can recover \( G \)
from the \( \varepsilon \)-Hopf \( C(G) \) algebra; and in
Theorem~\ref{theo:bab58124687e3b85} that we can even recover the Pontryagin dual
group \( \widehat{G} \) from the \( \varepsilon \)-Hopf-\( \ast \) algebra
\( C(G) \) when \( G \) is abelian. It is also shown in
Theorem~\ref{theo:fc44c913131be011} that the \( \varepsilon \)-Hopf algebra
\( C(G) \) is \( (\varepsilon, \pi) \)-polar reflexive, when \( G \) is
\( \sigma \)-compact, or in particular, if \( G \) is second countable.

For second countable locally compact groups, it might be desirable to find a
locally convex Hopf algebra counterpart that is strongly reflexive instead of
polar reflexive. We now describe a way to do this, using a variant of Bruhat's
notion regular functions on an arbitrary locally compact groups
(\cite{MR0140941}*{pp48--50}), which in turn heavily relies on some structural
results of locally compact groups that are discovered during a series of
spectacular papers \cites{MR0039730,MR0049203}, \cite{MR0049204},
\cites{MR0054613,MR0058607}, aiming to resolve Hilbert's fifth problem of
characterizing Lie groups among locally compact groups. See also the books
\cite{MR0073104} and \cite{MR3237440} for a systematic treatment.  An
affirmative answer for one version of the formulation of Hilbert's fifth problem
goes as follows: we say a topological group \( G \) \textbf{has no small
  subgroups}, if there exists a neighborhood \( U \) of \( e \), such that there
exists no nontrivial subgroup \( H \) of \( G \) such that \( H \subseteq U \).
Yamabe \cite{MR0058607}*{p364, Theorem~3} showed that a locally compact group
that has no small subgroups is a Lie group, and it is elementary to use
exponential maps to see that any Lie group has no small subgroups.

For the rest of \S~\ref{sec:de1592686b0dd205}, we fix a second countable locally
compact group \( G \). Following Bruhat \cite{MR0140941}, we say a compact
subgroup \( K \) of \( G \) is \textbf{good}, if the quotient group \( G / K \)
is a Lie group, which is necessarily second countable since \( G \) is so (by
\cite{MR0722297}*{p110, 3.41}, there's also no ambiguity of the Lie group
structure on \( G/K \)).

We shall also need the following weaker version of \cite{MR0058607}*{p364,
  Theorem~5'}.

\begin{theo}[\cite{MR0073104}*{p175, Theorem}]
  \label{theo:8f669c476b6c7e6f}
  Let \( G \) be a second countable locally compact group. Then every
  neighborhood of the identity \( e \) contains a good subgroup of \( G \).
\end{theo}

The following notion is due to Bruhat~\cite{MR0140941}.

\begin{defi}
  \label{defi:ea78b2fb50d6a4a9}
  A function \( f \) on \( G \) is said to be a \textbf{regular function with
    compact support}, if there exists a good subgroup \( K \) of \( G \), and a
  smooth function \( \widetilde{f} \) on \( G / K \) with compact support such
  that \( f = \widetilde{f} \pi \), where \( \pi : G \to G / K \) is the
  canonical projection. The space of all regular functions with compact support
  on \( G \) is denoted by \( \mathcal{D}(G) \).
\end{defi}

We now introduce the functions with which we are going to work.
\begin{defi}
  \label{defi:b3c85e6746f05253}
  A function \( f \) on \( G \) is called \textbf{liftably smooth}, if there
  exists a good subgroup \( K \) of \( G \), such that
  \( f = \widetilde{f} \pi \) for some \( \widetilde{f} \in C^{\infty}(G/K) \),
  where \( \pi : G \to G / K \) is the canonical projection. The space of all
  liftably smooth functions on \( G \) is denoted by \( \mathcal{E}_{l}(G) \).
\end{defi}

As vector spaces, it is clear that \( \mathcal{D}(G) \) is a subspace of
\( \mathcal{E}_{l}(G) \).

\begin{rema}
  \label{rema:3675ab0e7d8e9fbd}
  We point out that our definition of a liftably smooth function is closely
  related to, but distinct from Bruhat's notion of a regular function
  (\cite{MR0140941}*{\S~p44}, and note that we assume \( G \) is second
  countable, so we don't need the general case in \cite{MR0140941}*{\S~p44}),
  the latter is defined as a function that can be locally represented by some
  regular function with finite support. Hence functions in
  \( \mathcal{E}_{l}(G) \) is always smooth in Bruhat's sense, but not vice
  versa.
\end{rema}

We now introduce a locally convex space structure on \( \mathcal{E}_{l}(G) \).
Take a countable \emph{decreasing} fundamental system \( (U_{n})_{n \geq 1} \)
of neighborhoods of \( e \in G \), and for each \( n \), choose a good subgroup
\( K_{n} \) of \( G \) with \( K_{n} \subseteq U \). It is shown in
\cite{MR0140941}*{p44} that the intersection of two good subgroups remains a
good subgroup, so without loss of generality, we may also assume the sequence
\( (K_{n})_{n \geq 1} \) is also decreasing. For each \( n \), let
\( \pi_{n} : G \to G / K_{n} \) be the canonical projection, consider the
\( \varepsilon \)-Hopf(-\( \ast \)) algebra \( C^{\infty}(G/K_{n}) \) as in
Theorem~\ref{theo:6b0cb8e408044129} (see also
Remark~\ref{rema:3faab4832d082f96}). Since \( G \) is second countable, the
locally convex space \( C^{\infty}(G/K_{n}) \) is of class \( (\mathcal{FN}) \)
(Lemma~\ref{lemm:975e581967ecec25}). By Proposition~\ref{prop:89675519c2bb97f8},
we see that
\begin{equation}
  \label{eq:49bd95c1d3c1c4b7}
  C^{\infty}(G/K_{n}) \otimes_{\iota} C^{\infty}(G/K_{n})
  = C^{\infty}(G/K_{n}) \otimes_{\pi} C^{\infty}(G/K_{n})
  = C^{\infty}(G/K_{n}) \otimes_{\varepsilon} C^{\infty}(G/K_{n})
\end{equation}
and the compatible topology on
\( C^{\infty}(G / K_{n}) \odot C^{\infty}(G / K_{n}) \) is unique. The similar
assertion also holds three-fold and four-fold tensor products by the same
reasoning, therefore, \( C^{\infty}(G/K_{n}) \) is also an \( \iota
\)-Hopf(-\( \ast \)) algebra.

\begin{lemm}
  \label{lemm:f306de8279a2c545}
  If \( K \), \( H \) are good subgroups of \( G \) with \( K \subseteq H \),
  and \( \pi: G/K \to G/H \) the canonical projection \( xK \mapsto xH \), then
  \( \pi^{\ast}: C^{\infty}(G/H) \to C^{\infty}(G/K) \) is a strict monomorphism
  of locally convex spaces with closed image.
\end{lemm}
\begin{proof}
  It is clear that \( \pi \) maps every compact sets in \( G / K \) to compact
  sets in \( G / H \). We claim for every compact \( D \subseteq G / H \), there
  exists a compact \( C \subseteq G / K \) such that \( D = \pi(C) \). If the
  claim holds, then it follows easily from the definition of locally convex
  topologies on \( C^{\infty}(G/H) \) and \( C^{\infty}(G/K) \) as described in
  \S~\ref{sec:374fdcc9344806b0}, that \( \pi^{\ast} \) is a strict
  monomorphism. The closeness of the image of \( \pi^{\ast} \) follows from
  completeness.

  We now prove the claim. For each \( x \in \pi^{-1}(D) \), consider a
  \emph{compact} neighborhood \( C_{x} \) of \( x \). Since \( \pi \) is an open
  map, \( \pi(\operatorname{Int}(C_{x})) \) is an open neighborhood of
  \( \pi(x) \), where \( \operatorname{Int}(\cdot) \) denotes the interior. By
  compactness of \( D \) (and surjectivity of \( \pi \)), there exists finitely
  many \( x_{1}, \ldots, x_{n} \in \pi^{-1}(D) \), such that
  \( \cup_{k=1}^{n}\bigl(\operatorname{Int}(C_{x_{k}})\bigr) \supseteq D
  \). Note that \( \pi^{-1}(D) \) is closed, it suffices to take
  \( C = \left(\cup_{k=1}^{n}C_{x_{k}}\right) \cap \pi^{-1}(D) \), and the proof is complete.
\end{proof}

By Lemma~\ref{lemm:f306de8279a2c545}, we may canonically identified
\( C^{\infty}(G/K_{n}) \) as a closed subspace of \( C^{\infty}(G/K_{n+1}) \),
which are in turn identified canonically as subspaces of
\( \mathcal{E}_{l}(G) \) via \( \pi_{n}^{\ast} \) and \( \pi_{n+1}^{\ast} \)
respectively. Denoting by \( E_{n} \) the subspace of
\( \pi_{n}^{\ast}\bigl(C^{\infty}(G / K_{n})\bigr) \) in
\( \mathcal{E}_{l}(G) \), we see that, by definition,
\( \mathcal{E}_{l}(G) = \cup_{n=1}^{\infty}E_{n} \). We transport the topology
of \( C^{\infty}(G / K_{n}) \) onto \( E_{n} \), and equip
\( \mathcal{E}_{l}(G) \) with the inductive locally convex topology with respect
to the inclusions \( E_{n} \hookrightarrow \mathcal{E}_{l}(G) \). It follows
from Lemma~\ref{lemm:f306de8279a2c545} that \( \mathcal{E}_{l}(G) \) is the
strict inductive limit of the strict inductive system
\( (E_{n}, u_{n})_{n \geq 1} \), where each
\( u_{n} : E_{n} \hookrightarrow E_{n+1} \) is the inclusion. By
Corollary~\ref{coro:7ae986967f0d1033}, we see that \( \mathcal{E}_{l}(G) \) is
complete.

Now we show that the topology on \( \mathcal{E}_{l}(G) \) defined as above
doesn't depend on the choice of \( (U_{n}) \), nor the choice of \( (K_{n}) \).
Indeed, it suffices to show that for any good subgroup \( K \) of \( G \), we
have \( K \supseteq K_{n} \) for \( n \) large enough (as this means
\( C^{\infty}(G / K) \) can be seen as a closed subspace of
\( C^{\infty}(K_{n}) \), and the closed subspaces \( (E_{n}) \) of
\( \mathcal{E}_{l}(G) \) shall be cofinal for any possible choice as above,
hence the topology \( \mathcal{E}_{l}(G) \) is uniquely determined via the above
procedure and does not depend on the choices). To see that
\( K \supseteq K_{n} \) for some \( n \), it suffices to note that
\( \pi(K_{n}) \) is a compact, hence closed subgroup of the Lie group \( G/K \)
for each \( n \), where \( \pi : G \to G/K \) is the canonical projection. It is
well-known that closed subgroups of a Lie group is still a Lie group (see, e.g.\
\cite{MR0722297}*p110, 3.42), we see that
\( K_{n}/(K \cap K_{n}) \simeq \pi(K_{n}) = (K K_{n})/ K \) is a Lie subgroup of
\( G/K \). Moreover, as \( \pi(U_{n}) \) is a fundamental system of
neighborhoods of the identity in \( G/K \) and
\( \pi(K_{n}) \subseteq \pi(U_{n}) \), since \( G/K \) has no small subgroups,
we must have \( \pi(K_{n}) \) is the trivial subgroup in \( G/K \) as long as
\( n \) is large enough, which in turn forces \( K_{n} \subseteq KK_{n} = K \).

Since \( E_{n} \) is a copy of \( C^{\infty}(G/K_{n}) \), we may also transport
the \( \varepsilon \)-Hopf(-\( \ast \)) algebra structure of
\( C^{\infty}(G/K_{n}) \) on \( E_{n} \). It follows from
Theorem~\ref{theo:6b0cb8e408044129} that \( E_{n} \) is
\( (\iota, \varepsilon) \)-reflexive, and the transpose of the inclusion
\( u_{n}: E_{n} \hookrightarrow E_{n+1} \) yields a surjective (by Hahn-Banach)
continuous linear map \( p_{n} : (E_{n+1})'_{b} \to (E_{n})'_{b} \). It is
easily seen that both \( u_{n} \) and \( p_{n} \) are morphisms of locally
convex Hopf algebras. We may, in particular, equip \( \mathcal{E}_{l}(G) \) with
an \( \iota \)-Hopf(-\( \ast \)) algebra structure by letting it be the strict
inductive limit of the strict inductive system \( (E_{n}, u_{n})_{n \geq 1} \)
as in \S~\ref{sec:32608de27100a358}.

\begin{theo}
  \label{theo:52f7e77af5cde009}
  Using the above notation, the \( \iota \)-Hopf algebra
  \( \mathcal{E}_{l}(G) \) is \( (\iota, \varepsilon) \)-reflexive, with its
  strong dual being the projective limit of the reduced projective system
  \( \bigl((E_{n})'_{b}, p_{n}\bigr)_{n \geq 1} \).
\end{theo}
\begin{proof}
  This is clear from the above discussion and
  Theorem~\ref{theo:7f86d4ede3701552}.
\end{proof}

It is quite interesting that we may still recover \( G \) as a topological group
(hence as a Lie group since the Lie group structure is unique once the
topological group is given, cf.\ \cite{MR0722297}*{p110, 3.41}).

\begin{theo}
  \label{theo:8ecf1ac899a864bb}
  Using the above notation, the map
  \( \delta: G \to \chi_{c}\bigl(\mathcal{E}_{l}(G)\bigr) \),
  \( x \mapsto \delta_{x} \) is an isomorphism of topological groups.
\end{theo}
\begin{proof}
  By Theorem~\ref{theo:6c2d8092e7020245}, it suffices to show that liftably
  smooth functions on \( G \) with values in \( \interval{0}{1} \) separate
  compact sets and closed sets. But this follows directly from Bruhat's version
  of partitions of unity \cite{MR0140941}*{p76, Proposition~2} and that
  \( \mathcal{D}(G) \subseteq \mathcal{E}_{l}(G) \).
\end{proof}

\begin{rema}
  \label{rema:6d232c9f122e1506}
  Since continuous group morphisms between Lie groups are automatically smooth
  \cite{MR0722297}*{p109, 3.39}, it is easy to see that
  \( \mathsf{Grp}^{\inv}\bigl(\mathcal{E}_{l}(G)\bigr) = \widehat{G} \) as
  groups, following the same argument as in
  Theorem~\ref{theo:bab58124687e3b85}. However, at the time of this writing, the
  author is unsure whether the subspace topology on
  \( \mathsf{Grp}^{\inv}\bigl(\mathcal{E}_{l}(G)\bigr) \) induced by the
  topology on \( \mathcal{E}_{l}(G) \) is exactly the topology on the Pontryagin
  dual \( \widehat{G} \).
\end{rema}
 
\subsection{Certain projective limits with non locally compact character groups}
\label{sec:3a2044e319b01429}

As a motivation for our theory of the strict inductive limits of separable
compact quantum groups, as well as another illustration of the general
constructions in \S~\ref{sec:7a5cda2c4078b89b}, we now consider some projective
limits coming from compact Lie groups.

Consider a strictly increasing sequence of \emph{compact} Lie groups
\( (G_{n})_{n \geq 1} \), connected by Lie group morphisms
\( u_{n}: G_{n} \to G_{n+1} \), \( n \geq 1 \) that are all embeddings. For
example, we may take \( G_{n} = O(n) \) (resp.\ \( SU(n) \) or \( U(n) \)), and
\( u_{n} \) embeds \( O(n) \) to the top left corner of \( O(n+1) \) (resp.\
\( SU(n) \) or \( U(n) \)).

We've seen that \( C^{\infty}(G_{n}) \) is an \( \varepsilon
\)-Hopf(-\( \ast \)) algebra of class \( (\mathcal{FN}) \). It is also clear
that the pull-back
\( p_{n}:= u_{n}^{\ast}: C^{\infty}(G_{n+1}) \to C^{\infty}(G_{n}) \) is a
morphism of \( \varepsilon \)-Hopf(-\( \ast \)) algebras. It is well-known that
by partitions of unity, smooth functions on a closed set (which means it extends
to a smooth function on an open set containing this closed set) in a smooth
manifold can be extended to a smooth function of the whole manifold. Thus
\( p_{n} \) is surjective as a map. We thus have a reduced projective system
\( \bigl(C^{\infty}(G_{n}), p_{n}\bigr) \) of \( \varepsilon
\)-Hopf(-\( \ast \)) algebras.  Now let \( \mathcal{H} \) be the projective
limit of this projective system (\S~\ref{sec:fd4b65d5b8ba0635}). The underlying
locally convex space of \( \mathcal{H} \) is
\( \varprojlim C^{\infty}(G_{n}) \), which is still of class
\( (\mathcal{FN}) \) by Proposition~\ref{prop:438d10ce1f2149df} and
Proposition~\ref{prop:a8139d3432a18d9e}. Thus it follows from
Theorem~\ref{theo:5a4f6899a7115559} that \( \chi_{c}(\mathcal{H}) \) (or
\( \chi^{\inv}_{c}(\mathcal{H}) \) in the involutive case) is a topological
group. To facilitate discussion, let
\( q_{n} : \mathcal{H} \to C^{\infty}(G_{n}) \) be the canonical projection for each \( n \).

\begin{prop}
  \label{prop:400deb0a75cc0c86}
  Using the above notation, we have
  \begin{enumerate}
  \item \label{item:5953e04bf2ed6c66} the map
    \( \delta_{n}: G_{n} \to \chi_{c}\bigl(C^{\infty}(G_{n})\bigr) \),
    \( x \mapsto \delta_{x} \) is an isomorphism of topological groups;
  \item \label{item:e464c86c07c7f0eb} the map
    \( v_{n}: \chi_{c}\bigl(C^{\infty}(G_{n})\bigr) \to
    \chi_{c}\bigl(C^{\infty}(G)\bigr) \), \( \chi \mapsto \chi q_{n} \) is an injective
    morphism of topological groups;
  \item \label{item:5348aa4be984958e} the union of all the images of \( v_{n} \)
    are strictly increasing, and their union is exactly
    \( \chi\bigl(C^{\infty}(G)\bigr) \);
  \item \label{item:d03e0133297b925c} if \( u_{n}(G_{n}) \) is nowhere dense in
    \( G_{n+1} \) for each \( n \), then the topological space \( \chi_{c}(G) \)
    is not a Baire space; in particular, the topological group
    \( \chi_{c}(\mathcal{H}) \) is not locally compact.
  \end{enumerate}
  If \( \mathbb{K} = \C \), and treating all the above \( \varepsilon \)-Hopf
  algebras as \( \varepsilon \)-Hopf-\( \ast \) algebras and replaces
  \( \chi_{c}(\cdot) \) (resp.\ \( \chi(\cdot) \)) with
  \( \chi^{\inv}_{c}(\cdot) \) (resp.\ \( \chi^{\inv}(\cdot) \)), then all of
  the above still hold.
\end{prop}
\begin{proof}
  We only treat the non-involutive case, as the involutive case follows from the
  same argument. \ref{item:5953e04bf2ed6c66} follows
  Theorem~\ref{theo:61825be42ff88744}, while \ref{item:e464c86c07c7f0eb} follows
  by noting that \( q_{n}: \mathcal{H} \to C^{\infty}(G_{n}) \) is a morphism of
  \( \varepsilon \)-Hopf algebras and maps precompact sets to precompact sets by
  continuity. The first half of \ref{item:5348aa4be984958e} is clear since the
  sequence \( (G_{n}) \) is strictly increasing. To see the latter half, it
  suffices to note that by Proposition~\ref{prop:bd7af82b294982b6} (we've seen
  several times that all of the relevant spaces are barrelled, hence Mackey),
  every character of \( \mathcal{H} \) factors through a character of some
  \( C^{\infty}(G_{n}) \). \ref{item:d03e0133297b925c}, the image of each
  \( v_{n} \), which we denote by \( C_{n} \), is a compact, hence closed subset
  of \( \chi_{c}(\mathcal{H}) \), and they are strictly increasing with union
  being the whole space \( \chi_{c}(\mathcal{H}) \). By compactness,
  \( v_{n}\delta_{n}: G_{n} \to C_{n} \) is a homeomorphism where the latter is
  equipped with the subspace topology induced by \( \chi_{c}(\mathcal{H}) \).
  The hypothesis of nowhere dense implies that each \( C_{n} \) has an empty
  interior, thus being the countable union \( \cup_{n \geq 1}C_{n} \), the
  topological space \( \chi_{c}(\mathcal{H}) \) can not be a Baire space.
\end{proof}

\begin{rema}
  \label{rema:40bccc5e25fe53cd}
  Note that our aforementioned examples of \( G_{n} = O(n) \),
  \( G_{n} = SU(n) \) and \( G_{n} = U(n) \) all satisfies the hypothesis in
  Proposition~\ref{prop:400deb0a75cc0c86}~\ref{item:d03e0133297b925c}, hence the
  resulting \( \mathcal{H} \) has a non locally compact character group.
\end{rema}

\subsection{Strict inductive limits of separable compact quantum groups}
\label{sec:3dcbf7486388ada5}

We say a compact quantum group
\( \mathbb{G} = \bigl(C(\mathbb{G}), \Delta\bigr) \) is \textbf{separable} if
the \( C^{\ast} \)\nobreakdash-algebra \( C(\mathbb{G}) \) is separable. From
the Peter-Weyl theory and the orthogonality relations for irreducible unitary
representations of \( \mathbb{G} \), one checks easily that \( \mathbb{G} \) is
separable if and only if \( \irr(\mathbb{G}) \) is countable, which in turn, is
equivalent to \( \pol(\mathbb{G}) \) being of countable dimension. From now on,
we will treat compact quantum groups as locally convex Hopf-\( \ast \) algebras
as in \S~\ref{sec:4d1fce794188b9cb}. In particular, we equip
\( \pol(\mathbb{G}) \) with the finest locally convex topology. It follows
Proposition~\ref{prop:12e473f3f887b8ab}, Proposition~\ref{prop:eb0c853264f8729c}
and Proposition~\ref{prop:a8139d3432a18d9e} that \( \mathbb{G} \) is separable
if and only if \( \pol(\mathbb{G}) \) is of class \( \mathcal{N} \).

Now consider an increasing sequence of compact quantum groups
\( (\mathbb{G}_{n})_{n \geq 1} \), by which we mean that we may view
\( \mathbb{G}_{n} \) as a compact quantum subgroup of \( \mathbb{G}_{n+1} \) in
the sense that there exists a \emph{surjective} homomorphism
\( p_{n}: \pol(\mathbb{G}_{n+1}) \to \pol(\mathbb{G}_{n}) \) of
(\( \varepsilon \)-)Hopf-\( \ast \) algebras.  We thus have a reduced projective
system \( \bigl(\pol(\mathbb{G}_{n}), p_{n}\bigr) \) of
\( \varepsilon \)-Hopf-\( \ast \) algebras. We've seen in
Theorem~\ref{theo:493ddf6e91b637cd} (and \S~\ref{sec:4d1fce794188b9cb}) that for
each \( n \), the dual \( \mathcal{H}_{\widehat{\mathbb{G}_{n}}} \) is an
\( \pi \)-Hopf algebra of type \( (\mathcal{FN}) \), hence can also be seen as
an \( \iota \)-Hopf algebra (Proposition~\ref{prop:a8139d3432a18d9e},
Corollary~\ref{coro:0be09f15f7e01dce} and
Proposition~\ref{prop:89675519c2bb97f8}). Taking the transpose of each
\( p_{n} \), we obtain a strict inductive system
\( \bigl(\mathcal{H}_{\widehat{\mathbb{G}_{n}}}, u_{n}\bigr) \) of
\( \iota \)-Hopf algebras. Hence we may take the projective limit of the reduced
projective system \( \bigl(\pol(\mathbb{G}_{n}), p_{n}\bigr) \), which we denote
by \( \mathcal{H}_{\mathbb{G}_{\infty}} \), and is still an
\( \varepsilon \)-Hopf-\( \ast \) algebra. We may also take the inductive limit
of the strict inductive system
\( \bigl(\mathcal{H}_{\widehat{\mathbb{G}_{\infty}}}, u_{n}\bigr) \), which we
denote by \( \mathcal{H}_{\widehat{G}_{\infty}} \), and is an
\( \iota \)-Hopf-\( \ast \) algebra.

By Theorem~\ref{theo:7f86d4ede3701552}, we see that
\( \mathcal{H}_{\widehat{\mathbb{G}_{\infty}}} \) and
\( \mathcal{H}_{\mathbb{G}_{\infty}} \) are the strong duals of each other as
locally convex Hopf-\( \ast \) algebras.

\begin{defi}
  \label{defi:b8cad9b442ca143d}
  We call the \( \varepsilon \)-Hopf-\( \ast \) algebra
  \( \mathcal{H}_{\mathbb{G}} \) the \textbf{strict inductive limit} of the
  increasing sequence \( (\mathbb{G}_{n}) \) of separable compact quantum
  groups.
\end{defi}

\subsection{A notion of \texorpdfstring{\( S^{+}_{\infty} \)}{S+\_infty},
  \texorpdfstring{\( O^{+}_{\infty} \)}{O+\_infty},
  \texorpdfstring{\( U^{+}_{\infty} \)}{U+\_infty} and their duals}
\label{sec:4c8d08ba8bcefe46}

Since their introduction by the work of S.\ Wang \cites{MR1637425,MR1316765}, the
quantum permutation group \( S^{+}_{N} \), the free orthogonal group
\( O^{+}_{N} \) and the free unitary group \( U^{+}_{N} \) received a lot
attention, and still play an important role in current research on
compact/discrete quantum groups. It is natural to ask, however vaguely, the
natural question of what should be the right notion of \( S^{+}_{\infty} \),
\( O^{+}_{\infty} \) and \( U^{+}_{\infty} \), as the respective inductive limit
as \( N \to \infty \). In the following, we follow in
\cite{MR3204665}*{\S~1.1} for the basics of \( S^{+}_{N} \),
\( O^{+}_{N} \) and \( U^{+}_{N} \).

First, for the infinite quantum permutation group \( S^{+}_{\infty} \). We first
define the quantum permutation group \( S^{+}_{N} \) on \( N \) symbols. The
universal version of \( S^{+}_{N} \) is described as follows (see
\cite{MR3204665}*{p6, Example~1.1.8}): let \( A_{s}(N) \) be the universal
\( C^{\ast} \)\nobreakdash-algebra generated by the family of symbols
\( U = (u^{N}_{i,j})_{1 \leq i, j \leq n} \), subject to the relations given by
requiring \( U \) to be a so-called ``magic unitary''. More precisely, this
means that each \( u^{N}_{i,j} \) should be a projection, and the sum of each
column and row of \( U \) is \( 1 \). The formula
\( \Delta_{N}(u^{N}_{i,j}) = \sum_{k=1}^{N}u^{N}_{i,k} \otimes u^{N}_{k,j} \)
for all \( 1 \leq i, j \leq n \) extends by the universal property to a unique
comultiplication on \( A_{s}(N) \), and it can be shown that
\( \bigl(A_{s}(N), \Delta_{N}\bigr) \) is a separable compact quantum group,
which we denote by \( S^{+}_{N} \). Now \( \mathcal{H}_{N}:= \pol(S^{+}_{N}) \)
is an \( \varepsilon \)-Hopf-\( \ast \) algebra, and is generated by
\( u^{N}_{i,j} \) as an algebra. We may obtain a surjective morphism
\( \pi_{N}: \mathcal{H}_{N+1} \to \mathcal{H}_{N} \) by requiring
\( \pi_{N}(u^{(N+1)}_{i,j}) = u^{N}_{i,j} \), if \( 1 \leq i, j \leq N \), and
\( \pi_{N}(u^{(N+1)}_{i,j}) = 0 \) otherwise. We call the projective limit
\( \varinjlim \mathcal{H}_{n} \) of the projective system
\( (\mathcal{H}_{n}, \pi_{n})_{n \geq 1} \) the \( \varepsilon \)-Hopf algebra
corresponding to \( S^{+}_{\infty} \), and its strong dual (see
\S~\ref{sec:3dcbf7486388ada5}) the \( \iota \)-Hopf algebra corresponding to the
dual of \( S^{+}_{\infty} \). Classically, this construction correspond to the
permutation group \( S_{\infty} \) of all permutations on \( \N_{+} \) that fix
all but finitely many elements, as the inductive limit of \( S_{N} \),
\( N \geq 1 \), with \( S_{N} \) embeds into \( S_{N+1} \) by letting all
elements in \( S_{N} \) fix any \( n \geq N \).

Next, we treat \( O^{+}_{\infty} \). For each \( n \in \N_{+} \), we describe a
separable compact quantum group \( O^{+}_{n} \) as follows. Consider the
universal \( C^{\ast} \)\nobreakdash-algebra \( A_{o}(n) \) generated by the
symbols \( U = (u^{(n)}_{i,j})_{1 \leq i, j \leq n} \), subject to the relations
given by requiring \( U \) to be unitary and each \( u^{(n)}_{i,j} \) being
self-adjoint (this corresponds to letting \( F \) to be the identity matrix in
\cite{MR3204665}*{p5, Example~1.1.7}). The formula
\( \Delta_{n}(u^{(n)}_{i,j}) = \sum_{k=1}^{n}u^{(n)}_{i,k} \otimes u^{(n)}_{k,j}
\) determines a unique comultiplication on \( A_{o}(n) \) the universal
property. It can be shown that \( \bigl(A_{o}(n), \Delta_{n}\bigr) \) is a
compact quantum group, and we denote it by \( O^{+}_{n} \). The
\( \varepsilon \)-Hopf algebra \( \pol(O^{+}_{n}) \) is generated by
\( u^{(n)}_{i,j} \), \( 1 \leq i, j \leq n \) as an algebra. And we have a
unique surjective morphism of \( \varepsilon \)-Hopf-\( \ast \) algebras
\( \pi_{n}: \pol(O^{+}_{n+1}) \to \pol(O^{+}_{n}) \) such that
\( \pi_{n}(u^{(n+1)}_{i,j}) = u^{(n)}_{i,j} \), if \( 1 \leq i, j \leq n \); and
\( \pi_{n}(u^{(n+1)}_{i,j}) = 0 \) otherwise. Now the projective limit of the
reduced projective system \( \bigl(\pol(O^{+}_{n}, \pi_{n})\bigr) \) of
\( \varepsilon \)-Hopf-\( \ast \) algebra is an \( \varepsilon
\)-Hopf-\( \ast \) algebra, which we shall denote by
\( \mathcal{H}_{O^{+}_{\infty}} \), and its strong dual, as an
\( \iota \)-Hopf-\( \ast \) algebra, is denoted by
\( \mathcal{H}_{\widehat{O}^{+}_{\infty}} \), both corresponding to
\( O^{+}_{\infty} \) (the former is the analogue of the function algebra, while
the latter the convolution algebra). Classically, \( O^{+}_{\infty} \)
corresponds to \( O(\infty) \)---the union of all \( O(n) \) as we've considered
in Remark~\ref{rema:40bccc5e25fe53cd}, which is shown there is not locally
compact. It is by this comparison that \( O^{+}_{\infty} \) seems lies outside
the framework of locally compact quantum groups.

By letting the matrix \( F \) to be identity in \cite{MR3204665}*{p6,
  Example~1.1.6}, and proceeds as in the previous paragraph, we may obtain an
\( \varepsilon \)-Hopf-\( \ast \) algebra \( \mathcal{H}_{U^{+}_{\infty}} \), as
well as its strong dual \( \mathcal{H}_{\widehat{U^{+}_{\infty}}} \), which is
an \( \iota \)-Hopf-\( \ast \) algebra, both related to \( U^{+}_{\infty} \). By
a similar line of thought as in the case of \( O^{+}_{\infty} \), again it seems
that \( U^{+}_{\infty} \) lies outside the framework of locally compact quantum
groups; and classically, \( U^{+}_{\infty} \) corresponds to \( U(\infty) \),
which is the union of all \( U(n) \), \( n \geq 1 \).

\begin{rema}
  \label{rema:c93d57c82ccd7a41}
  In \cite{MR1382726}, other versions of universal quantum groups are
  considered, which are more general than \( O^{+}(n) \) and \( U^{+}(n) \). By
  suitably choosing the invertible matrices
  \( F_{n} \in \operatorname{Mat}_{n \times n}(\C) \) for different \( n \), in
  principle, we may consider certain strict inductive limit of separable compact
  quantum groups of the form \( \bigl(A_{o}(F_{n})\bigr) \), and of the form
  \( \bigl(A_{u}(F_{n})\bigr) \) etc. (see the notation in
  \cite{MR3204665}*{p6}).
\end{rema}

\begin{rema}
  \label{rema:7673e61fbb4d9b73}
  A version of the object \( S^{+}_{\infty} \) is already considered in
  \cite{MR3059661} using a much involved method. We also point out that
  recently, Voigt \cite{MR4543448} has constructed the full quantum permutation
  group on an infinite set, as a discrete quantum group, so still remains in the
  framework of locally compact quantum groups. Our construction of
  \( O^{+}_{\infty} \) and \( U^{+}_{\infty} \) seems new to the best of the
  author's knowledge.
\end{rema}

\section*{Acknowledgement}
\label{sec:f2fa53f2ae6541ff}

This work is supported by the Polish National Science Center NCN grant No.\
2020/39/I/ST1/01566. The author would also like to express his sincere gratitude
to Prof.\ Adam Skalski for his encouragement and stimulating discussions. The
author would also like to thank Prof.\ Paweł Kasprzak, for his question of why
not considering \( C(G) \) of all continuous functions on a Lie group \( G \)
after hearing a talk of the author on some part of this paper, which eventually
lead to the constructions in \S~\ref{sec:ce6223de3fc1fab6}.


% \addcontentsline{toc}{section}{References}
% \bibliography{math,mref}

\end{document}